\documentclass[letterpaper,12pt,titlepage,oneside,final]{book}
 
\usepackage{amssymb}
\usepackage{subcaption}
\usepackage[hyphens]{url} 
\usepackage[margin=1in]{geometry} 
\usepackage{amsmath,amsthm,amssymb}
 \usepackage{pifont}
\usepackage{mathrsfs}

\usepackage{caption}

\usepackage{tcolorbox}
\usepackage{tabularx}
\usepackage[all,cmtip]{xy}

\usepackage{float}

\theoremstyle{plain}
 \newtheorem{thm}{Theorem}[section]
 \newtheorem{prop}[thm]{Proposition}
 \newtheorem{lem}[thm]{Lemma}
 \newtheorem{cor}[thm]{Corollary}
\theoremstyle{definition}
\newtheorem{notation}{Notation}[section]
 \newtheorem{exm}{Example}[section]
 \newtheorem{dfn}{Definition}[section]
  \newtheorem{rem}{Remark}[section]
 \newtheorem{nota}{Notation}[section]

 \numberwithin{equation}{section}

\renewcommand{\leq}{\leqslant}
\renewcommand{\geq}{\geqslant}

\newcommand{\href}[1]{#1} 

\usepackage{ifthen}
\newboolean{PrintVersion}
\setboolean{PrintVersion}{false}

\usepackage{amsmath,amssymb,amstext} 
\usepackage{graphicx} 

\usepackage[pagebackref=false]{hyperref} 
\hypersetup{
    plainpages=false,       
    unicode=false,          
    pdftoolbar=true,        
    pdfmenubar=true,        
    pdffitwindow=false,     
    pdfstartview={FitH},    
    pdfnewwindow=true,      
    colorlinks=true,        
    linkcolor=blue,         
    citecolor=green,        
    filecolor=magenta,      
    urlcolor=cyan           
}
\ifthenelse{\boolean{PrintVersion}}{   
\hypersetup{	
    citecolor=black,%
    filecolor=black,%
    linkcolor=black,%
    urlcolor=black}
}{} 

\let\origdoublepage\cleardoublepage
\newcommand{\clearemptydoublepage}{%
  \clearpage{\pagestyle{empty}\origdoublepage}}
\let\cleardoublepage\clearemptydoublepage

\begin{document}


\pagestyle{empty}
\pagenumbering{roman}

\begin{titlepage}
        \begin{center}
        \vspace*{1.0cm}

        \Huge
        {\bf On the Enumerative Structures in Quantum Field Theory }

        \vspace*{1.0cm}

        \normalsize
        by \\

        \vspace*{1.0cm}

        \Large
        Ali Mahmoud \\

        \vspace*{3.0cm}

        \normalsize
        A thesis \\
        presented to the University of Waterloo \\ 
        in fulfillment of the \\
        thesis requirement for the degree of \\
        Doctor of Philosophy \\
        in \\
        Mathematics \\

        \vspace*{2.0cm}

        Waterloo, Ontario, Canada, 2020 \\

        \vspace*{1.0cm}

        \copyright\ Ali Mahmoud 2020 \\
        \end{center}
\end{titlepage}

\pagestyle{plain}
\setcounter{page}{2}

\cleardoublepage 

\begin{center}\textbf{Examining Committee Membership}\end{center}
  \noindent
The following served on the Examining Committee for this thesis. The decision of the Examining Committee is by majority vote.
  \bigskip
  
  \noindent
\begin{tabbing}
Internal-External Member: \=  \kill 
External Examiner: \>  Jo Ellis-Monaghan
 \\ 
\> Professor, Dept. of Mathematics and Statistics, \\\>Saint Michael's College  \\
\end{tabbing} 
  \bigskip
  
  \noindent
\begin{tabbing}
Internal-External Member: \=  \kill 
Supervisor(s): \> Karen Yeats\\
\> Professor, Dept. of Combinatorics and Optimization, \\\>University of Waterloo \\
\\
\> Achim Kempf \\
\> Professor, Dept. of Applied Mathematics,\\\> University of Waterloo  \\
\end{tabbing}
  \bigskip
  
  \noindent
  \begin{tabbing}
Internal-External Member: \=  \kill 
Internal Member: \> Jon Yard \\
\> Professor, Dept. of Combinatorics and Optimization, \\\>University of Waterloo \\
\end{tabbing}
  \bigskip
  
  \noindent
   \begin{tabbing}
Internal-External Member: \=  \kill 
Internal Member: \> David Wagner \\
\> Professor, Dept. of Combinatorics and Optimization, \\\>University of Waterloo \\
\end{tabbing}
\bigskip
  
  \noindent
\begin{tabbing}
Internal-External Member: \=  \kill 
Internal-External Member: \> Eduardo Martin-Martinez\\
\> Professor, Dept. of Applied Mathematics,\\\> University of Waterloo \\
\end{tabbing}
  \bigskip
  
  \noindent

\cleardoublepage

\begin{center}
    \Large Author's Declaration
\end{center}
  \noindent
I hereby declare that I am the sole author of this thesis. This is a true copy of the thesis, including any required final revisions, as accepted by my examiners.

  \bigskip
  
  \noindent
I understand that my thesis may be made electronically available to the public.

\cleardoublepage


\begin{center}\textbf{Abstract}\end{center}
This thesis addresses a number of enumerative problems that arise in the context of quantum field theory and in the process of renormalization. In particular, the enumeration of rooted connected chord diagrams is further studied and new applications in quenched QED and Yukawa theories are introduced. Chord diagrams appear in quantum field theory in the context of Dyson-Schwinger equations, where, according to recent results, they are used to express the solutions. In another direction, we study the action of point field diffeomorphisms on a free theory. We give a new proof of a vanishing phenomenon for tree-level amplitudes of the transformed theories.

A functional equation for $2$-connected chord diagrams is discovered, then is used to calculate the full asymptotic expansion of the number of $2$-connected chord diagrams by means of alien derivatives applied to factorially divergent power series. The calculation extends the older result by D. J. Kleitman on counting irreducible diagrams. Namely,  Kleitman's result calculates the first coefficient of the infinite asymptotic expansion derived here and is therefore an approximation of the result presented. In calculating the asymptotics this way  we are following the approach M. Borinsky used for solving the asymptotic counting problem of general connected chord diagrams. The numbers of $2$-connected chord diagrams and the sequence of coefficients of their asymptotic expansion amazingly also appeared, without being recognized, in more physical situations in the work of Broadhurst on 4-loop Dyson-Schwinger-Johnson anatomy, and among the renormalized quantities of quenched QED calculated by M. Borinsky. The underlying chord-diagrammatic structure of quenched QED and Yukawa theory is unveiled here. Realizing this combinatorial structure  makes  the asymptotic analysis of the Green functions of these theories easier, and no longer requires singularity analysis. As a subsidiary outcome, a combinatorial interpretation of sequence \href{https://oeis.org/A088221}{A088221} in terms of chord diagrams is obtained from analysing some of the expressions in the asymptotic expansion of connected chord diagrams in the work of M. Borinsky.

Lastly, a problem from another context is considered. The problem was to reprove the cancellation of interaction terms produced by applying a point field diffeomorphism to a free field theory. The original proof by D. Kreimer and K. Yeats was intricate in a way that did not provide much insight, and, therefore, a proof on the level of generating functions was desired. This is what we do here. We show that the series of  the tree-level amplitudes is exactly the compositional inverse of the diffeomorphism applied to the free theory. The relation to the combinatorial Legendre transform defined by D. Jackson, A. Kempf and A. Morales is outlined although still not settled. Simple combinatorial proofs are also given for some Bell polynomial identities. 
\cleardoublepage


\begin{center}\textbf{Acknowledgements}\end{center}

Thanks are all to God, for showing me the way. Thanks are to God, from whom I sought guidance to understand, learn, build up the bijections and solve the problems. Thanks are to God for giving me the strength to endure hardships and keep working. 

I would like to thank Karen Yeats, would like to sincerely thank you Karen for being so  caring and thoughtful for your students, you do this genuinely in a way that I shall always remember. I would like to thank you for introducing me to this interesting and interdisciplinary topic of which the benefit has been two-folded, namely learning some quantum field theory beside working on enumerative problems. Thanks for the weekly meetings, the group meetings and for the many useful remarks that came along them. Finally, thank you for your great help with my job applications.

I would like to thank Achim Kempf for his support and his ever friendly advice. Whether on campus or  while exercising pull-ups in the CIF, we had plenty of interesting discussions and funny moments.

I would like to thank Ian P. Goulden. Thanks for your continuous support and guidance. I remember how you welcomed my questions in the Symmetric functions course in Spring 2017, and before the comprehensive exams in Spring 2018. I will miss your ``Hello Ali" and your light sense of humour every time you passed in the corridor.

I would like to also give special thanks to David Wagner, from whom I have learnt enumerative combinatorics, and whose brilliant lecturing of CO630 in Winter 2017 was the main reason I chose this research direction.

A great portion of this thesis is a continuation of the work by M. Borinsky. Thank you Michi for your impressive research and for your interesting questions and suggestions. I hope we will be able to work together on similar problems again in the future.

Many thanks to Jo Ellis-Monaghan, Eduardo Martin-Martinez, David Wagner and Jon Yard for reading my thesis.

I would like to also thank Alfred Menezes, Richard Cleve, and (again) Jon Yard. The topics of cryptography and quantum information that I first learnt from you eventually helped me find my next job. Thank you Alfred for your special support. 

I would like to thank everyone in Karen's group and in the Combinatorics and Optimization department. In particular, thanks are to Nick Olson-Harris, Kazuhiro Nomoto, Lukas Nabergall and William Dugan for many useful and unuseful discussions! Thanks are also to the Dirk Kreimer group at Humboldt University in Berlin for their hospitality; special thanks to Paul-Hermann Balduf for insightful comments on the diffeomorphisms problem.

I was lucky to meet many Egyptian friends in Waterloo, with whom the earnest friendship and warmth of Egypt were restored. Thanks are to you Ahmed Wagdy, Ahmed Hamza, Ahmed Ali, Ayman Eltaliawy, Mahmoud Allam, Mortada, Mohamed Ibrahim, Mohamed Arab, and Abdallah Abdelaziz, among many others. 

Special thanks to my old friends in Egypt who always wished me the best: Thank you Mahmoud Esmat, Ahmed Magdy, Abdelaziz Mabrouk and Sayed Mohamed. Thanks are also to all my former students in Egypt. Thank you Mohamed Khaled and Bishoy Girgis.

I must also thank all my professors and instructors at Cairo University who have been always there to teach, help, and guide. Special thanks to professors Fatma Ismail, Nahed Sayed, Alaa-Eldin Hamza,  Soaad Badawy, Hany El Hosseiny, Layla Soeif,, Mohamed Atef Helal, Nefertiti Megahed, Mohamed Asaad, Tarek and Amr Sid Ahmed, Ahmed Ghaleb, Mohamed Adel Hosny, Hossam M. Hassan and Mustafa Ashry. 

Finally, I am totally and ever indebted to my father and mother, without whom I would never have climbed this way up. Thanks are to my sister Amena Mahmoud who has always devoted herself for caring for me and the whole family especially after the loss of our mother in 2018. Thanks are to my brother Mohamed Mahmoud who has always been our spearhead in the academic trail.  Thanks are also to all of my extended family for their constant remembrance and care.

The final thanks are to you Israa, my dear fianc\'ee and future wife in God's blessing. Thank you for supporting me everyday in the past five years, and for enduring being oceans apart for such long times, with your blessed tenderness and patience. May God bless our life together on earth and in heaven with happiness and soundness ever after.  $\sim\;\sim\;\sim$

\cleardoublepage




\renewcommand\contentsname{Table of Contents}
\tableofcontents
\cleardoublepage
\phantomsection    

\addcontentsline{toc}{chapter}{List of Figures}
\listoffigures
\cleardoublepage
\phantomsection		

\addcontentsline{toc}{chapter}{List of Tables}
\listoftables
\cleardoublepage
\phantomsection		


\cleardoublepage
\phantomsection		

\pagenumbering{arabic}


\chapter{Introduction}\setlength{\parindent}{1cm}
The goal of this thesis is to display a number of situations where enumerative combinatorics and algebraic methods are used to reveal more of a basic  structure in quantum field theory problems. This matches with the lines of thought followed for example in \cite{conneskreigeom} to describe the process of renormalization in quantum field theory in the pure terms of Hopf algebras; and in the robust combinatorial definition of the Lgendre transform in \cite{kjm2}; or in \cite{Karenmarkushihn} for expressing solutions of Dyson-Schwinger equations in terms of connected chord diagrams; among many other examples \cite{michi1,michiq, karenthesis}. The problems we address here can be divided into two parts: one part is focused on chord diagrams represented in Chapter \ref{chchords1} and Chapter \ref{chapterchords2}; whereas in the second part we study diffeomorphisms of quantum fields applied to free theories, this is Chapter \ref{difchapter}.  

However, beside the essential difficulty of proving, there is a big difficulty in presenting results that are of this nature. Namely, the difficulty lies in providing the reader with the necessary tools and concepts from physics that are used or talked about to formulate the problems combinatorially. The good news is that, here, we managed to compress an overview of quantum theory in a way that makes the thesis self-contained and gives the reader unfamiliar with quantum field theory a convenient feeling of comprehension. Besides, going through this physics trying to wrap them up was the only way and the only time I truly understood the physical content abode by which I have been working.

Thus, we start by an intense overview of quantum theory in Chapter \ref{chapterqft}. In this first chapter we start with the very basic Newtonian mechanics seen almost by every mathematician at some point. Then we move  through an elementary justification of the correction introduced by quantum mechanics, we then provide all the tools leading to the `postulate' (and not the derivation) of Schr\"{o}dinger equation. Then we will see why this falls short before interpreting more advanced physical phenomena, which leads us to the field point of view. A review of classical field theory is quickly made and then we become finally able to define the notions of quantum field theory used in the sequel of the thesis, including renormalization, Green functions and Dyson-Schwinger equations. We end the first chapter with the abstract algebra concepts of Hopf algebras, as we will use them to describe renormalization quantities in papers coming from the D. Kreimer approach of renormalization.

In Chapter \ref{chchords1} we finally make the jump into the combinatorial class of chord diagrams by showing why this structure is important in QFT, namely by relating to the work of N. Narie and K. Yeats in \cite{yu}  for solving Dyson-Schwinger equations in the context of Yukawa theory, and more generally in \cite{Karenmarkushihn}. The goal of this chapter is to study the expressions appearing in an asymptotic expansion of $C_n$, the number of connected chord diagrams on $n$ chords, obtained by M. Borinsky \cite{michi}. M. Borinsky in \cite{michi} applies \textit{alien} derivatives to factorially divergent power series (chord diagrams are an example of such generating series) to get asymptotic information. In the calculation for $C(x)$, the generating series for connected chord diagrams, the result is a rational function in $C(x)$ times an exponential function in $C(x)$. The power series expansion for the expression in this exponent was noticed to follow sequence \href{https://oeis.org/A088221}{A088221} from the OEIS, for which the best known information was its relation to indecomposable chord diagrams. Our result in this chapter is proving that in fact \href{https://oeis.org/A088221}{A088221} counts pairs of connected chord diagrams when allowing the empty diagram, through providing suitable bijections.

Chapter \ref{chapterchords2} then comes with more about chord diagrams. As M. Borinsky suggested to the author in one meeting, in order to derive asymptotic information about $\mathcal{C}_{\geq2}$, the class of $2$-connected chord diagrams, in an approach similar to the one used in \cite{michi} for chord diagrams, one must first derive a functional equation involving connected and $2$-connected chord diagrams. We show how to derive such an equation and then use it in deriving the desired asymptotic information about the number of $2$-connected chord diagrams. Besides, we also derive a decomposition for connectivity-$1$ chord diagrams (these are diagrams with cuts). Our result extends the older result by D. J. Kleitman on counting irreducible diagrams. Namely,  Kleitman's result calculates the first coefficient of the infinite asymptotic expansion we get here.

It turned out that this calculation is more is interesting, namely,
the number of $2$-connected chord diagrams and the sequence of coefficients of their asymptotic expansion amazingly also appeared, without being recognized, in  the physical context of the work of Broadhurst on 4-loop Dyson-Schwinger-Johnson anatomy \cite{broadhurst}, and among the renormalized quantities of quenched QED calculated by M. Borinsky in \cite{michiq}. 

We give a chord-diagrammatic interpretation of the renormalized quantities in quenched QED and the proper Green functions of Yukawa theory. Realizing this combinatorial structure  makes  the asymptotic analysis of the Green functions of these theories easier, and no longer requires singularity analysis as used in \cite{michiq}. 
In particular, unlike the case of quenched QED, it was intricate to get the relation between Yukawa theory 1PI graphs and connected chord diagrams.

Finally, in Chapter \ref{difchapter}, we consider a problem from another context. The problem was to reprove the cancellation phenomenon of interaction terms produced by applying a point field diffeomorphism to a free field theory. In \cite{karendiffeo}, D. Kreimer and K. Yeats suggested that a proof on the level of generating functions would give more insight, and that is what we will do in Chapter \ref{difchapter}. Namely, we show that the series of  the tree-level amplitudes is exactly the compositional inverse of the diffeomorphism applied to the free theory. This is also conjectured to be related to the combinatorial Legendre transform defined by D. Jackson, A. Kempf and A. Morales, although it is still not very clear how this works. Simple combinatorial proofs are also given for some Bell polynomial identities. 

Chapter \ref{chfurther} outlines a number of future questions and research directions that may be considered.
\setlength{\parindent}{0cm}
\chapter{Quantum Field Theory Brief}\label{chapterqft}

\section{Quantum Mechanics}

In this section we shall start our journey with a quick, yet intense, rehearsal of the required physical background. This introduction is not meant to give a full account of quantum theory, but the promise is that a reader unfamiliar with physics will gain a sense of fulfilled comprehension while reading through the rest of the thesis, and will not have to  search for the many scattered elements of quantum theory every now and then. The key for achieving this satisfaction will simply be through historical justification: we will see how to justify quantization of Newtonian dynamics, how to postulate (and not derive) Schr\"{o}dinger equation, and how the field point of view is necessary when quantum mechanics is combined with special relativity to give what is known as quantum field theory. The reader can refer to \cite{zee, qftgifted, Das} for an overview of quantum theory.

\subsection{Limits of the Classical Theory}
In the study of macroscopic dynamics it has long been believed that Newton's law $$\mathbf{F}=m\textbf{a}$$ is unquestionable, and it was therefore natural to try using it for describing the dynamics of atomic particles around the start of the 20th century, when measurement technologies were revolutionized. The Bragg diffraction experiment (1913) was one of the experiments that pushed towards quantization. The experiment considers the diffraction of X-rays over atoms aligned in planes in a crystal. It was found (experimentally) that the intensity of diffracted rays of wave length $\lambda$ attained its peaks at angles $\theta$ according to the relation now known as Bragg's law: \[n\lambda=2 d \sin \theta,\]
where $d$ is the planes spacing in the crystal, and $n$ is a positive integer. This relation may be seen from Figure \ref{figureBragg} below, where the two diffracted rays are going to interfere constructively when the distance (bold) between the two dots is a multiple of the wave length.
\begin{figure}[!htb]
        \center{\includegraphics[scale=0.5]{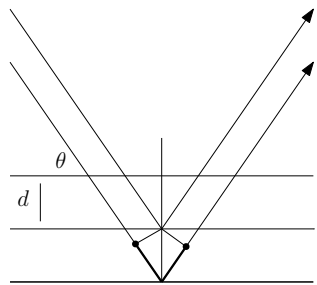}}
        \caption{\label{figureBragg} Bragg Diffraction}
\end{figure}
      
 The intriguing thing about this experiment is that it offered a plot of behaviour that has been reproduced by later experiments which involved substantial differences. In the Davisson-Germer experiments (1927) the X-rays (first granted as being of wave-nature) were replaced by electrons (thought of as pure particles) of certain energy $E$, yet, in the detected positions of scattered electrons the peaks followed a relation similar to that of Bragg. More precisely, the situation was found to fit Bragg's relation if for these electrons we assigned a virtual wavelength given by
      
 \[\lambda=\displaystyle\frac{h}{\sqrt{2m_e E}},\] where $m_e$ is the electron (empirical) mass, and $h=6.626 \times 10^{-34}$ J.s is Planck's constant introduced in the work of Max Planck and Albert Einstein on the photo-electric effect around 1905 (in their work $h$ related a photon's frequency $\nu$ to its energy as per $E=h\nu$). These findings also matched the suggestions of de Broglie about a wave nature of electrons (1925). It is also worth mentioning that the relation $\lambda=h/p$ ($p$ is momentum) played a key
role in establishing the circular-orbit model for atoms, as in the pioneering work of Niels Bohr, in effort to interpret the discrete emission spectra of heated atoms or ions. 
      
To summarize, these experiments, among others, established two things: (1) Wave-particle duality is inevitable; and (2) Newtonian mechanics is inadequate to phrase the new findings consistently. 
      
\subsection{Schr\"{o}dinger Equation}
     
The study of waves and oscillators was already very mature when the need to address wave-particle duality became insistent. In that way Schr\"{o}dinger equation was postulated through the available knowledge of waves and how oscillatory motion can take quantized wavelengths and frequencies. Functions of the form $A(x,t)=\exp[\pm2\pi i(\nu t-x/\lambda)]$ were considered as they were known to correspond to waves evolving in $x$ and $t$ under no external forces. $A(x,t)$ is called the \textit{amplitude} in wave terminology. These amplitudes also satisfy
      
\[\displaystyle\frac{\partial^2A}{\partial x^2}=-\left(\displaystyle\frac{2\pi}{\lambda}\right)^2 A, \quad \text{and}\quad  i\left(\displaystyle\frac{h}{2\pi}\right)\displaystyle\frac{\partial A}{\partial t}= EA. \]
     
Now, with de Broglie relation we will use $\lambda=h/p$, and if we assume no external forces we also have $E=p^2/2m$. The first equation then becomes
     
\[-\left(\displaystyle\frac{h}{2\pi}\right)^2\displaystyle\frac{\partial^2A}{\partial x^2}= EA, \]
     
and thus we get \[-\left(\displaystyle\frac{h}{2\pi}\right)^2\displaystyle\frac{\partial^2A}{\partial x^2}= i\left(\displaystyle\frac{h}{2\pi}\right)\displaystyle\frac{\partial A}{\partial t}. \tag{1}\label{primitive sch}\]
     
Before postulating Schr\"{o}dinger equation, we may think of some hidden analogy between 
\begin{equation}
         E \;\; \text{and} \;\; i\left(\displaystyle\frac{h}{2\pi}\right)\displaystyle\frac{\partial }{\partial t}\;, \;\; \text{and between}\;\;p^2 \;\; \text{and} \;\; -\left(\displaystyle\frac{h}{2\pi}\right)^2\displaystyle\frac{\partial^2}{\partial x^2}.\;\;
\end{equation}
     
The latter also says that we can identify the momentum operator as  \begin{equation}\label{momentum operator}
         p\;\;\; \equiv\;\;\; -i\left(\displaystyle\frac{h}{2\pi}\right)\displaystyle\frac{\partial }{\partial x}.
\end{equation}

This exchange will be important in the next postulates, and it should always be remembered that the equations are not exactly derived, they are more being guessed to comply with experimental findings. 
     
Every measurable physical quantity in classical mechanics can be expressed in terms of positions $\{q_i\}$ and momenta $\{p_i\}$ of the system. To get into the quantized picture we just rewrite the expression by replacing the $q_i$'s and $p_i$'s with the corresponding differential operators. Thus, every measurable quantity is expressed as an operator in quantum mechanics, and the measurements yielded experimentally are eigenvalues of the corresponding operator. An operator of special interest is the one that corresponds to the total energy of the system, called the \textit{Hamiltonian} $H$, its eigenstates (eigenfunctions) are called the \textit{wave functions}.
     
The \textbf{Schr\"{o}dinger equation} can now be postulated as \[H|\psi\rangle=i\hbar\frac{\partial}{\partial t} |\psi\rangle, \label{shrodinger}\tag{2}\]
where $|\psi\rangle$ is called a \textit{quantum state} (a state is a vector in a Hilbert space of \textit{states}, in which the inner product is denoted using the bra-ket notation $\langle\;|\;\rangle$).
In many cases the Hamiltonian is time independent, in which case one obtains the time-independent Schr\"{o}dinger equation: 
\[|\psi(q_i,t)\rangle=e^{-iHt/\hbar}|\psi(q_i,0)\rangle\label{schrodingerr}\tag{3}.\]
The value of the inner product $\langle \psi|\psi\rangle$ gives the probability density function for the positioning to be at $\{q_j\}$ at time $t$.

\section{Necessity of the Field Picture}
     
     In light of the work by Dirac, the attempt to incorporate Einstein's principles of special relativity with  Schr\"{o}dinger equation is impossible. It turns out that the Schr\"{o}dinger equation is strongly functional for particles moving with velocities much less than the speed of light $c$. However, problems arise for relativistic particles due to the uneven treatment of time and space in Schr\"{o}dinger equation. The Klein-Gordon equation was the first relativistic version of Schr\"{o}dinger, however, it was only consistent with spinless particles \cite{peskin, Das}. 
     
     Dirac derived his famous equation to incorporate special relativity in the theory, he realized  `new' wave functions as  four-component objects, which  leads to defining  \textit{spin} theoretically. We are not going to do the math of Dirac's equation, but we shall lightly talk about its consequences for the sake of justification. The reader interested in the whole story can refer to \cite{zee, particle}. 
     
     One aspect of Dirac's equation is that it allowed the existence of `particles' with negative-energy. Dirac accepted this new feature as a telling and not a misleading addition. Dirac gave an elementary picture that negative-energy electrons exist but are unobservable. The only conceivable way they would become observable is an interaction in the form of a real electron collapsing in energy level to become a negative-energy electron. To prevent this Dirac suggested that this unobservable `sea' of negative-energy electrons is already full and no electron takes energy in the gap between $-m_0c^2$ and $m_0c^2$ where $m_0$ is the electron mass.
     
     Elaborating these ideas and the problems that arise from them leads to defining the \textit{positron}, the \textit{antiparticle} of the electron. The existence of positrons was experimentally proven in 1931 by Carl Anderson. \textit{Anti-protons} were also discovered experimentally in the 1950's. The four components of the Dirac wave function were then well understood as representing the two spin components of the electron and the positron.
     
     The idea that an electron and a positron can annihilate each other and produce photons with positive energy (and vice versa) made it more clear that the corpuscular picture of electrons is no longer adequate.
     
     Just as photons were realized as the quanta of the electromagnetic field, an electron must be taken as a manifestation of an `electron field' defined over the whole space-time, and a proton is a manifestation of a `proton field'. 



\section{Elements of Classical Field Theory}

To overview field theory we will follow the lectures given at the summer school \textit{Geometric and Topological Methods for Quantum Field Theory}, in Villa de Leyva (Colombia). July 2-20, 2007. The lectures were based on a graduate course held at Universit\'e Lyon 1 in spring 2006, by Alessandra Frabetti and Denis Perrot \cite{aless}. For a general reference of classical and quantum field theory see \cite{qftgifted, itz, peskin, zee}.

\subsection{Classical Fields}
Here we review classical field theory in order to understand the nature of the questions involved.
The spacetime coordinates are denoted by $x=(x^\mu)$ where $\mu=0,1,\ldots, D-1$. For the usual Minkowski space $D=4$, and the space is endowed with the  metric signature $g=(1,-1,-1,-1)$. The \textit{Wick rotation} transformation allows us to work in the Euclidean space. In a flat space, as we do consider here, a \textit{field} is a vector-valued function. 

\begin{dfn}
A \textit{classical field} $\phi$ will be a $C^\infty$ real function $\phi:\mathbb{R}^D\longrightarrow \mathbb{R}$ with the extra property that all of its derivatives $\partial^n_\mu \phi$ decay rapidly to zero as $|x|\longrightarrow \infty$. An \textit{observable} of the system described by $\phi$ is a real functional in $\phi$, which expresses an observable quantity of the system. 
\end{dfn} 

\begin{dfn}
A field $\phi:\mathbb{R}^D\longrightarrow \mathbb{C}$ that attains complex values is called a \textit{wave function}. In this case a measurement value $|\phi(x)|^2$ describes the \textit{probability amplitude} of finding the particle at position $x$.
\end{dfn}

\subsection{Euler-Lagrange Equation}

In the usual treatment of classical mechanics, the framework of Lagrangian mechanics was found to be the most suitable. We will assume familiarity with basic Lagrangian mechanics, a thorough treatment can be found in any intermediate level mechanics textbook. What we will care for is that there is a \textit{Lagrangian density} $L:\mathbb{R}^D\longrightarrow\mathbb{R}$ associated to any given system. If the system is described by a field $\phi$, then The Lagrangian density at $x$ shall generally depend on $\phi(x)$ as well as the gradient $\partial\phi(x)$. The \textit{action} of the field is then naturally defined to be the functional
\[\phi\mapsto S[\phi]=\int_{\mathbb{R}^D}L(x,\phi(x),\partial\phi(x)) \;d^Dx.\]

By \textit{Hamilton's principle of least action} and the usual calculus of variation discussion \cite{qftgifted} one gets the well-known \textit{Euler-Lagrange equation}

\[\displaystyle\frac{\partial L}{\partial\phi}-\sum_\mu \partial_\mu\left(\displaystyle\frac{\partial L}{\partial(\partial_\mu\phi)}\right)=0.\]

To find the classical field $\phi$ is to solve this equation. The equation can be non-homogeneous and/or non-linear. The non-homogeneous terms appear if the system is not isolated (see below), and the non-linear terms appear in case of interactions. For example, consider the field with Lagrangian density 

\[L(x,\phi, \partial\phi)=\frac{1}{2}(\partial_\mu\;\phi\partial^\mu\phi+m^2\phi^2)-J\phi-\frac{\lambda}{3!}\phi^3-\frac{\mu}{4!}\phi^4.\]

Working out the Euler-Lagrange equation gives

\[(-\Delta+m^2)\phi(x)=J(x)+\frac{\lambda}{2}\phi^2-\frac{\mu}{3!}\phi^3,\]
where $\Delta\phi(x)=\sum_\mu \partial_\mu(\partial_\mu\phi(x))$ is the \textit{Klein-Gordon operator}, the last equation is also known as the \textit{Klein-Gordon equation}. $m$ is the mass, and $J(x)$ is a field and is usually referred to as the $source$ as we will discuss again shortly.

A general Lagrangian for a relativistic particle can be of the form

\begin{equation}\label{use}
L(x,\phi, \partial\phi)=\frac{1}{2}\phi^\dagger A\phi-J\phi-\frac{\lambda}{3!}\phi^3-\frac{\mu}{4!}\phi^4,\end{equation}

where $A$ is a differential operator. The \textit{free Lagrangian} is  $L_\text{free}:=\frac{1}{2}\phi^\dagger A\phi-J\phi$. The source $J(x)$ is  external to the system, hence the system is said to be \textit{isolated (or in vacuum)} if $J=0$. The parameters $\lambda, \mu$ govern the strength of the system's self-interactions (for example, sometimes they are the electric charge), therefore, they are called the \textit{coupling constants}. The \textit{interacting Lagrangian} is defined to be the sum of the interaction terms, namely $L_\text{int}:=-\frac{\lambda}{3!}\phi^3-\frac{\mu}{4!}\phi^4$.

\subsection{Free and Interacting Fields}

1) A field describing a system with a free Lagrangian will be called a \textit{free field} and will be denoted by $\phi_0$. In that case $L=\frac{1}{2}\phi^\dagger A\phi-J\phi $, and whence the Euler-Lagrange equation is \[A\phi_0=J(x).\]

Then $\phi_0$ is found to be the convolution \[\phi_0(x)=\int G_0(x-y) \;J(y) \;d^Dy,\]
where $G_0(x)$ is called the \textit{Green function} of operator $A$, this means it is the distribution satisfying $AG_0=\delta(x)$, where $\delta(x)$ is the Kronecker delta function. The interpretation is that from any point $y$ where the source is non-vanishing, i.e. $J(y)\neq0$, the source affects $\phi$ at position $x$ according to the distribution $G_0(x-y)$. $G_0(x-y)$ is then considered as the \textit{field propagator}.

As an example, for $A$ taken to be the Klein-Gordon operator $-\Delta+m^2$, $G_0$ is found to be the Fourier transform 
\begin{equation}\label{kleinnn}
G_0(x-y)=\int_{\mathbb{R}^D} \displaystyle\frac{e^{-i\;p.(x-y)}}{(2\pi)^D(p^2+m^2)}\; \; d^Dp.    
\end{equation}

2) A field describing a system with an interacting Lagrangian as in equation (\ref{use}) is called a \textit{self-interacting field}, and will have the Euler-Lagrange equation as 
\[A\phi(x)=J(x)+\frac{\lambda}{2}\phi^2-\frac{\mu}{3!}\phi^3.\]

This non-linear equation however generally are hard to solve (if a solution exists). An approach to it when the coupling constants are small enough is to consider the RHS as a perturbation of the free analogue. In that case we get the recursive solution 

\[\phi(x)=\int_{\mathbb{R}^D}G_0(x-y)\; \left[J(y)+\frac{\lambda}{2}\phi(y)^2-\frac{\mu}{3!}\phi(y)^3\right]\;d^Dy,\]

which is then solved as a series in $\lambda$ and $\mu$. As an example we assume for the sake of simplicity that $\mu=0$, in that case we have the recursive equation
\begin{equation}\label{to}
    \phi(x)=\int_{\mathbb{R}^D}G_0(x-y)\; \left[J(y)+\frac{\lambda}{2}\phi(y)^2\right]\; d^Dy.
    \end{equation}

Repeated substitution of $\phi$ into the integral gives

\begin{align}\label{to1}
    \phi(x)&=\int_{\mathbb{R}^D}G_0(x-y)\; J(y)\nonumber\\
           &+\frac{\lambda}{2}\int_{\mathbb{R}^D}G_0(x-y)G_0(y-z)G_0(y-u)\; J(z)\;J(u) \;d^Dy\;d^Dz\;d^Du\nonumber\\
           &+\frac{2\lambda^2}{4}\int_{\mathbb{R}^D}G_0(x-y)G_0(y-z)G_0(y-u)G_0(z-v)G_0(z-w)\;\times \nonumber\\
           &\qquad \qquad \times\; J(z)J(u)J(v)J(w)\;\;d^Dy\;d^Dz\;d^Du\;d^Dv\;d^Dw+ \mathcal{O}(\lambda^3).
\end{align}

Thus the moral is that a classical field with interactions whose Lagrangian density is something like 
\[L(x,\phi, \partial\phi)=\frac{1}{2}\phi^\dagger A\phi-J\phi-\frac{\lambda}{3!}\phi^3-\frac{\mu}{4!}\phi^4 \]

can be perturbatively described as a power series 
\[\phi(x)=\sum_{m,n=0}^\infty \lambda^n \mu^m \phi_{n,m}(x),\]

where each coefficient $\phi_{n,m}(x)$ is a finite sum of integrals that only involves the field propagator $G_0$ and the source $J$. We shall see later how to encrypt these integrals as Feynman graphs.

\section{Elements of Quantum Field Theory}

Here the observables of a quantum systems will be represented as self-adjoint operators on a Hilbert space representing the \textit{state-space} in which a vector is referred to as a \textit{(quantum) state}, the quantum system can be found in one such state. The eigenvalues of an observable represent the possible values the observable can have (note that they are real numbers since the operator is self-adjoint). The expectation value $\langle\psi|F|\psi\rangle$ is interpreted as the probability that the measurement of the observable $F$ is the value for state $\psi$. In this context, we will have \textit{field operators} instead of real-valued functions. 

The standard approach for quantum fields is to make a \textit{Wick rotation} for the time $t$, by $t\mapsto -it$. This transforms the Minkowski spacetime into a usual Euclidean space. A quantum field is then regarded as classical field  (wave function) that fluctuates around its expectation value (a statistical field). This procedure is referred to as the \textit{path integral quantization}.

\subsection{Path Integrals Formulation}

Let $\phi$ be a quantum field. The first expectation value that we will be interested in is $\langle\phi(x)\rangle$, the mean value for $\phi$ at position $x$, or generally $\langle\phi(x_1)\cdots\phi(x_m)\rangle$, the Green function that represents the probability that the field $\phi$ takes the path $x_mx_{m-1}\cdots x_1$. The drawback now is that a quantum field is not necessarily satisfying Hamilton's principle of least action, but can only be interpreted as a perturbation around the classical solution of the Euler-Lagrange equation. In Minkowski space, the probability of observing the quantum field at value $\phi$ is proportional to $\exp\big(i\frac{S[\phi]}{\hbar}\big)$, where $\hbar=h/2\pi$ (the reduced Planck's constant). Translated in the Euclidean space, the probability is proportional to $\exp\big(-\frac{S[\phi]}{\hbar}\big)$. With the limit $\hbar\rightarrow 0$ we verify the classical limit as we are to get a maximal probability to observe the field $\phi$ at a stationary point of the action. 

Thus, as a probability, the Green function is the path integral given by

\[\langle\phi(x_1)\cdots\phi(x_m)\rangle=\displaystyle\frac{\int \phi(x_1)\cdots\phi(x_m) \exp\big(-\frac{S[\phi]}{\hbar}\big)|_{J=0}\; d\phi}{\int  \exp\big(-\frac{S[\phi]}{\hbar}\big)|_{J=0}\; d\phi}.\]

This however is not well-defined, we can't exactly tell what is the precise meaning of integrating over the measure $d\phi$. It is nevertheless a soundness certificate that we can correctly retract to the classical values by taking the limit $\hbar\rightarrow 0$.

\subsection{The Dyson-Schwinger Equation and Connected 
Green Functions}
We now define what is called a \textit{partition function}, which will serve as a generating function for Green functions. As a function in $J$, the partition function is defined as 

\[Z[J]=\int e^{-\frac{S[\phi]}{\hbar}}\; d\phi,\] with the normalization constraint that $Z[J]|_{J=0}=1$. In these terms we get 

\begin{equation}\label{partitionfun}
\langle\phi(x_1)\cdots\phi(x_m)\rangle=\left(\displaystyle\frac{\hbar^m}{Z[J]}\right)\; \left.\displaystyle\frac{\delta^mZ[J]}{\delta J(x_1)\cdots\delta J(x_m)}\right|_{J=0},\end{equation}

where $\delta/\delta J$ is a functional derivative. Working out the calculations using the previous equations and definitions we get 

\[\displaystyle\frac{\delta S[\hbar\delta/\delta J]}{\delta \phi(x)}\;Z[J]=0.\]
This can be seen as a generalization of the classical Euler-Lagrange equation.

As is conventional, we define the \textit{free energy} to be $W[J]=\hbar \log\;Z[J]$ (i.e. $Z[J]=e^{W[J]/\hbar}),$ with the constraint that $W[J]|_{J=0}=0$. Introducing this into equation (\ref{partitionfun}) gives

\begin{align*}
    \langle\phi(x)\rangle&= \left.\displaystyle\frac{\delta W[J]}{\delta J(x)}\right|_{J=0},\\
    \langle\phi(x_1)\phi(x_2)\rangle&=\langle\phi(x_1)\rangle\langle\phi(x_2)\rangle+\hbar \left.\displaystyle\frac{\delta^2W[J]}{\delta J(x_1)\delta J(x_2)}\right|_{J=0},\\
    &\;\;\vdots
\end{align*}

In this we see the Green functions have been expressed recursively with the addition of terms involving the derivatives of the free energy, namely they involve the expressions

\[G(x_1,\ldots,x_m)=\left.\displaystyle\frac{\delta^mW[J]}{\delta J(x_1)\cdots\delta J(x_m)}\right|_{J=0},\]

which are referred to in the literature as the \textit{connected Green functions}. We will see later that the name connected relates to graph theoretic connectivity for Feynman diagrams.

\begin{exm}
Consider an arbitrary quantum field with action
\[S[\phi]=\frac{1}{2}\phi^\dagger A\phi-J\phi-\frac{\lambda}{3!}\int \phi(x)^3 d^Dx,\] and set $G_0=A^{-1}$. The \textit{Dyson-Schwinger equation} for the 1-point Green function in $J$ becomes

\begin{equation}
\langle\phi(x)\rangle_J=\displaystyle\frac{\delta W[J]}{\delta J(x)}=
\int G_0(x-u)\left[J(u)+\frac{\lambda}{2}\bigg[\left(\displaystyle\frac{\delta W[J]}{\delta J(u)}\right)^2+\hbar\displaystyle\frac{\delta^2 W[J]}{\delta J(u)^2}\bigg]\right]d^Du.\label{dys1}\end{equation}

Taking the functional derivative $\delta/\delta J(y)$ and evaluating at $J=0$, we get the Dyson-Schwinger equation for $2$-point connected Green function:
\begin{equation}\label{2point}
G(x,y)=G_0(x-y)+\frac{\lambda}{2}\int G_0(x-u)[2G(u) G(u,y)+\hbar G(u,u,y)] \;d^Du.    
\end{equation}

This can be repeated to get the $n$-point connected Green function in general. Also notice that this can be used to provide a perturbative expansion for the mean value $\langle\phi(x)\rangle_J$ of $\phi$ as a power series in $\lambda$ and involving $J$.

\end{exm}

Thus, we have seen that if $\phi$ is a quantum field with a typical Lagrangian density of \[L(\phi)=\frac{1}{2}\phi^\dagger A\phi-J\phi-\frac{\lambda}{3!} \phi(x)^3, \] then the $n$-point connected Green function can be represented as 
\[G(x_1,\ldots,x_m)=\sum_{n=0}^\infty \lambda^n G_n(x_1,\ldots,x_m), \]where $G_n(x_1,\ldots,x_m)$ are the coefficients obtained as in the example above.

In the next section we will see how to represent the the $G_n$'s diagrammatically by means of the Feynman diagrams. For the sake of clarity we shall split the classical and quantum cases to highlight the comparison.

\subsection{Diagrammatic Representation and Feynman Graphs}
\subsubsection{Classical Fields and Tree Graphs:}

Imagine that we wish to represent the equations we got so far diagrammatically. Appealing to Feynman's notation we set
\begin{align*}
  \text{the field}\;\;  \phi(x)\;&\leftrightarrow\hspace{0.1cm} \raisebox{-0.3cm}{\includegraphics[scale=0.8]{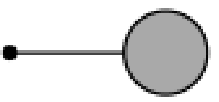}} \\
    &\\
\text{the source}\;\; J(y)\;&\leftrightarrow\;\raisebox{0cm}{\includegraphics[scale=0.8]{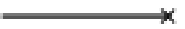}}\\
  &\\
 \text{the propagator}\;\; G_0(x-y)\;&\leftrightarrow\;\raisebox{0cm}{\includegraphics[scale=0.8]{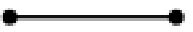}}
\end{align*}

The \textit{amplitude} of a Feynman diagram will be its analytical value.

The Euler-Lagrange equation (eq. \ref{to})
\[\phi(x)=\int_{\mathbb{R}^D}G_0(x-y)\; \left[J(y)+\frac{\lambda}{2}\phi(y)^2\right]\; d^Dy\] can now be represented as 

\begin{equation}\label{eulerdiag}
    \raisebox{-0.2cm}{\includegraphics[scale=0.7]{Figures/field.eps}}\;\;=\;\;\raisebox{0.02cm}{\includegraphics[scale
    =0.7]{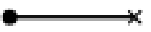}}\;\;+\;\;\frac{\lambda}{2}\;\;\raisebox{-0.9cm}{\includegraphics[scale
    =0.7]{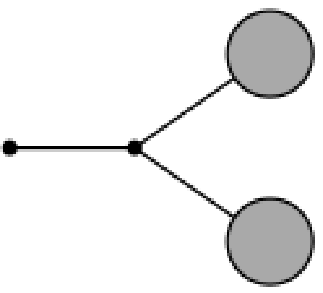}}\;\;,
\end{equation}

notice that the filled circle is standing for the recursive nature of the field. Now, what we did next was to substitute $\phi$ into itself to obtain the perturbative expansion, diagrammatically this will stand for forming trees with a double branching at each vertex (note that there will be no grey disks):

\begin{equation}
  \raisebox{-0.23cm}{\includegraphics[scale=0.7]{Figures/field.eps}}\;=
  \;\includegraphics[scale
    =0.7]{Figures/dottedsource.eps}\;+\;\frac{\lambda}{2}\; \raisebox{-0.42cm}{\includegraphics[scale=0.7]{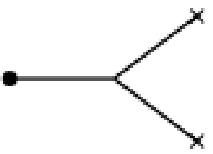}}\;+\;\frac{\lambda^2}{2}\;\raisebox{-0.42cm}{\includegraphics[scale
    =0.7]{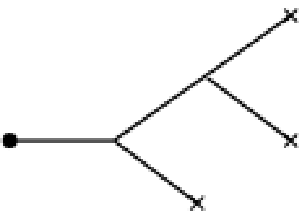}}\;+\frac{\lambda^3}{8}\;
    \raisebox{-0.85cm}{\includegraphics[scale
    =0.7]{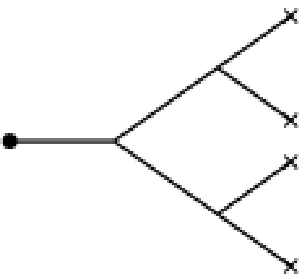}}\;+\;\cdots
\end{equation}

The branchings are double due to the term $\phi(y)^2$. Notice that coefficient of each tree $T$ in the expansion is $\lambda^{v(T)}/ \text{Sym}(T)$, where $v(T)$ is meant to be the number of internal vertices in the tree $T$, and $\text{Sym}(T)$ is the size of the automorphism group of $T$ (graph-theoretic isomorphisms). Also note that if the interacting term in the Lagrangian density was cubic, the trees would have had $4$-valent internal vertices instead of $3$-valent ones in our case. 

The comparison between the analytic solution and the diagrammatic expansion shows that each tree corresponds to one of the integrals in equation (\ref{to1}). The process of moving from the integrals (the amplitudes) to the diagrams is governed by the so-called \textit{Feynman rules}.

\textbf{Feynman Rules:} Assume that the field is analytically solved to be $\phi(x)=\sum_n \lambda^n \phi_n(x)$, then the coefficient $\phi_n(x)$ is represented by the sum of all tree graphs with $n$ internal vertices together with their symmetry factors. The trees are obtained as follows:

\begin{enumerate}
    \item Draw all rooted trees in which the degree of any internal vertex is $3$ (or generally the power of the field in the interacting term of the Lagrangian). The leaves together with their edges are referred to as the \textit{external legs}.
    
    \item Label the root as $x$, and label the vertices (internal and external ones)  by other free variables $y,z,u,v,\ldots$ as we did in the analytic solution.
    
    \item If vertices $\alpha$ and $\beta$ are adjacent, give weight $G_0(\alpha-\beta)$ for the incident edge.
    
    \item Assign weight $\lambda$ to every internal vertex.
    
    \item Assign weight $J(\alpha) $ to a leaf (external leg) $\alpha$.
    
    \item The amplitude $\phi_T(x)$ of a tree $T$ is then obtained by multiplying all the contributing weights and integrating over the free variables involved in the tree.
    
    \item Finally multiply with a symmetry factor $\displaystyle\frac{1}{\text{Sym}(T)}$.
    
\end{enumerate}

It is worth mentioning that $\phi_n(x)=\underset{v(T)=n}{\sum} \displaystyle\frac{1}{\text{Sym}(T)}\; \phi_T(x)$, where $\phi_T(x)$ is the amplitude (the corresponding integral) of the tree $T$.

\subsubsection{Quantum Fields and Loop Corrections:}\label{ddd}

As before, let us first the new set of Feynman notations: 
\begin{align*}
  \text{$m$-point Green function}\;\;  \langle\phi(x_1)\cdots\phi(x_m)\rangle\;&\leftrightarrow\; \raisebox{-0.61cm}{\includegraphics[scale=0.7]{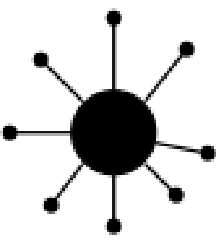}} \\
    &\\
    \text{$m$-point connected Green function}\;\; G(x_1,\ldots,x_m)\;&\leftrightarrow\;\raisebox{-0.61cm}{\includegraphics[scale=0.7]{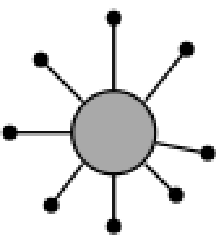}}\\
\text{the source}\;\; J(y)\;&\leftrightarrow\;\raisebox{0cm}{\includegraphics[scale=0.7]{Figures/source.eps}}\\
  &\\
 \text{the propagator}\;\; G_0(x-y)\;&\leftrightarrow\;\raisebox{0cm}{\includegraphics[scale=0.7]{Figures/G.eps}}
\end{align*}

\textbf{Examples:}\\

1) The Dyson-Schwinger equation (\ref{dys1}) for the $1$-point connected Green function  is represented as

\begin{equation}
    \raisebox{-0.2cm}{\includegraphics[scale=0.7]{Figures/field.eps}}\;=\;\includegraphics[scale
    =0.7]{Figures/dottedsource.eps}\;+\;\frac{\lambda}{2}\;\raisebox{-0.91cm}{\includegraphics[scale
    =0.7]{Figures/branch1.eps}}\;+\;\hbar\;\frac{\lambda}{2}\;\raisebox{-0.365cm}{\includegraphics[scale
    =0.7]{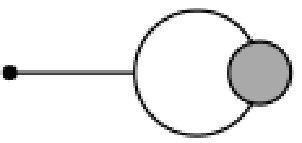}}.
\end{equation}

It is beautiful here to notice that taking the limit $\hbar\rightarrow 0$ gives the classical Euler-Lagrange equation of $\phi$ (equation \ref{eulerdiag}). The loop term corresponds to the correction term in equation (\ref{dys1}) containing the second derivative of the free energy $W[J]$. In case the field is isolated, we only get rid the first term which represents the source.\\

2) The Dyson-Schwinger equation for the $2$-point connected Green function is obtained similarly and we can expect it will contain diagrams that are reduced to two external legs if all edges are contracted (we will see this again later on):

\begin{equation}
    \raisebox{-0.24cm}{\includegraphics[scale=0.7]{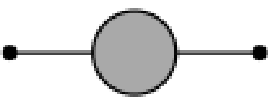}}\;=
    \;\includegraphics[scale=0.7]{Figures/G.eps}\;+
    \;\lambda\;\raisebox{-0.85cm}{\includegraphics[scale=0.7]{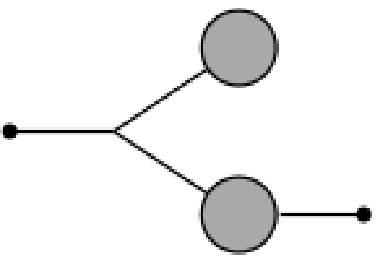}}\;+
    \;\hbar\;\frac{\lambda}{2}\;\raisebox{-0.36cm}{\includegraphics[scale
    =0.7]{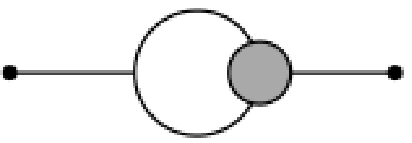}}.
\end{equation}\\

3) The Dyson-Schwinger equation for the $3$-point connected Green function is diagrammatically equivalent to 

\begin{equation}
    \raisebox{-0.5cm}{\includegraphics[scale=0.7]{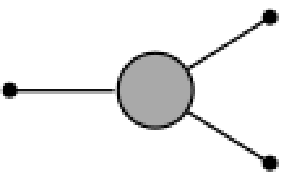}}\;=
    \;\lambda\;\;\raisebox{-1.1cm}{\includegraphics[scale=0.7]{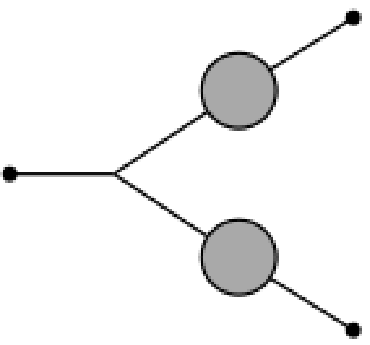}}\;+
    \;\lambda\;\raisebox{-1.1cm}{\includegraphics[scale=0.7]{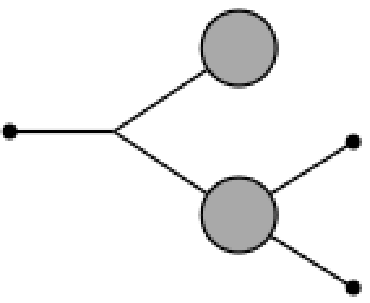}}\;+
    \;\hbar\;\frac{\lambda}{2}\;\raisebox{-0.57cm}{\includegraphics[scale
    =0.7]{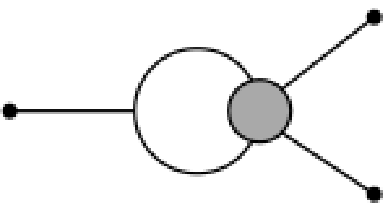}}.
\end{equation}

Now, as we did in the classical case, let us obtain the perturbative solution of the recursive Dyson-Schwinger equation for the $1$-point connected Green function. We expect the solution to include the same trees from the classical setting in addition to the loop correction terms with $\hbar$ factors:

\begin{align*}
  \raisebox{-0.2cm}{\includegraphics[scale=0.7]{Figures/field.eps}}\;
  =&
  \;\includegraphics[scale=0.7]{Figures/dottedsource.eps}\;+\;\frac{\lambda}{2}\;\raisebox{-0.4cm}{\includegraphics[scale=0.7]{Figures/branchcross1.eps}}\;+\;\hbar\frac{\lambda^2}{2}\;
  \raisebox{-0.2cm}{\includegraphics[scale=0.45]{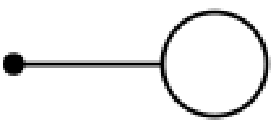}}\;+
    \;
    \frac{\lambda^2}{2}\;
    \raisebox{-0.38cm}{\includegraphics[scale
    =0.7]{Figures/branchcross2.eps}}\;+\;\hbar\;\frac{\lambda^2}{2}\;\raisebox{-0.38cm}{\includegraphics[scale
    =0.7]{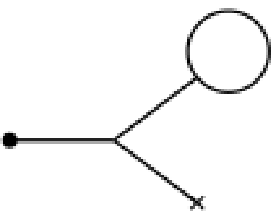}}\;+\\
    &+\;
    \hbar\;\frac{\lambda^2}{2}\;
    \raisebox{-0.2cm}{\includegraphics[scale
    =0.7]{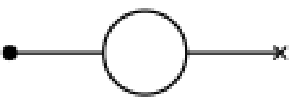}}\;+\;\frac{\lambda^3}{8}\;\raisebox{-0.8cm}{\includegraphics[scale
    =0.7]{Figures/branchcross3.eps}}\;+\;\frac{\lambda^3}{2}\;\raisebox{-0.29cm}{\includegraphics[scale
    =0.45]{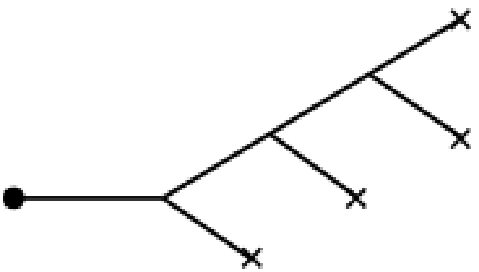}}\;+\\
    &+\;
    \hbar\;\frac{\lambda^3}{2}\;
    \raisebox{-0.78cm}{\includegraphics[scale
    =0.43]{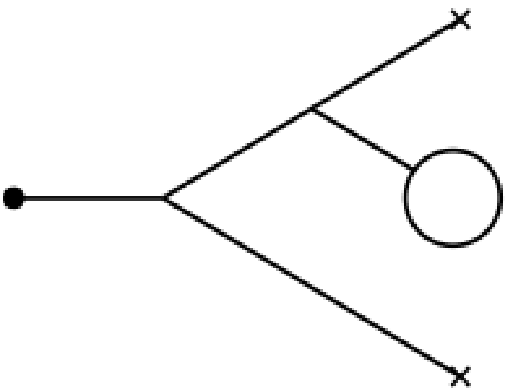}}\;+\;\hbar\;\frac{\lambda^3}{4}\;\raisebox{-0.6cm}{\includegraphics[scale
    =0.43]{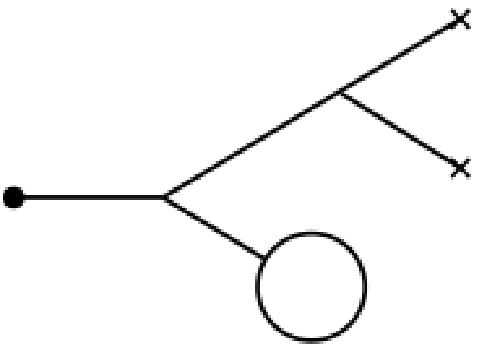}}\;+\;\hbar\;\frac{\lambda^3}{4}\;\raisebox{-0.41cm}{\includegraphics[scale
    =0.45]{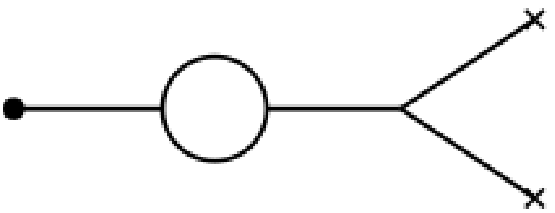}}\;+\\
    &+\;
    \hbar\;\frac{\lambda^3}{2}\;
    \raisebox{-0.78cm}{\includegraphics[scale
    =0.43]{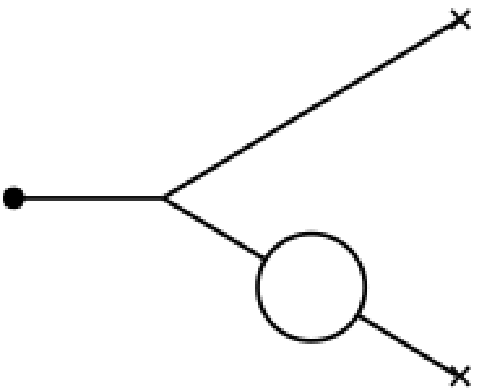}}\;+\;\hbar^2\;\frac{\lambda^3}{4}\;\raisebox{-0.2cm}{\includegraphics[scale
    =0.45]{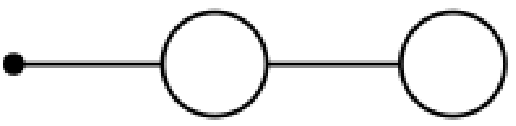}}\;+\;\hbar\;\frac{\lambda^3}{4}\;\raisebox{-0.78cm}{\includegraphics[scale
    =0.7]{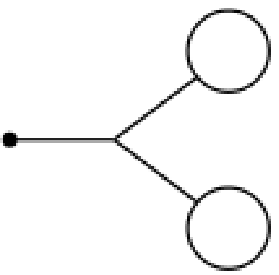}}\;+\\
    &+\;
    \hbar\;\frac{\lambda^3}{2}\;
    \raisebox{-0.75cm}{\includegraphics[scale
    =0.45]{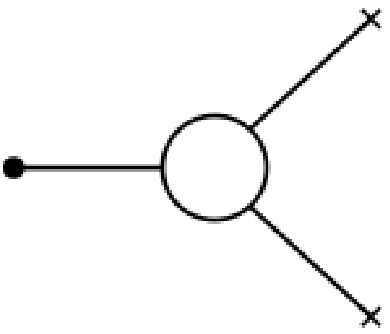}}\;+\;\hbar^2\;\frac{\lambda^3}{4}\;\raisebox{-0.2cm}{\includegraphics[scale
    =0.7]{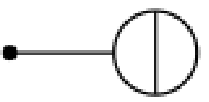}}\;+ \;\mathcal{O}(\lambda^4).
\end{align*}

One again observes that this boils down to the classical counterpart by taking $\hbar\rightarrow 0$. Each graph $\Gamma$ has a factor of $\lambda^{v(\Gamma)}$, where, as in the tree case, $v(\Gamma)$ is the number of internal vertices in $\Gamma$. Also notice that each loop contributes $2$ to $\text{Sym}(\Gamma)$, where $\text{Sym}(\Gamma)=|\text{Aut}(\Gamma)|$, both symbols will be used interchangeably. The power of $\hbar$ expresses the loop number (the number of independent cycles in the graph) of $\Gamma$. The $2$-point connected Green function can be drawn similarly. 

\textbf{Feynman Rules:} Assume that $\phi$ is a quantum field with Lagrangian density given by \[L(\phi)=\frac{1}{2}\phi^\dagger A\phi-J\phi-\frac{\lambda}{3!}\phi^3,\] for which the $m$-point connected Green function is solved as \[G(x_1,\ldots,x_m)=\sum_{n=0}^\infty \lambda^n G_n(x_1,\ldots,x_m).\]

For a given $n$, the coefficient $G_n(x_1,\ldots,x_m)$ will be obtained diagrammatically as the finite sum 
\[G_n(x_1,\ldots,x_m)=\underset{v(\Gamma)=n}{\sum} \displaystyle\frac{\hbar^{\ell(\Gamma)}}{\text{Sym}(\Gamma)}\;A(\Gamma;x_1,\ldots,x_m),\]

where $\ell(\Gamma)$ is the loop number; and $A(\Gamma;x_1,\ldots,x_m)$ is the \textit{amplitude} of the graph $\Gamma$, that is the analytic value of the Feynman diagram. The channel between  the analytic and diagrammatic representations is subject to the following Feynman rules:

\begin{enumerate}
    \item Draw all graphs with $n$ internal vertices and $m$ external legs, in which every internal vertex is $3$-valent ($3$-regular graphs).
    
    \item Give the external legs the fixed labels $x_1,\ldots,x_m$. On the other hand, label the internal vertices with the free variables $y,z,v,u,\ldots$.
    
    \item If vertices $\alpha$ and $\beta$ are adjacent, give weight $G_0(\alpha-\beta)$ for the incident edge.
    
    \item Assign weight $\lambda$ to every internal vertex (these are the $3$-valent ones).
    
    \item Assign weight $\hbar$ to each loop.
    
    \item Assign weight $J(\alpha) $ to an external leg $\alpha$.
    
    \item The amplitude $A(\Gamma;x_1,\ldots,x_m)$ of a graph $\Gamma$ is then obtained by multiplying all the contributing weights and integrating over the free variables involved in the tree.
    
    \item Finally multiply with the symmetry factor $\displaystyle\frac{1}{\text{Sym}(\Gamma)}\;$.
    
\end{enumerate}

\section{Momentum Space Formulation}

All of our work so far took place in the position space, were we expressed all relations in terms of the position coordinates. In this section we will translate all quantities and relations in terms of the momentum operator $p$ (in Minkowski space it is called the \textit{four-momentum}). We recall equation (\ref{momentum operator}) from the background section, namely 

\begin{equation}\label{momentum operatoragain}
         p\;\;\; \equiv\;\;\; -i\hbar\displaystyle\frac{\partial }{\partial x},
     \end{equation}
which says that the operators of momentum and position are conjugate to each other. The eigenstates for the momentum operator are seen to be the states whose wave functions are  

$\langle x|p\rangle=e^{ip\cdot x/\hbar}$, or simply $$\langle x|p\rangle=e^{ip\cdot x}$$ if we suppress $\hbar$ into the definition of $p$.

Thus we can use the Fourier transform to move from the momentum basis to position basis. In other words $|p\rangle=\int e^{ip\cdot x} |x\rangle \;d^Dx$.

This leads to the following representation for each of the expressions in the Dyson-Schwinger equation:
\begin{align*}
    \hat\phi(p)&=\int_{\mathbb{R}^D}\phi(x)  e^{ip\cdot x}\;d^Dx,\\
    \hat J(p)&=\int_{\mathbb{R}^D}J(x)  e^{ip\cdot x}\;d^Dx,\\
    \hat G_0(p)&=\int_{\mathbb{R}^D}G_0(x-y)  e^{ip\cdot (x-y)}\;d^Dx.
\end{align*}

As an example of how things are changed in the momentum point-of-view let us look back at the Green function distribution for the Klein-Gordon equation (see equation (\ref{kleinnn})); we get that in this case
$\hat G_0(p)=\displaystyle\frac{1}{p^2+m^2}$.

The $m$-point connected Green functions transforms as 
\begin{equation}\label{may}
 \hat G^{(m)}(p_1,\ldots,p_m)=\int G(x_1,\ldots,x_m) \; e^{ip_1(x_1-x_m)}\cdots e^{ip_m(x_{m-1}-x_m)}\;d^Dx_1\cdots d^Dx_m.   
\end{equation}
Note that the translation invariance of $G(x_1,\ldots,x_m)$ implies a very important property in the momentum formulation, namely that

which we will refer to as \textit{conservation of momentum}.

For example, this enables us to express the Dyson-Schwinger equation (\ref{2point}) for the $2$-point function  in terms of external momenta: 

\begin{equation}
    \hat G^{(2)}(p)=\hat G_0(p)+\lambda\; \hat G_0(p) \;\hat G^{(1)}(0)\; \hat G^{(2)}(p)+\hbar\; \frac{\lambda}{2}\;\hat G_0(p) \int \frac{1}{(2\phi)^D} \;\hat G^{(3)}(q,p-q,-p) \;d^Dq,
\end{equation}

and the higher order functions are obtained similarly. Now let's see what changes are seen on Feynman graphs in the transition to momentum space.

\subsection{1PI Diagrams}\label{section}

The diagrams in momentum space are structurally the same as before. The only difference is that propagators and external half-edges are now oriented and labeled by momenta $p_1,\ldots, p_k$, taking into regards the conservation of momenta. A propagator gets its direction from the truncated edge that it is matched with. In the previous setting we didn't worry about direction and so we only determined a `place' for a propagator. Diagrammatically, this is 
\begin{center}
    \includegraphics[scale=0.6]{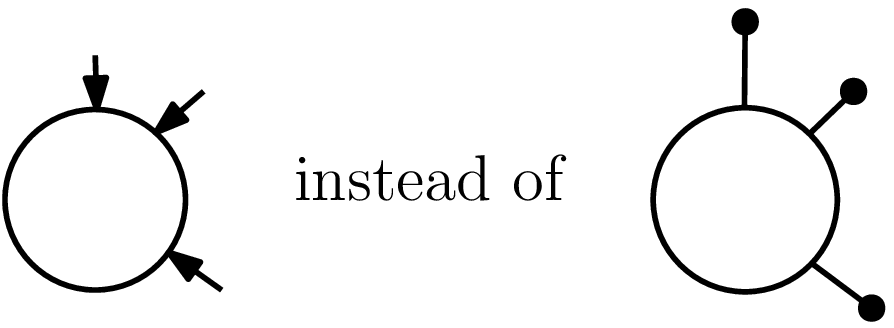}
\end{center}

So, there will be a difference between the short (\textit{amputated}) edges and the propagator edges. Nevertheless, this tells that the diagrammatic equations we derived before are still intact, and therefore, for the sake of simplicity, we shall now and forth ignore the hat notation and will often suppress the orientation for propagators. 

One can verify that, due to the conservation of momenta, the amplitude for a graph $\Gamma$ given by 

\[\Gamma\;=\;\raisebox{-0.52cm}{\includegraphics[scale=0.7]{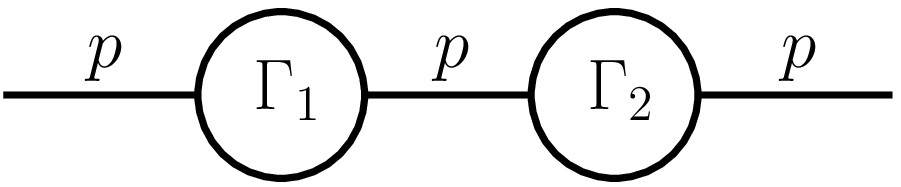}}\;,\]

is \[A(\Gamma;p)=G_0(p)\;A(\Gamma_1;p)\;G_0(p)\;A(\Gamma_2;p)\;G_0(p).\]

This leads to a nice reduction that distinguishes the momentum formulation, namely, we can only focus on bridgeless graphs. One may think of this step as similar to the sufficiency of considering connected graphs made earlier.

\begin{dfn}
A Feynman graph is said to be \textit{one-particle irreducible}, or $1PI$, if it is $2$-edge-connected in the sense of graph theory (has no bridges). That is, a $1PI$ diagram is a graph that stays connected after the removal of any single edge.
\end{dfn}
For example, the graphs 
\begin{center}
    \includegraphics[scale=0.8]{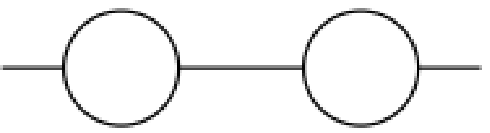} \raisebox{0.41cm}{\;\;and\;\;} \raisebox{0.05cm}{\includegraphics[scale=0.8]{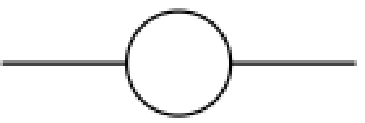}}
\end{center}are not $1PI$, whereas
\begin{center}
    \includegraphics[scale=0.7]{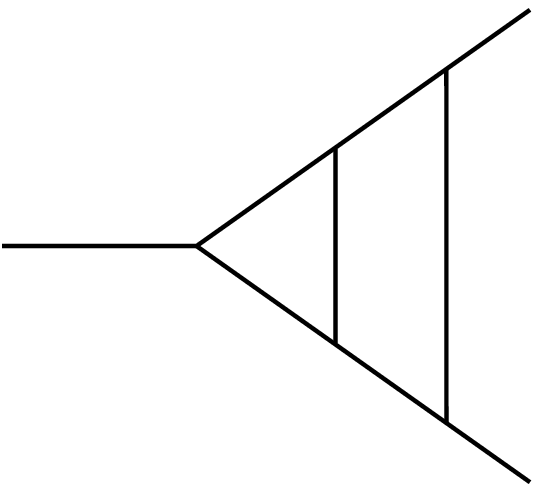}
\end{center}
is a $1PI$ diagram. It is important here to notice that one of the graphs is not $1PI$ because it has long external edges, which are read as propagators. The $1PI$ diagram however is amputated and all the external legs are not propagators. Thus, every connected graph can be decomposed into a tree whose vertices are $1PI$ diagrams, this observation has been used for example in \cite{kjm1,kjm2} to give a combinatorial meaning of the Legendre transform.

For our purposes, this implies that we only need to consider the \text{proper} or \textit{1PI Green functions}:

\[G_{1PI}^{(m)}(p_1,\ldots,p_m)=\underset{\Gamma \;\text{is}\;1PI}{\underset{E(\Gamma)=m}{\sum}} \lambda^{v(\Gamma)}\;\displaystyle\frac{\hbar^{\ell(\Gamma)}}{\text{Sym}(\Gamma)}\;A(\Gamma;p_1,\ldots,p_m),\]

where the sum is over $1PI$ diagrams with $m$ amputated edges. The explicit relation on the level of Green functions is a result in \cite{green}.

\section{Renormalization}

\subsection{Ultraviolet and Infrared Divergencies}
We have seen earlier that the amplitudes, or the analytic versions of the Feynman diagrams, are repeated integrals of products involving the propagator $G_0$ and the source $J$. Now recall that $G_0(x)$ was a distribution of $x$ and has a singularity at $x=0$, which means that higher powers of $G_0$ will be undefined at that point. In the momentum formulation, this leads to a divergency of integrals containing powers of $G_0$. 

\textbf{Example:} Let us consider the Klein-Gordon field in dimension $D$, and consider the amplitude for the simplest case with one loop and two external half edges, say,

\begin{center}
    $\Gamma$\;=\;\;\raisebox{-0.39cm}{\includegraphics[scale=0.78]{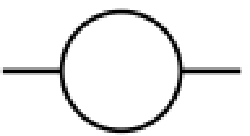}}.
\end{center}

In space-time formulation we get the term $G_0(u-v)^2$ in the integral, and in momentum space we assume that the momentum is $q$ and that $|q|$ is its module in spherical coordinates. The amplitude is 

\[A(\Gamma;p)=\int \displaystyle\frac{1}{q^2+m^2}\;\frac{1}{(p-q)^2+m^2}\;\frac{d^Dq}{(2\pi)^D},\]
which behaves like 

\[\int_{|q|_\text{min}}^\infty \displaystyle\frac{1}{|q|^4}\;{d|q|^D} \;\sim\;\int_{|q|_\text{min}}^\infty \displaystyle\frac{1}{|q|^{4-(D-1)}}\;{d|q|}  \;.\]

This tells that the integral converges if and only if $4-(D-1)>1$, i.e. $D<4$. So the amplitude diverges whenever $D\geq 4$. 

In this section we will be dealing only with \textit{ultraviolet} divergencies. A divergency is  ultraviolet if it occurs as $|q|\rightarrow \infty$ where $q$ is one of the integrated momenta; whereas an \textit{infrared} divergency occurs as $|q|\rightarrow|q|_{\text{min}}$. Infrared divergencies often appear for a vanishing mass where $|q|_{\text{min}}=0$.

\subsection{The Problem of Renormalization}

The ultimate goal in physics is set to obtaining accurate calculations of predictions. for physics, it is more important to use a successful  mathematical model than to prove how much flawless it is. Sometimes a physicist encounters inaccessible values, a divergent integral for example, and in that case the goal will be reset to be obtaining the largest amount of  useful information from this divergent quantity. This situation occurs with Feynman integrals. The amplitude of a Feynman diagram is often divergent, and it has to be \textit{renormalized} to give access to a useful finite value. See \cite{aless, renorm} for a more comprehensive reading.

Thus, given a Feynman graph $\Gamma$ with a divergent amplitude $A(\Gamma)$ (note that we omit the momenta for simplicity), the \textit{renormalization} process eventually gives us a modified finite amplitude $A^{ren}(\Gamma)$, which we call the \textit{renormalized amplitude}. A theory is \textit{normalizable} if it requires satisfying only a finite list of conditions in order to have finite renormalized amplitudes for all Feynman graphs. 

The study of renormalization led to the recognition of the relation between loops and divergence \cite{renorm,conneskreimer}.

\begin{dfn}[Superficial Degree of Divergence]\label{spdivdefinition}
The \textit{superficial degree of divergence} of a graph $\Gamma$ is the integer $\omega(\Gamma)$ for which the amplitude transforms as 
\[A(\Gamma)\;\longrightarrow \;t^{\omega(\Gamma)}\;A(\Gamma)\] when the integrated momenta are transformed as $q_i\longrightarrow tq_i $. The superficial degree of divergence is also  found to be \[\omega(\Gamma)=D\;\ell(\Gamma)-\sum_a w(a),\]where the sum is over all the \textit{power counting weights} (Definition \ref{power counting weight})  of vertices and internal edges in $\Gamma$ determined by the  QFT theory considered, and where $\ell(\Gamma)$ is the number of loops in $\Gamma$.
\end{dfn}

\begin{rem}
For example, since the propagator $G_0$ has denominator of degree $2$ in momentum, whereas the vertices has degree $0$, we will have 

\begin{equation}
    w\big(\;\raisebox{0.1cm}{\includegraphics[scale=0.4]{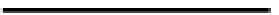}}\;\big)=2
    \quad\text{and}\quad w\big(\;\raisebox{-0.4cm}{\includegraphics[scale=0.34]{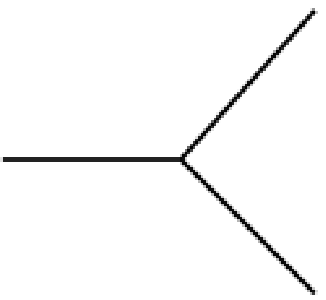}}\;\big)=0
\end{equation}
\end{rem}

If $\omega(\Gamma)\geq 0$ we say that $\Gamma$ is \textit{divergent}; and when equality holds we say that $\Gamma$ is \textit{logarithmically divergent}. These names are justified by thinking of $\omega(\Gamma)$ as being the difference \[ (\text{power of momentum in the numerator}) - (\text{power of momentum in the denominator})\] in $A(\Gamma)$.

\begin{rem}
\textbf{But why is it `superficial' though?} We said that if $\omega(\Gamma)\geq 0$ then the graph is divergent, but is it an `only if'? This is the crucial part. In the case of graphs with one loop this will be an `if and only if'; however, for a many-loop graph $\Gamma$, it may happen that the $\omega(\Gamma)< 0$ but it contains a divergent subgraph (due to the excess of loops), which should still destroy the amplitude for $\Gamma$. This is the reason it is the `superficial', and not the `actual', degree of divergence. 
\end{rem}

\subsubsection{Regularized Amplitudes:}

Before going through the renormalization of $A(\Gamma)$ we may need a finite value to work with, and for whom a certain limit shall reproduce the divergence. The \textit{regularized amplitude} $A_z(\Gamma)$ is chosen such that 

\[A^{ren}(\Gamma)=\lim_{z\rightarrow z_0} A^{ren}_z(\Gamma),\] for some $z_0$.

We are typically considering an amplitude that looks like
\[A(\Gamma)=\int I(\Gamma;q_1,\ldots,q_m)\;\Pi_i\frac{d^Dq_i}{(2\pi)^D},\]
and displays an ultraviolet divergence. So, an example for a regularized amplitude can be obtained by \[A_z(\Gamma)=\int_{|q_i|\leq z} I(\Gamma;q_1,\ldots,q_m)\;\Pi_i\frac{d^Dq_i}{(2\pi)^D},\]
which gives the original problem as $z\rightarrow \infty$.

\subsection{The BPHZ Scheme}\label{BP}

\subsubsection{Single Loop BPHZ:}

Consider the amplitude \[A(\Gamma)=\int I(\Gamma;\textbf{p};q) \;\frac{d^Dq}{(2\pi)^D}\]
of a divergent $1PI$ diagram $\Gamma$ with one loop and superficial degree of divergence $\omega(\Gamma)\geq 0$, where $\textbf{p}=(p_1,\ldots,p_m)$ are the momenta labels for the external legs of $\Gamma$.

Now let $T^{\omega(\Gamma)}$ be the operator that gives a truncated Taylor series up to degree $\omega(\Gamma)$ in the momenta variables $\textbf{p}=(p_1,\ldots,p_m)$ around $\textbf{p}=0$. The Bogoliubov-Parasiuk-Hepp-Zimmermann (BPHZ) subtraction scheme for renormalization gives 

\[A^{ren}(\Gamma)=\int \bigg(I(\Gamma;\textbf{p};q)- T^{\omega(\Gamma)}[I(\Gamma;\textbf{p};q)]\bigg)\;\frac{d^Dq}{(2\pi)^D}.\]

This value is finite as proven by Bogoliubov and Parasiuk, see \cite{bogoliubov,bogoliubovv,bogoliubovvv,bogoliubovvvv}. 

Fixing one regularized amplitude $A_z$, the renormalized amplitude is 

\begin{align*}
    A_z^{ren}(\Gamma;\textbf{p})
    &=A_z(\Gamma;\textbf{p})-T^{\omega(\Gamma)}[A_z(\Gamma;\textbf{p})]\\
    &=A_z(\Gamma;\textbf{p})-\left\{\left.\int I(\Gamma)\frac{d^Dq}{(2\pi)^D}\right|_{\textbf{p}=0}+\sum p_i^\mu\left.\int \frac{\partial I(\Gamma)}{\partial p_i^\mu}\frac{d^Dq}{(2\pi)^D}\right|_{\textbf{p}=0}+\cdots \right\},
\end{align*}

where the terms in the curly brackets continue up to degree $\omega(\Gamma)$ in momenta.
The terms in the curly brackets are called the \textit{(local) counterterms} of $\Gamma$, and they are to be denoted as 
\begin{equation}\label{counterterms}
    z_n(\Gamma)=\displaystyle\frac{-1}{n!}\left.\int \partial^n_\textbf{p} I(\Gamma)\frac{d^Dq}{(2\pi)^D}\;\right|_{\textbf{p}=0}.
\end{equation}

\begin{exm}

Consider the graph 
\begin{center}
    $\Gamma$\;=\;\raisebox{-0.39cm}{\includegraphics[scale=0.78]{Figures/nn.eps}}.
\end{center} as before, and let us assume that $D$ is larger than usual this time to increase the risk of superficial divergence, say $D=6$. In this case $\omega(\Gamma)=6\times 1-2\times 2- 2\times 0=2$, and the graph is divergent.

The amplitude is \[A(\Gamma;p)=\int \displaystyle\frac{1}{q^2+m^2}\;\frac{1}{(p-q)^2+m^2}\;\frac{d^6q}{(2\pi)^6},\]
and we should calculate the renormalized amplitude 

\begin{align*}
A^{ren}(\Gamma;p)
&=A(\Gamma;p)-T^2[A(\Gamma;p)]\\
&=A(\Gamma;p)+\{z_0(\Gamma)+p \;z_1(\Gamma)+p^2\; z_2(\Gamma)\},\end{align*}
by working out the counterterms.

The counterterms are as follows

\begin{align*}
z_0(\Gamma)&=-\int \displaystyle\frac{1}{(q^2+m^2)^2}\;\frac{d^6q}{(2\pi)^6},\\
z_1(\Gamma)&=-\int \displaystyle\frac{2q}{(q^2+m^2)^3}\;\frac{d^6q}{(2\pi)^6}=0,\\
z_2(\Gamma)&=-\int \displaystyle\frac{3q^2-m^2}{(q^2+m^2)^4}\;\frac{d^6q}{(2\pi)^6}.
\end{align*}

and so we get 

\[A^{ren}(\Gamma;p)=\int \displaystyle\frac{4p^3q^3-3p^4q^2-4m^2p^3q+m^2p^4}{(q^2+m^2)^4((p-q)^2+m^2)}\;
\frac{d^6q}{(2\pi)^6}.\]

The integrand has the order $1/|q|^7$ as $|q|\rightarrow\infty$. Thus $A^{ren}(\Gamma;p)$  behaves like 
\[\int_{|q|_\text{min}}^\infty \displaystyle\frac{1}{|q|^7}\;{d|q|^6} \;\sim\;\int_{|q|_\text{min}}^\infty \displaystyle\frac{1}{|q|^2}\;{d|q|}= \left.\displaystyle\frac{-1}{|q|}\right|_{|q|_\text{min}}^\infty=\frac{1}{|q|_\text{min}} \;.\]

\end{exm}

\subsubsection{Many Loops BPHZ:}

The idea for tackling the many-loops case is to recursively renormalize the graph by renormalizing all of its divergent subgraphs first.

\begin{dfn}[Contractions and Residue of a Graph]\label{contr}
Let $\gamma$ be a subgraph of a graph $\Gamma$. A new graph $\Gamma/\gamma$ can be obtained by contracting all the internal edges in $\gamma$. $\Gamma/\gamma$ is called a \textit{contraction} or a \textit{cograph}. Notice that $\gamma$ transforms into either a  vertex (\textit{vertex-type}) or a propagator (\textit{edge-type}).  The \textit{residue} res$(\Gamma)$ of a graph $\Gamma$ is the graph obtained by shrinking all internal edges so that only the external leg structure of $\Gamma$ survives.
\end{dfn}

For example,

\begin{equation}
    \text{res}\bigg(\raisebox{-1.4cm}{\includegraphics[scale=0.27]{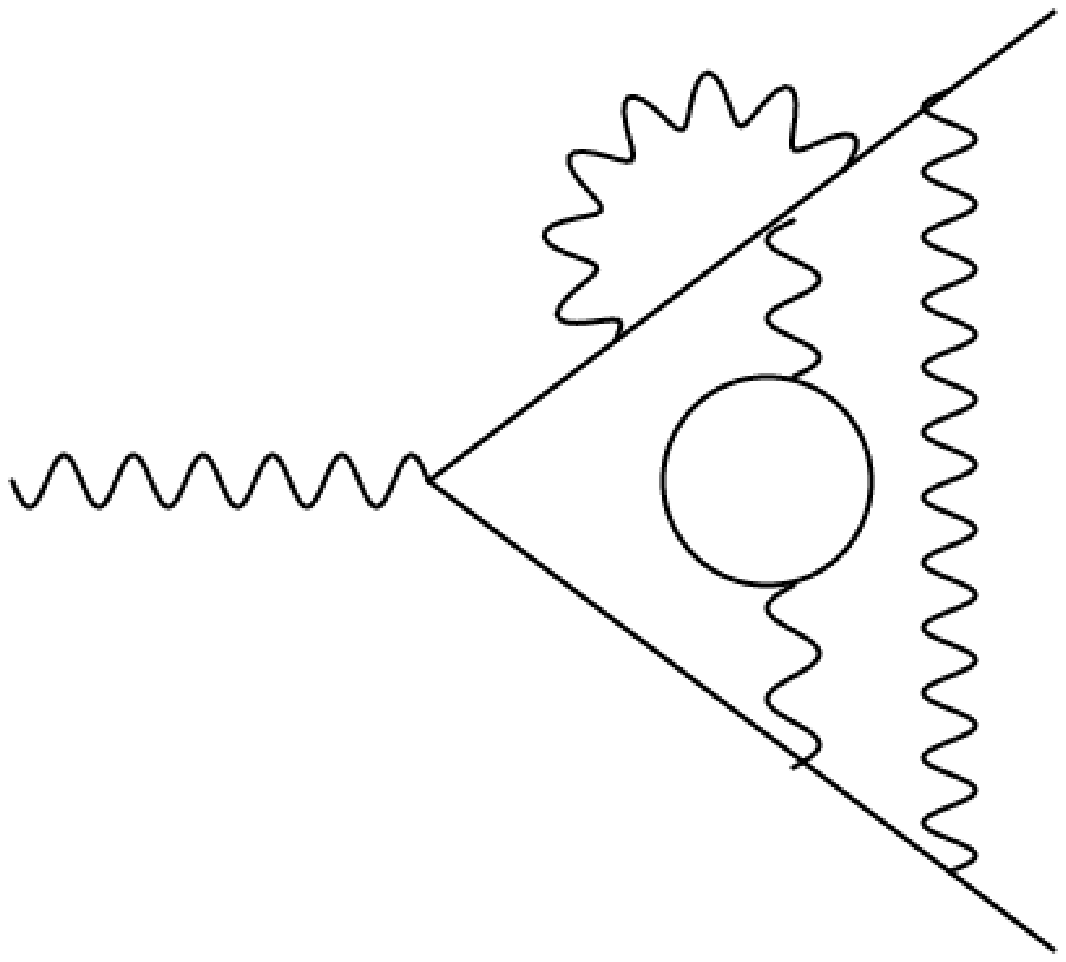}}\bigg)
    \;=\;\text{res}\bigg(\raisebox{-1.28cm}{\includegraphics[scale=0.27]{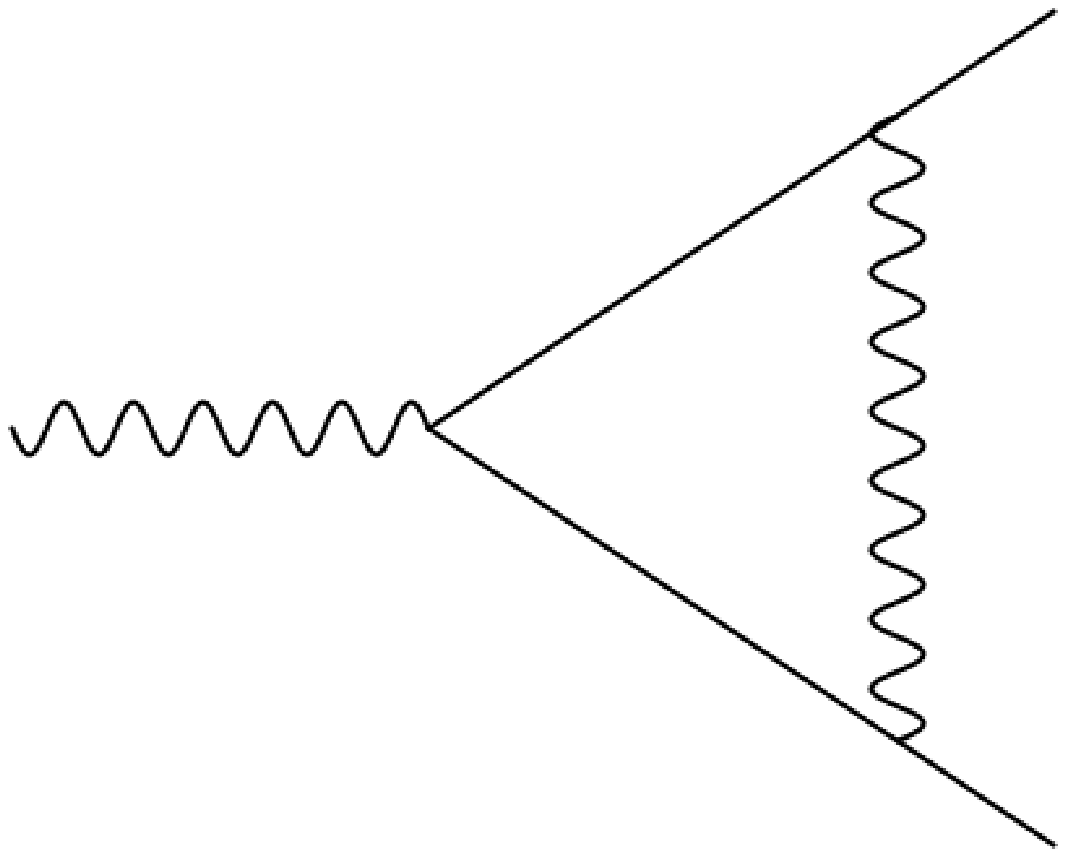}}\bigg)
    \;=\;\raisebox{-0.7cm}{\includegraphics[scale=0.3]{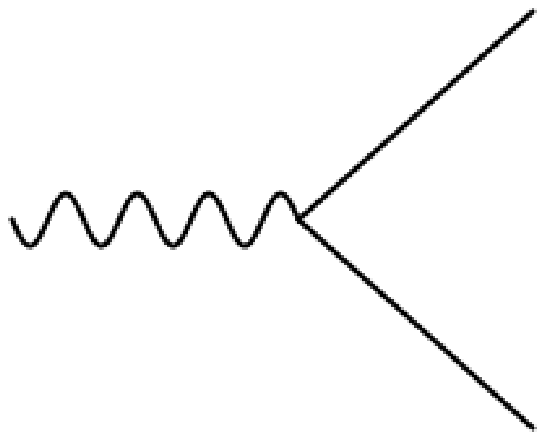}}.
\end{equation}

Now, let $\Gamma$ be a $1PI$ diagram with loop number $\ell$. Assume that either $\omega(\Gamma)\geq0$ or that $\Gamma$ contains divergent subgraphs. Also assume that $\textbf{q}=(q_1,\ldots,q_\ell)$ are the momenta integrated in the amplitude, and as before $\textbf{p}$ is for the external momenta.
The BPHZ scheme for a graph $\Gamma$ with many loops is described as

\begin{equation}
    A^{ren}(\Gamma)=\int \bigg(I^{prep}(\Gamma;\textbf{q})- T^{\omega(\Gamma)}[I^{prep}(\Gamma;\textbf{q})]\bigg)\;\frac{d^Dq_1}{(2\pi)^D}\cdots\frac{d^Dq_\ell}{(2\pi)^D},
\end{equation}

where $I^{prep}(\Gamma) $ is the \textit{prepared} version of $I(\Gamma)$ in which all the divergent subgraphs have been renormalized. This term is defined recursively as 

\begin{equation}
    I^{prep}(\Gamma;\textbf{q})=I(\Gamma;\textbf{q})+\sum_{\gamma_i}\prod_i \bigg(-T^{\omega(\gamma_i)}[I^{prep}(\gamma_i;\textbf{q}_i)]\bigg) \;\displaystyle\frac{I(\Gamma;\textbf{q})}{\prod_i I(\gamma_i;\textbf{q}_i)},
\end{equation}

where the subgraphs considered by the sum are all the disjoint $1PI$ divergent subgraphs of $\Gamma$ (also called the \textit{subdivergences} of the graph). This was partially proven by Bogoliubov and Parasiuk in \cite{bogoliubov} in 1957, then developed by Hepp in 1966 \cite{bogoliubovvvv}, and finally established by Zimmermann in 1969 \cite{bogoliubovvvv}. In 1998, D. Kreimer \cite{kreimerr, kreimerrr, conneskreimer} could show that the BPHZ renormalization scheme is captured by a combinatorial Hopf algebra. This opened the door for many questions about the interaction between  QFT and other abstract areas in mathematics.

By a re-indexing argument  one can also write 

\begin{equation}
     I^{prep}(\Gamma;\textbf{q})=I(\Gamma;\textbf{q})+\sum_{\gamma_i}\prod_i \bigg(-T^{\omega(\gamma_i)}[I^{prep}(\gamma_i;\textbf{q}_i)]\bigg) \;I(\Gamma/\{\gamma_i\};\textbf{q}'),
\end{equation}

counterterms are found similarly (for details see \cite{aless}).

\begin{exm}\label{divex}

We Let $D=6$ and consider the graph 
\[\Gamma\;=\;\raisebox{-1.6cm}{\includegraphics[scale=0.45]{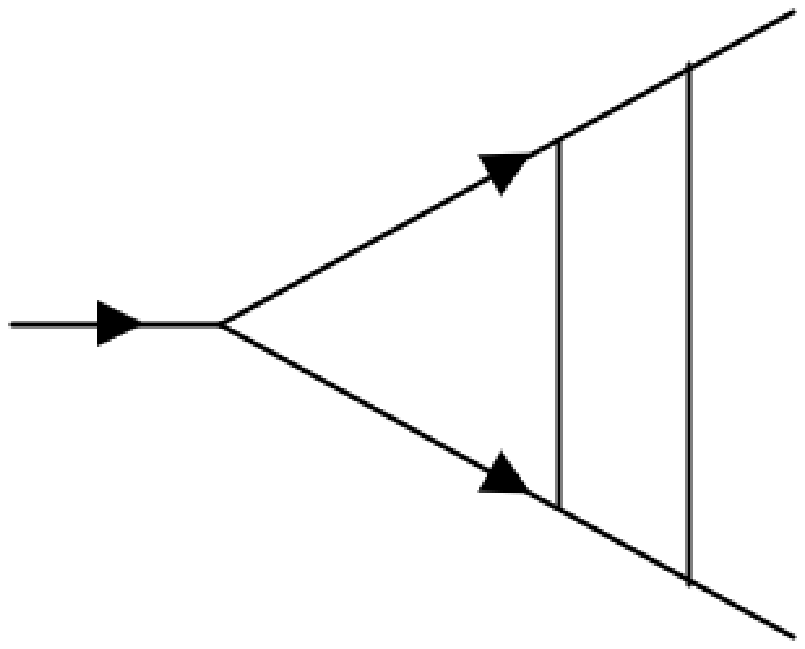}}\;\;,\]

whose amplitude is given by 
\begin{align*}
A(\Gamma;p_1,p_2)=\int& \displaystyle\frac{1}{q_1^2+m^2}\;\frac{1}{(p_1-q_1)^2+m^2}\;\frac{1}{(q_1-q_2)^2+m^2}\\
&\times\frac{1}{q_2^2+m^2}\;\frac{1}{(p_1-q_2)^2+m^2}\;\frac{1}{(q_2-p_2)^2+m^2}\;\frac{d^6q_1}{(2\pi)^6}\frac{d^6q_2}{(2\pi)^6}.\end{align*}

Notice that $\omega(\Gamma)=6\times2-2\times6-2\times0=0$, and so the graph has a logarithmic superficial divergence. Moreover, the following subgraphs are inside $\Gamma$:

\begin{itemize}
    \item The $1PI$ subgraph $\;\;\gamma_1=$\; \raisebox{-1cm}{\includegraphics[scale=0.3]{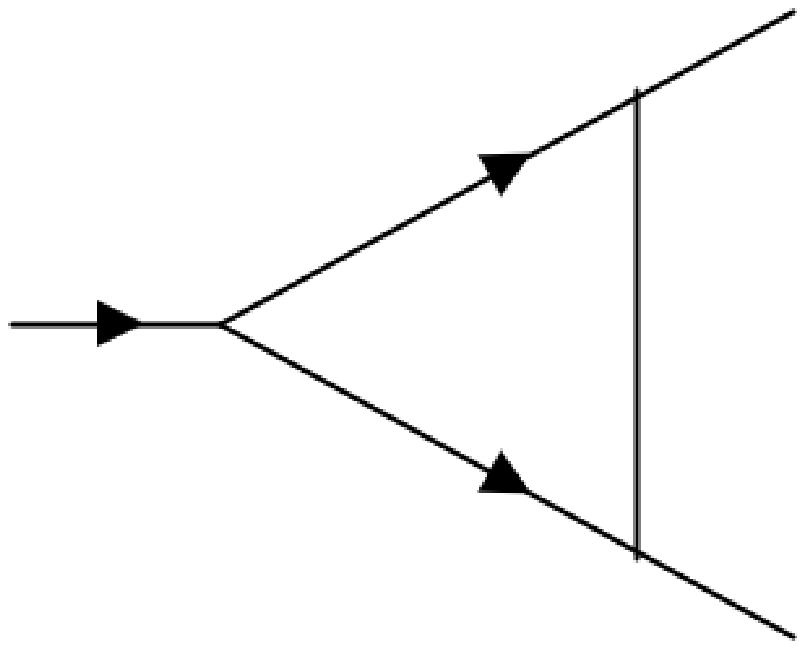}} \;, with one more logarithmic divergence.
    
    \item The $1PI$ subgraph $\;\;\gamma_2=$\; \raisebox{-1.2cm}{\includegraphics[scale=0.3]{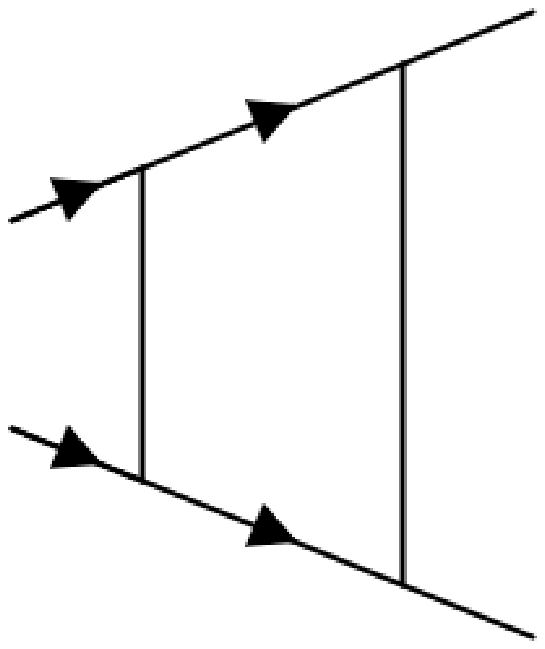}} \;, which is convergent and hence not considered in the BPHZ recursion.
\end{itemize}
    
Thus, the prepared amplitude shall be

\[I^{prep}(\Gamma;p_1,p_2;q_1,q_2)=I(\Gamma)- T^0[I(\gamma_1)] \;I(\Gamma/\{\gamma_1\}),\]

with $$-T^0[I(\gamma_1)]=-I(\gamma_1;p_1,q_2;q_1)|_{p_1=q_1=q_2=0}=\displaystyle\frac{-1}
{(q_1^2+m^2)^3},$$

and \[I(\Gamma/\gamma_1;p_1,p_2;q_2)=\displaystyle\frac{1}{q_2^2+m^2}\;\frac{1}{(p_1-q_2)^2+m^2}\;\frac{1}{(q_2-p_2)^2+m^2}.\] 

\end{exm}

In the next section we will see how the BPHZ subtraction scheme can be recognized in a purely algebraic method in terms of Hopf algebras.

\section{The Hopf Algebra of Feynman Diagrams}

This section is a quick review of the algebraic treatment of renormalization in terms of Hopf algebras. The definitions in this section will be needed to give sense of some of the expressions that our results are related to. We have seen what  renormalization is about analytically in the overview of the work by Bogoliubov, Parasiuk, Hepp, and Zimmermann around 1960's. Almost five decades later, D. Kreimer \cite{kreimerr} showed that the BPHZ scheme is captured by Hopf algebras and the recursive definition of the antipode. In this approach, the Feynman diagrams are used to define a connected graded and commutative Hopf algebra. For an extensive treatment of Hopf algebras see \cite{sweedler}, and see \cite{kurusch,conneskreimer,kreimerr,kreimerrr, manchonhopf, renorm} for a spectrum of results and developments of the approach in renormalization in QFT.

\subsection{Basic Definitions of Hopf Algebras}
We let $\mathbb{K}$ be an infinite  field (characteristic 0). The unit of an algebra will be treated as a map in the sense of category theory. $\mathcal{L}(V,W)$ will mean the group of $\mathbb{K}$-linear maps from the $\mathbb{K}$-vector space $V$ to the $\mathbb{K}$-vector space $W$.

\begin{dfn}[Associative unital algebra]\label{aslg}
An \textit{ associative unital $\mathbb{K}$-algebra} $(A,m,\mathbb{I})$ is a $\mathbb{K}$-vector space $A$ together with two linear maps $m:A\otimes A\rightarrow A$ (product) and $\mathbb{I}:\mathbb{K}\rightarrow A$ (unit) such that :
\begin{align*}
    m\circ(\text{id}\otimes m)&=m\circ(m\otimes\text{id})\\
    m\circ(\mathbb{I}\otimes\text{id})&=m\circ(\text{id}\otimes \mathbb{I}).
\end{align*}
Furthermore, the algebra is said to be \textit{commutative} if $m=m\circ\tau$, where $\tau$ is the twist map $a\otimes b\mapsto b\otimes a$.\end{dfn}

The image $\mathbb{I}(1)$ of $1$ under the unit map $\mathbb{I}$ will often be denoted also by $\mathbb{I}$ with no confusion.

In terms of commutative diagrams, $A$ is an associative unital algebra if the following diagrams commute

\begin{center}
\begin{minipage}{0.5 \textwidth}
\xymatrixcolsep{5pc}\xymatrix{
A\otimes A\otimes A \ar[d]_{\text{id}\;\otimes\; m} \ar[r]^{m\;\otimes\;\text{id}} & A\otimes A \ar[d]^m \\
A\otimes A \ar[r]_m          & A }\end{minipage}
\begin{minipage}{0.5 \textwidth}\quad
\xymatrixcolsep{5pc}\xymatrix{
\mathbb{K}\otimes A\ar[rd]_{\cong} \ar[r]^{\mathbb{I}\;\otimes\;\text{id}} & A\otimes A \ar[d]^m &A\otimes\mathbb{K} \ar[l]_{\text{id}\;\otimes\; \mathbb{I}}\ar[ld]^{\cong}\\
&A& }
\end{minipage}
\end{center}

The categorical dual then becomes

\begin{dfn}[Coassociative counital coalgebra]\label{coaslg}
A \textit{ coassociative counital $\mathbb{K}$-coalgebra} $(C,\Delta,\hat{\mathbb{I}})$ is a $\mathbb{K}$-vector space $C$ together with two linear maps $\Delta:C\rightarrow C\otimes C$ (coproduct) and $\hat{\mathbb{I}}:C\rightarrow \mathbb{K}$ (counit) such that :
\begin{align*}
    (\text{id}\otimes \Delta)\circ\Delta&=(\Delta\otimes\text{id})\circ \Delta\\
    (\hat{\mathbb{I}}\otimes\text{id})\circ\Delta&=(\text{id}\otimes \hat{\mathbb{I}})\circ\Delta.
\end{align*}
Furthermore, the coalgebra is said to be \textit{cocommutative} if $\Delta=\tau\circ\Delta$.\end{dfn}

In terms of commutative diagrams, $A$ is an associative unital algebra if the following diagrams commute

\begin{center}
\begin{minipage}{0.5 \textwidth}
\xymatrixcolsep{5pc}\xymatrix{
C\otimes C\otimes C  & C\otimes C \ar[l]_{\Delta\;\otimes\;\text{id}}  \\
C\otimes C \ar[u]^{\text{id}\;\otimes\; \Delta}         & C\ar[l]^\Delta \ar[u]_\Delta }\end{minipage}
\begin{minipage}{0.5 \textwidth}\quad
\xymatrixcolsep{5pc}\xymatrix{
\mathbb{K}\otimes C  & C\otimes C \ar[l]_{\hat{\mathbb{I}}\;\otimes\;\text{id}}\ar[r]^{\text{id}\;\otimes\; \hat{\mathbb{I}}} &C\otimes\mathbb{K} \\
&C\ar[ur]_{\cong}\ar[u]_\Delta\ar[ul]^{\cong}& }
\end{minipage}
\end{center}

\textit{Sweedler's } notation for the coproduct is often useful: $\Delta (x)=\sum_x\;x'\otimes x''$.

\begin{dfn}[Algebra morphism]
Let $(A,m_A,\mathbb{I}_A)$ and $(B,m_B,\mathbb{I}_B)$ be two associative unital $\mathbb{K}$-algebras. A $\mathbb{K}$-linear map $\varphi:A\longrightarrow B$ is an \textit{algebra morphism}  if 
\begin{align*}
    \varphi\circ\mathbb{I}_A&=\mathbb{I}_B, \text{\;and\;}\\
    \varphi\circ m_A&=m_B\circ(\varphi\otimes \varphi). 
\end{align*}
\end{dfn}

Dually, one defines

\begin{dfn}[Coalgebra morphism]
Let $(C,\Delta_C,\hat{\mathbb{I}}_C)$ and $(D,\Delta_D,\hat{\mathbb{I}}_D)$ be two coassociative counital $\mathbb{K}$-coalgebras. A $\mathbb{K}$-linear map $\psi:C\longrightarrow D$ is a  \textit{coalgebra morphism}  if 
\begin{align*}
    \hat{\mathbb{I}}_D\circ\psi&=\hat{\mathbb{I}}_C, \text{\;and\;}\\
    \Delta_D\circ\psi&=(\psi\otimes \psi)\circ\Delta_C. 
\end{align*}
\end{dfn}

\begin{dfn}[Bialgebras]
A \textit{$\mathbb{K}$-bialgebra} $(B,m,\mathbb{I},\Delta, \hat{\mathbb{I}})$ is a $\mathbb{K}$-vector space such that $(B,m,\mathbb{I})$ is an algebra and $(B,\Delta, \hat{\mathbb{I}})$ is a coalgebra and that the two structures are compatible in the sense that  $m$ (and $\mathbb{I}$) is a coalgebra morphism and $\Delta$ (and  $\hat{\mathbb{I}}$) is an algebra morphism.

Note that only one of the compatibility conditions is enough; one can verify that
$m$ is a coalgebra morphism if and only if $\Delta$ is an algebra morphism.
\end{dfn}

Note that in a bialgebra it must be that $ \hat{\mathbb{I}}(\mathbb{I})=1$ and vanishes for all other elements.

\begin{dfn}[Hopf algebras and the antipode]\label{hopf}
A \textit{Hopf algebra}  $(\mathcal{H},m,\mathbb{I},\Delta, \hat{\mathbb{I}},S)$ is a $\mathbb{K}$-bialgebra $(\mathcal{H},m,\mathbb{I},\Delta, \hat{\mathbb{I}})$
together with a linear map $S:\mathcal{H}\rightarrow \mathcal{H}$ such that  

\[m\circ(S\otimes\text{id})\circ\Delta=\mathbb{I}\circ\hat{\mathbb{I}}=m\circ(\text{id}\otimes S)\circ \Delta.\]

The map $S$ is called the \textit{antipode} of the Hopf algebra.
\end{dfn}

Diagrammatically this is equivalent to the following diagram being commutative:

\begin{center}
\begin{minipage}{0.5 \textwidth}
\xymatrixcolsep{5pc}\xymatrix{
\mathcal{H}\otimes \mathcal{H}\ar[r]^{S\;\otimes\;\text{id}}& 
\mathcal{H}\otimes \mathcal{H}\ar[dr]^m& \\
     \mathcal{H} \ar[d]_\Delta \ar[u]^\Delta\ar[r]^{\hat{\mathbb{I}}} &\mathbb{K}\ar[r]^{\mathbb{I}} &            \mathcal{H}\\
\mathcal{H}\otimes \mathcal{H}\ar[r]^{\text{id}\;\otimes\;S}& \mathcal{H}\otimes\mathcal{H}\ar[ur]_m
}
\end{minipage}

\end{center}

\begin{dfn}[Convolution Product]\label{convo}
Let $f,g$ be two linear maps in $\mathcal{L}(\mathcal{H},\mathcal{H})$. Then their \textit{convolution product} is defined as 
\[f\ast g := m\circ(f\otimes g)\circ \Delta.\]

This product gives again a linear map on $\mathcal{H}$. It can be shown that $(\mathcal{L}(\mathcal{H},\mathcal{H}),\ast,\; \mathbb{I}\circ\hat{\mathbb{I}})$ becomes an algebra. Moreover, $f\circ S$ is the inverse of $f$ with respect to the convolution product, and in that sense, the antipode $S$ may be thought of as the $\ast$-inverse of the identity map $\text{id}_\mathcal{H}$.
\end{dfn}

\subsubsection{Filtration and Connectedness of Hopf Algebras}

\begin{dfn}[Gradedness and Connectedness]\label{Aug}
A Hopf algebra $\mathcal{H}$ is said to be \textit{graded} ($\mathbb{Z}_{\geq0}$-graded to be precise) if it decomposes into a direct sum 
$\mathcal{H}=\oplus_{n=0}^\infty\mathcal{H}_n$, such that

\begin{align*}
m(\mathcal{H}_n\otimes\mathcal{H}_m)&\subseteq \mathcal{H}_{n+m},\\  
\Delta(\mathcal{H}_n))&\subseteq \oplus_{k=0}^n\mathcal{H}_{k}\otimes\mathcal{H}_{n-k},\\
S(\mathcal{H}_n)&\subseteq \mathcal{H}_n.
\end{align*}
If, in addition, $\mathcal{H}_0\cong \mathbb{K}$, the Hopf algebra is said to be \textit{connected}.

\end{dfn}
Given a graded Hopf algebra as above, one finds that $\text{ker}\; \hat{\mathbb{I}}=\text{Aug}\mathcal{H}:=\oplus_{n\geq1}^\infty\mathcal{H}_{n}$, called the \textit{augmentation ideal}.

\begin{dfn}[Filtration]
A Hopf algebra $\mathcal{H}$ is \textit{filtered} if there exists a tower of subspaces $\mathcal{H}^n\subseteq \mathcal{H}^{n+1}, n\in \mathbb{N}$, such that 

\begin{align*}
\mathcal{H}&=\sum_{n=0}^\infty\mathcal{H}^n,\\
m(\mathcal{H}^n\otimes\mathcal{H}^m)&\subseteq \mathcal{H}^{n+m},\\  
\Delta(\mathcal{H}^n))&\subseteq \sum_{k=0}^n\mathcal{H}^{k}\otimes\mathcal{H}^{n-k},\\
S(\mathcal{H}^n)&\subseteq \mathcal{H}^n.
\end{align*}
Note that every graduation implies a filtration by taking $\mathcal{H}^n=\oplus_{k=0}^n\mathcal{H}_k$.
\end{dfn}

\begin{dfn}[Primitive and group-like elements]
An element $x\in\mathcal{H}$ is \textit{primitive} if $\Delta(x)=\mathbb{I}\otimes x+x\otimes {\mathbb{I}}$. An element $x$ is \textit{group-like} if $\Delta(x)=x\otimes x$.
\end{dfn}

\subsection{Physical Theories as Combinatorial Classes of Graphs}

Now we are ready to define the renormalization Hopf algebras of Feynman diagrams, but first let us emphasize the combinatorial rephrasing of the physical setup already seen in the previous sections.

All of the enumerative aspects of QFT considered in this thesis are about Feynman diagrams. Feynman diagrams and their Hopf algebras will be key ingredients in the later chapters, although not explicitly affecting the combinatorial problems we consider. We will proceed by defining  combinatorial QFT theories, Feynman graphs, and Feynman rules. Then we will recover some of the concepts of renormalization. Recommended references for similar treatments are \cite{karenbook,manchonhopf, michi}.

The building block for Feynman graphs is going to be \textit{half edges}. An edge is intuitively understood to be formed from two half edges.

\begin{dfn}
A graph (or diagram) $G$ is a set of half edges for which there is 
\begin{enumerate}
    \item a partition $V(G)$ into disjoint classes of half edges, a class $v\in V(G)$ will be called a \textit{vertex};
    \item a collection $E(G)$ of disjoint pairs of half edges. $E(G)$ will be called the set of \textit{internal edges};
    \item half edges that are not occurring in any of the pairs in $E(G)$ will be called \textit{external edges} or \textit{external legs}.
\end{enumerate}

\end{dfn}

The size of a graph will be the size of its set of half edges. Half edges can be labelled or unlabelled, and sometimes we will use many types of half edges to represent a certain physical theory. We will be concerned at some point with graphs which have a prescribed set of external legs. The \textit{loop number} of a graph is the dimension of its cycle space, or in other words the number of independent cycles. There exists, in any graph, a family of independent cycles with each cycle having an edge not occurring in any of the other cycles in the family. The size of the largest such family is the number of independent cycles in the graph. Such a family of cycles can be obtained by starting with a spanning tree and reading off the new cycle created by adding one of the edges, one edge at a time (and with removing any edge added earlier). The loop number will be very important in our later considerations and will express some sort of size for diagrams with a prescribed scheme of external legs.

By $\text{Aut}(G)$ we mean the group of automorphisms (self isomorphisms) of the graph $G$. In the perturbative expansions that we will see, a Feynman diagram  will have a \textit{symmetry} factor of $1/|\text{Aut}(G)|$ (it is more common to write it as $1/\text{Sym}(G)$. We will usually work with unlabelled graphs, nevertheless, the symmetry factor will allow us to use the exponential relation between connected and disconnected objects.

The good thing about many of the aspects of quantum field theory is that they can be transformed into purely combinatorial and enumerative problems. 

\begin{dfn}\label{power counting weight}
A \textit{combinatorial physical theory} consists of 
\begin{enumerate}
    \item a dimension of spacetime (nonnegative integer);
    \item a number of half edge types, and a set of pairs of half edge types, with each pair representing an admissible edge type in the theory (note that the half edge types in one pair are not necessarily distinct nor identical);
    
    \item a collection of multisets of half edge types to define the options for a vertex in the theory;
    
    \item an integer weight for each edge or vertex type, called a \textit{power counting weight}.
    
\end{enumerate}
\end{dfn}

Thus, a graph in a certain theory $T$ will be a graph whose edges are of the types formed by the admissible pairs of $T$, and each of whose vertices is incident to an admissible multiset of half edges. Note that even oriented and unoriented edges can be formed this way: oriented edges arise from an admissible pair of half edges in which the two types are different, whereas unoriented ones arise from pairs with the two types identical.\\

\begin{exm}
\begin{enumerate}
    \item \textbf{QED: Quantum electrodynamics.} In QED there are 3 half edge types: a half photon, a front half fermion, and a back half fermion. The admissible combinations of half edges to form edges are: (1) a pair  of two half photons to give a photon edge, drawn as a wiggly line \raisebox{-0.1cm}{\includegraphics[scale=0.3]{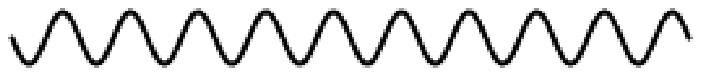}}, with power counting weight 2; and (2) a pair consisting of a front and back halves fermion to give a directed fermion edge \includegraphics[scale=0.8]{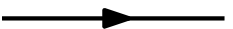}, with power counting weight 1. There is one type of a vertex, namely, a vertex is 3-valent and is incident to one of each half edge type, with power counting weight 0. The spacetime dimension is taken to be 4.
    
    \item \textbf{Yukawa theory:} This theory also has 3 types of half edges: a half meson, a front half fermion, and a back half fermion. The admissible edges are: (1) a meson edge formed by two half mesons (front and back), drawn as \raisebox{-0.1cm}{\includegraphics[scale=0.3]{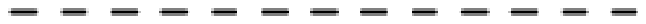}}, and has power counting weight 2; and (2) a pair of a front and back halves fermion to give a directed fermion edge \includegraphics[scale=0.8]{Figures/fermionedge.eps}, with weight 1. Just as in QED, there is one type of vertices, namely, a vertex is 3-valent and is incident to one of each half edge type, with power counting weight 0. The spacetime dimension is taken to be 4. The difference from QED lies in the Feynman rules.
\end{enumerate}
\end{exm}

As mentioned earlier, the significance of quantum field theory is the ability to describe how particles interact and scatter. In an idealized experiment some particles are sent in, they interact and scatter, and then the outcomes are detected. This picture can be visualized as a diagram in which the edges describe propagating particles. The idea then is that, on an atomic scale, we never know what exactly happened and every possible interaction is assigned a probability \textit{scattering amplitude}. This amounts into a weighted sum, known as a \textit{perturbative expansion}.  The probabilities in the theory are computed through what is known as a \textit{Feynman integral}. These integrals encountered by physicists are  often divergent and have to undergo renormalization to retrieve useful information.  As we saw before, Feynman graphs encode these complicated integrals, and the rules for this encoding  in a given QFT are known as \textit{Feynman rules}.

\begin{dfn}[Feynman Graphs]
A Feynman graph in a theory $T$ is combinatorially a graph structure in which edges fall into certain types and vertices are subject to conditions (pertinent to $T$) on the number of edges of a certain types attached to it. A Feynman graph represents an integral through the Feynman rules of the theory, which assigns an integrand factor contribution to every internal edge or vertex. The power counting weights give the degree of an integrated momentum variable. \end{dfn}

\begin{exm}
We have already seen  Example \ref{divex}, in which the integral in Figure \ref{F}  is a divergent Feynman integral and its corresponding Feynman diagram:

\centering
\begin{figure}[H]
   \centering
   \[   \begin{minipage}[h]{0.4\linewidth}
	\vspace{0pt}
	\includegraphics[width=\linewidth]{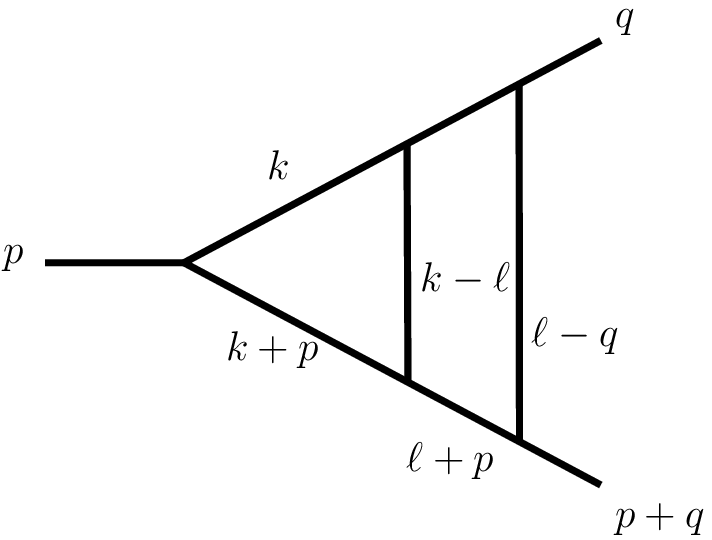}
    \end{minipage}
     = 
   \begin{minipage}[h]{0.5\linewidth}
	\vspace{0pt}
	$\displaystyle \int \int\frac{d^D\ell}{\ell^2(\ell-q)^2(\ell+p)^2} \frac{d^Dk}{k^2(k-\ell)^2(k+p)^2}.$
    \end{minipage}
   \]
   \caption{The Feynman integral of a Feynman graph}\label{F}
\end{figure}

\end{exm}

\subsection{Characters and Cocycles}

\begin{dfn}[Characters]
Let $\mathcal{H}$ be a connected bialgebra and $(A,\cdot,1_A)$ be a $\mathbb{K}$-algebra. A \textit{character} from $\mathcal{H}$ to $A$ is defined to be an algebra morphism with the extra property that $\phi(\mathbb{I})=1_A$. The set of all characters from $\mathcal{H}$ to $A$ is denoted by $G_A^\mathcal{H}$. Further, if $A$ is commutative, $G_A^\mathcal{H}$ becomes a group under convolution product \cite{panzer}. The inverses are denoted $\varphi^{\ast -1}:=\varphi\circ S$.

\end{dfn}

\begin{dfn}
Let $\mathcal{H}$ be a connected bialgebra and $A$ be a commutative algebra that can be written as a direct sum of two vector spaces. A \textit{Birkhoff decomposition} of a character $\phi$ is a pair of characters $\phi_+,\phi_-\in G_A^\mathcal{H}$ such that \[\phi=\phi_-^{\ast -1}\ast\phi_+ \;\text{\;and\;}\; \phi_{\pm}(\mathrm{ker}\hat{\mathbb{I}})\subseteq A_{\pm}.\] 
\end{dfn}

In \cite{conneskreigeom} it was shown that dimensional regularization (viewing the integral over dimension $D-2\epsilon$ and expanding in $\epsilon$) can be studied in terms of characters into the algebra of meromorphic functions in $\epsilon$.

\begin{thm}[\cite{manchonhopf}]\label{birkhoff}

Let $\mathcal{H}$  be a connected filtered Hopf algebra, and let $G_A^\mathcal{H}$ be the group of characters with the convolution product. Then any character $\varphi\in G_A^\mathcal{H}$ has a unique Birkhoff decomposition 
\[\varphi=\varphi^{\ast -1}_-\ast\varphi_+,\] where $\varphi_- ,\varphi_+\in G_A^\mathcal{H}$, with $\varphi_-$ mapping the augmentation ideal into $A_-$, and  with $\varphi_+$ mapping $\mathcal{H}$ into $A_+$. The characters naturally satisfy $\varphi_-(\mathbb{I})=1_A=\varphi_+(\mathbb{I})$ and are defined recursively over the augmentation ideal as 

\begin{align*}
    \varphi_-(x)&=-\pi\big(\varphi(x)+\sum_x\varphi_-(x')\varphi(x'')\big), \;\text{and}\\
    \varphi_+(x)&=(\mathrm{id}-\pi)\big(\varphi(x)+\sum_x\varphi_-(x')\varphi(x'')\big),\end{align*}
    where $\pi$ is the projection of $A$ onto $A_-$, and the sum is making use of Sweedler's notation for the coproduct.

\end{thm}

\begin{dfn}[Bogoliubov map]\label{hopf bog}
The \textit{Bogoliubov map} is  the map $b:G\longrightarrow \text{Hom}(\mathcal{H},A)$ defined recursively by

\[b(\varphi)(x)=\varphi(x)+\sum_x\varphi_-(x')\varphi(x'').\]

In particular, the decomposition in Theorem \ref{birkhoff} is  now seen via the Bogoliubov map as 
\[\varphi_-=-\pi\circ b(\varphi)\;,\qquad\text{and}\qquad\varphi_+=(\mathrm{id}-\pi)\circ b(\varphi).\]
\end{dfn}

 Before starting the next part, it must be noted that this section does not give a full account of the Hopf-algebraic treatment of renormalization. We are only interested in defining the expressions that we will encounter in our problems. The reader can refer to \cite{sweedler} for an in depth account on Hopf algebras. The Hopf algebra of Feynman graphs is also surveyed in the review article of D. Manchon \cite{manchonhopf}.

\subsection{The Hopf Algebra of divergent 1PI Diagrams}\label{hopfalg1PIsection}

Let $T$ be a fixed combinatorial physical theory in the sense of the previous section, and consider the $\mathbb{Q}$-vector space $\mathcal{H}$ generated by the set of disjoint unions of divergent $1PI$ Feynman graphs in the theory $T$, including the empty graph which we denote by $\mathbb{I}$.

We can define a multiplication $m$ on $\mathcal{H}$ to be taking the disjoint union, and the unit is the empty graph $\mathbb{I}$, this makes $(\mathcal{H},m,\mathbb{I})$ a commutative associative algebra.

Now we need to define a compatible coalgebra structure for $\mathcal{H}$.

Recall Definition \ref{contr} of residues and the contraction $\Gamma/\gamma$ of the subgraph $\gamma$ in $\Gamma$. For the sake of a precise general definition in the new terms we have

\begin{dfn}[Contraction of a Subgraph]\label{contractions}
Let $\Gamma$ be a Feynman graph in a theory $T$, and let $\gamma\subseteq \Gamma$ be a subgraph each of whose connected components is $1PI$ and divergent. The \textit{contraction graph} $\Gamma/\gamma$ is constructed as follows:
\begin{enumerate}
    \item A component of $\gamma$ with a vertex residue (external leg structure)  is contracted in $\Gamma$ into a vertex of the same type as the residue.
    
    \item A component of $\gamma$ with an edge residue (external leg structure)  is contracted in $\Gamma$ into an edge of the same type as the residue.
\end{enumerate}
\end{dfn}

Recall Definition \ref{spdivdefinition} of the superficial degree of divergence $\omega(\Gamma)$ of a graph $\Gamma$, and that $\omega(\Gamma)=D\;\ell(\Gamma)-\sum_a w(a)$ where the sum is over all the power counting weights (Definition \ref{power counting weight})  of vertices and internal edges in $\Gamma$ determined by the  QFT theory considered, and where $\ell(\Gamma)$ is the number of loops in $\Gamma$.

\begin{dfn}[Subdivergence]\label{subdiv}
A subgraph $\gamma$ of $\Gamma$ with divergent $1PI$ connected components   is called a \textit{subdivergence}. 
\end{dfn}

Then we define the coproduct as

\begin{dfn}\label{coprod}
The coproduct $\Delta:\mathcal{H}\longrightarrow\mathcal{H}\otimes\mathcal{H}$ is defined for a connected Feynman graph $\Gamma$ to be 
\[\Delta(\Gamma)=\underset{1PI \;\text{subgraphs}}{\underset{\gamma\;\text{product of divergent }}{\underset{\gamma\subseteq\Gamma}{\sum}}}\gamma\otimes\Gamma/\gamma\]

and extended as an algebra morphism.
\end{dfn}

Note that since we are considering graphs that are themselves divergent, the coproduct sum for any element in $\mathcal{H}$ will  always start as 

\[\Delta(\Gamma)=\mathbb{I}\otimes\Gamma+\Gamma\otimes\mathbb{I}+\widetilde{\Delta}(\Gamma).\]
 
The part $\widetilde{\Delta}(\Gamma)$ of the coproduct is called the \textit{reduced coproduct}. 

\begin{dfn}[Primitive Elements]\label{primitivediags}
An element $\Gamma\in\mathcal{H}$ is said to be \textit{primitive } if 
$\widetilde{\Delta}(\Gamma)=0$. That is, ${\Delta}(\Gamma)=\mathbb{I}\otimes\Gamma+\Gamma\otimes\mathbb{I}$.
 In particular, a primitive $1PI$ graph $\Gamma$ is a $1PI$ graph that contains no subdivergences in the sense of Definition \ref{subdiv}.\end{dfn}

For example let us calculate the coproduct 

\begin{flalign*}
    &\Delta\bigg(\;\raisebox{-0.5cm}{\includegraphics[scale=0.3]{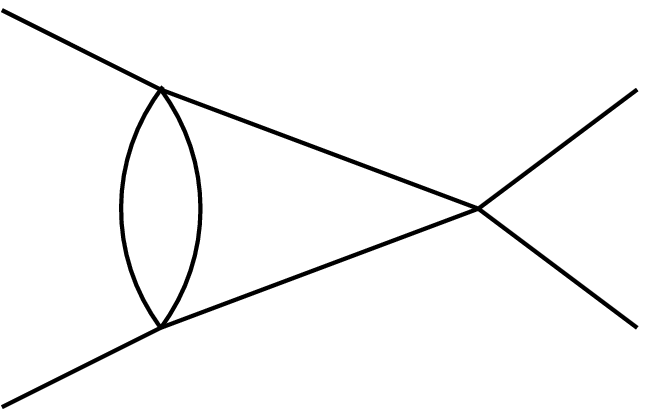}}-\raisebox{-0.30cm}{\includegraphics[scale=0.3]{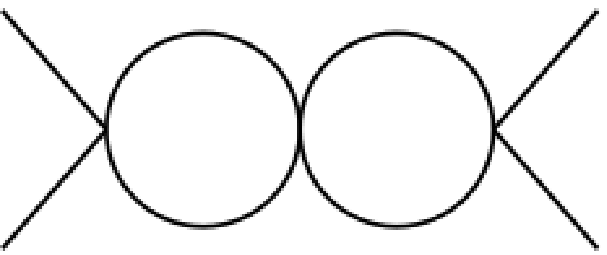}}\;\bigg)\\
    =
    &\;\mathbb{I}\otimes
    \bigg(\;\raisebox{-0.5cm}{\includegraphics[scale=0.3]{Figures/coprodsmall1.eps}}-\raisebox{-0.30cm}{\includegraphics[scale=0.3]{Figures/coproduct3.eps}}\;\bigg)+
  \bigg(\;\raisebox{-0.5cm}{\includegraphics[scale=0.3]{Figures/coprodsmall1.eps}}-\raisebox{-0.30cm}{\includegraphics[scale=0.3]{Figures/coproduct3.eps}}\;\bigg)\otimes\mathbb{I}\;+\qquad\qquad\qquad\qquad\\
    &\;+\;\raisebox{-0.3cm}{\includegraphics[scale=0.3]{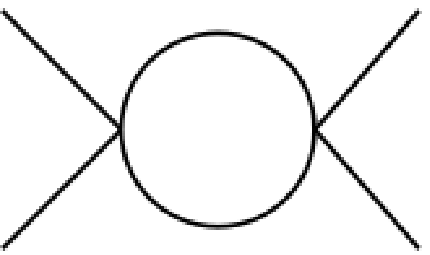}}\;\otimes\;\raisebox{-0.3cm}{\includegraphics[scale=0.3]{Figures/coproduct2.eps}}\;-\;2\;\raisebox{-0.3cm}{\includegraphics[scale=0.3]{Figures/coproduct2.eps}}\;\otimes\;\raisebox{-0.3cm}{\includegraphics[scale=0.3]{Figures/coproduct2.eps}}\\
    =
    &\;\mathbb{I}\otimes\bigg(\;\raisebox{-0.5cm}{\includegraphics[scale=0.3]{Figures/coprodsmall1.eps}}-\raisebox{-0.30cm}{\includegraphics[scale=0.3]{Figures/coproduct3.eps}}\;\bigg)+
  \bigg(\;\raisebox{-0.5cm}{\includegraphics[scale=0.3]{Figures/coprodsmall1.eps}}-\raisebox{-0.30cm}{\includegraphics[scale=0.3]{Figures/coproduct3.eps}}\;\bigg)\otimes\mathbb{I}-
\raisebox{-0.3cm}{\includegraphics[scale=0.3]{Figures/coproduct2.eps}}\otimes\raisebox{-0.3cm}{\includegraphics[scale=0.3]{Figures/coproduct2.eps}}.
\end{flalign*}
   
   Finally, let\;\; $\hat{\mathbb{I}}:\mathcal{H}\longrightarrow\mathbb{Q}$\;\; be the map defined on the empty graph by sending $q\mathbb{I}$ to $q\in\mathbb{Q}$ and sending every other element in $\mathcal{H}$ to zero.
Then it is not hard to prove the following proposition (see \cite{karenbook,cutcut,karenthesis} for a proof)

\begin{prop}
As per the above definitions, $(\mathcal{H},m,\mathbb{I},\Delta,\hat{\mathbb{I}})$ is a bialgebra. Further, if we define a map $S:\mathcal{H}\longrightarrow\mathcal{H}$ recursively by

\begin{align*}
    S(\mathbb{I})&=\mathbb{I}\;, \text{\;and}\\
    S(\Gamma)&=-\Gamma-\underset{1PI \;\text{subgraphs}}{\underset{\gamma\;\text{product of divergent }}{\underset{\mathbb{I}\neq\gamma\neq\Gamma}{\underset{\gamma\subseteq\Gamma}{\sum}}}}\;S(\gamma)\;\Gamma/\gamma,\end{align*}
then $(\mathcal{H},m,\mathbb{I},\Delta,\hat{\mathbb{I}},S)$ becomes a Hopf algebra woth antipode $S$. (Note that the product in the second term is the product $m$ abbreviated). Moreover, the Hopf algebra $\mathcal{H}$ is commutative and is graded by the loop number.
\end{prop}

\begin{rem}\label{core hopf}
In the next section we will broadly see how renormalization is represented in this algebraic context of Hopf algebras. Our job ends with learning the meaning of some of the expressions that will show up again in our problems. It should be noted however that, as expected, this is not the only meaningful appearance of Hopf algebras in quantum field theory. Namely, if the condition of divergence is dropped from the elements summed over in the definitions of the coproduct and the antipode,  we get the so-called  \textit{the core Hopf algebra}, denoted $\mathcal{H}_c$. It turns out that $\mathcal{H}_c$ interplays with Cutkosky cuts in graphs, this is related to the unitarity of the $S$-matrix \cite{cutcut}.
\end{rem}

\subsection{Renormalization in Hopf algebras}

\subsubsection{Feynman Rules and Characters of $\mathcal{H}$:}

Let a theory $T$ be fixed as before, and let $\mathcal{H}$ be the Hopf algebra generated by sets of divergent $1PI$ Feynman graphs in $T$. We start by thinking of Feynman rules as a map $\phi$ that assigns formal integrals to elements in $\mathcal{H}$, and we investigate what conditions should be imposed on $\phi$ to fully interpret the Feynman rules.

For the Feynman rules, we need to satisfy certain criteria:

\begin{enumerate}
    \item The map $\phi$ should be multiplicative on disjoint unions of graphs. Moreover, the map should also have a multiplicative property for bridges (remember the discussion at the beginning of Section \ref{section}). The latter requirement enables us to start defining $\phi$ over $1PI$ diagrams. The leap from all Feynman graphs to $1PI$ graphs is done through the Legendre transform, which has been redefined recently as a purely combinatorial map \cite{kjm2,kjm1}. 
    
    \item The map $\phi$, representing Feynman rules, has also to adapt with the combinatorial Dyson-Schwinger equations. Precisely, it has
    to interplay nicely with the process of \textit{insertion} which we discuss in the next section.
    \end{enumerate}

\textbf{(A)}\label{AAA} All of this was seen to suggest that the Feynman rules are to be represented by a character $\phi\in G^\mathcal{H}_A$, where $A$ is a suitably chosen commutative algebra. The target algebra $A$ is usually taken to be the algebra $\mathbb{C}[L][[z^{-1},z]]$ of Laurent series whose coefficients are polynomials in an energy scale $L$. For example, in \cite{renorm}, $L=\log (q^2/\mu^2)$ where $q$ and $\mu$ are the external momenta and the renormalization scale respectively.

\subsubsection{Rota-Baxter Operators:}

Let $A$ be an algebra as before. An operator (linear map on $A$) $R:A\longrightarrow A$ is said to be a \textit{Rota-Baxter operator} if it satisfies

\[R[ab]\;+\;R[a]\;R[b]\;=\;R\big[R[a]\;b\;+\;a\;R[b]\big],\]
for all $a,b\in A$.

It turns out that the truncated Taylor operator  $T^{\omega(\Gamma)}$ (see Section \ref{BP}) is a Rota-Baxter operator. This relation between renormalization and Rota-Baxter operators has been extensively studied in \cite{kurusch, kurusch2}.

\textbf{(B)} \label{BBB} In general, a  Rota-Baxter operator will be used to express a map which sends a formal integral to the evaluation of the integral at the subtraction point in the renormalization scheme. In other words, $R$ produces the counterterms as in equation \ref{counterterms}. If $\Gamma$ is a divergent graph with no subdivergences, $R\phi(\Gamma)$ will stand for the ill part of the integral $\phi(\Gamma)$.
\\

It remains to setup a technology for dealing  with subdivergences recursively.

Define a linear map $S^{\phi}_R:\mathcal{H}\longrightarrow A$
by $S^{\phi}_R(\mathbb{I})=1_A$ and 

\begin{equation}\label{S_R}
    S^{\phi}_R(\Gamma)=-R(\phi(\Gamma))-\underset{1PI \;\text{subgraphs}}{\underset{\gamma\;\text{product of divergent }}{\underset{\mathbb{I}\neq\gamma\subsetneq\Gamma}{\sum}}}S^{\phi}_R(\gamma)\; R(\phi(\Gamma/\gamma)),
\end{equation}
 and extended to all of $\mathcal{H}$ as a morphism of algebras. 

\textbf{(C)} Then the \textit{renormalized} Feynman rules are defined to be 
\begin{equation}
    \phi_R=S^{\phi}_R\ast\phi.
\end{equation}

It can be shown that $\phi_R(x)=(S^{\phi}_R\ast\phi)(x)=(\mathrm{id}_A-R)b(\phi)(x),$
where $b$ is the Bogoliubov map defined before (Definition \ref{hopf bog}) \cite{renorm}.
\\

By (A), (B), and (C), the conclusion is that the approach of renormalization is as follows:
(1) We express Feynman graphs in a graded Hopf algebra $\mathcal{H}$, and interpret the Feynman rules as characters from $\mathcal{H}$ to some commutative algebra $A$. (2) A renormalization scheme is determined via a Rota-Baxter operator on $A$, this also determines a Birkhoff decomposition $A=A_-\oplus A_+$ into two subalgebras. (3) The renormalized Feynman rules are obtained through the coproduct and the map $S^{\phi}_R$. For explicit examples and applications of this approach the reader can refer to \cite{renorm,karenbook,manchonhopf}.

\chapter{The Dyson-Schwinger Equations and The Enumeration of Chord Diagrams}\label{chchords1}

This chapter  presents some results about the combinatorial class of \textit{chord diagrams}  (also known as \textit{linked diagrams} \cite{steinandeveret}), which are basically matchings on finite ordered sets.

The interest in the combinatorial structure of chord diagrams is also motivated by questions arising from quantum field theory, more specifically, in the context of the Dyson-Schwinger equations \cite{yu, Karenmarkushihn}. As we will see in Section \ref{DSE}, chord diagrams are used in providing series solutions the Dyson-Schwinger equations. For the sake of completeness, the chapter will start with an overview of the combinatorial treatment of the Dyson-Schwinger equation, then in Section \ref{chorddiagrams section} we will  proceed with the basics of chord diagrams; Section \ref{factorially} is a necessary preparation for the approach applied by M. Borinsky \cite{michi, michi1} for studying the asymptotics of factorially divergent series in general. Finally, in Section \ref{asymptotics o connected} we will prove a result related to the asymptotics of the number of connected chord diagrams, namely we will show that sequence
\href{https://oeis.org/A088221}{A088221} of the OEIS counts pairs of rooted connected chord diagrams.

\section{Combinatorics of Dyson-Schwinger Equations}\label{DSE}

\subsection{Insertions }
Definition \ref{contractions} introduces the notion of contracting a subgraph within a bigger graph. One can think of a reverse operation in terms of \textit{inserting} a graph $\gamma$ into another graph $\Gamma$ as a subgraph, in one of the potential positions (\textit{insertion places}) in $\Gamma$ that can host $\gamma$. An insertion place has to be compatible with the external leg structure of the graph being inserted.

\begin{dfn}[Insertion]\label{insertion}
Let $\gamma$ be a Feynman graph with external leg structure $r=\mathrm{res}(\gamma)$. Let $\Gamma$ be a Feynman graph with a vertex or an internal edge of the same type as $r$. 
\begin{enumerate}
    \item If $r$ is of edge type, and $e$ is an internal edge in $\Gamma$ of the same type,  then we can \textit{insert} $\gamma$ into $\Gamma$ as follows:
    
    Break the edge $e$ into two half edges, each of which is identified with one of the two compatible external legs of $\gamma$. 
    
    \item If $r$ is of vertex type, and $v$ is a vertex of the same type in $\Gamma$, then we can \textit{insert} $\gamma$ into $\Gamma$ as follows:
    Break every edge incident to $v$, and, in a compatible way, which may not be unique, attach the external legs of $r$ to the resulting half edges in $\Gamma-v$.
\end{enumerate} 

The places $e$ or $v$ in the above scenarios are called \textit{insertion places}. Notice that the way to insert $\gamma$ into $\Gamma$ at a certain insertion place is not unique and depends on the symmetries of the graphs.
\end{dfn}

We wish now to define an operator $B_+^\Gamma$ that inserts graphs into a fixed graph $\Gamma$. 

\begin{dfn}[\cite{karenbook}]
For a connected $1PI$ Feynman graph $\Gamma$ we define
\[B_+^\Gamma(X)=\underset{G\;\text{1PI\;graph}}{\sum} \displaystyle\frac{\text{bij}(\Gamma,X,G)}{|X|_*}\frac{1}{\text{maxf}(G)}\frac{1}{(\Gamma|G)}\;G,\]

where 
\begin{enumerate}
    \item $\text{maxf}(G)$ is the number of insertion trees corresponding to $G$,
    \item $|X|_*$ is the number of distinct graphs obtained from permuting the external legs in $X$,
    \item $\text{bij}(\Gamma,X,G)$ is the number of bijections of the external legs of $X$ which have an insertion place in $\Gamma$ so that the insertion gives $G$.
    \item $(\Gamma|G)$ is the number of insertion places for $X$ in $\Gamma$.
\end{enumerate}
\end{dfn}

\begin{rem}\label{mercy}
See \cite{kreimerB} Theorem 4 for a justification of this definition. In the case of trees, the operation $B_+(T_1\cdots T_m)$ takes the rooted trees $T_1,\cdots,T_m$ and attach all of their roots as children of a new added root, getting a single rooted tree (the name \textit{grafting} operator makes sense in this case).  
\end{rem}

In the case of rooted trees described in the remark above, if $\mathcal{H}_C$ denotes the Connes-Kreimer Hopf algebra of rooted trees \cite{conneskreigeom,karenbook}, then the grafting operator $B_+$ is characterized as being a \textit{Hochschild  $1$-cocycle}, that is:  \[\Delta\circ B_+(t)=(\mathrm{id}\otimes B_+)\circ\Delta(t)+B_+(t)\otimes \mathbb{I}.\]

 This property will be highlighted in the next section as it is crucial to the algebraic reconstruction of renormalization in the approach pioneered by D. Kreimer and his collaborators. For more about this algebraic treatment and concepts see the original paper by D. Kreimer and A. Connes \cite{conneskreigeom} or \cite{karenbook}.

\begin{rem}\label{B+cocycle}
In order to get a $1$-cocycle from the operators $B_+^\gamma$ it turns out that we can not work with individual primitive graphs, rather, we should sum over all primitive graphs of a given loop number \cite{karenbook,kreimerB}
\end{rem}

The thing we used intuitively in drawing up the combinatorial Dyson-Schwinger equations in  section \ref{ddd} was exactly insertions of graphs. To express the Dyson-Schwinger equations in terms of the operators $B_+^\Gamma$ we need to know more about the number of insertion places in a given graph.

Let us assume that the combinatorial theory we are considering now has only one vertex type $v$, and let $d$ be the  degree of any such vertex. Also set $n(e)$ to be the number of half edges of type $e$ appearing in the external legs of vertex-type $v$. By definition we set $n(v)=1$.

\begin{prop}[\cite{karenbook}]
Let $\Gamma$ be a $1PI$ graph in a QFT theory of the type described above, that is, the theory has only one vertex type $v$ with $d$ being the degree of such a vertex.  Let $r=\mathrm{res}(\Gamma)$, and $\ell=\ell(\Gamma)$ (the loop number). Also let $n(e)$  be the number of half edges of type $e$ appearing in the external legs of vertex-type $v$. Besides, define $n(v)=1$. Then 

\begin{enumerate}
    \item $\Gamma$ has $\displaystyle\frac{2\ell n(s)}{d-2}$ insertion places for every type $s\neq r$;
    \item If $r$ is vertex-type, then $\Gamma$ has\; $1+\displaystyle\frac{2\ell n(r)}{d-2}$ insertion places for type $r$; and
    \item If $r$ is not vertex-type, then $\Gamma$ has\; $-1+\displaystyle\frac{2\ell n(r)}{d-2}$ insertion places for type $r$.
\end{enumerate}
\end{prop}

\begin{exm}

In QED (quantum electrodynamics) we have only one vertex type, namely \raisebox{-0.67cm}{\includegraphics[scale=0.25]{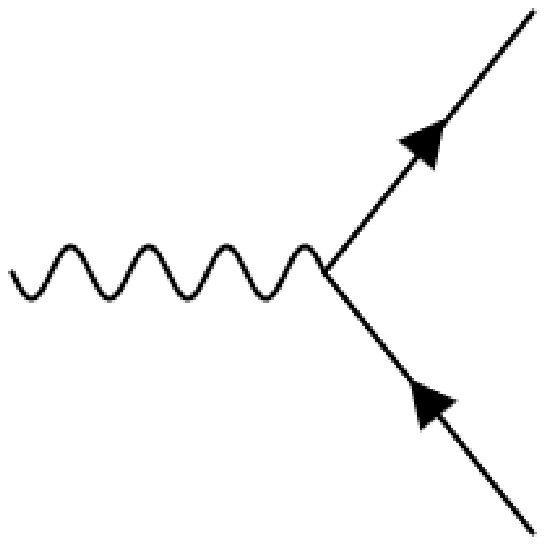}}, and two edge types: a photon edge  \raisebox{-0.1cm}{\includegraphics[scale=0.26]{Figures/photonedge.eps}}, and a fermion edge \;\raisebox{-0.0cm}{\includegraphics[scale=0.7]{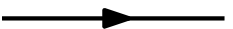}}.

We are going to follow the notation used in \cite{karenbook}, namely

\begin{itemize}
    \item $X^{\text{vertex}}$ is the \textit{vertex series}, whose $n$th coefficient is the sum of all $1PI$ QED diagrams with residue  \raisebox{-0.57cm}{\includegraphics[scale=0.2]{Figures/QEDvertex.eps}} and loop number $n$.
    
    \item $X^{\text{photon}}$ is the \textit{photon edge series}, whose $n$th coefficient is $(-1)\times$\{the sum of all $1PI$ QED diagrams with residue  \raisebox{-0.14cm}{\includegraphics[scale=0.2]{Figures/photonedge.eps}} and loop number $n$\}.
    
    \item $X^{\text{fermion}}$ is the \textit{fermion edge series}, whose $n$th coefficient is $(-1)\times$\{the sum of all $1PI$ QED diagrams with residue  \raisebox{0cm}{\includegraphics[scale=0.48]{Figures/ferm1.eps}} and loop number $n$\}.
\end{itemize}

\begin{rem}
Notice that the negative signs with the edge series arise as we will be actually interested in sequences of such diagrams, and so if $Y$ is the original generating function then we are to get a geometric series $\displaystyle\frac{1}{1-Y}$. Then we use $X=1-Y$.
\end{rem}

\setlength{\parindent}{0.5cm}
Then we have 
\begin{align}
    X^{\text{vertex}}&=\mathbb{I}+ \underset{\text{vertex residue}}{\underset{\gamma \text{\;primitive with}}{\sum}} x^{\ell(\gamma)}
    \;B_+^\gamma\left(\displaystyle\frac{\left(X^{\text{vertex}}\right)^{1+2\ell(\gamma)}}{\left(X^{\text{photon}}\right)^{\ell(\gamma)}\left(X^{\text{fermion}}\right)^{2\ell(\gamma)}}\right),\label{e1}\\
    &\nonumber\\
    X^{\text{photon}}&=\mathbb{I}-  x
    \;B_+^{\raisebox{-0.84cm}{\includegraphics[scale=0.182]{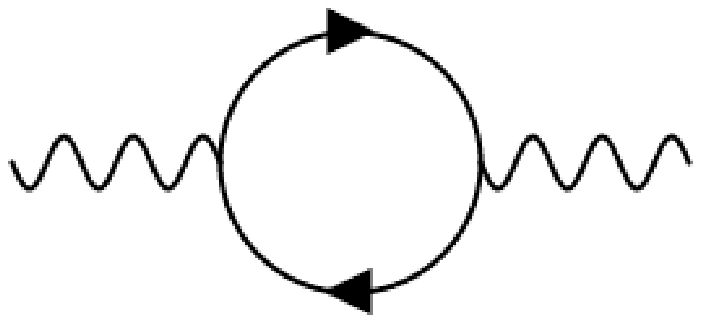}}}
    \left(\displaystyle\frac{\left(X^{\text{vertex}}\right)^2}{\left(X^{\text{fermion}}\right)^2}\right),\label{e2}\\
    &\nonumber\\
    X^{\text{fermion}}&=\mathbb{I}- x
    \;B_+^{\raisebox{-0.57cm}{\includegraphics[scale=0.18]{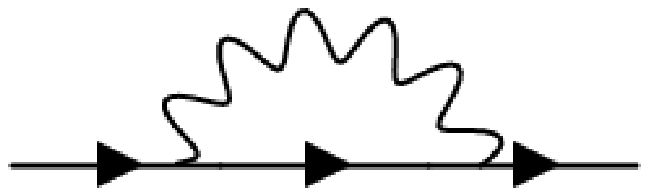}}}
   \left(\displaystyle\frac{\left(X^{\text{vertex}}\right)^2}{X^{\text{photon}}\;\;X^{\text{fermion}}}\right).\label{e3}
\end{align}

These equations are obtained by a direct counting argument. For example, the second equation can be illustrated through Figure{\ref{lloop}}.

\begin{figure}[h]
 \begin{center}

\raisebox{-0.84cm}{\includegraphics[scale=0.6]{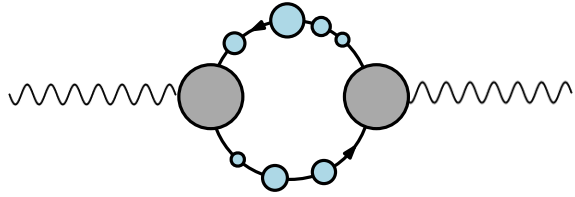}}
\end{center} \caption{Graphs with photon edge residue}\label{lloop}
\end{figure}

In Figure \ref{lloop} the blue bubbles represent the two sequences of fermion-type $1PI$ graphs that can be inserted along the original two fermion edges, hence the $1/\left(X^{\text{fermion}}\right)^2$; whereas the two larger grey bubbles represent insertion of a vertex-type $1PI$ graph and correspond to the $\left(X^{\text{vertex}}\right)^2$. The minus sign follows from the definition of $X^{\text{photon}}$.
\end{exm}

\subsection{The Invariant Charge}\label{invariantch}

If we set $Q=\displaystyle\frac{\left(X^{\text{vertex}}\right)^2}{
\left(X^{\text{photon}}\right)^{}\;\left(X^{\text{fermion}}\right)^2},$
then we can rewrite equations (\ref{e1}), (\ref{e2}), and (\ref{e3}) as

\begin{align}
    X^{\text{vertex}}&=\mathbb{I}+ \underset{\text{vertex residue}}{\underset{\gamma \text{\;primitive with}}{\sum}} x^{\ell(\gamma)}
    \;B_+^\gamma\left(X^{\text{vertex}} Q^{\ell(\gamma)}\right),\label{e11}\\
    X^{\text{photon}}&=\mathbb{I}-  x
    \;B_+^{\raisebox{-1cm}{\includegraphics[scale=0.18]{Figures/loop1.eps}}}
    \left(X^{\text{photon}} Q\right),\label{e22}\\
    X^{\text{fermion}}&=\mathbb{I}- x
    \;B_+^{\raisebox{-0.57cm}{\includegraphics[scale=0.18]{Figures/loop2.eps}}}
   \left(X^{\text{fermion}} Q\right).\label{e33}
\end{align}

The expression $Q$ is to be called the \textit{invariant charge}. In general, for a theory with only one vertex type $v$, the system of Dyson-Schwinger equations takes the form \begin{equation}\label{generaldys}
    X^r=1\pm \sum_k B_+^{\gamma_{r,k}}(X^rQ^k),
\end{equation}
where the sum is over $k$ and over all primitive 1PI diagrams $\gamma_{r,k}$ with loop number $k$ and residue $r$.

\begin{equation}\label{qgeneral} 
    Q=\left(\displaystyle\frac{X^v}{\sqrt{\prod_{e\in v}(X^e)}}\right)^{\frac{2}{d-2}},
\end{equation}
where $d$ is the degree of the vertex-type, and the product is over all half edges making up the vertex. Note that the orientation of a half edge is ignored in counting the types. In some references, the convention for the invariant charge is to be the square root of our definition \cite{michiq}. Also in \cite{michiq} the $X^r$ is defined to be the negative of ours in case $r$ is an edge type.
\section{From Combinatorial to Analytic Dyson-Schwinger Equations}

In Section \ref{DSE} we have seen how to use the insertion  to express the Dyson-Schwinger equations combinatorially as functional equations involving the operators $B_+^\gamma$ and generating functions indexed by the loop number. Applying the Feynman rules to these equations we should somehow get functional equations of the Green functions. The first unclear thing we need to discuss is how the Feynman rules interact with the operators $B_+^\gamma$.

The following theorem establishes a general universal property for Hopf algebras. We will use this theorem in translating the insertion operators.

\begin{thm}[Universal Property \cite{conneskreigeom}]
Let $\mathcal{H}_C$ be the Connes-Kreimer Hopf algebra of rooted trees, and $B_+$ be the grafting operator as in Remark \ref{mercy}. Let $A$ be a commutative algebra and $L:A\longrightarrow A$ be a map. Then there exists a unique algebra morphism $\rho_L:\mathcal{H}_C\longrightarrow A$ such that 
$$\rho_L\circ B=L\circ\rho_L,$$ or equivalently, such that the diagram 
\begin{center}
\begin{minipage}{0.5 \textwidth}
\xymatrixcolsep{5pc}\xymatrix{
H\ar[d]_B\ar[r]^{\rho_L}& A\ar[d]^L\\
     H \ar[r]_{\rho_L} &A
}
\end{minipage}

\end{center}
commutes. If further $A$ is a Hopf lgebra and $L$ is a $1$-cocycle then $\rho_L$ is a Hopf algebra morphism.
\end{thm}

Now, remember from Remark \ref{B+cocycle} that somehow we eventually get $1$-cocycles from the operators $B_+^\gamma$ of the Hopf algebra $\mathcal{H}$ of Feynman graphs, and thus, by the above universal property, the renormalized Feynman rules only replace $B_+^\gamma$ with a new operator.
The Feynman rules will also take the counting variable $x$ to become the coupling constant. Feynman rules bring scale variables $L_j$ coming from the external momenta $q_i$ and the fixed renormalization points $\mu_i$. A single scale variable is $L=\log(q^2/\mu^2)$ (this is not necessarily true in case there are more than one scale variable). By applying the renormalized Fenman rules we get  a relation in which Green functions appear  as functions in $x$ and $L$.

More precisely, every $X_0$ appearing in the combinatorial Dyson-Schwinger equation is replaced with the corresponding Green function $G_0$.  For a factor $(X_0)^s$ in an operator $B_+^\gamma$ in the combinatorial equation we take the $\gamma$ integrand and multiply it with $(G_0)^s$. Take the momenta of the edges where $X_0$ is inserted and use it as a scale argument for $G_0$. Finally subtract the resulting integral at the fixed external momenta $\mu_i$ as is done when renormalizing.

\begin{exm}\label{massless yukawa theory}
(See \cite{karenthesis} for the complete example).  In the Dyson-Schwinger equation for massless Yukawa theory we consider graphs that are obtained by inserting \raisebox{-0.3cm}{\includegraphics[scale=0.34]{Figures/loop2.eps}} into itself. The graphs obtained are generally nested graphs like the one below.

\begin{center}
    \raisebox{-0.2cm}{\includegraphics[scale=0.34]{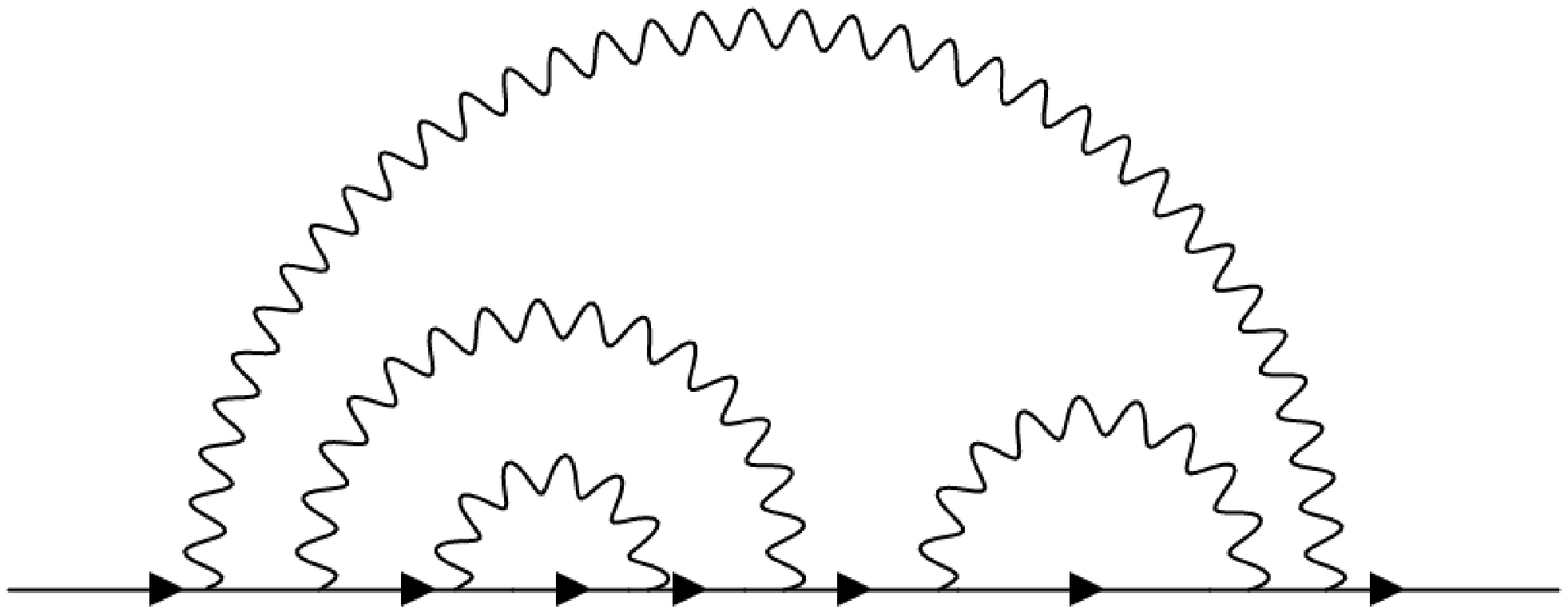}} 

\end{center}

 It follows that the combinatorial Dyson-Schwinger equation  is of the form
 
\[X(x)=\mathbb{I}-xB_+\left(\displaystyle\frac{1}{X(x)}\right).\]

Applying the Feynman rules and following the approach described above we get the equation in terms of Green functions as

\begin{equation}\label{G(x,L)}
    G(x,L)=1-\left(\left.\frac{x}{q^2}\int d^4k\displaystyle\frac{k\cdot q}{k^2\;
    G(x,L_1)(k+q)^2}\;-\cdots\right|_{q^2=\mu^2}\right),
\end{equation}

where $L=\log(q^2/\mu^2), L_1=\log(k^2/\mu^2)$ and the $\cdots$ stands for the same integrand evaluated at the renormalization point $q^2=\mu^2$ for some fixed $\mu$. 
\end{exm}

\subsection{Relation to Chord Diagrams}

Consider the recursive integral equation (\ref{G(x,L)}) from Example \ref{massless yukawa theory}. This equation can be transformed into a differential equation as follows (here we follow \cite{karenbook}):

Substituting the ansatz 
\[G(x,L)=1-\sum_{k\geq1}\gamma_k(x) L^k\] into the equation we get

\begin{align*}
    \sum_{k\geq1}\gamma_k(x) L^k&= \left.\frac{x}{q^2}\int d^4k\underset{\ell_1+\cdots+ \ell_s=\ell}{\sum}\displaystyle\frac{(k\cdot q) \gamma_{\ell_1}(x)\cdots\gamma_{\ell_s}(x) L_1^\ell}{k^2\;
    (k+q)^2}\;-\cdots\right|_{q^2=\mu^2},
\end{align*}

where as before $L=\log(q^2/\mu^2), L_1=\log(k^2/\mu^2)$, and the $\cdots$ stands for the same integrand evaluated at renormalization point $q^2=\mu^2$ (subtraction scheme). Now notice that we can turn the logarithms into differentail operators by adding a new variable, namely 
\[\left.\displaystyle\frac{d^ky^\rho}{d\rho^k}\right|_{\rho=0}=(\log (y))^k.\]

Thus, back to our equation,

\begin{align*}
    \sum_{k\geq1}\gamma_k(x) L^k&= \frac{x}{q^2}\underset{\ell_1+\cdots+ \ell_s=\ell}{\sum}\gamma_{\ell_1}(x)\cdots\gamma_{\ell_s}(x)\int d^4k\displaystyle\frac{(k\cdot q)  (-1)^\ell \left.\frac{d^\ell}{d\rho^\ell}(k^2/\mu^2)^{-\rho}\right|_{\rho=0}}
    {k^2(k+q)^2}\;-\cdots\Bigg|_{q^2=\mu^2}
    \\
    &=\frac{x}{q^2}\underset{\ell_1+\cdots+ \ell_s=\ell}{\sum}\gamma_{\ell_1}(x)\cdots\gamma_{\ell_s}(x)(-1)^\ell
    \times\\
    &\qquad\qquad\times
    \frac{d^\ell}{d\rho^\ell}
    \left\{(\mu^2)^\rho 
    \int d^4k
    \displaystyle\frac{(k\cdot q)}{(k^2)^{1+\rho}(k+q)^2}   \;-\cdots\Bigg|_{q^2=\mu^2}\right\}
    \Bigg|_{\rho=0}
    \end{align*}
    
    \begin{align*}
    &=x\left(1-\sum_{k\geq1}\bigg(\displaystyle\frac{d}{d(-\rho)}\bigg)^k\right)^{-1}
    \frac{(\mu^2)^\rho}{q^2}
    \int d^4k
    \displaystyle\frac{(k\cdot q)}{(k^2)^{1+\rho}(k+q)^2}   \;-\cdots\Bigg|_{q^2=\mu^2}
    \Bigg|_{\rho=0}
    \\
    &=x\left(1-\sum_{k\geq1}\bigg(\displaystyle\frac{d}{d(-\rho)}\bigg)^k\right)^{-1}
    \frac{(\mu^2)^\rho}{(q^2)^\rho}
    \int d^4k_0
    \displaystyle\frac{(k_0\cdot q_0)}{(k_0^2)^{1+\rho}(k_0+q_0)^2}   \;-\cdots\Bigg|_{q^2=\mu^2}
    \Bigg|_{\rho=0}\\
    &\qquad\text{(where we set $q=rq_0$ with $r\in\mathbb{R}, r^2=q^2, q^2=1$ and $k_0=k/r$)}
    \\
    &=x\left(1-\sum_{k\geq1}\bigg(\displaystyle\frac{d}{d(-\rho)}\bigg)^k\right)^{-1}
    (e^{-L\rho}-1)F(\rho)\Bigg|_{\rho=0},
\end{align*}
where \[F(\rho)=\frac{1}{q^2}
    \int d^4k
    \displaystyle\frac{(k\cdot q)}{(k^2)^{1+\rho}(k+q)^2}   \Bigg|_{q^2=1}.\]

Thus, equation (\ref{G(x,L)}) is equivalent to the following differential equation

\begin{equation}\label{G(x,L)diff}
    G(x,L)=1-xG\left(x,\;\displaystyle\frac{d}{d(-\rho)}\right)^{-1}
    (e^{-L\rho}-1)F(\rho)\Bigg|_{\rho=0}.
\end{equation}

In the literature, $F(\rho)$ is called the \textit{Mellin transform} \cite{karenbook}.

For our purposes, it will be enough to mention here that a series solution for equation (\ref{G(x,L)diff}) can be expressed in terms of rooted connected chord diagrams, which we shall study next. More specifically, in \cite{yu}, it is shown that the series solution for the equation can be written as

\[G(x,L)=1-\sum_{i\geq1}\displaystyle\frac{(-L)^i}{i!}\underset{C,\;b(C)\geq i}{\sum}x^{|C|} \;f_0^{|C|-k}f_{b(C)-i}\prod^k_{j=2} f_{t_j-t_{j-1}},\]

where the inner sum is over all rooted connected chord diagrams $C$ (which we properly define in the next section) in which, with respect to intersection order, the first terminal chord comes at least  at the $i$th place.  $k$ is the number of terminal chords, and the places for terminal chords with respect to intersection order are $b(C)=t_1<t_2<\cdots<t_k$.  This shows the significance of the combinatorial structure of chord diagrams, which we are about to study, in the context of the Dyson-Schwinger equations. 


\section{Chord Diagrams}\label{chorddiagrams section}

(Connected) chord diagrams stand as a rich structure that becomes handy and informative in a variety of contexts, including bioinformatics \cite{bioinfo}, quantum field theory \cite{karenbook, michi1, michi}, and data structures \cite{datast}. Our interest in chord diagrams comes from the context of quantum field theory, in particular, from the Dyson-Schwinger equations as we have seen above. The solutions to Dyson-Schwinger were recently shown to be described as series indexed by connected chord diagrams with extra conditions on the placing of terminal chords \cite{yu}. Our work on chord diagrams here shall henceforth be purely combinatorial.

\begin{dfn}[Chord diagrams]
A \textit{chord diagram} on $n$ chords (i.e. of size $n$) is geometrically perceived simply as a circle with $2n$ nodes that are matched into disjoint pairs, with each pair corresponding to a \textit{chords}. 
\end{dfn}

\begin{dfn}[Rooted chord diagrams]
A \textit{rooted} chord diagram is a chord diagram with a selected node. The selected node is called the \textit{root vertex}, and the chord with the root vertex is called the \textit{root chord}. In other words, a rooted chord diagram of size $n$ is a matching of the set $\{1,\ldots,2n\}$. For an algebraic definition, this is the same as a fixed-point free involution in $S_{2n}$. Then the generating series for rooted chord diagrams is 

\begin{equation}\label{rootedchorddiagsgen}
    D(x):=\sum_{n=0}(2n-1)!!\; x^n
\end{equation}

All chord diagrams considered here are going to be rooted and so, when we say a chord diagram we tacitly mean a rooted one.
\end{dfn}

Now, a rooted chord diagram can be represented in a linear order, by numbering the nodes in counterclockwise order, starting from the root which receives the label `$1$'. A chord in the diagram may be referred to as $c=\{a<b\}$, where $a$ and $b$ are the nodes in the linear order.

\begin{dfn}[Intervals]
In the linear representation of a rooted chord diagram, an $interval$ is the space to the right of one of the nodes in the linear representation. Thus, a rooted diagram on $n$ chords has $2n$ intervals.
\end{dfn}
 
For example, this includes the space to the right of the last node in the linear order). 

\begin{figure}[!htb]
    \centering
    \includegraphics[scale=0.8]{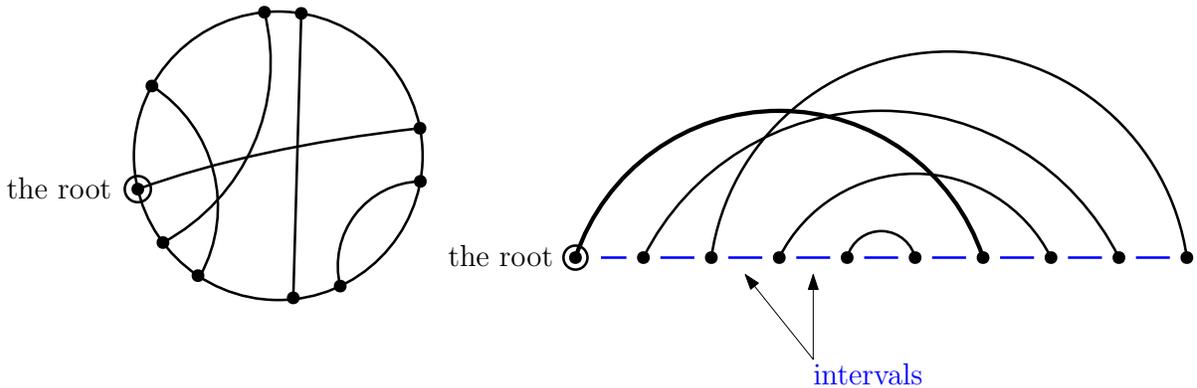}
    \caption{A rooted chord diagram and its linear representation}
\end{figure}

As may be expected by now, the crossings in a chord diagram encode much of the structure and so we ought to give proper notation for them. Namely, in the linear order, two chords $c_1=\{v_1<v_2\}$ and $c_2=\{w_1<w_2\}$ are said to \textit{cross} if $v_1<w_1<v_2<w_2$ or $w_1<v_1<w_2<v_2$.
Tracing all the crossings in the diagram leads to the following definition:

\begin{dfn}[The Intersection Graph] 
Given a (rooted) chord diagram $D$ on $n$ chords, consider the following graph $\mathcal{G}_D$: the chords of the diagram will serve as vertices for the new graph, and there is an edge between the two vertices $c_1=\{v_1<v_2\}$ and $c_2=\{w_1<w_2\}$ if $v_1<w_1<v_2<w_2$ or $w_1<v_1<w_2<v_2$, i.e. if the chords \textit{cross} each other. The graph so constructed is called the \textit{intersection graph} of the given chord diagram.  
\end{dfn}

\begin{rem}
A labelling for the intersection graph can be obtained as follows: give the label $1$ to the root chord;  order the components obtained if the root is removed according to the order of the first vertex of each of them in the linear representation, say the components are $C_1,\ldots,C_n$; and then recursively label each of the components. It is easily verified that a rooted chord diagram can be uniquely recovered from its labelled intersection graph.\\
\end{rem}

\begin{dfn}[Connected Chord Diagrams]\label{c}
A (rooted) chord diagram is said to be \textit{connected} if its  intersection graph is connected (in the graph-theoretic sense). A \textit{connected component} of a diagram is a subset of chords which itself forms a connected chord diagram. The term \textit{root component} will refer to the connected component containing the root chord.
\end{dfn}

\begin{exm}
The diagram $D$ below is a connected chord diagram in linear representation, where the root node is drawn in black.

\begin{center}
\includegraphics[scale=0.5]{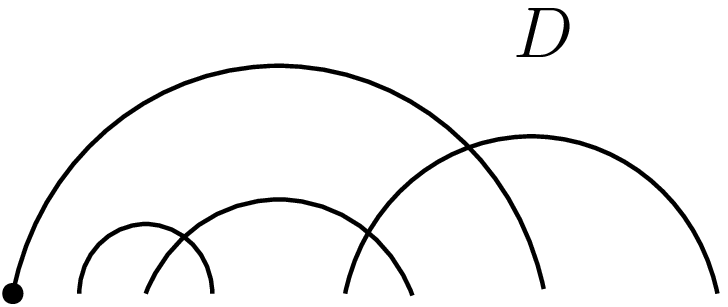}\end{center}
\end{exm}

The generating function for connected chord diagrams (in the number of chords) is denoted by $C(x)$. Thus $C(x)=\sum_{n=0}C_n x^n$, where $C_n$ is the number of connected chord diagrams on $n$ chords. The first terms of $C(x)$ are found to be 
\[C(x)=x+x^2+4x^3+27\;x^4+248\;x^5+\cdots\;;\]
the reader may refer to OEIS sequence \href{https://oeis.org/A000699}{A000699}  for more coefficients. 
The next lemma lists some classic decompositions for chord diagrams (see \cite{flajoletchords} for example).

\begin{lem}\label{cd}
  If $D(x), C(x)$ are the generating series for chord diagrams and connected chord diagrams respectively, then 
  \begin{enumerate}
      \item[$\mathrm{(i)}$]  $D(x)=1+C(xD(x)^2)$, \label{i1}
      \item[$\mathrm{(ii)}$] $D(x)=1+xD(x)+2x^2D'(x)$, and \label{i2}
      \item[$\mathrm{(iii)}$] $2xC(x)C'(x)=C(x)(1+C(x))-x$.\label{i3}
  \end{enumerate}

  \end{lem}
  
  \begin{proof} We sketch the underlying decompositions as follows: 
  \begin{enumerate}
    \item[$\mathrm{(i)}$] The `one' term is for the empty chord diagram. Now, given a nonempty chord diagram, we see that for every chord in the root component there live two chord diagrams to the right of its two ends. This gives the desired decomposition.
     
     \item[$\mathrm{(ii)}$] There are three situations for a root chord: it is either non-existent (empty diagram); or it is concatenated with a following diagram; or the root chord has its right end landing in one of the intervals of a diagram. These situations correspond respectively with the terms in (ii).
     
   \item[$\mathrm{(iii)}$] Can be derived from (i) and (ii). Nevertheless, it can be also shown as follows: if we remove the root chord what is left is a sequence of connected components, with each component having a special interval (through wich the root used to pass) which cannot be the last interval (see the figure below). Thus each of these components is counted according to the generating function $2xC'(x)-C(x)$.
   \begin{center}
   \includegraphics[scale=0.65]{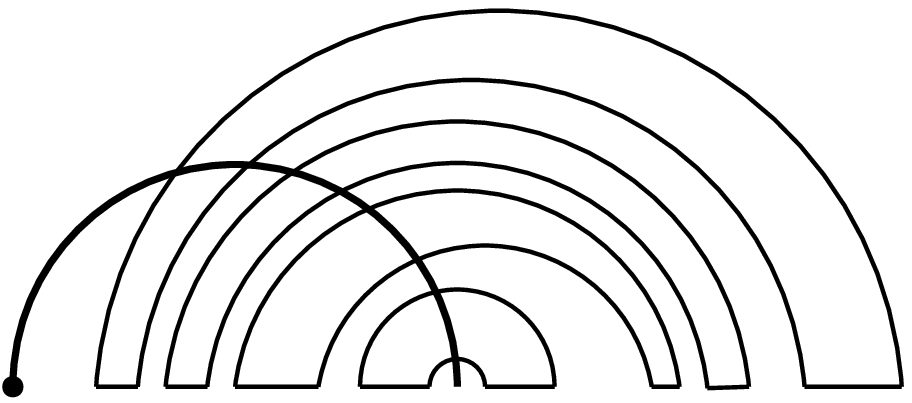}   
   \end{center}
   
   This decomposition gives that
   \[C(x)=\displaystyle\frac{x}{1-(2xC'(x)-C(x))},\]
   and the result follows.
     \end{enumerate}
  \end{proof}

  \setlength{\parindent}{0cm}
  
    We end this section with the definition of an \textit{indecomposable} chord diagram, these diagrams will become a key ingredient later on. 
    \begin{dfn}\label{indecompo}
    A chord diagram is said to be \textit{indecomposable} if, when represented linearly, it is not the concatenation of disjoint chord diagrams. The empty diagram is vacuously indecomposable by definition. The generating function for indecomposable chord diagrams is denoted here by $I(x)$. We shall also use $I_0(x)$ to denote the generating function for nonempty indecomposable chord diagrams (that is $I(x)=1+I_0(x)$).
    \end{dfn}
    
 \begin{exm}

 Consider the following two diagrams:
 
$D_1=\;$ \raisebox{-0.1cm}{\includegraphics[scale=0.56]{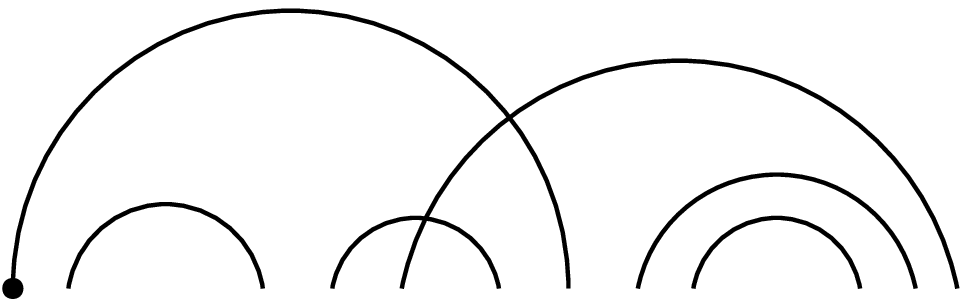}}\;, \;and \;$D_2=\;$ \raisebox{-0.1cm}{\includegraphics[scale=0.54]{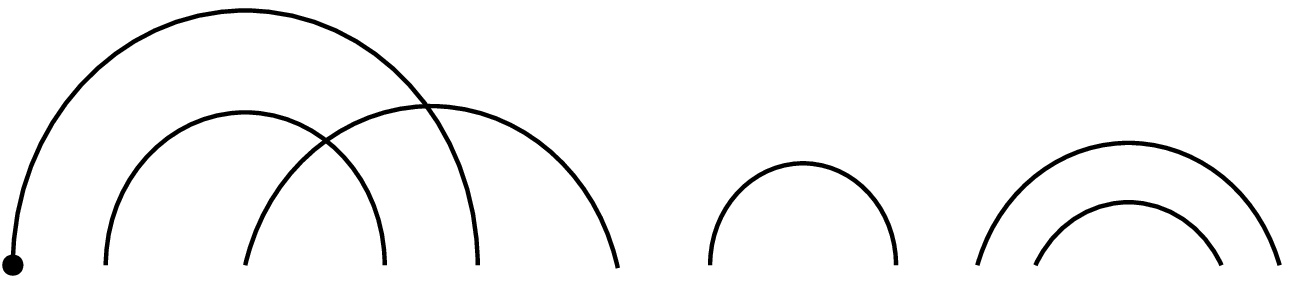}}

\quad Then $D_1$ is indecomposable, whereas $D_2$ is not since it is the concatenation of three (indecomposable chord diagrams). Notice that an indecomposable chord diagram is not necessarily connected, but the converse is clearly true, namely, any connected diagram is indecomposable. 
 
 \end{exm}
Sequence \href{https://oeis.org/A000698}{A000698} of the OEIS counts indecomposable chord diagrams, the first terms start as 
$$I(x)=1+x+ 2x^2+ 10x^3+74 x^4+ 706x^5+\cdots,$$ where $I(x)$ is the generating power series for indecomposable chord diagrams.

\section{Factorially Divergent Power Series}\label{factorially}

 This section aims to provide the necessary background for \textit{factorially divergent power series}, as introduced in Chapter 4 in \cite{michi}. In \cite{michi, michi1}, M. Borinsky  studied sequences $a_n$ whose asymptotic behaviour for large $n$ follows a relation like

\begin{equation}
    a_n=\alpha^{n+\beta} \Gamma(n+\beta)\bigg(c_0+\displaystyle\frac{c_1}{\alpha(n+\beta-1)}+\displaystyle\frac{c_2}{\alpha^2(n+\beta-1)(n+\beta-2)}+\cdots\bigg),
\end{equation}

where $\alpha\in\mathbb{R}_{>0}$, and $\beta, c_k \in \mathbb{R}$. We will need to use the usual big and small o-notation for asymptotic analysis: Given a sequence $a_n$, $\mathcal{O}(a_n)$ will denote the class of sequences $b_n$ satisfying $\limsup_{n\rightarrow\infty}|\frac{b_n}{a_n}|<\infty$; whereas $o(a_n)$ shall denote the sequences $b_n$ such that $\lim_{n\rightarrow\infty}\frac{b_n}{a_n}=0$. Moreover, $a_n=b_n+\mathcal{O}(c_n) $ should mean that $a_n-b_n\in\mathcal{O}(c_n)$. Following \cite{michi}, we adopt the notation $\Gamma_\beta^\alpha(n):=\alpha^{n+\beta}\Gamma(n+\beta)$, where $\Gamma(z)=\int_0^\infty x^{z-1}e^{-x}dx$ for $\text{Re}(z)>0$ is the gamma function.

\begin{dfn}[Factorially Divergent Power Series]\label{fdps} For real numbers $\alpha$ and $\beta$, with $\alpha>0$, the subset $\mathbb{R}[[x]]_\beta^\alpha$ of $\mathbb{R}[[x]]$ will denote the set of all formal power series $f$ for which there exists a sequence $(c_k^f)_{k\in \mathbb{N}}$ of real numbers such that

\begin{equation}
    f_n=\overset{R-1}{\underset{k=0}{\sum}}c_k^f \Gamma_\beta^\alpha(n-k)+\mathcal{O}(\Gamma_\beta^\alpha(n-R)),\;\;\text{for all} \;R\in\mathbb{N}_0 \;\;\; \text{(positive integers).}\label{asy}\end{equation}
 \end{dfn}

\begin{rem}
From the definition it follows that $\mathbb{R}[[x]]_\beta^\alpha$ is a linear subspace of $\mathbb{R}[[x]]$. 
\end{rem}

\begin{rem}\label{notinj}
Also, by the above definition all real power series with a non-vanishing radius of convergence belong to $\mathbb{R}[[x]]_\beta^\alpha$, with $c_k^f=0$ for all $k$ since in this case $f_n=o(\Gamma^\alpha_\beta(n-R))$ for all $R\in \mathbb{N}_0$.
\end{rem}

The following proposition also follows directly from the definition.

\begin{prop}[\cite{michi}, Ch.4]
 The sequence   $(c_k^f)_{k\in \mathbb{N}}$ is unique for every $f\in\mathbb{R}[[x]]_\beta^\alpha$; actually $c_N^f=\lim_{n\rightarrow\infty}\frac{f_n-\sum^{N-1}_{k=0}c_k^f \Gamma_\beta^\alpha(n-k)}{(\Gamma_\beta^\alpha(n-N))}$ for $N\in\mathbb{N}_0$.
\end{prop}

\begin{prop}[\cite{michi}, Prop 4.3.1]\label{subring}
Given $\alpha,\beta\in \mathbb{R}$, with $\alpha>0$, the set $\mathbb{R}[[x]]_\beta^\alpha$ is a subring of $\mathbb{R}[[x]]$.
\end{prop}

Note that  the identity in (\ref{asy}) stands for an asymptotic expansion with asymptotic scale $\alpha^{n+\beta}\Gamma(n+\beta)$ (refer to \cite{asymptotic} for a detailed literature on the topic). The ring $\mathbb{R}[[x]]_\beta^\alpha$ is referred to as \textit{a ring of factorially divergent power series}. Now, given $f\in\mathbb{R}[[x]]_\beta^\alpha$, we can associate the coefficients $(c_k^f)_{k\in \mathbb{N}}$ of the asymptotic expansion  with a new ordinary power series:

\begin{dfn}[\cite{michi1}]\label{map}
For $\alpha,\beta\in \mathbb{R}$, with $\alpha>0$, let $ \mathcal{A}_\beta^\alpha:\mathbb{R}[[x]]_\beta^\alpha\rightarrow\mathbb{R}[[x]]$ be the map that has the following action for every $f\in\mathbb{R}[[x]]_\beta^\alpha$
\[(\mathcal{A}_\beta^\alpha f)(x)=\overset{\infty}{\underset{k=0}{\sum}}c_k^fx^k.\]
\end{dfn}

This interpretation allows us to use the coefficients $c_k^f$ to get various results as we will see later. In \cite{michi} M. Borinsky provides an extensive analysis for the map $\mathcal{A}_\beta^\alpha $, we will include some of the properties without proof and will try to include proofs only as much as needed.

\begin{rem}

A map of this type is called an \textit{alien derivative (operator)}  in the context of resurgence theory \cite{resurgence}. We will use this terminology occasionally.
\end{rem}

\begin{rem}

Form the definition we see that $\mathcal{A}_\beta^\alpha$ is linear. $\mathcal{A}_\beta^\alpha$ is not injective since it vanishes for power series with nonzero radius of convergence as mentioned above in Remark \ref{notinj}.
\end{rem}

\begin{prop}[\cite{michi}, Prop 4.1.1]\label{plusm1} For $m\in\mathbb{N}_0$,
$f\in \mathbb{R}[[x]]_\beta^\alpha$ if and only if $f\in\mathbb{R}[[x]]_{\beta+m}^\alpha$ and $\mathcal{A}_{\beta+m}f\in x^m\mathbb{R}[[x]]$. In this case $x^m\big(\mathcal{A}_\beta^\alpha f\big)(x)=\big(\mathcal{A}_{\beta+m}^\alpha f\big)(x)$.
\end{prop}

\begin{proof}
First note that \[\Gamma^\alpha_\beta(n)=\alpha^{n-m+\beta+m}\Gamma(n-m+\beta+m)=\Gamma_{\beta+m}^\alpha(n-m).\]

Now, 
$f_n=\overset{R-1}{\underset{k=0}{\sum}}c_k^f \Gamma_{\beta}^\alpha(n-k)+\mathcal{O}(\Gamma_{\beta}^\alpha(n-R)),\;\;\text{for all} \;R\in\mathbb{N}_0 ,$ can be re-indexed as 
\[f_n=\overset{R'-1}{\underset{k=m}{\sum}}c_{k-m}^f \Gamma_{\beta}^\alpha(n-k+m)+\mathcal{O}(\Gamma_{\beta+m}^\alpha(n-R')),\;\;\text{for all} \;R'\geq m.\]

By the observation at the beginning the latter is equivalent to
\[f_n=\overset{R'-1}{\underset{k=m}{\sum}}c_{k-m}^f \Gamma_{\beta+m}^\alpha(n-k)+\mathcal{O}(\Gamma_{\beta+m}^\alpha(n-R')),\;\;\text{for all} \;R'\geq m.\]

Which proves the first part of the statement. Also, in that case

$\big(\mathcal{A}_{\beta+m}^\alpha f\big)(x)=\sum_{k=m}^\infty c_{k-m}^fx^k=x^m\big(\mathcal{A}_{\beta}^\alpha f\big)(x)\in x^m\mathbb{R}[[x]]$.
\end{proof}

The following corollary now follows.
\begin{cor}\label{corplusm1}
For all $m\in\mathbb{N}_0$,
$\mathbb{R}[[x]]_\beta^\alpha \subset \mathbb{R}[[x]]_{\beta+m}^\alpha$. 
\end{cor}

Thus we can always assume that $\beta>0$, which is very convenient in deriving many results for the ring $\mathbb{R}[[x]]_\beta^\alpha$ \cite{michi}.

\begin{prop}[\cite{michi}, Prop 4.1.2]\label{plusm2} For $m\in\mathbb{N}_0$,
$f\in \mathbb{R}[[x]]_\beta^\alpha \cap x^m\mathbb{R}[[x]]$ if and only if $\displaystyle\frac{f(x)}{x^m}\in\mathbb{R}[[x]]_{\beta+m}^\alpha$. In this case $\big(\mathcal{A}_\beta^\alpha f\big)(x)=\big(\mathcal{A}_{\beta+m}^\alpha \displaystyle\frac{f(x)}{x^m}\big)(x)$.
\end{prop}

\begin{proof}
Again, since \[\Gamma^\alpha_\beta(n+m)=\alpha^{n+\beta+m}\Gamma(n+\beta+m)=\Gamma_{\beta+m}^\alpha(n),\]
 we can argue as follows:
 
 The `only if part' follows by Proposition \ref{plusm1}. For the `if' part, assume that $g(x)=\displaystyle\frac{f(x)}{x^m}\in\mathbb{R}[[x]]_{\beta+m}^\alpha$. This gives that 

\[f_{n+m}=g_n=\overset{R-1}{\underset{k=0}{\sum}}c_k^g \Gamma_{\beta+m}^\alpha(n-k)+\mathcal{O}(\Gamma_{\beta+m}^\alpha(n-R)),\;\;\text{for all} \;R\in\mathbb{N}_0.\] 

By the observation above this is equivalent to
\[f_{n+m}=\overset{R-1}{\underset{k=0}{\sum}}c_k^g \Gamma_{\beta}^\alpha(n+m-k)+\mathcal{O}(\Gamma_{\beta}^\alpha(n+m-R)),\;\;\text{for all} \;R\in\mathbb{N}_0.\] 

Thus, for all $n\geq m$,
\[f_n=\overset{R-1}{\underset{k=0}{\sum}}c_k^g \Gamma_{\beta}^\alpha(n-k)+\mathcal{O}(\Gamma_{\beta}^\alpha(n-R)),\;\;\text{for all} \;R\in\mathbb{N}_0,\] 

which gives the desired result. Also, from the equations above we see that $\big(\mathcal{A}_\beta^\alpha f\big)(x)=\big(\mathcal{A}_{\beta+m}^\alpha \displaystyle\frac{f(x)}{x^m}\big)(x)$.

\end{proof}

The next two theorems will be used later in the thesis, the proofs however are lengthy and require many lemmas, and shall thereby be omitted. The reader can refer to \cite{michi} for the complete treatment.

\begin{thm}[\cite{michi}, Prop 4.3.1] \label{derivation}
Let $\alpha,\beta\in \mathbb{R}$, with $\alpha>0$. The linear map $\mathcal{A}_\beta^\alpha$ is a derivation over the ring $\mathbb{R}[[x]]_{\beta}^\alpha$, that is 
\[(\mathcal{A}_\beta^\alpha(f\cdot g))(x)=f(x)(\mathcal{A}_\beta^\alpha g)(x)+g(x)(\mathcal{A}_\beta^\alpha f)(x), \] for all $f,g\in \mathbb{R}[[x]]_{\beta}^\alpha $.
\end{thm}

More interestingly, the next theorem serves as a powerful tool for our purposes. For notation, we set $\mathrm{Diff_{id}}(\mathbb{R},0)=(\{g\in\mathbb{R}[[x]]:g_0=0, g_1=1\},\circ)$, the group of formal diffeomorphisms tangent to the identity, under composition of maps. Similarly, we set  $\mathrm{Diff_{id}}(\mathbb{R},0)_\beta^\alpha=(\{g\in\mathbb{R}[[x]]_\beta^\alpha:g_0=0, g_1=1\},\circ)$ (easily checked to be a monoid).

\begin{thm}[\cite{michi}, Th. 4.4.2]\label{chaintheorem}
Let $\alpha,\beta\in \mathbb{R}$, with $\alpha>0$. Then $\mathrm{Diff_{id}}(\mathbb{R},0)_\beta^\alpha$ is a subgroup of $\mathrm{Diff_{id}}(\mathbb{R},0)$; moreover, for any $f\in\mathbb{R}[[x]]_\beta^\alpha$ and $g\in\mathrm{Diff_{id}}(\mathbb{R},0)_\beta^\alpha$ the following statements  hold:
\begin{enumerate}
    \item $f\circ g$ and $g^{-1}$ are again elements in $\mathbb{R}[[x]]_\beta^\alpha$.
    \item The derivation $ \mathcal{A}_\beta^\alpha$ satisfies a chain rule, namely 
    \begin{equation}\label{chain}
        (\mathcal{A}_\beta^\alpha(f\circ g))(x)=f'(g(x))(\mathcal{A}_\beta^\alpha g)(x)+\bigg(\displaystyle\frac{x}{g(x)}\bigg)^\beta e^{\frac{g(x)-x}{\alpha x g(x)}} (\mathcal{A}_\beta^\alpha f)(g(x)), and
    \end{equation}
     \begin{equation}\label{chain1}
        (\mathcal{A}_\beta^\alpha g^{-1})(x)=-(g^{-1})'(x)\bigg(\displaystyle\frac{x}{g^{-1}(x)}\bigg)^\beta e^{\frac{g^{-1}(x)-x}{\alpha x g^{-1}(x)}}(\mathcal{A}_\beta^\alpha g)(g^{-1}(x)).
    \end{equation}
\end{enumerate}
\end{thm}

It is worth mentioning here that this theorem offers more flexibility than the result by E. Bender in \cite{benderalone}.


\section{Asymptotics of $C_n$}\label{asymptotics o connected}

We pursue a combinatorial interpretation for expressions that appear in the asymptotic expansion of $C_n$, the number of connected chord diagrams on $n$ chords. When we present M. Borinski's calculation of $\mathcal{A}^2_{\frac{1}{2}}C(x)$ in this section, we will find that it is expressed as a rational function of $C(x)$ times an exponential function in $C(x)$. Then it was conjectured that the exponent is following sequence \href{https://oeis.org/A088221}{A088221}. The main result presented in Section \ref{mainhere}  proves this and gives a new combinatorial interpretation for entry  \href{https://oeis.org/A088221}{A088221} of the OEIS.  We will show that \href{https://oeis.org/A088221}{A088221} surprisingly counts pairs of connected chord diagrams (allowing empty diagrams).

\subsection{Overview}

In \cite{michi1}, M. Borinsky studied the asymptotic behaviour of $C_n$, the number of connected chord diagrams on $n$ chords, as an instance of the work on factorially divergent power series. Namely, if we start with an asymptotic expansion that consists of a sum where terms are scalar multiples of (modified) gamma functions, we can then associate the sequence of scalar coefficients to an ordinary power series instead of the gamma functions and study the algebraic advantage of this process. We discussed this in detail in Section \ref{factorially}. So, roughly speaking, it is shown in \cite{michi1} that, after factoring out the factorial divergencies from the coefficients of the generating function $C(x)$ of   connected chord diagrams, we obtain the following power series in terms of $C(x)$:
\begin{equation}\label{exxx}
    \frac{x}{C(x)} \exp\Big(-\frac{1}{2x}(2C(x)+C(x)^2)\Big).
\end{equation}
By (iii) in Lemma \ref{cd}  we can rewrite the expression inside the exponential as $-1$ times
\[1+\frac{1}{2x} C(x) (4x\frac{d}{dx}-1)C(x).\]
Ignoring the 1 and the 1/2, this can be interpreted as the generating function for rooted chord diagrams with at most two connected components, counted by one less than the number of chords. Indeed, a $2x\displaystyle\frac{d}{dx}$ means distinguishing an interval; now, except for the last one, there are two ways of using an interval: we can just place the other $C(x)$ in an interval, or we can place it and pull its root chord to the very front to become the new root. The last interval can only be used in the first way. The coefficients of the expression in \ref{exxx} with after ignoring the 1 and the 1/2 start as $3,10,63,558,6226,82836,\ldots$, which coincide with those of the sequence \href{https://oeis.org/A088221}{A088221} of the OEIS: $1,2,3,10,63,558,6226,82836,\ldots$. The only definition available for the latter is in terms of another sequence: \href{https://oeis.org/A000698}{A000698} which interestingly counts indecomposable chord diagrams. Namely, the definition tells that (in abuse of notation!) $[x^n](\href{https://oeis.org/A088221}{A088221})^n=[x^{n+1}] I(x)$.
  So, there has to be some bridge between chord diagrams with at most two connected components and indecomposable chord diagrams  as we shall here prove. 
  
The problem of finding a better combinatorial interpretation for \href{https://oeis.org/A088221}{A088221} lies as a piece in a more general context. The ultimate goal is to interpret, on combinatorial  basis, the action described by the map $\mathcal{A}_\beta^\alpha:\mathbb{R}[[x]]_\beta^\alpha\rightarrow\mathbb{R}[[x]]$, acting over the subring of $(\alpha,\beta)$-factorially divergent power series (see Capter \ref{chfurther}). 

Throughout the paper we shall stick to the following notation:
\begin{enumerate}
    \item $\mathcal{D}$ is the class of chord diagrams,
    \item $\mathcal{C}$ is the class of connected chord diagrams,
    \item $\mathcal{C}^*$ is the class of connected chord diagrams excluding the one chord diagram,
    \item $\mathcal{D}_{\leq2}$ is the class of chord diagrams with at most two connected components,
    \item $\mathcal{I}$ is the class indecomposable chord diagrams, and finally
    \item $\mathcal{I}_2$  will stand for indecomposable chord diagrams with exactly two components.\\
\end{enumerate}


    
    \subsection{Asymptotic Analysis for Connected Chord Diagrams}
  \setlength{\parindent}{0.5cm}  
  
   Here we follow \cite{michi1} in applying the results discussed in Section \ref{factorially} to chord diagrams. The full study of this approach, as well as more combinatorial applications, can be found in \cite{michi1}. We have $(2n-1)!!=\displaystyle\frac{2^{n+\frac{1}{2}}}{\sqrt{2\pi}}
   \Gamma(n+\frac{1}{2})=
   \displaystyle\frac{1}{\sqrt{2\pi}}\Gamma^2_{\frac{1}{2}}(n)$, and so $D(x)\in \mathbb{R}[[x]]_{\frac{1}{2}}^2$. According to Definition \ref{map} we have   $(\mathcal{A}_{\frac{1}{2}}^2D)(x)=1/\sqrt{2\pi}$. From the recursion (i) in Lemma \ref{cd} we see that $C(xD(x)^2)\in\mathbb{R}[[x]]_{\frac{1}{2}}^2$. Also, $xD(x)^2\in\mathbb{R}[[x]]_{\frac{1}{2}}^2$ since  $\mathbb{R}[[x]]_{\frac{1}{2}}^2$ is a ring by Proposition \ref{subring}. Now, applying Theorem \ref{chaintheorem} to dissolve the composition, we get that $C(x)\in\mathbb{R}[[x]]_{\frac{1}{2}}^2$.\\
   \setlength{\parindent}{0cm}  
   
   From the other way round, we can apply the chain rule of $\mathcal{A}_{\frac{1}{2}}^2$ (Theorem \ref{chaintheorem}), doing so we obtain the equation
   \begin{align*}
       \frac{1}{\sqrt{2\pi}}&=(\mathcal{A}_{\frac{1}{2}}^2D)(x)
       =\Big(\mathcal{A}_{\frac{1}{2}}^2\big(1+C(xD(x)^2)\big)\Big)(x)=\Big(\mathcal{A}_{\frac{1}{2}}^2C(xD(x)^2)\Big)(x)\\&=
       2xD(x)C'(xD(x)^2)(\mathcal{A}_{\frac{1}{2}}^2D)(x)+\bigg(\displaystyle\frac{x}{xD(x)^2}\bigg)^{\frac{1}{2}} e^{\displaystyle\frac{xD(x)^2-x}{2x^2D(x)^2}} \Big(\mathcal{A}_{\frac{1}{2}}^2C\Big)(xD(x)^2)\\&=
       \frac{2}{\sqrt{2\pi}}xD(x)C'(xD(x)^2)+\displaystyle\frac{1}{D(x)} e^{\frac{D(x)^2-1}{2xD(x)^2}} \Big(\mathcal{A}_{\frac{1}{2}}^2C\Big)(xD(x)^2).
   \end{align*}
    This tells us that \[\Big(\mathcal{A}_{\frac{1}{2}}^2C\Big)(xD(x)^2)=\displaystyle
    \frac{D(x)-2xD(x)^2C'(xD(x)^2)}{\sqrt{2\pi}}\;e^{\frac{1-D(x)^2}{2xD(x)^2}}. \]
    
Now notice that $xD(x)^2$ has a compositional inverse, $h(x)$, say. That is, $h(x)D(h(x))^2=x$, and hence, if we replace $x$ by $h(x)$ in the last equation we get 
\[\Big(\mathcal{A}_{\frac{1}{2}}^2C\Big)(x)=\displaystyle
    \frac{D(h(x))-2xC'(x)}{\sqrt{2\pi}}\;e^{\frac{1-D(h(x))^2}{2x}}. \]
But $D(h(x))=1+C(x)$ by (i) in Lemma \ref{cd}, so we get  
\begin{align}\label{A}
    \Big(\mathcal{A}_{\frac{1}{2}}^2C\Big)(x)&=\displaystyle
    \frac{1+C(x)-2xC'(x)}{\sqrt{2\pi}}\;e^{-\frac{1}{2x}(2C(x)+C(x)^2)}\\
    &=\frac{x}{\sqrt{2\pi}C(x)}\;e^{-\frac{1}{2x}(2C(x)+C(x)^2)}\;, \tag{$\dagger$}
\end{align}
where the second equality is achieved by appealing to (iii) in Lemma \ref{cd}.
Obtaining such a computable formula for $\mathcal{A}_{\frac{1}{2}}^2C$ means that we have all the coefficients for the asymptotic expansion of $C(x)$, in the sense of Definition \ref{fdps}. As provided in \cite{michi1}, the first coefficients are 

\begin{equation}\label{calc}
   \Big(\mathcal{A}_{\frac{1}{2}}^2C\Big)(x)=\frac{1}{e\;\sqrt{2\pi}}
   \Bigg(1-\frac{5}{2}x-\frac{43}{8}x^2-\frac{579}{16}x^3-\frac{44477}{128}x^4-\frac{5326191}{1280}x^5\cdots\Bigg).
\end{equation}

Accordingly, by Definition \ref{fdps}, and since $(2n-1)!!=
   \displaystyle\frac{1}{\sqrt{2\pi}}\Gamma^2_{\frac{1}{2}}(n)$, we have for all $\;R\in\mathbb{N}_0$:
\begin{align*}
    C_n&= \overset{R-1}{\underset{k=0}{\sum}} \Gamma^2_{\frac{1}{2}}(n-k)\; [x^k]\Big(\mathcal{A}_{\frac{1}{2}}^2C\Big)(x)+
    \mathcal{O}\Big(\Gamma^2_{\frac{1}{2}}(n-R)\Big)\\
    &= \sqrt{2\pi}\overset{R-1}{\underset{k=0}{\sum}} (2(n-k)-1)!!\; [x^k]\Big(\mathcal{A}_{\frac{1}{2}}^2C\Big)(x)+
    \mathcal{O}\big((2(n-R)-1)!!\big).
\end{align*}

Consequently, for large $n$ we have 
\[C_n= e^{-1}\Big((2n-1)!!-\frac{5}{2}(2n-3)!!-\frac{43}{8}(2n-5)!!-\frac{579}{16}(2n-7)!!
\cdots\Big).\]

This result by M. Borinsky provides a full generalization for the computations in the work of Kleitman  \cite{kleit}, Stein and Everett \cite{steinandeveret} and Bender and Richmond \cite{benderandrichmond}, where only the first term in the expansion has been known. Finally, this also tells us that the probability for a diagram on $n$ chords to be connected is 
$e^{-1}(1-\frac{5}{4n})+\mathcal{O}(1/n^2)$.



\section{The Main Bijection}\label{mainhere}
\setlength{\parindent}{0.6cm}

In this section we derive a bijection described by a reversible algorithm to move
 between the class of two lists of indecomposable chord diagrams (allowing empty lists) and the class of rooted trees, in which vertices are of special type, and where a $\mathcal{D}_{\leq2}$-structure is set over the children of every vertex. Our goal by proving this is to eventually prove that \href{https://oeis.org/A088221}{A088221} counts pairs of connected chord diagrams.
 By the results in \cite{con}, it might be nice to pass this problem into the context of maps on oriented surfaces, but we shall not discuss this here.

First we recall the decomposition of a chord diagram by means of extracting the root component (Lemma \ref{cd}-(i))\;:
\[D(x)=1+C(xD(x)^2).\]
This decomposition will be of great help in the construction presented here, hence it may be wise to accompany it with a suitable notation.\\
  
 \begin{nota}\label{notationdangling} The two diagrams that correspond to each chord in the root component will be referred to as the \textit{right dangling} and the \textit{left dangling} diagrams. Given a chord diagram $D$, the root component will be denoted as $C_\bullet(D)$, while the dangling diagrams will then be $d_r$ and $d_l$. The symbols $d_r,d_l$ and $C_\bullet$ are to be often used as operators.\end{nota}
   
  \begin{exm}\label{exampledangling}
   Consider the following chord diagram.
   \begin{center}
   \raisebox{0cm}{\includegraphics[scale=0.68]{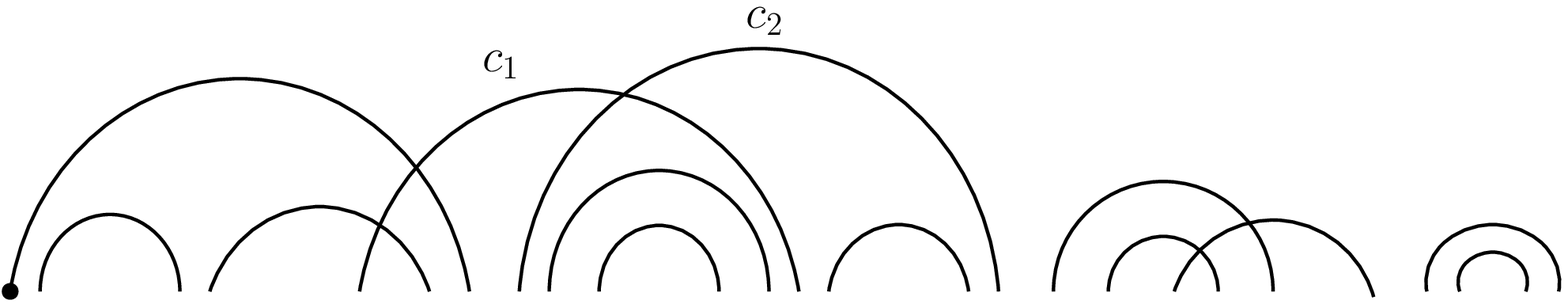}}
   \end{center}

  Yet, this can be also seen as
    \begin{center}
   \raisebox{0cm}{\includegraphics[scale=0.68]{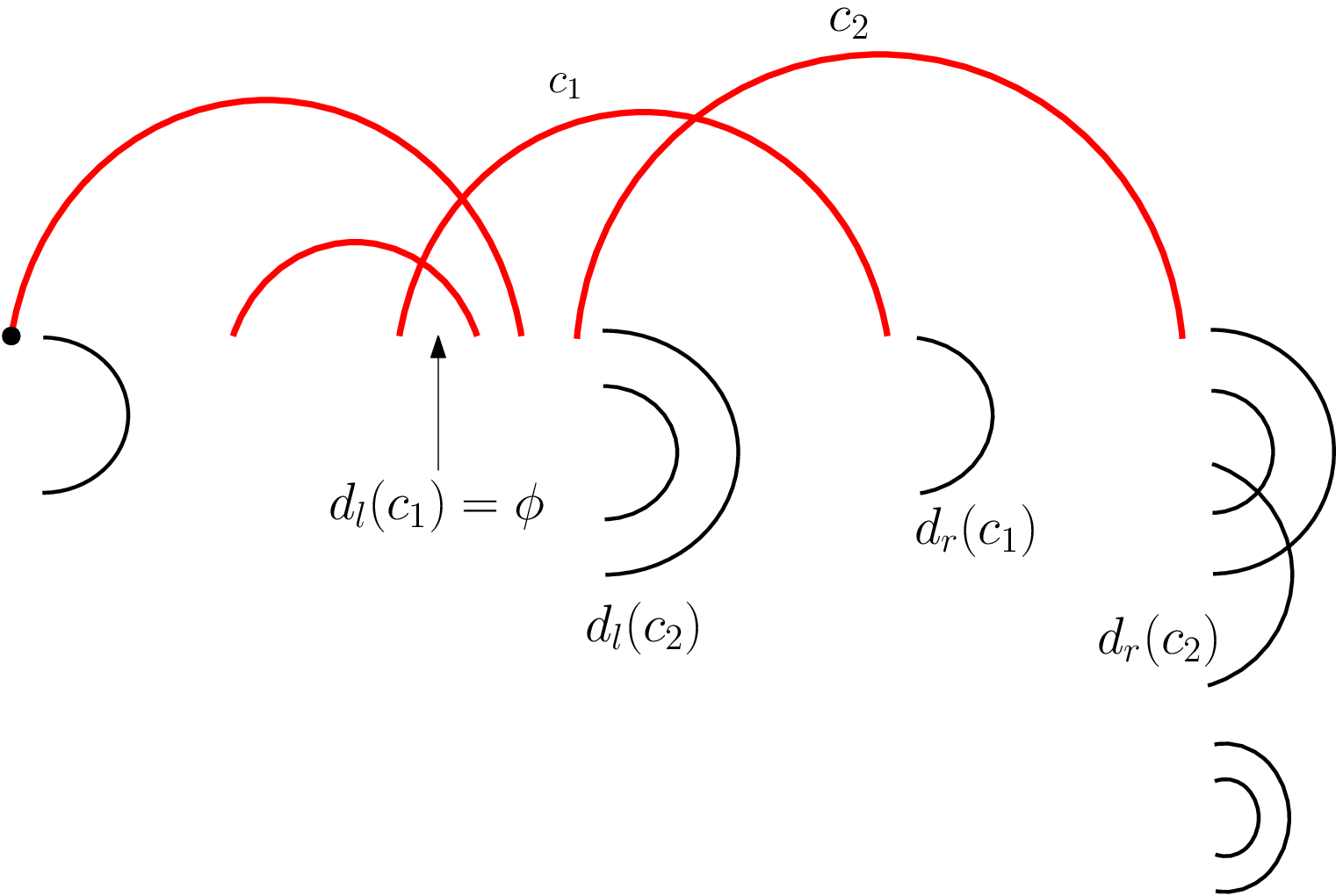}}    
   \end{center}

    which clarifies the decomposition of the lemma. Note that the thick red diagram is the root component $C_\bullet$ of the original diagram. Also, notice that, for example, $d_l(c_1)=\varnothing$. The reason for the nomenclature is now hopefully justified.\\
  \end{exm}

\begin{lem}\label{bij}There is a bijection $\Phi$ between    the class $\mathcal{C}^*$ of rooted connected chord diagrams excluding the one chord diagram, and the class $\mathcal{I}_2$ of indecomposable chord diagrams with exactly two components. Thus, in terms of generating functions
    $\mathcal{I}_2(x)=C(x)-x$.
    \end{lem}
\begin{proof}
The bijection $\Phi$ defined here works almost the same as what is known as the {\it{root share composition}}: Let $C$ be a rooted connected chord diagram. Removing the root chord shall generally leave us with a list of rooted connected components ordered in terms of intersections with the original root. The first  of these components is denoted as $C_2$, while $C_1$ is obtained by removing $C_2$ from the original diagram $C$. Then $(k,C_1,C_2)$ where $1\leq k\leq 2|C_2|-1$, is the {\it{root share decomposition}}  of $C$ (root share decomposition is defined in \cite{yu} by Karen Yeats and N. Marie). $\Phi$ then places the whole $C_1$ at the interval of $C_2$ where the last end of the original root of $C$ would be if the non-root chords of $C_1$ are removed. As seen, the image $\Phi(C)$ in this case is an indecomposable  chord diagram with exactly two components. This definition is reversible. Indeed, given an indecomposable chord diagram with exactly two connected components, $C_2$ will be the outer component and $C_1$ should be the inner one, and $C$ is obtained by just pulling out the first end of the root of $C_1$ to the leftmost position.
\end{proof}

\begin{exm}
Under the map $\Phi$, the chord diagram
\begin{center}
 \raisebox{-0.1cm}{\includegraphics[scale=0.5]{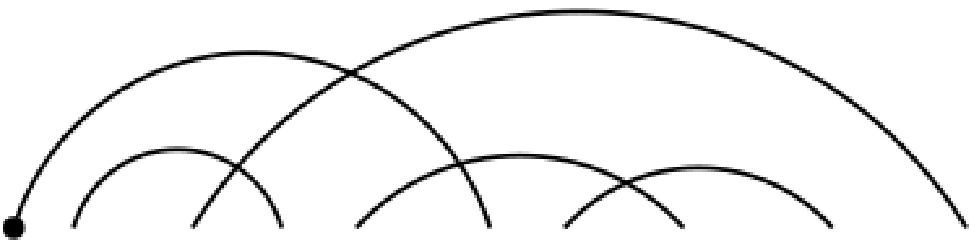}}
 \;\;\text{is mapped to}\;\;
 \raisebox{-0.1cm}{\includegraphics[scale=0.5]{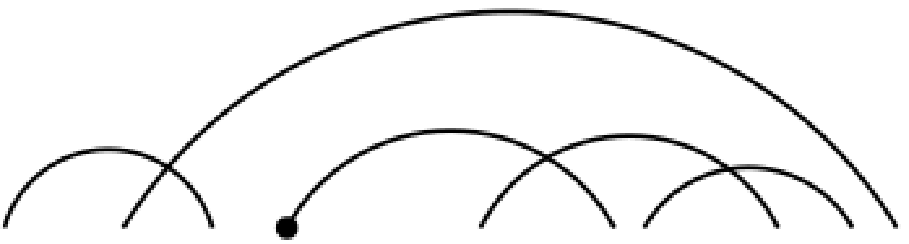}}\;,
\end{center}
\end{exm}

\setlength{\parindent}{0cm}
where, of course, the original root (black) is no longer the root for the resulting diagram.\\

\begin{nota}
In the next theorem, given a finite set $S$ and a class $\mathcal{G}$ of combinatorial objects, the term $\mathcal{G}$-structure on $S$ will simply mean an arrangement of the elements of $S$ into an object from $\mathcal{G}$. The operation $\ast$ stands for the usual ordered product for combinatorial classes. For example, if $\mathcal{T}$ is the class of trees and $\mathcal{K}_{or}$ is the class of oriented complete graphs, then an element from the class $\mathcal{T}\ast\mathcal{K}_{or}$ will be an ordered pair $(T,K)$ where $T\in\mathcal{T}$ and $K\in\mathcal{K}_{or}$. Notice that, for such a product structure to be applied on a finite set there has to be a partition of the set. The reader unfamiliar with this notation from enumerative combinatorics can refer to Appendix \ref{Appendix1}.

\end{nota}

    \begin{thm}\label{mainth}Let $\mathcal{Z}$ be the class of rooted trees where vertices are nonempty ordered sets (paths) and where there is a $\mathcal{D}_{\leq2}$-structure over the  children of every vertex. Then there is a bijection $\Theta$ between $\mathcal{Z}$ and the class $\mathcal{X\ast(D\ast D)}$, where $\mathcal{D}$ is the class of chord diagrams. Consequently, if $Z(x)$ is the generating series for $\mathcal{Z}$, then 
    \begin{equation}
        Z=x\big(\displaystyle\frac{1}{1-I_0}\big)^2,
    \end{equation} where $I_0$ is the generating series for nonempty indecomposable chord diagrams.\end{thm}
    
    \begin{proof}
    Well, for simplicity, we will assume that the objects are given a fixed labelling. let's start with a labelled object $P$ from the class $\mathcal{X\ast(D\ast D)}$. In abuse of notation we shall write $d_r(P)$ and $d_l(P)$ for the two diagrams involved. Also \textit{the chord of} $P$ will mean the unique chord  of which $P$ starts (represented by $\mathcal{X}$ in the decomposition). Then the corresponding $\mathcal{Z}$-tree is obtained through the following algorithm: \\
    \noindent\rule{\textwidth}{0.4pt}
    \textbf{Algorithm 1\label{alg}: Make $\mathbf{\mathcal{Z}}$-Tree}\\
     \noindent\rule{\textwidth}{0.4pt}
   
   \texttt{Input:} $P\;$=\;\raisebox{0cm}{\includegraphics[scale=0.3]{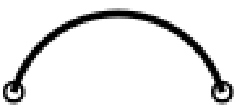}}
   $(d_l,d_r) \;\;$ 
   
  \texttt{initially} $Q_1=P$; \setlength{\parindent}{1.3cm} 
  
  \texttt{queue} $Q=(Q_1)$;
  
  \texttt{integer} $L=\text{length}(Q)$   (automatically modified by any alteration of $Q$);
  
  \texttt{vertex} $v=\odot$\;;
  
  \texttt{label}($v$)\;=\;label given to the chord of $Q_1$; 
  
  \texttt{diagrams} $D_l=D_r=\varnothing$;
  
  \texttt{tree} $Z=v$;\\

   \setlength{\parindent}{0cm} 
 Set $v$ as the root vertex of $Z$\\
  While $Q\neq\varnothing$ \{\\\setlength{\parindent}{1.3cm} 
  \begin{enumerate}
      \item Set $D_l=d_l(Q_1)$ and $D_r=d_r(Q_1)$;\\
  
      \item If $D_l=\varnothing= D_r$ then:

      - push $Q_1$ out of $Q$ (i.e. for all $1\leq k<L, Q_k\leftarrow Q_{k+1}$
      );    
      
      - Go to step (1) again.\\
      
      \item If $D_l= \varnothing$ and $D_r\neq\varnothing$  then:
      
      - Create $|C_\bullet(D_r)|$ children attached to $v$,
      and set their labels to be the same
      
      as the chords in $C_\bullet(D_r)$. Namely, let $\{w_1,\ldots,w_{|C_\bullet(D_r)|}\}$ be the children
      
      and set
      label($w_i$)=label($i^{\text{th}}$ \text{chord})) in the obvious meaning;
      
      - For each $i\in\{1,\ldots,|C_\bullet(D_r)|\}$ add $Q_{L+i}$ to the queue $Q$, where
      
      $Q_{L+i}$:=\;
  \raisebox{0cm}{\includegraphics[scale=0.3]{Figures/onechord.eps}} $\big(d_l(i^{\text{th}} \text{chord}),d_r(i^{\text{th}} \text{chord})\big)$, where the single chord 
   
   is standing for the $i^\text{th}$ chord in $C_\bullet(D_r)$ and the $d_l,d_r$ are the dangling
   
   diagrams of this chord in $D_r$;
   
   - Set $C_\bullet(D_r)$ as the $\mathcal{D}_{\leq2}$-structure over the children of $v$;
   
   - Push $Q_1$ out of $Q$; 
   
   - Set $v=$ vertex for the chord of $Q_1$ (where $Q_1$ has been updated);
   
   - Go to step (1);\\

   \item If $D_l\neq \varnothing$ and $D_r\neq\varnothing$  then:
   
  - Create $|C_\bullet(D_l)|+|C_\bullet(D_r)|$ children attached to $v$,
      and set their labels to
      
      be the same
      as the corresponding chords, as before;
      
      - For each $i\in\{1,\ldots,|C_\bullet(D_l)|+|C_\bullet(D_r)|\}$  add $Q_{L+i}$ to the queue $Q$, where
      
      $Q_{L+i}$:=\;
  \raisebox{0cm}{\includegraphics[scale=0.3]{Figures/onechord.eps}} $\big(d_l(i^{\text{th}} \text{chord}),d_r(i^{\text{th}} \text{chord})\big)$, where the single chord 
   
   is standing for the $i^\text{th}$ chord in $C_\bullet(D_l)$ if $1\leq i\leq|C_\bullet(D_l)|$ and for the
   
   $(i-|C_\bullet(D_l)|)^\text{th}$ chord in $C_\bullet(D_r)$ otherwise;
   
   - Set the concatenation $C_\bullet(D_l) C_\bullet(D_r)$ as the $\mathcal{D}_{\leq2}$-structure over the children 
   
   of $v$;
   
   - Push $Q_1$ out of $Q$; 
   
   - Set $v=$ vertex for the chord of $Q_1$;
   
   - Go to step (1);\\
   \item If $D_l\neq \varnothing$ and $D_r=\varnothing$  then:
   \begin{enumerate}
       \item[(a)] In case $C_\bullet(D_l)$ is a \textbf{single chord} $c$ then:
       
       - vertex $v$ \textbf{absorbs} another node with the label given to $c$. It is appropriate to think of a vertex here as some sort of stack comprising labelled nodes:\\
       
       \begin{center}
           \includegraphics[scale=0.8]{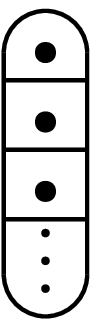}
      \end{center}
       - Add $Q_{L+1}$ to $Q$, where $Q_{L+1}$ consists of $c$ and its dangling diagrams, i.e. $d_l(Q_{L+1})=d_l(c)$ and $d_r(Q_{L+1})=d_r(c)$;  
       
       - Push $Q_1$ out of $Q$; 
   
       - Set $v=$ vertex for the chord of $Q_1$;
   
       - Go to step (1);\\
       
       \item[(b)] Otherwise if $C_\bullet(D_l)$ is \textbf{not a single chord} then:
       
       - Create $|C_\bullet(D_l)|$ children attached to $v$,
       and set their labels to
       be the same
      as the corresponding chords, as before;
      
      - For each $i\in\{1,\ldots,|C_\bullet(D_l)|\}$  add $Q_{L+i}$ to the queue $Q$, where
      
      $Q_{L+i}$:=\;
  \raisebox{0cm}{\includegraphics[scale=0.3]{Figures/onechord.eps}} $\big(d_l(i^{\text{th}} \text{chord}),d_r(i^{\text{th}} \text{chord})\big)$, where the single chord 
   is standing for the $i^\text{th}$ chord in $C_\bullet(D_l)$;

   - Set  $\Phi(C_\bullet(D_l))$ as the $\mathcal{D}_{\leq2}$-structure over the children 
   of $v$;
   
   - Push $Q_1$ out of $Q$; 
   
   - Set $v=$ vertex for the chord of $Q_1$;
   
   - Go to step (1);\;\;\; \}\\
    \end{enumerate}

  \end{enumerate}
   
 \texttt{Output:} $\Theta(P)=Z$.\\
 \noindent\rule{\textwidth}{0.4pt}\\
 
\setlength{\parindent}{0cm}
This algorithm uniquely generates the corresponding tree. Indeed, to see this it shall be enough to see that every branching from a vertex is uniquely translated into chord diagrams:

\begin{enumerate}
    \item If the $\mathcal{D}_{\leq2}$-structure over the children is the concatenation of two connected components, then we know simply that there were nonempty right and left dangling diagrams for the chord corresponding to the vertex. Further, the two connected components are, respectively, the root components of  the dangling diagrams. The order of components in the (rooted) $\mathcal{D}_{\leq2}$-structure dictates which component is for the left or right dangling diagram. 
    
    \item If the $\mathcal{D}_{\leq2}$-structure is just a connected chord diagram, then for the chord corresponding to the vertex only the right dangling diagram existed (nonempty). In particular, this connected structure is the root component for the right dangling diagram.
    The next two cases are actually the most crucial.
    \item If the $\mathcal{D}_{\leq2}$-structure is an indecomposable chord diagram with exactly two connected components, then we learn that for the corresponding chord only the left dangling diagram existed. The root component of which is determined by applying $\Phi^{-1}$ to the $\mathcal{D}_{\leq2}$-structure. This process is well-defined by virtue of $\Phi$ being a bijection.
    
    \item The only remaining case is when the vertex itself is a stack. This marks that, as in the previous case, only the left dangling diagram existed for the chord corresponding to the vertex, and that, further, the root component for this dangling diagram was a single chord with the label given next in the stack. Uniqueness in this case is clear as the information is encoded into the tree in a way that does not interfere with the previous cases, hence no ambiguity arises.\\
\end{enumerate}

This outlines by which means the above algorithm is reversible, and hence establishes the desired bijection. To prove the second part of the theorem just notice that any rooted chord diagram can be viewed as a (possibly empty) list of nonempty indecomposable chord diagrams. 
    \end{proof}
    
    \begin{exm}\label{exampleofZbijection}
    Let $P\in\mathcal{X\ast(D\ast D)}$ be given by $P=\;$ \raisebox{0cm}{\includegraphics[scale=0.3]{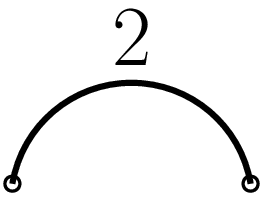}} $\big(d_l(P),d_r(P)\big)$, where

 $d_r(P)=\;$ \raisebox{-0.12cm}{\includegraphics[scale=0.7]{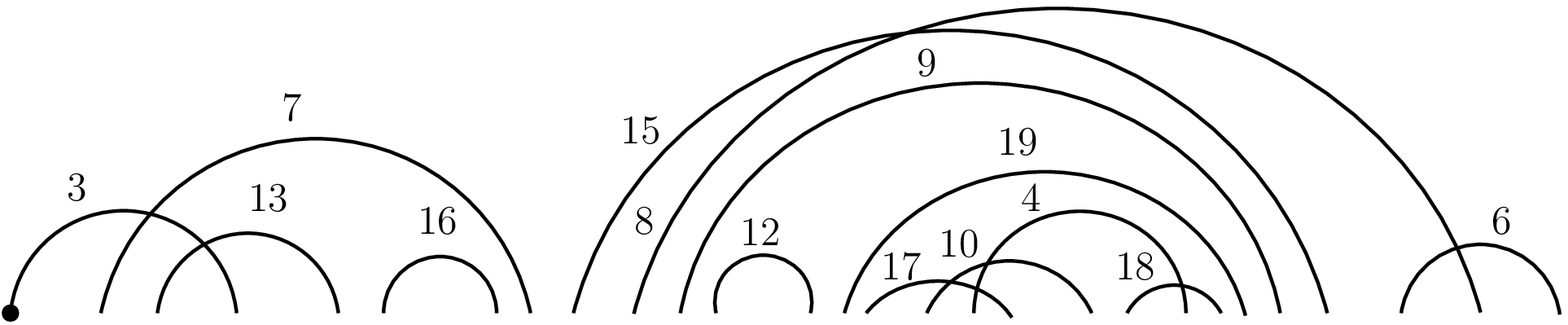}}

  and
    
  $d_l(P)=\;$ \raisebox{-0.12cm}{\includegraphics[scale=0.7]{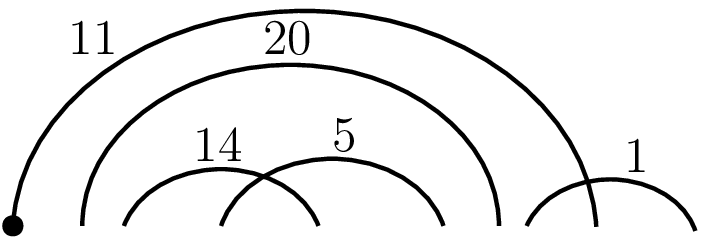}}\;.\\ 

The first iterations in the algorithm are going to be as follows:\\

[1] Initially:

$Q_1=P$,

$Q=(Q_1)$,

$L=1$,

$Z=$ \raisebox{-0.12cm}{\includegraphics[scale=0.7]{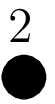}}\;.
 
\noindent\rule{\textwidth}{0.4pt}\\

[2]
$D_l\neq\varnothing$ and $D_r\neq\varnothing$, so we attach the children as of

$Z=\;$ \raisebox{-0cm}{\includegraphics[scale=0.65]{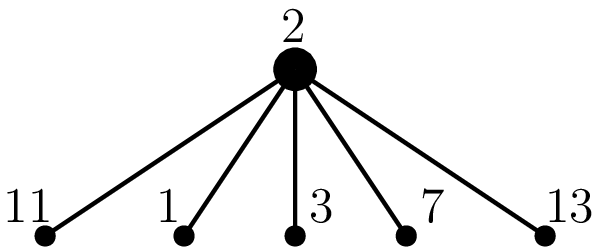}}
   \;with the $\mathcal{D}_{\leq2}$-structure
      \raisebox{-0cm}{\includegraphics[scale=0.6687]{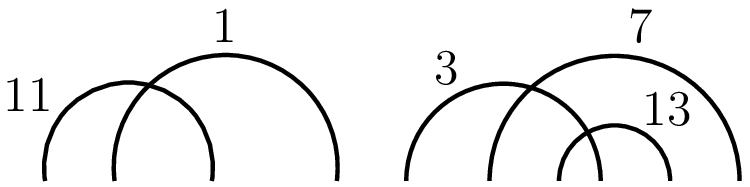}}.
 The updated queue becomes (after including the new entries and pushing the old $Q_1$ out of $Q$):

\[ Q=\bigg(Q_1= \raisebox{-2.6cm}{\includegraphics[scale=0.75]{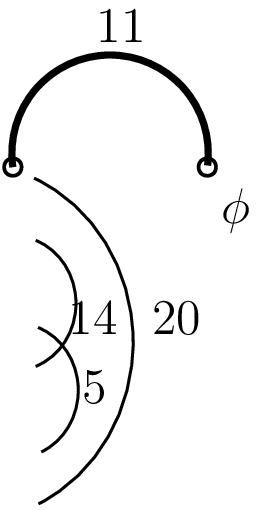}}\;,\;
 \raisebox{-0.5cm}{\includegraphics[scale=0.75]{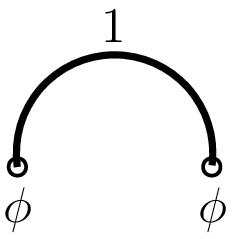}}\;,\;
 \raisebox{-0.5cm}{\includegraphics[scale=0.75]{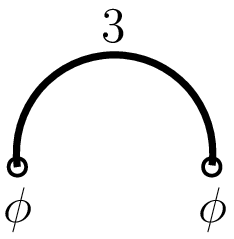}}\;,\;
 \raisebox{-0.5cm}{\includegraphics[scale=0.75]{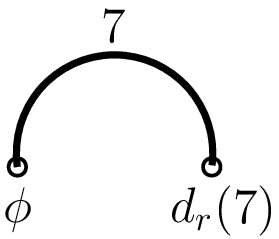}}\;,\;
 \raisebox{-1.65cm}{\includegraphics[scale=0.75]{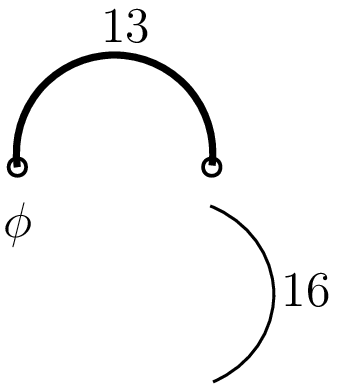}} \bigg)\]

and then the vertex $v$ is set to be $11$.

\noindent\rule{\textwidth}{0.4pt}\\

[3] In this iteration we find that $D_l\neq\varnothing$, whereas $D_r=\varnothing$, moreover, $C_{\bullet}(D_l)$ is the single chord labelled $`20$'. Thus, following the algorithm, one more node is appended to the vertex $v$ which, before this moment, only contained the node labelled $11$. Thus, vertex $v$ is now given by 
\raisebox{-0.12cm}{\includegraphics[scale=0.5]{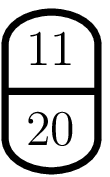}}. Then we add the entry 
$Q_6=\;\raisebox{-0.1cm}{\includegraphics[scale=0.42]{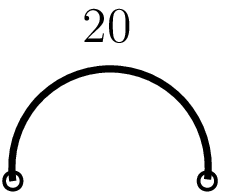}}\;\big(\raisebox{-0.1cm}{\includegraphics[scale=0.7]{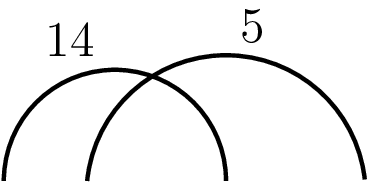}}\;,\varnothing\big)$
to the queue $Q$ (from the end); update $Q$ by pushing out $Q_1$; and set the vertex $v$ to be at chord `1', since it is the chord of the new $Q_1$.\\


Following the algorithm to the end we generate the tree $\Theta(P)$ to be as in Figure \ref{thetaofP} below, where the right column displays the $\mathcal{D}_{\leq2}$-structures pertinent to the children of each vertex (recall that vertices here are generally stacks of nodes). For clarity, structures are displayed level-wise.

 \begin{figure}[h]
     \centering
    \raisebox{0cm}{\includegraphics[scale=0.8]{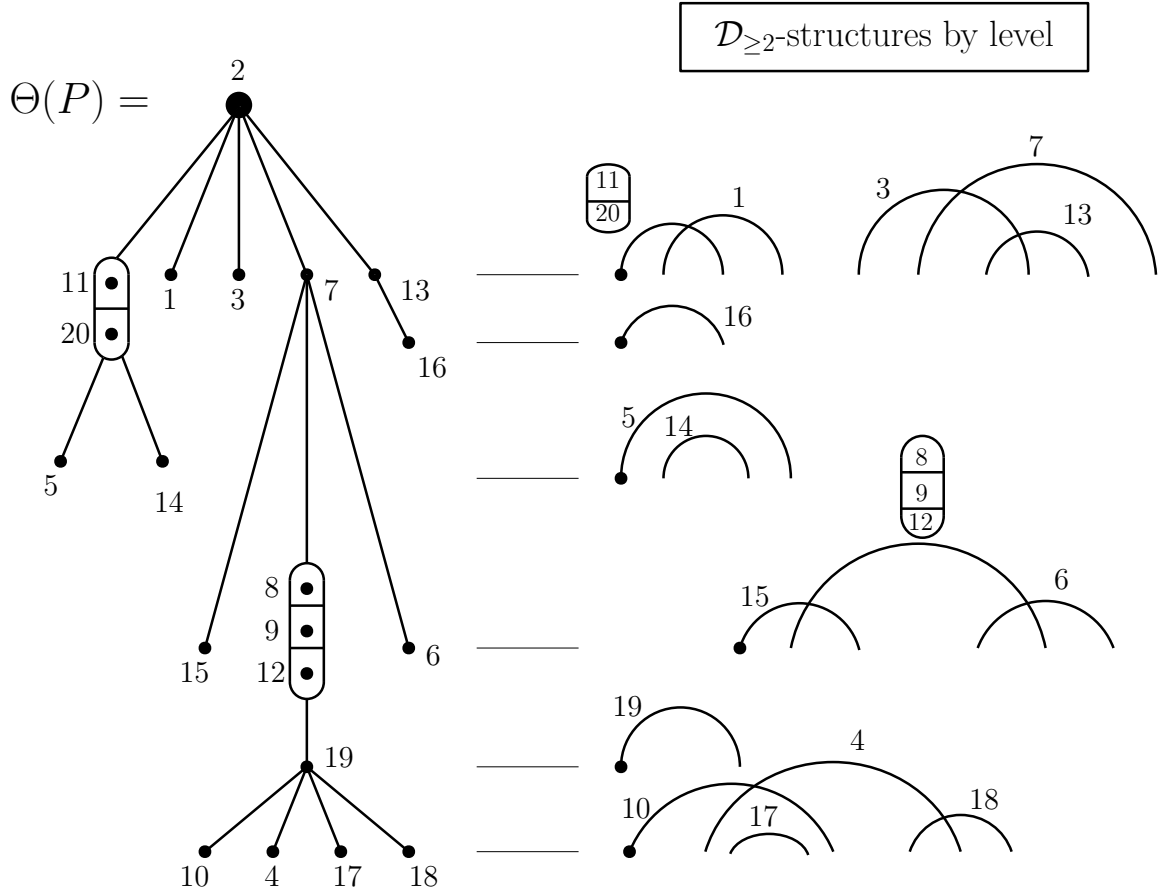}}
     \caption{$\Theta(P)$, with $\mathcal{D}_{\leq2}$-structures displayed by level on the right.}
     \label{thetaofP}
 \end{figure}

 \end{exm} 
 
 \begin{cor}\label{coro} Let $I_0(x)$ be the generating series for nonempty indecomposable chord diagrams as before, and set $B(x)=D_{\leq2}(x)+x$, where $D_{\leq2}(x)$ is the generating function for the class $\mathcal{D}_{\leq2}$. Then $$I_0(x)=\displaystyle\frac{x}{1-xB'(Z)}\;,$$ where $Z$ is the generating series for the class $\mathcal{Z}$ as before.
 
    \end{cor}
 Before proving Corollary \ref{coro} we will prove a decomposition of indecomposable diagrams:
 
 \begin{lem}\label{inde}
   The generating series $I_0$ for nonempty indecomposable chord diagrams satisfies the relation \[I_0(x)=x+\displaystyle\frac{2x^2I_0'(x)}{1-I_0(x)}\;.\label{ind rel}\]
 \end{lem}
 \begin{proof}
     Given a nonempty indecomposable chord diagram we can argue as follows. If the diagram is not a single chord, then removing the root chord generally leaves us with a list of nonempty indecomposable chord diagrams. Moreover, the last diagram in this list carries all the information about the removed root chord encoded as a marked interval that carried the right end of the root chord. Recall that the intervals are meant to be the spacings to the right of every end-node (this includes the last space to the right of the diagram), thus we have $2m$ intervals in a size-$m$ chord diagram. The relation in the lemma is exactly the translation of this decomposition into the world of generating series. 
 \end{proof}

 \begin{exm}
 In the following diagram, the diagram is decomposed into: the root, $D_1$, $D_2$, and ($D_3$, interval 4), where, among the 8 intervals in $D_3$, the root originally landed in interval 4 (marked by a red dotted line).  
 
\includegraphics[scale=0.85]{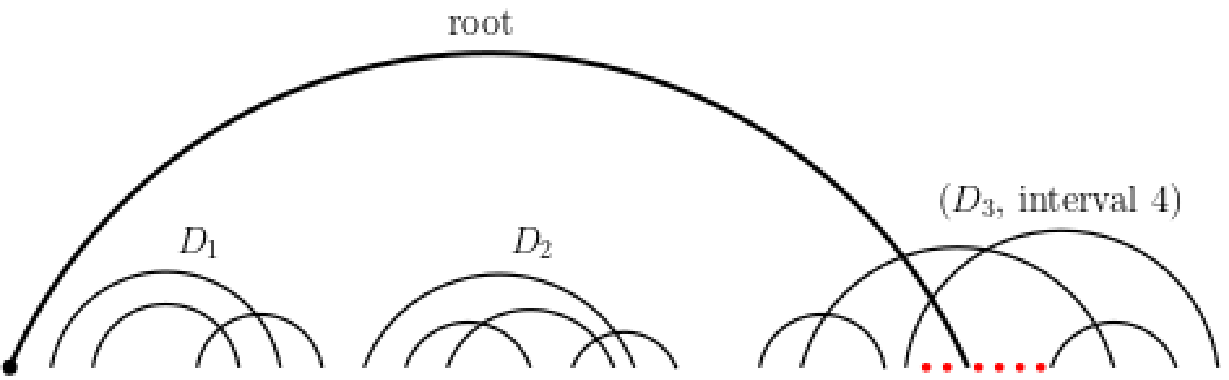}\\
  
 \end{exm}

    \begin{proof}[Proof of Corollary \ref{coro}:]
    First of all, notice that by the definition of the class $\mathcal{Z}$, the generating series $Z$ satisfies the recursion \[Z(x)=xD_{\leq2}(Z)+xZ=x B(Z),\] where $B(t)=D_{\leq2}(t)+t$. Thus $Z'(x)=B(Z)+xB'(Z)Z'=\displaystyle\frac{Z}{x}+xB'(Z)Z',$ and hence \[Z=xZ'(1-xB'(Z)).\]By Theorem \ref{mainth}, we have that $Z=x\Big(\displaystyle\frac{1}{1-I_0}\Big)^2$. Taking the logarithmic derivative of both sides and making use of the above identities we get
    \begin{align*}
        1+2x\frac{d}{dx}\log\Big(\displaystyle\frac{1}{1-I_0}\Big)&=x\displaystyle\frac{d}{dx}\log Z= x\displaystyle\frac{Z'}{Z}
        =\displaystyle\frac{1}{1-xB'(Z)},
    \end{align*}
    
and hence $$\frac{1}{1-xB'(Z)}=1+\displaystyle\frac{2xI_0'}{(1-I_0)}\;.$$

Multiplying by $x$ we get that, by Lemma \ref{ind rel}, $\displaystyle\frac{x}{1-xB'(Z)}=x+\displaystyle\frac{2x^2I_0'}{(1-I_0)}=I_0\;,$  which completes the proof.
    \end{proof}

    \begin{cor}\label{endcoro} Let $A(X)$ be the generating series for the sequence \href{https://oeis.org/A088221}{A088221}. Then $$A(x)=D_{\leq2}(x)+x.$$\end{cor}
\begin{proof}
From Corollary \ref{corolift} (Lagrange inversion), we know that 
\[[x^n]\displaystyle\frac{1}{1-xB'(Z)}=[x^n]B^n(x).\]
 Now, by the definition of the sequence \href{https://oeis.org/A088221}{A088221}, we know that it is the sequence for which $[x^n]A^n(x)=[x^{n+1}]I_0(x)$, where $A(x)$ is assumed to be the generating series for the sequence \href{https://oeis.org/A088221}{A088221}. This gives that $[x^n]A^n(x)=[x^{n+1}]I_0(x)=
 [x^{n+1}]\displaystyle\frac{x}{1-xB'(Z)}=[x^n]\displaystyle\frac{1}{1-xB'(Z)}=[x^n]B^n(x)$, and so $A(x)=B(x)=C_{\leq2}(x)+x.$
\end{proof}

\begin{prop}
The $n^\text{th}$entry in the sequence \href{https://oeis.org/A088221}{A088221} counts the number of pairs $(C_1,C_2)$ of connected chord diagrams (allowing empty diagrams) with total number of chords being $n$.\end{prop}

\begin{proof}
Indeed, any chord diagram with at most two connected components is either: 
(1) empty, (2) connected, (3) concatenation of two connected diagrams, (4) or is indecomposable with exactly two connected components.
    By using Lemma \ref{bij} for the last case we thus get \[A(x)=D_{\leq2}(x)+x=1+C(x)+C^2(x)+C(x)-x+x=(C(x)+1)^2,\] and the result is established.
\end{proof}

\chapter{k-connected Chord Diagrams and Quenched QED}\label{chapterchords2}

In this chapter we study the asymptotic behaviour of the number of $2$-connected chord diagrams, informally speaking these are chord diagrams which require the removal of at least two chords to get them disconnected. Here we obtain an asymptotic expansion for $(C_{\geq2})_n$, the number of $2$-connected chord diagrams on $n$ chords.  As we have mentioned earlier, in \cite{steinandeveret}, it is shown that the proportion of connected chord diagrams approaches $e^{-1}$ as the number of chords goes to infinity. The work of Stein and Everett in \cite{steinandeveret} addresses a special case of a more general result by Kleitman in \cite{kleit}, where the argument was less detailed. Kleitman argues that the proportion of $k$-connected chord diagrams goes to $e^{-k}$. In \cite{michi}, M. Borinsky showed that the asymptotic behaviour of connected chord diagrams is approximated by a series expansion, in which the first term corresponds to the $e^{-1}$ obtained by Stein and Everett, and earlier by Kleitman; whereas the extra terms provide higher precision as needed. Thus, our result here will extend Kleitman's result in very much the same way, this time for the case of $2$-connected chord diagrams. Namely, we obtain an asymptotic expansion for $2$-connected chord diagrams, in which the first term corresponds to the $e^{-2}$ in Kleitman's argument. To be able to extract such information about this class of chord diagrams, we will need to work first on producing a recursion that relates $2$-connected chord diagrams with connected chord diagrams. Finally we will see that this computation surprisingly matches with the asymptotic number of primitive quenched QED vertex diagrams.
It turns out however that the general case of $k$-connected diagrams does not follow in a similar way.

\section{$k$-Connected Chord Diagrams}

The main object we use throughout this chapter is chord diagrams with certain degrees (strengths) of connectivity.
\begin{dfn}[$k$-Connected Chord Diagrams]
A chord diagram on $n$ chords is said to be $k$-\textit{connected} if there is no set $S$ of consecutive endpoints, with $|S|< 2n-k$, $S$ is paired with less than $k$ endpoints not in $S$ (here we assume the endpoints are consecutive in the sense of the linear representation). In other words, the diagram requires the deletion of at least $k$ chords to become disconnected. A $k$-\textit{connected} diagram which is not $k+1$-\textit{connected} will be said to have \textit{connectivity} $k$.
\end{dfn}

\begin{exm} The diagram in Figure \ref{3not4} is 3-connected since it can not be disconnected with the removal of fewer than 3 chords, but it is not 4-connected.

\begin{figure}
    \centering
    \includegraphics[scale=0.65]{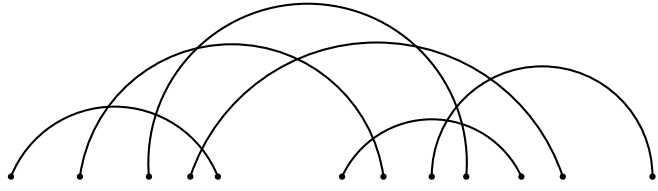}
    \caption{A diagram that is 3-connected but not 4-connected}
    \label{3not4}
\end{figure}
\end{exm}

\begin{dfn}[Cuts and Reasons for Connectivity-$k$]\label{connectv1dfn}
Given a connectivity-$k$ diagram, a set of size $k$ of chords is called a $cut$ if its removal disconnects the diagram. Equivalently, a set $T$ of $k$ chords in a connectivity-$k$ diagram is a cut if there exists a sequence $S$ of consecutive end points such that $|S|< 2n-k$ and all the end points in $S$ are paired together except for $k$ endpoints of the $k$ chords in $T$. Such a sequence $S$ will be called a \textit{reason for connectivity-$k$}.  See Figure \ref{reason} below for illustration.
\end{dfn}

\begin{figure}[h]
    \centering
    \includegraphics[scale=0.85]{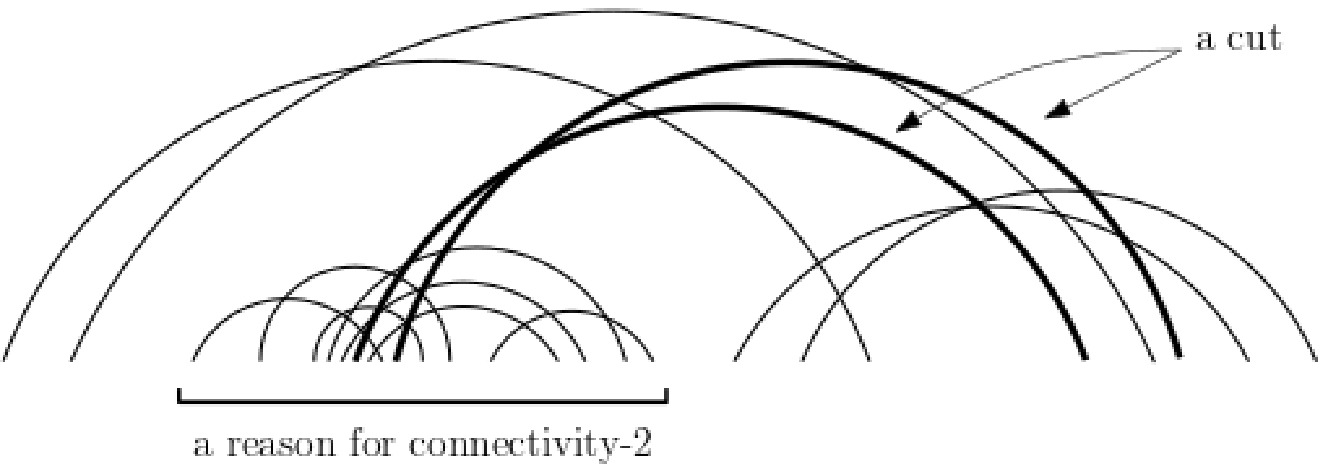}
    \caption{}
    \label{reason}
\end{figure}

\begin{notation}

 For generating functions we shall use the following notation: $C_{\geq k}(x)$ (or $C_{\geq k}$) will denote $k$-connected diagrams whereas $C_k(x)$ (or $C_k$) denotes diagrams with connectivity $k$. So for example $C(x)=C_1(x)+C_{\geq 2}(x)$. 
 \end{notation}

A computation of the first coefficients gives 
\begin{equation}
    \begin{split}
        C(x)&=x+x^2+4x^3+27x^4+248x^5+2830x^6+\cdots\\
        C_1(x)&=x+3x^3+20x^4+185x^5+2101x^6+\cdots\\
        C_{\geq2}(x)&=x^2+x^3+7 x^4+63 x^5+729 x^6+\cdots
    \end{split}
\end{equation}

\section{Root Decomposition for Connectivity-1 Diagrams}
The next results will focus on $2$-connected chord diagrams. Specifically, we will see how a connectivity-$1$ diagram decomposes upon the removal of its root chord. Recall that the defining differential equation for connected chord diagrams 
\begin{equation}
    C(x)=\displaystyle\frac{x}{1-(2xC'(x)-C(x))} \;,\end{equation}
    is obtained through the same root removal process.

\begin{lem}\label{B}
Let $\mathcal{C}$ be the class of connected chord diagrams as before, and let $\mathcal{B}$ be the set of pairs in $\mathbb{N}\times\mathcal{C}$ corresponding to connectivity-1 diagrams, with a special interval $I$ for which the diagram becomes $2$-connected in case a root chord is inserted to the diagram with its far end placed in $I$. The generating function for such pairs is then given by
\begin{equation}
    B(x)=x+\displaystyle\frac{4(xC'_{\geq2}-C_{\geq2})^2}{x-(2xC'_{\geq2}-C_{\geq2})}. \label{B(x)}
\end{equation}
\end{lem}

\begin{proof} 
First notice that in a connectivity-1 diagram the reasons for connectivity-1 must be either disjoint or contained inside each other. By the definition of $\mathcal{B}$, a connectivity-1 diagram $\mathbf{C}$ appearing amongst the considered pairs must be either the one-chord-diagram or must have a nonempty intersection $S$ of all the reasons of connectivity-1, for otherwise there would be a reason of connectivity-1 unaffected by the insertion of the new root. Thus, in that case, and by the observation at the beginning, the set of all reasons of connectivity-1 form a chain totally ordered by inclusion, and $S$ is the minimum reason for connectivity-1 in the diagram. Moreover, the pair $(i,\mathbf{C})$ belongs to $\mathcal{B}$ if and only if the $i$th interval is within $S$.

Let us denote the cut chord of $S$ by $c_1$, and denote by $c_1*S$ the diagram  induced by $c_1$ and the chords of $S$. Then $c_1*S$ is $2$-connected by the minimality of $S$. \\

There are only two possible forms for the diagram $c_1*S$:
\begin{enumerate}
    \item $c_1$ is the root chord (first chord) for $c_1*S$. We say such formation is \textit{root-fixed}, or simply that $S$ is \textit{root-fixed}.
    \begin{figure}[h]
        \centering
       \includegraphics[scale=0.66]{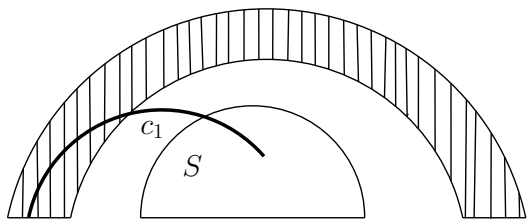}
        \caption{$S$ is root-fixed}
       
    \end{figure}

       \item $c_1$ is the last chord for $c_1*S$. In this case $S$ is said to be 
       \textit{end-fixed}.
       
       \begin{figure}[h]
        \centering
       \includegraphics[scale=0.66]{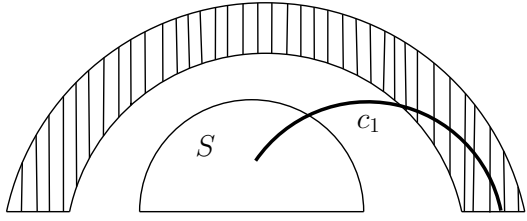}
        \caption{$S$ is end-fixed}
       
    \end{figure}
   \end{enumerate} 
    Now for the sake of clarity, let's decompose the diagram in the following steps. Remember that a pair $(i,\mathbf{C})$ in $\mathcal{B}$ stands for a connectivity-$1$ diagram $\mathbf{C}$ together with an interval with the extra property that we get a $2$-connected diagram upon inserting a new root to land in the chosen interval. 
    
    \textbf{Step 1:} We have seen above that, unless $\mathbf{C}$ is the one-chord-diagram, it should have a minimal reason $S$ for connectivity-$1$; the intervals within $S$ are all the valid intervals for the first entry $i$ so that $(i,\mathbf{C})\in \mathcal{B}$; $S$ and its cut $c_1$ form a $2$-connected diagram and are either root-fixed or end-fixed. In other words, we are counting $2$-conected chord diagrams with a special interval, where 
    
    \begin{enumerate}
        \item For root-fixed pieces, all intervals within $S$ are allowed, this means all intervals in $c_1*S$ are allowed except the first and last intervals.
        \item For end-fixed pieces, all intervals in $c_1*S$ are allowed except the last and the second-last intervals.
        
        \begin{figure}[h]
            
        \centering
       \includegraphics[scale=0.65]{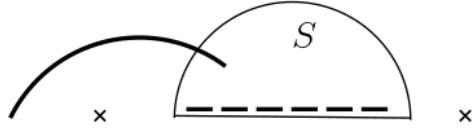}
        \caption{Two intervals are excluded for root-fixed pieces}
       
    \end{figure}
    \end{enumerate} Thus the generating function for this core part in our diagram is 
    
    \begin{equation}
        2(2xC'_{\geq2}-2C_{\geq2}). \label{step1}
    \end{equation}
        
\textbf{Step 2:} We then move outwards, out of the core part. Before the appearance of a next cut chord $c_2$ (if any), there may appear a sequence of components connected to the diagram only through $c_1$. Note that we are making use of the fact that all the reasons for connectivity-$1$ are totally ordered by inclusion, and hence the cuts also appear in a linear order (determined by their reasons) as we move out of the core $S$. This means that $S$ is generally followed by covering layers connected to the diagram only through $c_1$, and then eventually there appears a piece $S_1$ playing the same role of $S$ with the next cut $c_2$. Moreover, with a copy of $c_1$, each of these covering layers (before $c_2$) forms a $2$-connected diagram since they are assumed to precede the next cut (Figure \ref{illustration}). The generating function for each layer alone with $c_1$ is just $C_{\geq2}$. The generating function contribution for sequences of these layers is \begin{figure}[h]
    \centering
    \includegraphics[scale=0.85]{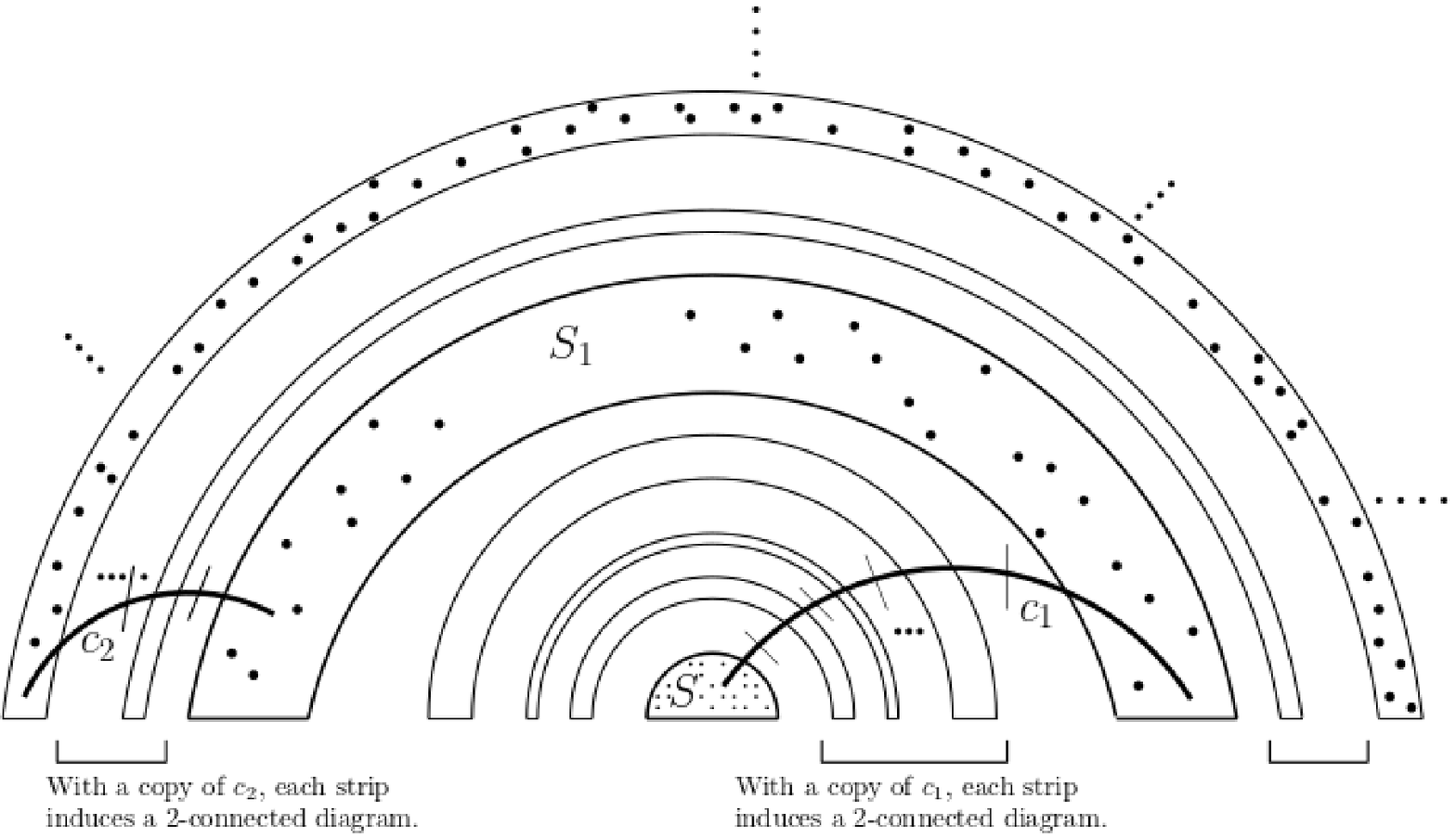}
    \caption{}
    \label{illustration}
\end{figure}

\begin{equation}
    \displaystyle\frac{1}{1-\displaystyle\frac{C_{\geq2}}{x}}\;,\label{step2}
\end{equation}

where we divide by $x$ to avoid overcounting $c_1$, which has already been counted in Step 1. We may think about this as if $c_1$ is divided into many pieces as in Figure \ref{illustration}, and counted only once.

\textbf{Step 3:} In a general situation, there may exist more cuts. Let $c_2$ be the next cut. Then we reach a piece $S_1$ that carries the new cut $c_2$, and, as before, it is either root-fixed or end-fixed. Please note that the previous cut $c_1$ lands within $S_1$ and not further. The chord $c_1$ can be any of the chords in $S_1$ and clearly it can not be $c_2$. One may falsely think that if $c_1$ is end-fixed (root-fixed) in $S$, and $c_2$ is also end-fixed (root-fixed) in $S_1$, then $c_1$ cannot be identified with the first (last) chord in $S_1$ lest we get two disjoint reasons for connectivity-$1$. This is not true however, because even this case can be remedied by the next chord covers: see Figure \ref{remedied} for example. 

\begin{figure}[h]
   
    \includegraphics[scale=0.7]{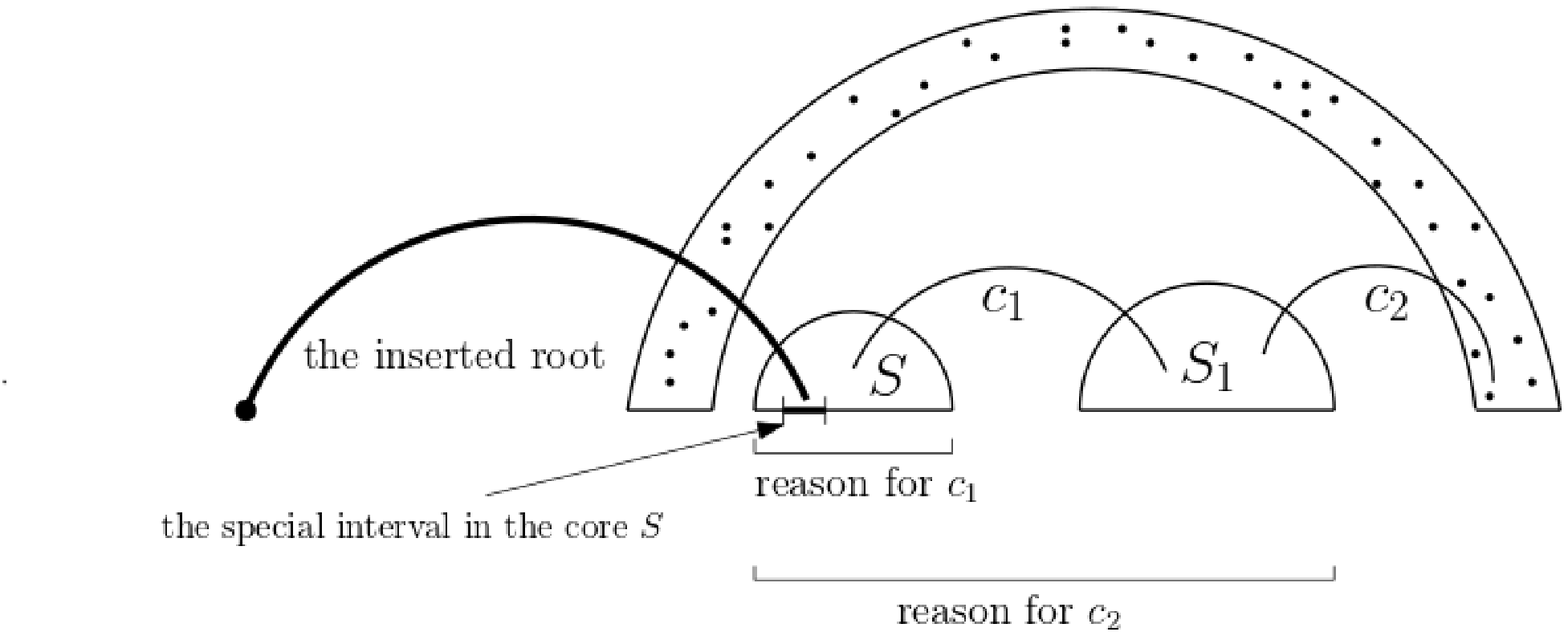}
    \caption{An example where $c_1$ merged with the root of $S_1$ even though $S_1$ is end-fixed.}
    \label{remedied}
\end{figure}

To be more precise, $S_1$ is obtained as follows: (1) remove the maximal reason for connectivity-$1$ that corresponds to $c_1$, without removing $c_1$. (2) in the modified diagram $S_1$ is the minimal reason for connectivity-$1$ corresponding to $c_2$ (note that in this case $S_1$ contains $c_1$). 

By this definition, and since $c_2$ is the first cut after $c_1$, it must be that the diagram $c_2*S_1$ induced by $c_2$ and $S_1$ is $2$-connected. Thus, $c_2*S_1$ is: (1) $2$-connected; (2) either root-fixed or end-fixed; (3) has a special chord (this is $C_1$); (4) the special chord must be different from the root (first chord) in case $c_2*S_1$ is root fixed, or different from the last chord in case $c_2*S_1$ is end-fixed.

Such a structure is counted by  
\begin{equation}
        \displaystyle\frac{2(xC'_{\geq2}-C_{\geq2})}{x},\label{step3}
    \end{equation}
where we divide by $x$ in order not to overcount the chord $c_1$ which has already been accounted for.

As in Step 2, $c_2$ may carry a list of components connected to the diagram only through $c_2$. Each of these components is $2$-connected when taken with $c_2$ (Figure \ref{illustration}). Thus we have a factor of \begin{equation}
    \displaystyle\frac{1}{1-\displaystyle\frac{C_{\geq2}}{x}}\;\label{step33}
\end{equation}

\textbf{Step 4:} 
Now note that Step 3 can be repeated an arbitrary number of times, or not occur at all. More precisely, Step 3 is repeated upon the appearance of every cut after $c_1$. In light of the calculations in Step 3, this amounts to the following part of the generating function 

\begin{equation}
    \displaystyle\frac{1}{1-\left[\displaystyle\frac{2(xC'_{\geq2}-C_{\geq2})}{x} \cdot 
    \displaystyle\frac{1}{\left(1-\displaystyle\frac{C_{\geq2}}{x}\right)}\right]}
    \;.\label{step4}
\end{equation}

\textbf{Step 5:} Let $c_m$ be the last cut, and $S_{m-1}$ be as in Step 2, and let $S_m$ be the last component, that carries $c_m$, which we refer to as the $lock$. We will see now how the lock $S_m$ needs to be considered in a special way. First note that this piece must exist and can not be empty, for otherwise $c_m$ is no longer a cut since it doesn't appear before $S_{m-1}$. Besides, $c_m*S_m$ is $2$-connected by construction. See Figure \ref{general case} for clarity.

So we may exhibit $c_m*S_m$ as a $2$-connected diagram with a special chord to be identified with $c_m$ ($c_m$ has already been determined). The choice of the special chord however is not arbitrary:
\begin{enumerate}
    \item it can be any of the middle chords (i.e. excluding the first and last chords);
    \item if $S_{m-1}$ is end-fixed, the special chord can be the last chord, but it cannot be the root (first) chord;
    
    \item if if $S_{m-1}$ is root-fixed, the special chord can be the root chord, but it cannot be the last chord.
\end{enumerate} 

In other words, assume we are given a $2$-connected diagram to play the role of $S_m$ and we would like to attach $c_m$ to it. If $c_m$ is intended to be identified with one of the middle chords then we should have distinguished one of the middle chords, this is counted by $\displaystyle\frac{(xC'_{\geq2}-2C_{\geq2})}{x}$. 
If $c_m$ is intended to be identified with  either the first or the last chord in $S_m$ the solution is unique: $c_m$ can only be identified with the first chord in $S_m$ if $S_{m-1}$ is root-fixed; and can only be identified with the last chord in $S_m$ if $S_{m-1}$ is end-fixed. Hence the possibilities are simply counted by $C_{\geq2}/x$.

So, in total, the generating function contribution for the lock is

\begin{equation}
   \displaystyle\frac{ C_{\geq2}+(xC'_{\geq2}-2C_{\geq2})}{x}
    \;.\label{step5}
\end{equation}

Finally, by combining the calculations from equations \ref{step1}, \ref{step2}, \ref{step4}, \ref{step5} we get

\begin{equation}
\begin{split}
       B(x)&=x+\displaystyle\frac{2(2xC'_{\geq2}-2C_{\geq2})}{\left(1-\displaystyle\frac{C_{\geq2}}{x}\right)}\cdot\displaystyle\frac{\left(\displaystyle\frac{C_{\geq2}+(xC'_{\geq2}-2C_{\geq2})}{x}\right)}{1-\left[\displaystyle\frac{2(xC'_{\geq2}-C_{\geq2})}{x} \cdot \displaystyle\frac{1}{\left(1-\displaystyle\frac{C_{\geq2}}{x}\right)}\right]}
       \\
          &= x+\displaystyle\frac{4(xC'_{\geq2}-C_{\geq2})^2}{x-(2xC'_{\geq2}-C_{\geq2})}.
    \end{split}
\end{equation}

\begin{figure}[h]
   \center
    \includegraphics[scale=0.8]{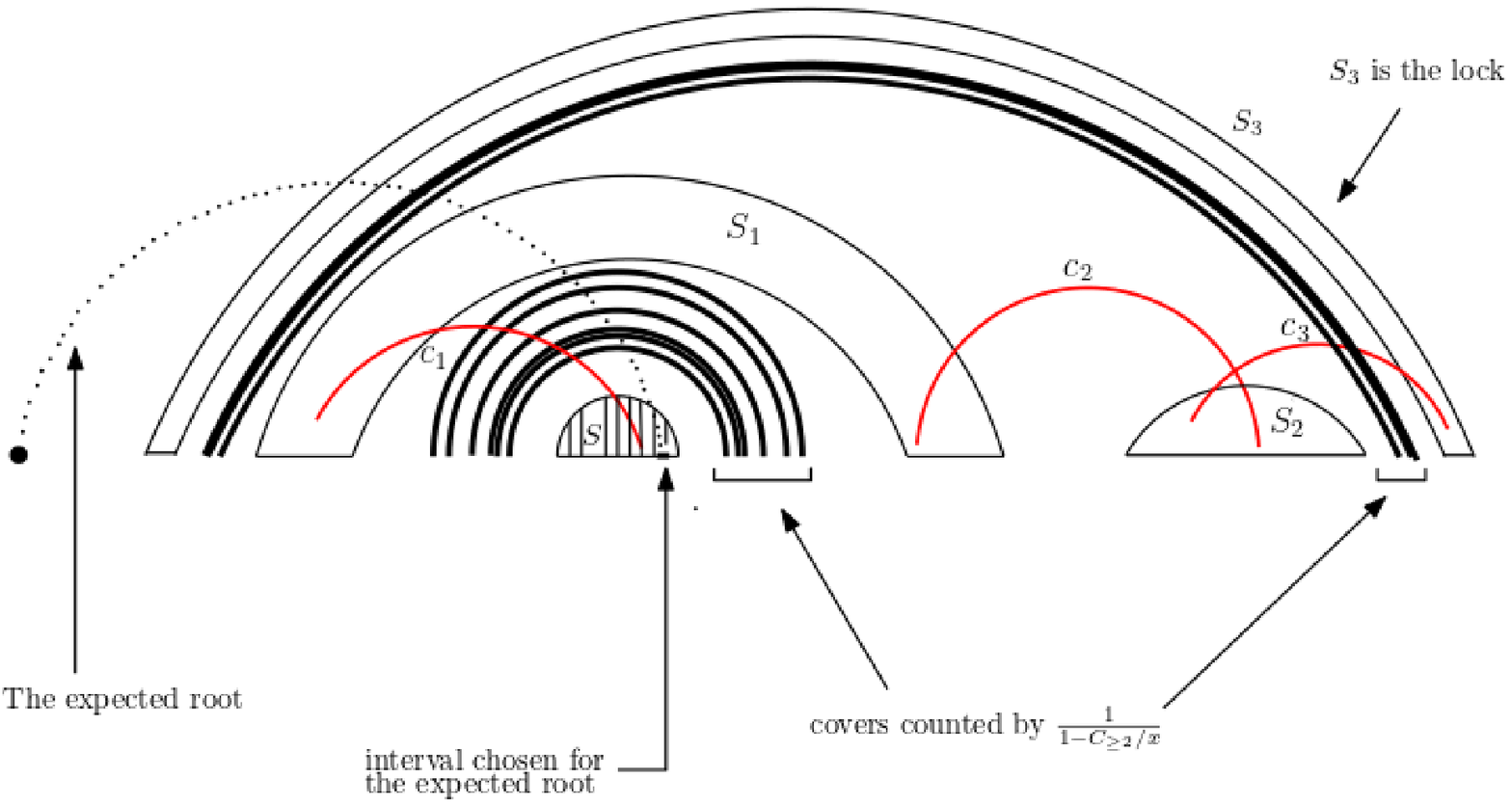}
    \caption{}
    \label{general case}
\end{figure}

\end{proof}

The first terms of $B(x)$ are computed to be

\begin{equation}
    B(x)=x+4x^3+28x^4+288x^5+3552x^6+50692x^7+\cdots. \label{B(x)firstentry}
\end{equation}

\begin{prop}
The following identity holds for connectivity-1 diagrams 
\begin{equation}
    C_1=x\left[1+\displaystyle\frac{(2xC'-C)^2}{1-(2xC'-C)}+2C_{\geq2}+(2xC'_1-C_1)-x-\displaystyle\frac{4(xC'_{\geq2}-C_{\geq2})^2}{x-(2xC'_{\geq2}-C_{\geq2})}\right]\label{C11eq}
\end{equation}
\end{prop}

\begin{proof}

Given a connectivity-1 chord diagram $\mathbf{C}$, the deletion of the root chord breaks the diagram in different ways:

\textbf{Case 1:} The root chord is itself a cut. In this case, upon removal of the root chord, we are left either with the empty diagram (and hence the 1 on the RHS), or we are left with a sequence of two or more connected diagrams nested into each other (Figure \ref{case1}). In the latter, this is the same as a sequence of  connected diagrams each having a special interval (excluding the last interval). The generating function for  connected diagrams with a special interval is $2xC'$, and if we were to exclude the last interval this amounts to $(2xC'-C)$. Thus, the generating function for this part is \[x+x\;\displaystyle\frac{(2xC'-C)^2}{1-(2xC'-C)}.\]

\begin{figure}[h]
    \centering
    \includegraphics[scale=0.6]{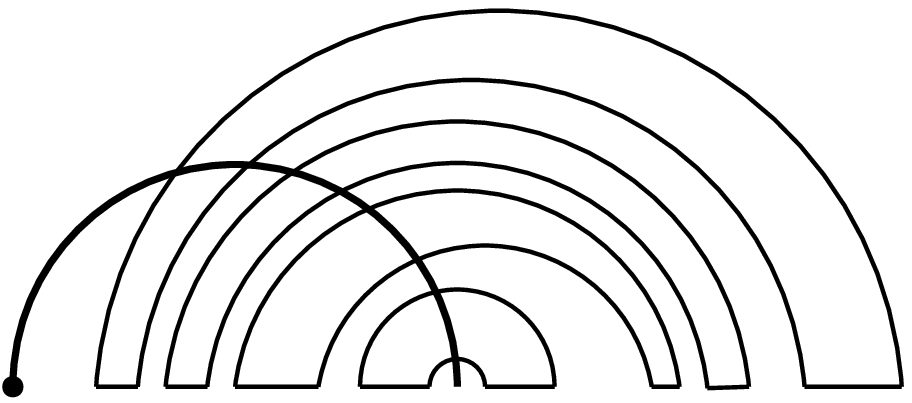}
    \caption{Connected components nested around the root chord.}
    \label{case1}
\end{figure}

\textbf{Case 2:} Removal of the root chord leaves a $2$-connected diagram. This can happen in only two situations, when the end point of the root chord comes third or  second last in the linear order of the diagram (Figure \ref{fig} below). Conversely, given any $2$-connected chord diagram,  adding a root chord with its far end positioned third or pre-last in the linear order results in a connectivity-1 diagram, hence the term $2C_{\geq2}$.

\begin{figure*}[h]\centering
\begin{subfigure}[t]{0.4\textwidth}
  \centering
  \includegraphics[scale=0.75]{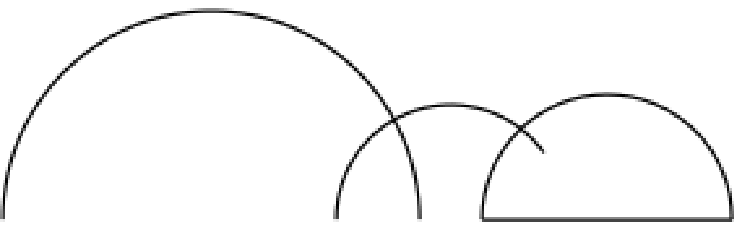}  
  \caption{}
  \label{fi1}
\end{subfigure}~
\begin{subfigure}[t]{0.4\textwidth}
\centering
  \includegraphics[scale=0.75]{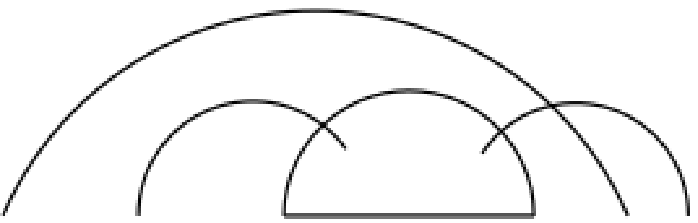}  
  \caption{}
  \label{fi2}
\end{subfigure}
\caption{ }
\label{fig}
\end{figure*}

\textbf{Case 3:} Removal of the root chord does not affect being connectivity-1. In this case we have to be careful: if removing the root chord leaves us again with a connectivity-1 diagram (with a special interval for the far end of the removed root chord), then we may add the term $2xC'_1-C_1$. However, by doing so we count some undesired diagrams. To see this notice that in reversing the process we add a root chord to a given connectivity-1 diagram with some special interval, the interval however is not arbitrary and must be chosen so that the addition of the root is not causing the diagram to become $2$-connected. So we need exactly to exclude the pairs in the class $\mathcal{B}$ of Lemma \ref{B}. Thus the generating function contribution is 

\begin{equation}
    x\left[(2xC'_1-C_1)-x-\displaystyle\frac{4(xC'_{\geq2}-C_{\geq2})^2}{x-(2xC'_{\geq2}-C_{\geq2})}\right].
\end{equation}
The proposition is now proved by combining the above contributions.

\end{proof}

In Section \ref{quenchedsec} we will see that the coefficients of the series $C_{\geq2}$ coincide with those of a renormalization quantity in \textit{quenched QED} \cite{michi, michiq}, namely the sequence counts the number of skeleton quenched QED vertex diagrams (also it is entry \href{https://oeis.org/A049464}{A049464} of the OEIS \cite{OEIS}). This relation between renormalized quantities in quenched QED and combinatorial $2$-connected chord diagrams is not obvious, and seemed not to be known to the authors in \cite{broadhurst, michiq}. Nevertheless, this relation maybe one of a family of similar relations and may lead to further results and observations in both perturbation theory and the underlying combinatorics.

It is therefore important to collect as much as possible of these evaluations of the first terms for the sequences we encounter here. This will make it easier for researchers to anticipate connections. Table \ref{table1} below displays the terms of the series used so far, up to the sixth term, 

As a sanity check we can use table \ref{table1} to calculate the first terms of the series on the right hand side of equation \ref{C11eq}:
\begin{align*}
RHS(\ref{C11eq})  &= x+ (1-1) x^2+ (1+2) x^3+ (7+2+15-4) x^4+(59+14+140-28)x^5+ \\
                  &                           \\
                  &                       +(598+126+1665-288)x^6+(7102+1458+23111-3552)x^7\cdots\\
                  &                           \\
                  &= x+ 3x^3+20 x^4+185 x^5+2101 x^6+28119 x^7+\cdots,
\end{align*}

which indeed coincides with $C_1$.

\begin{table}[h]
\center
\begin{tabular}{|c|c||c|c|c|c|c|c|c|c|c|}\hline
&           &$x^0$ & $x$   & $x^2$  & $x^3$ & $x^4$ & $x^5$ & $x^6$ & $x^7$ & $x^8$       \\\hline\hline
1&$C$         &   0  & 1     & 1      & 4     & 27    & 248   & 2830  & 38232      &  593859   \\\hline
2&$C_1$       &   0  & 1     & 0      & 3     & 20    & 185   & 2101  & 28119      &   431924 \\\hline
3&$C_{\geq2}$ &   0  & 0     & 1      & 1     & 7     & 63    & 729   &   10113    & 161935
                  \\\hline\hline
4&$xC'$       &   0  & 1     & 2      & 12    & 108   & 1240  & 16980 &       &
\\\hline
5&$xC'_1$     &   0  & 1     & 0      & 9     & 80    & 925   & 12606 &       &
\\\hline
6&$xC'_{\geq2}$
              &   0  & 0     & 2      & 3     & 28    & 315   & 4374  &       &
                  \\\hline\hline
7&$(2xC'-C)$  &   0  & 1     & 3      & 20    & 189   & 2232  & 31130 &       &
\\\hline 
8&$x\frac{(2xC'-C)^2}{1-(2xC'-C)}$
              &   0  & 0     & 0      & 1     & 7     & 59    & 598   &  7102 &
                  \\\hline\hline
9&$2xC_{\geq2}$
              &   0  & 0     & 0      & 2     & 2     & 14    & 126   & 1458  & 
                  \\\hline\hline

10&$x(2xC'_1-C_1)$
              &   0  & 0     & 1      & 0     & 15    & 140   & 1665  & 23111 &
                  \\\hline\hline
      
11&$xB$       &   0  & 0     & 1      & 0     & 4     & 28    & 288   & 3552  & 50692  \\\hline
 \end{tabular}\caption{The first terms of some of the series used in this section.}
\label{table1}
\end{table}

\section{Functional Recurrence for $2$-Connected Diagrams}\label{functionalrecurrence2connected}

In \cite{michi}, the functional relation $D(x)=1+C(xD^2)$ provided the suitable grounds for deriving information about the asymptotic behaviour of $C_n$, the number of connected chord diagrams on $n$ chords. The composition of maps in the second term in this relation transforms nicely into a product when taking the alien derivative $\mathcal{A}_{1/2}^2$. In the aftermath of our meeting in the Canadian Mathematical Society session about chord diagrams (Dec. 2019), M. Borinsky suggested to the author that it may be possible to obtain similar functional relations for the higher connectivity diagrams. This was motivated by the asymptotic pattern shown in Kleitman's results \cite{kleit}. In this section we derive such a functional relation for $2$-connected chord diagrams, and will use it later to study the asymptotic behaviour of the number of $2$-connected chord diagrams. However, just as the case for general graphs, it is not clear whether $3$-connected diagrams and $k$- connected diagrams in general do follow similar relations. 
\begin{prop}\label{myproposition2connected}
The following functional relation between connected and $2$-connected diagrams holds:

\begin{equation}
    C=\displaystyle\frac{C^2}{x}-C_{\geq2}\left(\displaystyle\frac{C^2}{x}\right). \label{c2con}
\end{equation}
\end{prop}

\begin{proof}
Assume that a connected chord diagram $\mathbf{C}$ is given. We can determine the maximal sequences of consecutive end points that are reasons for connectivity-$1$. A sequence $S=s_1s_2\ldots s_m$ of consecutive end points is of this type if and only if \begin{enumerate}
    \item $S$ is a reason of connectivity-$1$ corresponding to a cut chord $c$ that has exactly one end point inside $S$, say this is $s_i\;$ where $1<i < m$.
    \item $S$ is not contained in any other reason for connectivity-$1$.

\end{enumerate}

However, these reasons for connectivity-$1$ may overlap (see Figure \ref{reasonprob}), and so we will need to devise a canonical way for partitioning our diagram in terms of these maximal sequences.  

\begin{figure}[h]
   \center
    \includegraphics[scale=0.8]{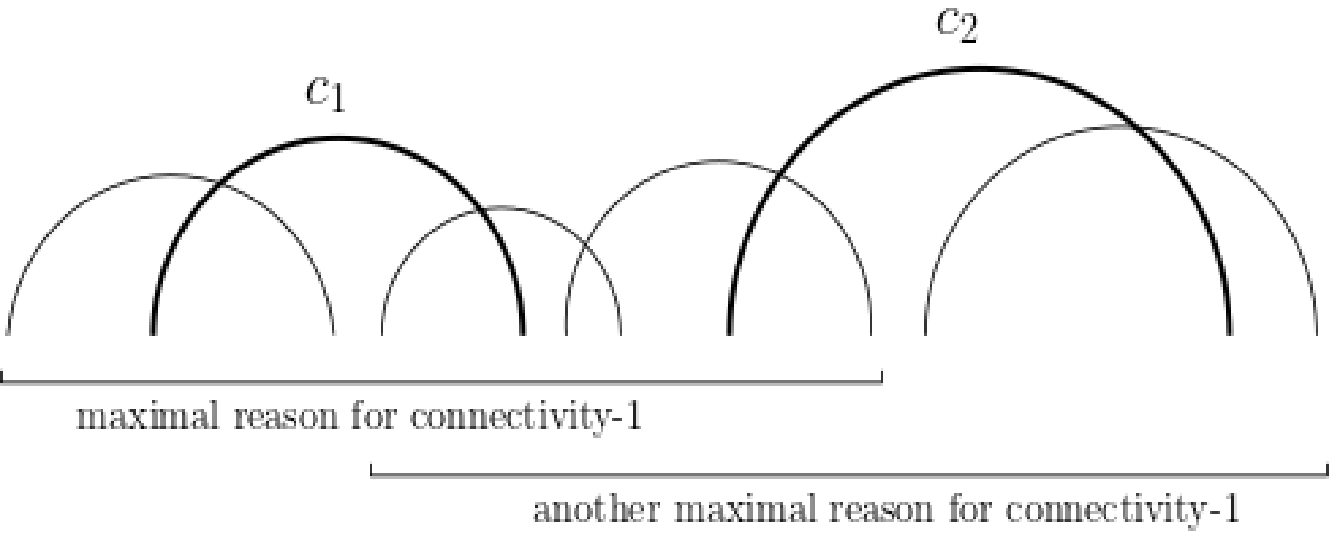}
    \caption{}
    \label{reasonprob}
\end{figure}

\textbf{Case 1:} $\mathbf{C}$ is the \textbf{single chord diagram}. In this case we do nothing, and the contribution to the generating function is just $x$.

In the next cases we generally assume $\mathbf{C}$ is not the single chord diagram.

\textbf{Case 2:} The root endpoint $r_0$ (left endpoint of the root chord of $\mathbf{C}$) is \textbf{not} contained in any reason for connectivity-$1$. In this case we determine the maximal reasons for connectivity-$1$ that are obtained through the next procedure by moving from left to right. Such a diagram generally looks like the example in  Figure \ref{looklike} below.

\begin{figure}[h]
   \center
    \includegraphics[scale=0.43]{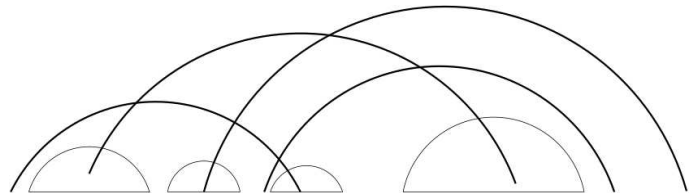}
    \caption{An example for Case 2}
    \label{looklike}
\end{figure}

Consider the diagram $\mathbf{C^\times}$ obtained from $\mathbf{C}$ as follows: 
\begin{enumerate}
    \item Starting from the left, determine the first endpoint that is included in some reason for connectivity-$1$. Let's denote it temporarily by $s_1$. Move to step 5 if the diagram is $2$-connected and no such endpoint exists. 
    
    \item Determine the maximal reason $S=s_1s_2\ldots s_m$ for connectivity-$1$ that contains $s_1$ by consecutively trying to include the next endpoints to the right. Assume $S$ corresponds to a cut chord $c$ that has the end point $s_i$, say.
    
    \item Let $\mathbf{C^\times}$ be the diagram obtained by removing the sub-diagram induced by $S-{s_i}$, i.e. we remove $S$ without removing $c$. 
    \item Update by setting $\mathbf{C}=\mathbf{C^\times}$, and go back to step 1.\\
    
    \item Output $\mathbf{C^\times}$.

\end{enumerate}

\textbf{Observation 1:} 
Notice that in the process of extracting $\mathbf{C^\times}$ the diagram remains connected, this is because any of the removed sub-diagrams has been connected to the rest of $\mathbf{C}$ through a single cut chord which is not removed.

Clearly, $\mathbf{C^\times}$ will not preserve any original reason for connectivity-1 in $\mathbf{C}$. Moreover, notice that again since each sub-diagram removed has only been connected to the rest of $\mathbf{C}$ through a single cut chord (which is kept), the process should not affect the connectivity of the rest of $\mathbf{C}$ neither will create new cuts.

\textbf{Observation 2:} 
Also, step 1 is exclusive throughout the procedure. Indeed, if there is no such endpoint in a connected diagram (Observation 1) then the diagram is either $2$-connected or is the single chord diagram (it can't be empty). The latter however never occurs: Initially the diagram is not the single chord diagram by our assumption. Further, $\mathbf{C}$ is not reduced to a single chord diagram at any iteration since this should imply that the root endpoint $r_0$ is contained in a reason for connectivity-1. Therefore the procedure eventually halts and the output is $2$-connected. 

\textbf{Observation 3:} It is important to note that also the last endpoint in $\mathbf{C}$ is not included in any reason for connectivity-1, for this will imply the same for $r_0$.
   
To summarize the procedure above, we are removing maximal reasons of connectivity-$1$ that appear in a certain order when moving from left to right, without removing their corresponding cuts. This is illustrated in Figures  \ref{con11} and \ref{con1removed}.

This gives a reversible decomposition into a $2$-connected where each endpoint, except the first and last endpoints, is assigned to a connected chord diagram counted by one less chord. In other words, we will count each middle chord (i.e. whose endpoints are not the root nor the last endpoint) in $\mathbf{C}^\times$ when counting the connected diagram for its right endpoint by keeping it as a root for this diagram, while on the other hand, the diagram for the left endpoint will be counted by one less chord to avoid overcounting.

This can also be viewed as follows:

Given a connected chord diagram $\mathbf{C}$ (which is not the single chord) we undergo the described procedure to get 

\begin{enumerate}
    \item a $2$-connected chord diagram $\mathbf{C}^\times$,
    \item the root chord $c_{r}$ corresponds to a connected chord diagram that consists of $c_r$ and the diagram attached to the right endpoint of the root, in which we will keep the root. 
    \item the chord $c_l$ carrying the last endpoint of $\mathbf{C}$ corresponds to a connected chord diagram that consists of $c_l$ kept as a root for whatever the diagram attached to the left endpoint.
    
    \item Every middle chord $c$ can be replaced with a pair of diagrams corresponding to right and left endpoints. The diagram for the left endpoint has its root a copy of $c$ that is not going to be counted and is connected; while the diagram for the right endpoint keeps $c$ and is connected as well.
\end{enumerate}

In terms of generating functions  the contribution of Case 2 is seen now to be:

\begin{equation}
    C(x)^2 \;\left[\left. \displaystyle\frac{C_{\geq2}(t)}{t^2}\right|_{t=C(x)^2/x}\right],\label{case2}
\end{equation}
where we divide by $t^2$ to account for the fact that two of the chords are treated differently (namely $c_r$ and $c_l$). Each of these two chords contributes with $C(x)$ as shown above.

\begin{figure}[h]
   \center
    \includegraphics[scale=0.7]{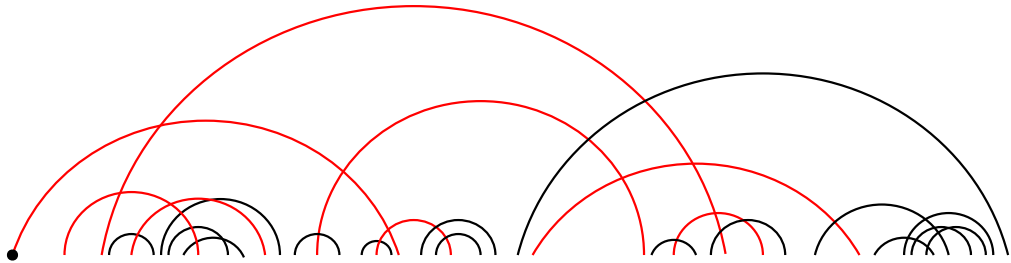}
    \caption{A diagram $\mathbf{C}$ with cuts highlighted (red).}
    \label{con11}
\end{figure}

\begin{figure}[h]
   \center
    \includegraphics[scale=0.56]{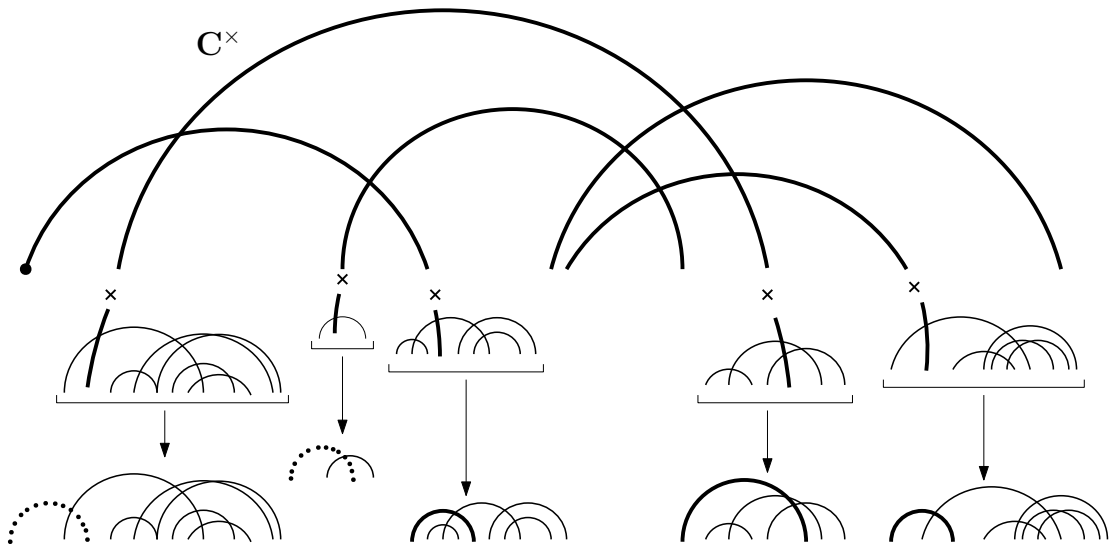}
    \caption{Maximal reasons obtained through the procedure (underlined), $\mathbf{C^\times}$ (bold).}
    \label{con1removed}
\end{figure}

\textbf{Case 3:} The root endpoint $r_0$ (left endpoint of the root chord of $\mathbf{C}$) is \textbf{contained} in a reason for connectivity-$1$. In this case we determine the maximal reason for connectivity-1 containing $r_0$, donted by $S$, by consecutively checking every endpoint to the right of $r_0$. Let $c^*$ be the corresponding cut for $S$. Now, by the maximality of $S$ it must be that none of the reasons for connectivity-1 that come later could be extended to contain $S$.  This means that the diagram obtained by removing $S$ (without removing $c^*$) is of the type considered in Case 2 above. The diagram will generally be structured as  in Figure \ref{struc}. Then the contribution to the generating function is

\begin{equation}
   \displaystyle\frac{C(x)-x}{x}\;. \;C(x)^2 \;\left[\left. \displaystyle\frac{C_{\geq2}(t)}{t^2}\right|_{t=C(x)^2/x}\right], \label{case3}
\end{equation}

where the factor of $\displaystyle\frac{C(x)-x}{x}$ corresponds to the sub-diagram induced by $S$ together with $c_r$: the $(C-x)$ since $S$ is always nonempty in this case,  and we divide by $x$ since $c_r$ is counted with the rest of the diagram. 
 \begin{figure}[h]
   \center
    \includegraphics[scale=0.67]{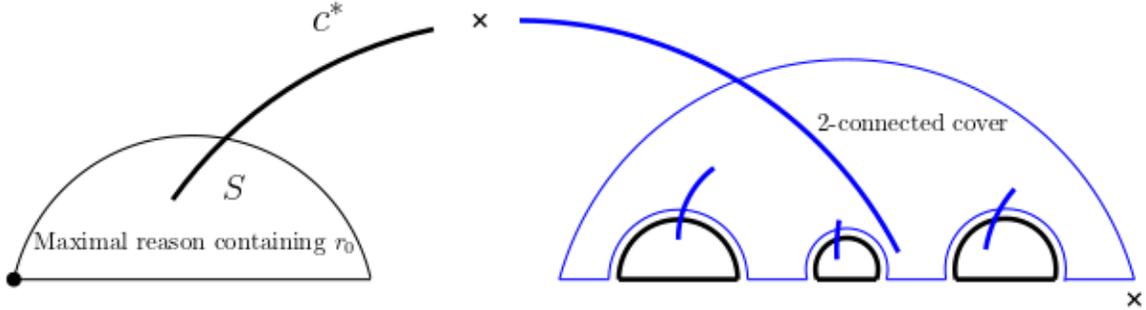}
    \caption{$S$ is the maximal reason containing $r_0$, and so the rest of the diagram should be covered with a $2$-connected sub-diagram (blue).}
    \label{struc}
\end{figure}

Thus, by combining the findings of the three cases we have 

\begin{align*}
   C(x)&=x+C(x)^2 \;\left[\left. \displaystyle\frac{C_{\geq2}(t)}{t^2}\right|_{t=C(x)^2/x}\right]+
   \displaystyle\frac{C(x)-x}{x}\;. \;C(x)^2 \;\left[\left. \displaystyle\frac{C_{\geq2}(t)}{t^2}\right|_{t=C(x)^2/x}\right]\\
   &\\
  &=x+\displaystyle\frac{C(x)^3}{x}\;.  \;\left[\left. \displaystyle\frac{C_{\geq2}(t)}{t^2}\right|_{t=C(x)^2/x}\right]\\
  &\\
  &=x+\displaystyle\frac{x}{C(x)}\;.\; \left(\displaystyle\frac{C(x)^2}{x}\right)^2\;. \;\left[\left. \displaystyle\frac{C_{\geq2}(t)}{t^2}\right|_{t=C(x)^2/x}\right]\\
  &\\
  &=x+\displaystyle\frac{x}{C(x)}\;.\;C_{\geq2}\left(\displaystyle\frac{C(x)^2}{x}\right),
\end{align*}

and the result now follows.

\end{proof}

For future reference, we include the first terms of the expressions involved in the previous decomposition. The reader can check that the sum of $x$ plus  lines 3 and 4 in the next table gives the first terms of $C(x)$.

\begin{table}[h]
\center
\begin{tabular}{|c|c||c|c|c|c|c|c|c|}\hline
&             &$x^0$ & $x$   & $x^2$  & $x^3$ & $x^4$ & $x^5$ & $x^6$ \\\hline\hline
1&$C^2/x$     &   0  & 1     & 2      & 9     & 62    & 566   & 6372      \\\hline
2&$\left. \displaystyle\frac{C_{\geq2}(t)}{t^2}\right|_{t=C(x)^2/x}$      
              &   1  & 1     & 9      & 100   & 1323  & 20088 & 342430         \\\hline
3&$C(x)^2 \;\big[\left. \displaystyle\frac{C_{\geq2}(t)}{t^2}\right|_{t=C(x)^2/x}\big]$      
              &   0  & 0     & 1      & 3     & 20    & 189   & 2232      \\\hline
4&$\displaystyle\frac{C(x)-x}{x}\;. \;C(x)^2 \;\big[\left. \displaystyle\frac{C_{\geq2}(t)}{t^2}\right|_{t=C(x)^2/x}\big]$ 
              &   0  & 0     & 0      & 1     & 7     & 59    & 598  
                  \\\hline

 \end{tabular}\caption{The first coefficients of the series involved in the terms of the decomposition of $C_{\geq2}$. }
\label{table2}
\end{table}

\section{Asymptotics of the number of 2-connected chord diagrams}

In this section we will see how to successfully estimate the number $(C_{\geq2})_n$ of $2$-connected diagrams when $n$ is large. The asymptotic behaviour obtained here will extend Kleitman's result \cite{kleit} and will shed light on an unexplained pattern for the images of the alien derivative. It turns out that $\mathcal{A}_{\frac{1}{2}}^2C_{\geq2}$ takes the form of a rational function in $C_{\geq2}$ times the exponential of a quadratic expression in the reciprocal of that rational function. This was exactly the same case for $\mathcal{A}_{\frac{1}{2}}^2C$ (as well as monolithic diagrams and simple permutations). We will proceed now by applying a suitable alien derivative as was done before for connected chord diagrams.

In the previous section we have seen that \begin{equation*}
    C=\displaystyle\frac{C^2}{x}-C_{\geq2}\left(\displaystyle\frac{C^2}{x}\right). 
\end{equation*}

We will start by applying the alien derivative $\mathcal{A}_{\frac{3}{2}}^2$, which is allowed since $C(x)\in\mathbb{R}[[x]]_{\frac{1}{2}}^2\subset \mathbb{R}[[x]]_{\frac{3}{2}}^2 $ by Corollary \ref{corplusm1}.

\begin{align*}
    \big(\mathcal{A}_{\frac{3}{2}}^2C\big)(x)
    &=\mathcal{A}_{\frac{3}{2}}^2\left(\displaystyle\frac{C(x)^2}{x}\right)-\mathcal{A}_{\frac{3}{2}}^2\left(C_{\geq2}\left(\displaystyle\frac{C(x)^2}{x}\right)\right)(x)\\
    &=\mathcal{A}_{\frac{1}{2}}^2\big(C(x)^2\big)-\mathcal{A}_{\frac{3}{2}}^2\left(C_{\geq2}\left(\displaystyle\frac{C(x)^2}{x}\right)\right)(x)\\
    &=2 C \big(\mathcal{A}_{\frac{1}{2}}^2C\big)(x)-\mathcal{A}_{\frac{3}{2}}^2\left(C_{\geq2}\left(\displaystyle\frac{C(x)^2}{x}\right)\right)(x),
\end{align*}

by the linearity of $\mathcal{A}_{\frac{3}{2}}^2$ and by Proposition \ref{plusm2}. Rearrange and use  Proposition \ref{plusm1} to get

\begin{align*}
   (2C-x)\big(\mathcal{A}_{\frac{1}{2}}^2C\big)(x)= \mathcal{A}_{\frac{3}{2}}^2\left(C_{\geq2}\left(\displaystyle\frac{C(x)^2}{x}\right)\right)(x).
\end{align*}

Now to get rid of the decomposition on the right we appeal to Theorem \ref{chaintheorem}:

\begin{align*}
   (2C-x)\big(\mathcal{A}_{\frac{1}{2}}^2C\big)(x) 
   &= C'_{\geq2}\bigg(\displaystyle\frac{C^2}{x}\bigg)  \mathcal{A}_{\frac{3}{2}}^2\bigg(\displaystyle\frac{C^2}{x}\bigg)
   +\bigg(\displaystyle\frac{x^2}{C^2}\bigg)^\frac{3}{2} e^{\frac{C^2/x-x}{2x C^2/x}} \left(\mathcal{A}_{\frac{3}{2}}^2C_{\geq2}\right)\left(\displaystyle\frac{C^2}{x}\right)\\
   &\\
   &=2C\big(\mathcal{A}_{\frac{1}{2}}^2C\big)(x)\; C'_{\geq2}\bigg(\displaystyle\frac{C^2}{x}\bigg)+\displaystyle\frac{x^3}{C^3}\; e^{\frac{C^2/x-x}{2x C^2/x}} \left(\mathcal{A}_{\frac{3}{2}}^2C_{\geq2}\right)\left(\displaystyle\frac{C^2}{x}\right).
\end{align*}

Equation (\ref{c2con}) and Lemma \ref{cd} give that 

\begin{align*}
C'&=\displaystyle\frac{2xCC'-C^2}{x^2}\left[1- C'_{\geq2}\bigg(\displaystyle\frac{C^2}{x}\bigg)\right]\\
&\\
&=\displaystyle\frac{C^2+C-x-C^2}{x^2}\left[1- C'_{\geq2}\bigg(\displaystyle\frac{C^2}{x}\bigg)\right]\\
&\\
&=\displaystyle\frac{C-x}{x^2}\left[1- C'_{\geq2}\bigg(\displaystyle\frac{C^2}{x}\bigg)\right].
\end{align*}

Substituting into our equation we get 

\begin{align*}
   \left(\displaystyle\frac{2x^2CC'}{C-x}-x\right)\big(\mathcal{A}_{\frac{1}{2}}^2C\big)(x)
   &=\displaystyle\frac{x^3}{C^3}\; e^{\frac{C^2/x-x}{2x C^2/x}} \left(\mathcal{A}_{\frac{3}{2}}^2C_{\geq2}\right)\left(\displaystyle\frac{C^2}{x}\right).
\end{align*}
   
Now recall that $\big(\mathcal{A}_{\frac{1}{2}}^2C\big)(x)=\displaystyle\frac{1}{\sqrt{2\pi}}\frac{x}{C(x)}\;e^{\frac{-1}{2x}(C^2+2C)}$ by equation (\ref{A}), and 
hence 
\begin{align*}
    \left(\mathcal{A}_{\frac{3}{2}}^2C_{\geq2}\right)\left(\displaystyle\frac{C^2}{x}\right)
   &=\displaystyle\frac{C^3}{x^3}\left(\displaystyle\frac{2x^2CC'}{C-x}-x\right)\big(\mathcal{A}_{\frac{1}{2}}^2C\big)(x)\; e^{\frac{x-C^2/x}{2x C^2/x}}\\
   &\\
   &=\displaystyle\frac{C^3}{x^3}\cdot\displaystyle\frac{x(C^2+C-x)-xC+x^2}{C-x}\cdot\big(\mathcal{A}_{\frac{1}{2}}^2C\big)(x)\;\cdot e^{\frac{x-C^2/x}{2x C^2/x}}\\
   &\\
   &=\displaystyle\frac{C^3}{x^3}\cdot\displaystyle\frac{xC^2}{C-x}\cdot\displaystyle\frac{1}{\sqrt{2\pi}}\frac{x}{C(x)}\;e^{\frac{-1}{2x}(C^2+2C)}\;\cdot e^{\frac{x-C^2/x}{2x C^2/x}}.
\end{align*}

Since the $LHS$ is a function in $\displaystyle\frac{C^2}{x}$, applying Proposition \ref{plusm1} gives that 

\[\left(\mathcal{A}_{\frac{3}{2}}^2C_{\geq2}\right)\left(\displaystyle\frac{C^2}{x}\right)=\displaystyle\frac{C^2}{x}\cdot \left(\mathcal{A}_{\frac{1}{2}}^2C_{\geq2}\right)\left(\displaystyle\frac{C^2}{x}\right) .\]

Back to our equation, we thus have 
\begin{equation}\label{pre}
    \left(\mathcal{A}_{\frac{1}{2}}^2C_{\geq2}\right)\left(\displaystyle\frac{C(x)^2}{x}\right)
   =\displaystyle\frac{1}{\sqrt{2\pi}}\cdot\displaystyle\frac{C^2}{C-x}\cdot e^{\frac{-1}{2x}(C^2+2C)}\;\cdot e^{\frac{x-C^2/x}{2x C^2/x}}
   =\displaystyle\frac{1}{\sqrt{2\pi}}\cdot\displaystyle\frac{C^2}{C-x}\cdot e^{\frac{-1}{2x}[C^2+2C+1-\frac{x^2}{ C^2}]}.
\end{equation}

Since the power series $\displaystyle\frac{C(x)^2}{x}$ is invertible, we let $y(x)$ be such that $\displaystyle\frac{C(y)^2}{y}=x$. In that case equation (\ref{c2con}) gives
\[C(y)=x-C_{\geq2}(x),\]

and hence \[y(x)=\displaystyle\frac{(x-C_{\geq2}(x))^2}{x}.\]
Substituting $y(x)$ for $x$ in equation (\ref{pre}) we get

\begin{align*}\label{done}
    \left(\mathcal{A}_{\frac{1}{2}}^2C_{\geq2}\right)(x)
   &=\displaystyle\frac{1}{\sqrt{2\pi}}\cdot\displaystyle\frac{(x-C_{\geq2})^2}
   {\left((x-C_{\geq2})-\displaystyle\frac{(x-C_{\geq2})^2}{x}\right)}\cdot  e^{\frac{-1}{2}\left[x+\frac{2x}{(x-C_{\geq2})}+\frac{x}{(x-C_{\geq2})^2}-\frac{1}{x}\right]}\\
   &\\
   &=\displaystyle\frac{1}{\sqrt{2\pi}}\cdot\displaystyle\frac{x\;(x-C_{\geq2})}
   {\left(x-x+C_{\geq2}\right)}\cdot  e^{\frac{-1}{2x}\left[x^2+\frac{2x}{(1-C_{\geq2}/x)}+\frac{1}{(1-C_{\geq2}/x)^2}-1\right]}\\
   &\\
   &=\displaystyle\frac{1}{\sqrt{2\pi}}\cdot\displaystyle\frac{x^2}
   {\left(\displaystyle\frac{C_{\geq2}}{(1-C_{\geq2}/x)}\right)}\cdot  e^{\frac{-1}{2x}\left[\left(\frac{1}{(1-C_{\geq2}/x)}+x\right)^2-1\right]}.
\end{align*}

In other words,

\begin{equation}\label{alien2con}
    \left(\mathcal{A}_{\frac{1}{2}}^2C_{\geq2}\right)(x)=
    \displaystyle\frac{1}{\sqrt{2\pi}}\cdot\displaystyle\frac{x^2}
   {C_{\geq2}S}\cdot  e^{\frac{-1}{2x}\left[\left(S+x\right)^2-1\right]},
\end{equation}

where $S(x)=\displaystyle\frac{1}{\Big(1-\displaystyle\frac{C_{\geq2}(x)}{x}\Big)}$ is the generating series for sequences of $2$-connected chord diagrams counted by one less chord.

Finally it is noteworthy to see that the image $\mathcal{A}_{\frac{1}{2}}^2C_{\geq2}$ of $C_{\geq2}(x)$ under the alien derivative is of the form of a rational function of $C_{\geq2}$ times the exponential of a quadratic expression in the rational function. The same pattern has been observed in the case of connected chord diagrams. From another point of view, one can see that $xS(x)$ also counts connectivity-1 diagrams in which only the root chord is a cut. 

The evaluation of $\mathcal{A}_{\frac{1}{2}}^2C_{\geq2}$ will enable us to derive information about the asymptotic behaviour which strongly extend the result by Kleitman in \cite{kleit}.  First let us list the first few terms of the functions involved.

\begin{table}[h]
\center
\begin{tabular}{|c|c||c|c|c|c|c|c|c|c|}\hline
 &            &$x^0$ & $x$   & $x^2$  & $x^3$ & $x^4$ & $x^5$ & $x^6$ & $x^7$                           \\\hline\hline
1&$S(x)=1/{\Big(1-\displaystyle\frac{C_{\geq2}(x)}{x}\Big)}$     
              &   1  & 1     & 2      & 10    & 82    & 898   & 12018 & 
              \\\hline
2&$(S+x)^2$      
              &   1  & 4     & 8      & 28    & 208   & 2164  & 28056 & 
              \\\hline
3&$\displaystyle\frac{1}{2x}\left[(S+x)^2-1\right]$      
              &   2  & 4     & 14     & 104   & 1082  & 14028 &       & 
              \\\hline
4&${C_{\geq2}\cdot S}$ 
              &   0  & 0     & 1      & 2     & 10    & 82    & 898   & 12018
              \\\hline
5&$\displaystyle\frac{x^2}{C_{\geq2}\cdot S}$
              &   1  & -2    &-6      & -50   & -574  & -8082 &       &  
              \\\hline
6&$e^2\;\cdot\;\exp\big\{\displaystyle\frac{-1}{2x}\left[(S+x)^2-1\right]\big\}$      
              &   1  & -4    & -6     & $\frac{-176}{3}$   & $\frac{-2008}{3}$  & $\frac{-46636}{5}$  &       &   
              \\\hline
 \end{tabular}\caption{The first terms of the series involved in calculating $\mathcal{A}^2_{\frac{1}{2}}C_{\geq2}(x)$.  }
\label{table3}
\end{table}

Note  that we are willing to display the factor of $e^{-2}$ that comes from $\;\exp\big\{\displaystyle\frac{-1}{2x}\left[(S+x)^2-1\right]\big\}$ and that is why the last row in  Table \ref{table3} is multiplied by $e^2$.

The computation then gives
\begin{align}\label{asymptoticsC2}
    \left(\mathcal{A}_{\frac{1}{2}}^2C_{\geq2}\right)(x)
    &=\displaystyle\frac{1}{\sqrt{2\pi}}\cdot\displaystyle\frac{x^2}
   {C_{\geq2}S}\cdot  e^{\frac{-1}{2x}\left[\left(S+x\right)^2-1\right]} \nonumber \\
    &\nonumber\\
    &=\displaystyle\frac{e^{-2}}{\sqrt{2\pi}}\big[1-6x-4x^2-\displaystyle\frac{218}{3}x^3-
    890x^4-\displaystyle\frac{196838}{15}x^5-\cdots\big]
\end{align}

Now, by Definition \ref{fdps} of factorially divergent power series and Definition \ref{map} of the alien derivative $\mathcal{A}_{\frac{1}{2}}^{2}$, and since $\Gamma^2_{\frac{1}{2}}(n)=\sqrt{2\pi} (2n-1)!!$, we obtain that, for all $R\in\mathbb{N}_0$, the number $(C_{\geq2})_n$
of $2$-connected diagrams on $n$ chords satisfies
\begin{align}
    (C_{\geq2})_n
    &=\overset{R-1}{\underset{k=0}{\sum}} [x^k] \left(\mathcal{A}_{\frac{1}{2}}^2C_{\geq2}\right)(x) \cdot \Gamma^2_{\frac{1}{2}}(n-k)+\mathcal{O}(\Gamma^2_{\frac{1}{2}}(n-R))\nonumber\\
    &=\sqrt{2\pi}\;\overset{R-1}{\underset{k=0}{\sum}} [x^k] \left(\mathcal{A}_{\frac{1}{2}}^2C_{\geq2}\right)(x) \cdot (2(n-k)-1)!! +\mathcal{O}((2(n-R)-1)!!),\nonumber
\end{align}

and hence the first few terms in this asymptotic expansion are given by 

\begin{align}
    (C_{\geq2})_n 
    & = e^{-2}  \bigg((2n-1)!!-6(2n-3)!!-4(2n-5)!!-
    \displaystyle\frac{218}{3}(2n-7)!!- \nonumber\\
    &\; \qquad       -890(2n-9)!!-\displaystyle\frac{196838}{15}(2n-11)!!-\cdots\bigg)
    \nonumber\\
    &\;        \nonumber\\
    & = e^{-2} (2n-1)!!\bigg(1-\displaystyle\frac{6}{2n-1}-\displaystyle\frac{4}{(2n-3)(2n-1)}-\displaystyle\frac{218}{3}\displaystyle\frac{1}{(2n-5)(2n-3)(2n-1)}- \nonumber\\
    &\;   \qquad     -\displaystyle\frac{890}{(2n-7)(2n-5)(2n-3)(2n-1)}-\displaystyle\frac{196838}{15}\displaystyle\frac{1}{(2n-9)\cdots(2n-1)}-\cdots\bigg).\nonumber\\
    &\label{computC2asympt}
\end{align}

The result by Kleitman \cite{kleit} corresponds to the first term in this expansion. By the above approach, any precision can be achieved and an arbitrary number of terms can be produced. 

Equation \ref{computC2asympt} also shows that a randomly chosen chord diagram on $n$ chords is $2$-connected with a probability of \[\displaystyle\frac{1}{e^2}\Big(1-\displaystyle\frac{3}{n}\Big)+\mathcal{O}(1/{n^2}).\]

In the next section we will see that this expansion also corresponds to the asymptotics of the number of skeleton quenched QED vertex diagrams \cite{michiq}. In that context the first five terms of the above expansion were conjectured  by D. J. Broadhurst on a numerical evidence (see page 38 in \cite{michiq}) in studying zero-dimensional field theory.


\section{Connection with Zero-Dimensional QFT}\label{quenchedsec}

In the next two sections we will see that some of the integer sequences produced in studying $2$-connected chord diagrams appear in the context of zero-dimensional quantum field theory. Note that in this situation the path integral transforms into a series of graphs since no Feynman integral shall remain. On the level of Feynman rules, they will be represented as a character from $\mathcal{H}$ to the algebra $\mathbb{R}[[\hbar]]$ \cite{michi}. We managed to establish the relation between $2$-connected chord diagrams and some of the observables in \textit{quenched QED} and in Yukawa theory. In \cite{michiq} the asymptotics for these sequences are obtained through a \textit{singularity analysis} approach. We will be able to get the same asymptotics through an enumerative approach. Factorially divergent power series are, as expected, used in both approaches, and hence we will regularly appeal to theorems from Section \ref{factorially}. First we will briefly set-up the context in perturbation theory. In most parts we follow the notation in \cite{michiq}.

Recall the basic path integral formulation of QFT and notice that for   zero-dimensional QFT the path integral for the partition function becomes an ordinary integral given for example by 

\[Z(\hbar):=\int_\mathbb{R} \displaystyle\frac{1}{\sqrt{2\pi\hbar}} e^{\frac{1}{\hbar}\left(-\frac{x^2}{2a}+V(x)\right)}dx,\]
where $V(x)\in x^3\mathbb{R}[[x]]$ is the \textit{potential} and the exponent $-\frac{x^2}{2a}+V(x)$ is the \textit{action} and denoted by $\mathcal{S}$. As known, this integral generally has singularities, and even as a series expansion it generally have a singularity at zero. In \cite{michiq}, the expansion is treated as a formal power series and the focus is on studying the asymptotics of the coefficients.

Recall that Gaussian integrals satisfy \[\int_\mathbb{R} \displaystyle\frac{1}{\sqrt{2\pi\hbar}} e^{-\frac{x^2}{2a\hbar}} x^{2n}dx=\sqrt{a}(a\hbar)^n (2n-1)!!, \] were only the even powers are considered since the integral vanishes for odd powers. This enables us to work with a well-defined power series instead of the path integral (actually this is the path integral in dimension 0):

\begin{dfn}
For a general formal action $\mathcal{S}{(x)}=-\frac{x^2}{2a}+V(x)\in x^2\mathbb{R}[[x]]$ we define the corresponding perturbative partition function to be the power series in $\hbar$ given by

\[\mathcal{F}[\mathcal{S}(x)](\hbar)=\sqrt{a}\sum_{n=0}^\infty (a\hbar)^n(2n-1)!![x^{2n}] e^{\frac{1}{\hbar}V(x)}.\]
\end{dfn}

This is a well-defined power series in $\hbar$ since the coefficient $[x^{2n}] e^{\frac{1}{\hbar}V(x)}$ is a polynomial in $\hbar^{-1}$ of degree less than $n$ because $V\in x^3\mathbb{R}[[x]]$. Just as the path integral, this map also has a diagrammatic meaning in terms of Feynman diagrams \cite{cvi}:

\begin{prop}
If $\mathcal{S}(x)=-\frac{x^2}{2a}+\sum_{k=3}^\infty \frac{\lambda_k}{k!}x^k$ with $a>0$, then 
\[\mathcal{F}[\mathcal{S}(x)](\hbar)=\sqrt{a} \sum_\Gamma \hbar^{|E(\Gamma)|-|V(\Gamma)|}\;\; \displaystyle\frac{a^{|E(\Gamma)|}\Pi_{v\in V(\Gamma)} \lambda_{n_v}}{|\mathrm{Aut}\;\Gamma|},\]

where the sum runs over all multigraphs $\Gamma$ in which the valency $n_v$ of every vertex $v$ is at least $3$, and where $|E(\Gamma)|$, $|V(\Gamma)|$, and $|\mathrm{Aut}\;\Gamma|$ are the sizes of the edge set, the vertex set, and the automorphism group of $\Gamma$ respectively. 
\end{prop}

So, in terms of Feynman diagrams, to compute the $n^\text{th}$ coefficient of $\mathcal{F}[\mathcal{S}(x)]$ we do the following 

\begin{enumerate}
    \item Draw all multigraphs $\Gamma$ with $|E(\Gamma)|-|V(\Gamma)|=n$. Note that this is one less than the loop number, it is the number of independent cycles in the graph (remember that independent cycles can be obtained by starting with a spanning tree and adding one edge at a time). The loop number is also known as the \textit{Betti number} of the graph. The number $|E(\Gamma)|-|V(\Gamma)|$ will be referred to as the \textit{excess} of $\Gamma$.
    
    \item Each vertex contributes with a factor that corresponds to its valency, that's how we get $\Pi_{v\in V(\Gamma)} \lambda_{n_v}$. Then we multiply with the factor $a^{|E(\Gamma)|}$. This process simply corresponds to the Feynman rules. The map that applies the Feynman rules will be denoted by $\phi_\mathcal{S}$.
    
    \item Divide by the size of the automorphism group of the graph. 
    \item Finally sum up all the contributions and multiply by $\sqrt{a}$.
\end{enumerate}

\begin{exm}
As an example, the action for $\varphi^3$-theory takes the form $\mathcal{S}(x)=-\frac{x^2}{2}+\frac{x^3}{3!}$. In that case 

\[Z^{\varphi^3}(\hbar)=\sum_{n=0}^\infty \hbar^n(2n-1)!! [x^{2n}]e^{\frac{x^3}{3!\hbar}}=\sum_{n=0}^\infty\displaystyle\frac{(6n-1)!!}{(3!)^{2n}(2n)!} \hbar^n.\]

In terms of Feynman diagrams only $3$-regular graphs will show up in $\varphi^3$-theory, hence we have 
\begin{align*}
Z^{\varphi^3}(\hbar)=\phi_{\mathcal{S}}&\Big(1+
\displaystyle\frac{1}{8}\;\raisebox{-0.2cm}{\includegraphics[scale=0.3]{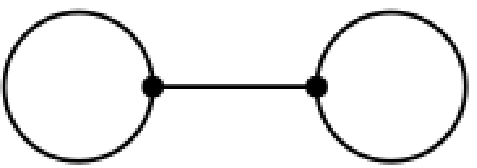}}+
\displaystyle\frac{1}{12}\;\raisebox{-0.2cm}{\includegraphics[scale=0.3]{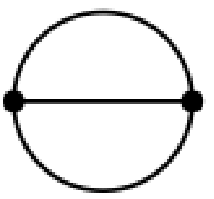}}+
\displaystyle\frac{1}{128}\;\raisebox{-0.2cm}{\includegraphics[scale=0.3]{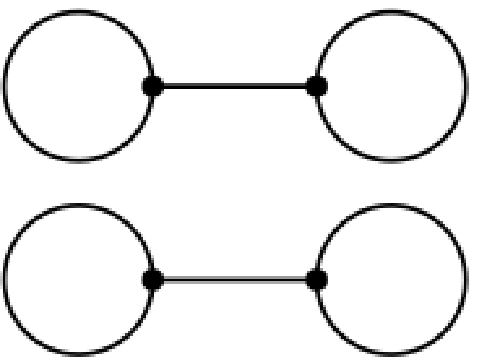}}+
\displaystyle\frac{1}{288}\;\raisebox{-0.2cm}{\includegraphics[scale=0.3]{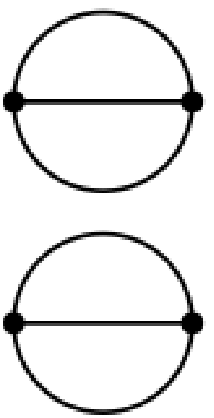}}+
\displaystyle\frac{1}{96}\;\raisebox{-0.2cm}{\includegraphics[scale=0.3]{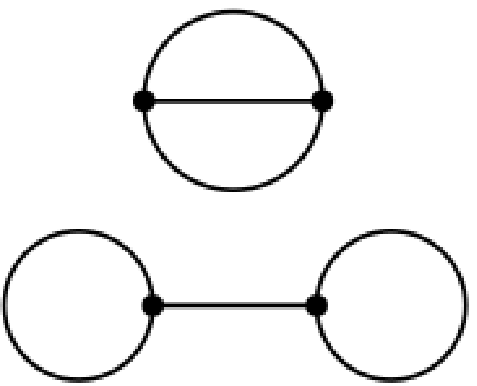}}+\\
&+\displaystyle\frac{1}{48}\;\raisebox{-0.2cm}{\includegraphics[scale=0.32]{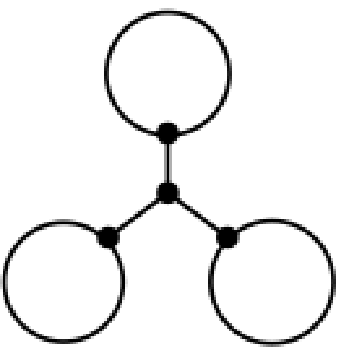}}+
\displaystyle\frac{1}{16}\;\raisebox{-0.2cm}{\includegraphics[scale=0.3]{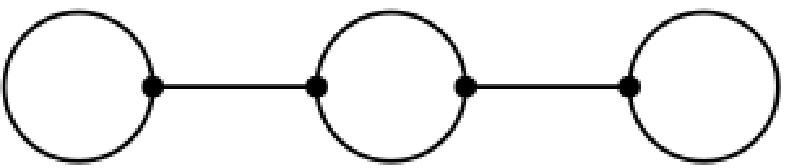}}+
\displaystyle\frac{1}{16}\;\raisebox{-0.2cm}{\includegraphics[scale=0.32]{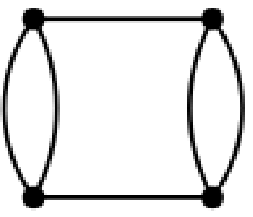}}+
\displaystyle\frac{1}{8}\;\raisebox{-0.2cm}{\includegraphics[scale=0.32]{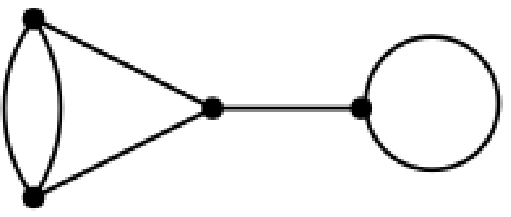}}+\\
&\\
&\displaystyle\frac{1}{24}\;\raisebox{-0.2cm}{\includegraphics[scale=0.32]{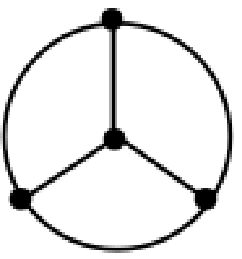}}+
\cdots\Big). \end{align*}

Then applying the Feynman rules  $\phi_\mathcal{S}$ does not change the coefficients in the sum since the contribution of any vertex is $1$ according to the described potential. Adding up the terms with the same loop number then gives 

\[Z^{\varphi^3}(\hbar)=1+\frac{5}{24}\hbar+\frac{385}{1152}\hbar^2+\cdots, \] which agrees with the first algebraic calculation.
\end{exm}

\subsection{Zero-Dimensional Scalar Theories with Interaction}

We will study such expansions that arise in QED theories, namely we shall consider quenched QED and Yukawa theory. These are examples of theories with interaction. It is impossible to completely cover the underlying physics, nevertheless we should be able to understand as much as needed for our purposes by anticipating the interrelations between the different entities defined. 

In the presence of interaction in the theory, the partition function takes the form 

\[Z(\hbar,j):=\int_\mathbb{R} \displaystyle\frac{1}{\sqrt{2\pi\hbar}} e^{\frac{1}{\hbar}\left(-\frac{x^2}{2}+V(x)+xj\right)}dx,\] 
where an additional term is added to the potential, namely $xj$, $j$ is called the \textit{source}.

With this extra term we can not directly expand the integral as we did before, but we can still achieve the same essence after a change of variables. Shift $x$ to $x+x_o$ where $x_o(j)$ is the unique power series solution to $x_o(j)=V'(x_o(j))+j$.  Then we get 

\begin{align*}
    Z(\hbar,j)&=\int_\mathbb{R} \displaystyle\frac{1}{\sqrt{2\pi\hbar}} e^{\frac{1}{\hbar}\left(-\frac{(x+x_o)^2}{2}+V(x+x_o)+(x+x_o)j\right)}dx\\
              &\\
              &=e^{\frac{1}{\hbar}\left(-\frac{x_o^2}{2}+V(x_o)+x_oj\right)}
              \int_\mathbb{R} \displaystyle\frac{1}{\sqrt{2\pi\hbar}} e^{\frac{1}{\hbar}\left(-\frac{x^2}{2}+V(x+x_o)-V(x_o)-xV'(x_o)\right)}dx\\
              &\\
              &=e^{\frac{1}{\hbar}\left(-\frac{x_o^2}{2}+V(x_o)+x_oj\right)}
              \mathcal{F}\left[-\displaystyle\frac{x^2}{2}+V(x+x_o)-V(x_o)-xV'(x_o)\right](\hbar).
\end{align*}

The exponential factor enumerates forests (collections of trees) with the corresponding conditions on vertices, these diagrams are referred to as the \textit{tree-level} diagrams. Tree-level diagrams contribute with negative powers of $\hbar$, and therefore we are going to isolate them so that the treatment for the main expansion remains clear. Remember that  Feynman diagrams are labeled, and so in order to restrict ourselves to connected diagrams we have to take the logarithm of the partition function:

\begin{align*}
    W(\hbar,j):&=\hbar \log 
    Z(\hbar,j)\\
    &=-\frac{x_o^2}{2}+V(x_o)+x_oj+\hbar \log\mathcal{F}\left[-\displaystyle\frac{x^2}{2}+
    V(x+x_o)-V(x_o)-xV'(x_o)\right](\hbar)
\end{align*}

generates all connected diagrams and is called the \textit{free energy}. Note that the extra $\hbar$ factor causes the powers to express the number of loops instead of the excess. Again we are using the notation in \cite{michiq} since we are eventually going to compare to parts of the work.

As customary in QFT, to move to the \textit{quantum effective action} $G$, which generates \textit{1PI} diagrams, one takes the Legendre transform of $W$:

\begin{align} \label{propergreenfnmichi}
    G(\hbar,\varphi_c):=W-j \varphi_c, 
\end{align}
where $\varphi_c:=\partial_jW$. The coefficients $[\varphi_c^n]G$ are called the \textit{(proper) Green functions} of the theory. 
Recall that from a graph theoretic point-of-view, being \textit{1PI} (1-particle irreducible)  is merely another way of saying $2$-connected. Thus, combinatorially, the Legendre transform, as in \cite{kjm2}, is seen to be the transportation from connected diagrams to $2$-connected or \textit{1PI} diagrams. In that sense, the order of the derivative  $\partial^n_{\varphi_c}G|_{\varphi_c=0}=[\varphi_c^n]G$ determines the number of external legs.

In the next part of the discussion we shall need the following physical jargon and terminology:

\begin{enumerate}
    \item The Green function $[\varphi_c^1]G=\partial_{\varphi_c}G|_{\varphi_c=0}$ generates all \textit{1PI} diagrams with exactly one external leg, which are called the \textit{tadpoles} of the theory (Figure \ref{tp}).
    
    \begin{figure}[h]
    \centering
    \includegraphics[scale=0.56]{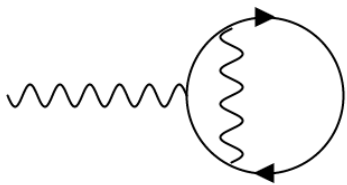}
    \caption{A tadpole diagram in QED}
    \label{tp}\end{figure}
    
    \item The Green function $[\varphi^2_c]G=\partial^2_{\varphi_c}G|_{\varphi_c=0}$ generates all \textit{1PI} diagrams with two external legs. Such a diagram is called a \textit{1PI propagator} (can replace an edge in the theory). 

    \begin{figure}[h]
    \centering
    \includegraphics[scale=0.56]{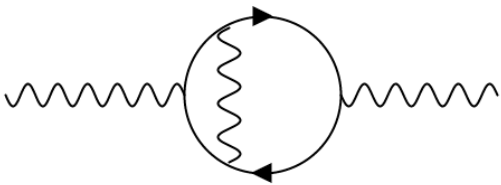}
    \caption{A propagator diagram}
    \label{propagator}
    \end{figure}
 
    \item For $n>2$, $\partial^n_{\varphi_c}G|_{\varphi_c=0}=[\varphi_c^n]G$ is called the \textit{$n$-point function}.
\end{enumerate}

In quenched QED, some of the quantities that we are going to compare their expansions with the generating series of $2$-connected chord diagrams are the \textit{renormalized} Green functions with respect to a chosen \textit{residue}. We shall therefore recall from Section \ref{hopfalg1PIsection} the basics of the Hopf-algebraic treatment of renormalization in the next section before proceeding into the real calculations. For more about this topic the reader can consult \cite{manchonhopf}, or the original paper by D. Kreimer and A. Connes \cite{conneskreimer}.


\subsection{Hopf-algebraic Renormalization Revisited}

 Recall from Section \ref{hopfalg1PIsection} that, for a given QFT, the superficially divergent 1PI Feynman graphs form a Hopf algebra $\mathcal{H}$. The product was defined to be the disjoint union, and the coproduct of a connected Feynman graph $\Gamma$ was defined according to Definition \ref{coprod}:
 
\[\Delta(\Gamma)=\underset{1PI \;\text{subgraphs}}{\underset{\gamma\;\text{product of divergent }}{\underset{\gamma\subseteq\Gamma}{\sum}}}\gamma\otimes\Gamma/\gamma,\]

and extended as an algebra morphism. The unit, counit, and antipode were denoted by $\mathbb{I}$, $\hat{\mathbb{I}}$, and $S$.

In Section \ref{DSE} we saw how to write the Dyson-Schwinger equations in terms of the elements $X^r$, where $X^r$ was defined as 

\[X^r=1\pm \underset{\text{with residue}\;r }{\underset{\text{1PI graphs}\;\Gamma}{\sum}}\displaystyle\frac{1}{\text{Sym}\; \Gamma}\;\Gamma,\]

where the negative sign is assumed only when $r$ is edge-type. Recall also that if we use insertions in case of a theory with a single vertex type we have equation (\ref{generaldys})
\[ X^r=1\pm \sum_k B_+^{\gamma_{r,k}}(X^rQ^k),\]

where the sum is over all primitive 1PI diagrams with loop number $k$ and residue $r$, and where $Q$ is the invariant charge as defined in Section \ref{invariantch}.

The identity 

\begin{equation}\label{most}
    \Delta X^r=\sum_{L=0}^\infty\left. X^r Q^L\otimes X^r\right|_L,
\end{equation}

is of most importance in the context of renormalization \cite{kreimerB}. The
 $|_L$, as used in \cite{michiq}, is the restriction of the sum to graphs with loop number $L$.

We have seen in Section \ref{hopfalg1PIsection} that the Feynman rules are simply characters from $\mathcal{H}$ to a commutative algebra $A$. For zero-dimensional field theories the Feynman rules will be $\phi:\mathcal{H}\longrightarrow\mathbb{R}[[\hbar]]$:
\begin{equation}\label{FFFF}
    \phi\{\Gamma\}(\hbar)=\hbar^{\ell(\Gamma)},
\end{equation}
where we follow the notation in \cite{michiq} for putting the arguments from $\mathcal{H}$ in curly brackets. 

In that case, the Green functions, or the generating function of $1PI$ Feynman graphs with residue $r$ are defined as 

\begin{equation}\label{g^r}
    g^r(\hbar):=\phi\{X^r\}(\hbar)=1\pm \underset{\text{with residue}\;r }{\underset{\text{1PI graphs}\;\Gamma}{\sum}}\displaystyle\frac{\hbar^{\ell(\Gamma)}}{\text{Sym}\; \Gamma}, 
\end{equation}
where $\ell(\Gamma)$ is the loop number as before. If residue $r$ is the $k$ external legs residue, then $g^r=\partial^k_{\varphi_c}G|_{\varphi_c=0}$ , the $k$th derivative of the quantum effective action. 

In our case of zero-dimensional QFT, the fact that the target algebra for the Feynman rules is $\mathbb{R}[[\hbar]]$  limits the choice for a Rota-Baxter operator\footnote{Remember that a Rota-Baxter operator $R$  is used in the renormalization scheme to extract (in terms of an induced Birkhoff decomposition) the divergent part of the integral. See Section \ref{hopfalg1PIsection}.}
$R:\mathbb{R}[[\hbar]]\longrightarrow\mathbb{R}[[\hbar]]$ that respects the grading of $\mathcal{H}$. The only choice for a meaningful renormalization scheme in this case is $R=\text{id}$ (see \cite{michiq}).

Thus, by our definitions in Section \ref{hopfalg1PIsection} (equation \ref{S_R}), the counterterm map for the renormalization scheme $R=\text{id}$ is given by

\[S^\phi=R\circ\phi\circ S=\phi\circ S.\] Then the renormalized Feynman rules is 

\begin{align}\label{tobeusednow}
\phi_{\text{ren}}:=S_R^\phi\ast\phi=S^\phi\ast\phi=(\phi\circ S)\ast\phi,   
\end{align} 
where $\ast$ is the convolution product (Definition \ref{convo}).
However, the action of the last expression on an arbitrary element of $\mathcal{H}$ is:

\begin{align*}\label{tobeusednow}
((\phi\circ S)\ast\phi)(\Gamma)
&=m\circ((\phi\circ S)\otimes\phi)\circ\Delta(\Gamma)&\\
&=\underset{1PI \;\text{subgraphs}}{\underset{\gamma\;\text{product of divergent }}{\underset{\gamma\subseteq\Gamma}{\sum}}} \phi(S(\gamma))\phi(\Gamma/\gamma)&\\
&&\\
&=\phi\Big(\sum S(\gamma)(\Gamma/\gamma)\Big)&
 \text{(since $\phi$ is a character)}\\
&=\phi(\mathbb{I}(\hat{\mathbb{I}}(\Gamma))& \text{(by definition of the antipode $S$),}
\end{align*}
which is zero for all nonempty elements in $\mathcal{H}$ since $\mathbb{I}\circ\hat{\mathbb{I}}$ maps all elements in $\mathcal{H}$ to zero except for the empty graph, which is mapped to itself.

Thus, equation (\ref{tobeusednow}) becomes
\begin{equation}\label{renphi}
    \phi_{\text{ren}}=\mathbb{I}\circ\hat{\mathbb{I}},
\end{equation}
and whence the renormalized Green functions are
\begin{equation}\label{reng}
    g^r_{\text{ren}}=\phi_{\text{ren}}\{X^r\}(\hbar)=1\pm0=1.
\end{equation}

Note that in $\cite{michiq}$ the signs are different since, as mentioned earlier, $X^r$ in their convention is $-1$ times ours.
 Finally, what we will care for the most are the counterterms:
\begin{equation}\label{countertermsz_r}
    z_r:=S^\phi\{X^r\}.
\end{equation}
Notice that since $\mathcal{H}$ is commutative, the definition of the convolution product together with equation (\ref{most}) now yield
\begin{align}
    1
    &=\phi_{\text{ren}}\{X^r\}(\hbar)\\\nonumber
    &=(m\circ(S^\phi\otimes\phi)\circ \Delta X^r)(\hbar)\\\nonumber
    &=(m\circ(\phi\otimes S^\phi)\circ \Delta X^r)(\hbar)\\\nonumber
    &=(m\circ(\phi\otimes S^\phi)\circ (\sum_{L=0}^\infty\left. X^r Q^L\otimes X^r\right|_L))(\hbar)\\\nonumber
    &=\sum_{L=0}^\infty \phi\{X^r\}(\hbar)\;\;(\phi\{Q\}(\hbar))^L\;\;[\hbar^L] S^\phi\{X^r\}(\hbar)\\\nonumber
    &= \phi\{X^r\}(\hbar)\;\;\;S^\phi\{X^r\}(\hbar \phi\{Q\}(\hbar))\\\nonumber
    &=g^r(\hbar) \;z_r(\hbar \alpha(\hbar)).
\end{align}

Let $\hbar(y)$ be the unique power series solution of $y=\hbar(y)\alpha(\hbar(y))$. In $\cite{michiq}$, $y$ is called the \textit{renormalized expansion parameter} and is denoted by $\hbar_{\text{ren}}$. Then, substituting in equation  we get

\begin{equation}
    z_r(\hbar_{\text{ren}})=\displaystyle\frac{1}{g^r(\hbar(\hbar_{\text{ren}}))}.\end{equation}

The following result was proven by M. Borinsky in \cite{michilattice}, and we shall depend on it in the combinatorial treatment in the next section.

\begin{thm}[\cite{michilattice}]\label{theorem important zr}
In a theory with a cubic vertex-type, the numeric coefficients in $z_r(\hbar_{\text{ren}})$ count the number of primitive diagrams if $r$ is vertex-type.
\end{thm}


\subsection{QED Theories, Quenched QED, and Yukawa Theory}

The two theories that we are concerned with here are quenched QED and Yukawa theory, which are examples of QED-type theories. In these theories we have two particles: fermion and boson (wiggly and dashed edges) particles, and we have only three-valent vertices of the type fermion-fermion-boson. We will compute the asymptotics of $\;z_{\phi_c|\psi_c|^2}(\hbar_{\text{ren}})\;$ in quenched QED, as well as the asymptotics of the green functions $\left.\partial^i_{\phi_c}(\partial_{\psi_c}\partial_{\bar{\psi}_c})^j\;G^{\text{Yuk}}\right|_{\phi_c=\psi_c=0}$. Our approach is completely combinatorial and depends on establishing bijections between the diagrams in the combinatorial interpretation of the considered series and different classes of chord diagrams. Unlike the approach applied in \cite{michiq}, we do not need to refer to singularity analysis nor the representation of $\mathcal{S}(x)$ by affine hyperelliptic curves.

\subsection{The Partition Function}\label{partitionfunctionsection}
The partition function takes the form 
\begin{equation}
 Z(\hbar,j,\eta)=\int_\mathbb{R} \displaystyle\frac{1}{\sqrt{2\pi\hbar}}\; e^{\frac{1}{\hbar}\left(-\frac{x^2}{2}+jx+\frac{|\eta|^2}{1-x}+\hbar \log \frac{1}{1-x}\right)}dx .\label{Yu}\end{equation}

We are not going to discuss the physical reasoning behind the above expression, the reader may refer to QFT books or surveys for more details, e.g. see \cite{michi}. We only hint that, combinatorially, $\hbar \log \frac{1}{1-x}$ generates fermion loops, while $\frac{|\eta|^2}{1-x}$ generates a fermion propagator. The special examples of Yukawa theory and quenched QED will be as follows:

\begin{enumerate}
    \item Quenched QED is an approximation of QED where fermion loops are not present. So that the term $\hbar \log \frac{1}{1-x}$ does not appear in the partition function. Thus, the partition function for quenched QED is given by 

\[Z^{QQED}(\hbar,j,\eta)=\int_\mathbb{R} \displaystyle\frac{1}{\sqrt{2\pi\hbar}}\; e^{\frac{1}{\hbar}\left(-\frac{x^2}{2}+jx+\frac{|\eta|^2}{1-x}\right)}dx.\]
    
    \item For zero-dimensional Yukawa theory the partition function is just the integral in equation (\ref{Yu}). That is, the partition function for zero-dimensional Yukawa theory is given by 

\[Z^{Yuk}(\hbar,j,\eta)=\int_\mathbb{R} \displaystyle\frac{1}{\sqrt{2\pi\hbar}}\; e^{\frac{1}{\hbar}\left(-\frac{x^2}{2}+jx+\frac{|\eta|^2}{1-x}+\hbar \log\frac{1}{1-x}\right)}dx.\]
\end{enumerate}

\section{Quenched QED}

For this theory we are interested in the asymptotics of the counterterm $\;z_{\phi_c|\psi_c|^2}(\hbar_{\text{ren}})\;$ obtained in \cite{michiq} (page 38). By Theorem \ref{theorem important zr}, since QQED has only one type of vertices which is three-valent, this series enumerates the number of primitive quenched QED diagrams with vertex-type residue (see sequence \href{https://oeis.org/A049464}{A049464} of the OEIS for the first entries).

Thus, this is the same as counting the number of all diagrams $\gamma$ with the following specifications:
\begin{enumerate}
    \item two types of edges, fermion and boson (photon) edges, represented as \raisebox{-0cm}{\includegraphics[scale=0.68]{Figures/fermionedge.eps}} and \raisebox{-0.23cm}{\includegraphics[scale=0.3]{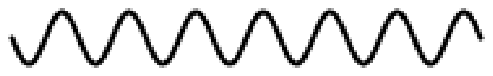}},  respectively;
    \item only three-valent vertices with the structure \raisebox{-0.64cm}{\includegraphics[scale=0.35]{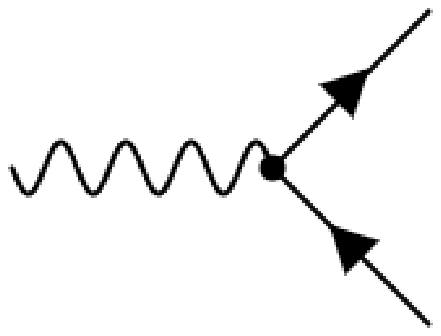}}, with one fermion in, one fermion out, and one photon;
    
    \item no fermion loops;
    \item the residue $\text{res} (\gamma)$ is vertex-type; and
    \item $\gamma$ is $1PI$ primitive, in other words it is edge-connected and contains no subdivergences (Definition \ref{primitivediags} and Definition \ref{subdiv}).
\end{enumerate}
We let $\mathcal{Q}$ be the class of all such diagrams.

\begin{exm}\label{exampleQQED diag}
The following diagram in Figure \ref{fig:my_label} is in $\mathcal{Q}$, whereas the diagrams in Figure \ref{notqqeddiags} are not in $\mathcal{Q}$.
\begin{figure}[h]
    \centering
    \includegraphics[scale=0.43]{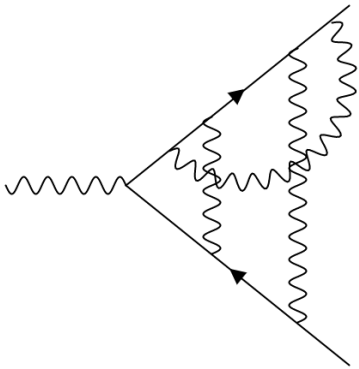}
    \caption{An example of a quenched QED diagram in $\mathcal{Q}$}
    \label{fig:my_label}
\end{figure}

\begin{figure}[h]
    \centering
    \includegraphics[scale=0.45]{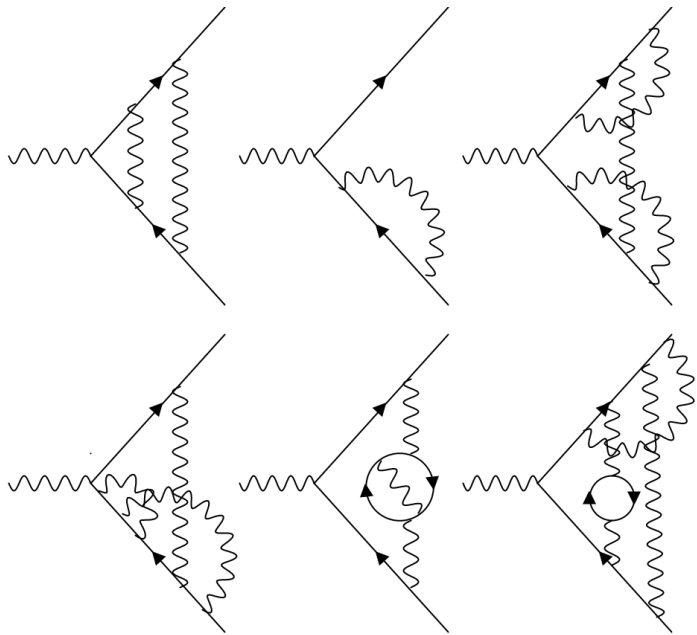}
    \caption{Diagrams not in $\mathcal{Q}$.}
    \label{notqqeddiags}
\end{figure}

In Figure \ref{notqqeddiags}, the second diagram is not 1PI, the rest of the first four diagrams are all 1PI, but they are not  primitive. The last two diagrams are not in $\mathcal{Q}$ and are not even quenched as they contain fermion loops. 

\end{exm}

\begin{thm}\label{myresultinquenched}
The generating series $\;z_{\phi_c|\psi_c|^2}(\hbar_{\text{ren}})\;$ and $\;z_{|\psi_c|^2}(\hbar_{\text{ren}})\;$ count $2$-connected chord diagrams. More precisely, $$[\hbar_{\text{ren}}^{n-1}]\;z_{\phi_c|\psi_c|^2}(\hbar_{\text{ren}})=
[\hbar_{\text{ren}}^{n}]\;z_{|\psi_c|^2}(\hbar_{\text{ren}})=[x^n]\;C_{\geq2}(x).$$
\end{thm}

\begin{proof}
As we mentioned above, the class of diagrams counted by $\;z_{\phi_c|\psi_c|^2}(\hbar_{\text{ren}})\;$ is to be denoted by $\mathcal{Q}$. Thus, $\mathcal{Q}$ consists of 1PI primitive quenched QED diagrams with two external fermion legs and one photon leg. By definition, every graph in $\mathcal{Q}$ has a (wiggly) photon external leg $r$, and two directed fermion external legs $f_1$ and $f_2$. 
We now start by proving the following claim:

\underline{\textbf{Claim 1:}} If $\Gamma$ is a graph in $\mathcal{Q}$, then there exists a unique fermion-only path $P$ from $f_1$ to $f_2$. Moreover, $P$ passes through the vertex at $r$ and every vertex in the graph is on $P$. In addition, the loop number in $\Gamma$ is equal to the number of internal photon edges, and so either is counted by the power of $\hbar_{\text{ren}}$ in the series.

\textbf{Proof:} Generally, if we remove all photon edges from a  graph with the 3-valent vertex residue we should  get a single directed path of fermion edges (because otherwise we will have more than 2 external fermion legs if the photon edges are restored) and a set fermion loops. Now, in our case, we can only get the path, which we denote by $P$ and which should then carry all the vertices in the original graph. In particular, the number of vertices in a graph $\Gamma\in\mathcal{Q}$ will be 1+the number of internal fermion edges. 

To see that the rest of the claim is indeed true first recall Euler's formula \[|V(\Gamma)|-|E(\Gamma)|+\ell(\Gamma)=1,\]
where as usual $|V(\Gamma)|$ is the number of vertices, $|E(\Gamma)|$ is the number of internal edges, and $\ell(\Gamma)$ is the number of loops or independent cycles in $\Gamma$.

Note that the external legs do not alter this relation. We have two types of edges, photons and fermions, so let us assume that $p$ is the number of internal photon edges and $f$ is the number of internal fermion edges, thus $|E(\Gamma)|=p+f$. Now, the most useful observation is that, in our case, we  should have $p=\ell(\Gamma)$. Indeed, we have seen that 
\[f=|V(\Gamma)|-1,\]
from which it follows that $p=\ell(\Gamma)$.
\textbf{This proves Claim 1.}

So, we can generally think of graphs in $\mathcal{Q}$ as in the figure below,
where $P$ is the unique path formed by all  directed fermion edges. $P$ goes from $f_1$ to $f_2$ and passes through the vertex at $r$. All vertices lie on $P$.
\begin{center}

\raisebox{0cm}{\includegraphics[scale=0.45]{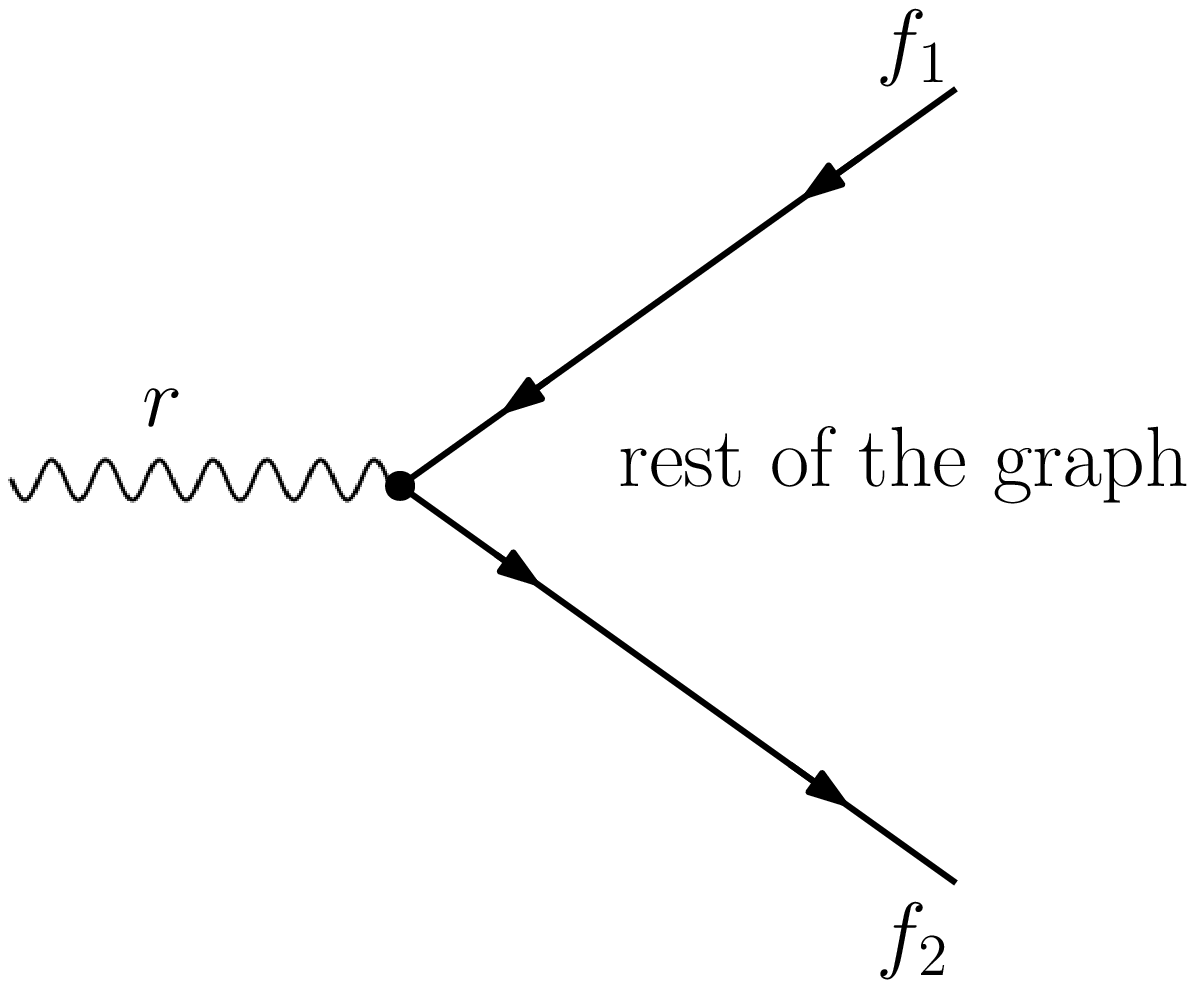}}
    
\end{center}

This means we can uniquely put any graph $\Gamma\in\mathcal{Q}$ in the form of a rooted chord diagram, namely by straightening $P$. See Figure \ref{intochords1} for an example. 
\begin{figure}[h]
    \centering
    \includegraphics[scale=0.42]{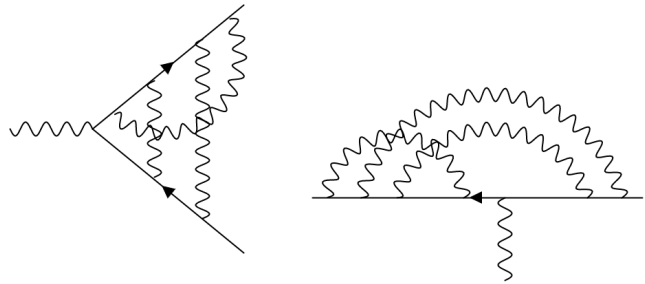}
    \caption{A primitive quenched QED graph and its representation as a chord diagram}
    \label{intochords1}
\end{figure}

For simplicity of drawing we shall now and forth in the proof use dashed or light lines for photons and drop the direction on the fermion edges on $P$. Also, let us agree that, in the chord diagram representation, we will bring $r$ to the front to play the role of a root, and still carry the information for the external leg position at its other end. Thus, for example, the graph in Figure \ref{intochords1} is now represented as follows:
\begin{center}

\raisebox{0cm}{\includegraphics[scale=0.42]{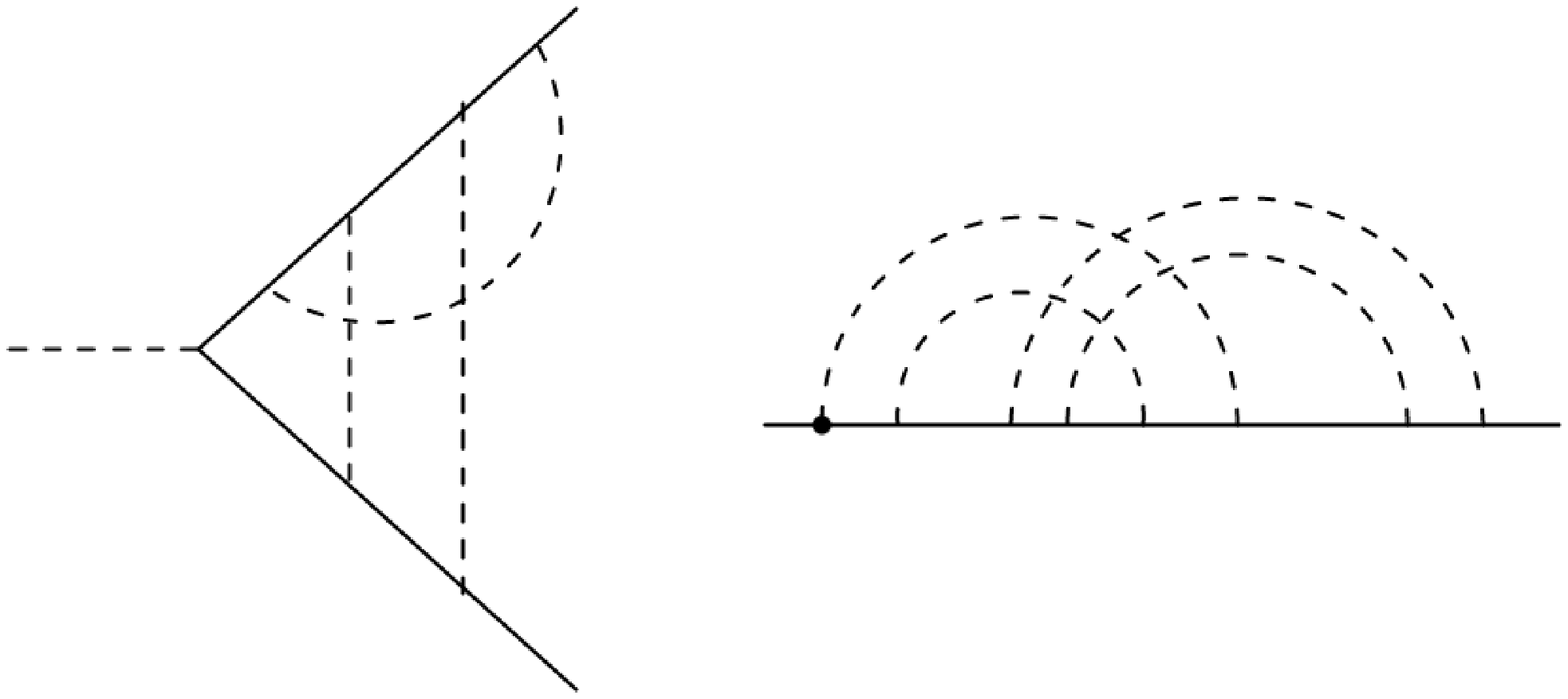}}
    
\end{center}
The chord diagram representation of $\Gamma\in\mathcal{Q}$ will be denoted by $C(\Gamma)$. Let us also denote the right end of the root $r$ by $v$. The only property of $\mathcal{Q}$ that we still haven't used is that a graph in $\mathcal{Q}$ is primitive.

\underline{\textbf{Claim 2:}} A 1PI quenched QED graph $\Gamma$ is primitive if and only if it is $2$-connected in the chord diagram representation. Subdivergences are translated into either a disconnection or a bridge.

\textbf{Case 1:} Assume that $C(\Gamma)$ is disconnected. This means that there exists an isolated component of chords to the right or left of $v$. On the original graph this is simply a propagator-type subdivergence inserted on one of the fermion edges. For an example see Figure \ref{disconnectandsubdiv} below.
\begin{figure}[H]
    \centering
\raisebox{0cm}{\includegraphics[scale=0.55]{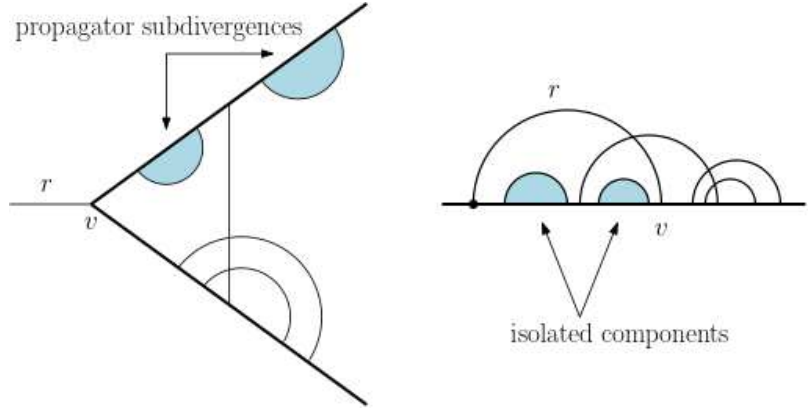}}
    \caption{Disconnections and propagator subdivergences.}\label{disconnectandsubdiv}
\end{figure}
The converse is also clearly true, a propagator subdivergence is translated into an isolated component in $C(\Gamma)$.

\textbf{Case 2: } Assume that $C(\Gamma)$ has a reason $S$ for connectivity-1, in the sense of Definition \ref{connectv1dfn}. Then the cut $c$ for $S$ is either the root chord $r$ or not.

(A) If $c=r$, then in $\Gamma$, $S$ together with $r$ correspond to a vertex-type subdivergence inserted at the vertex of the photon edge $r$.

(B) If $c\neq r$, then $S$ lies to the right or left of $v$ in $C(\Gamma)$. On $\Gamma$, this is a vertex-type subdivergence inserted an end of the photon edge $c$. 

Conversely, by the same means, every vertex-type subdivergence in $\Gamma$ gives a reason for connectivity-1 in $C(\Gamma)$. \textbf{This proves Claim 2}. Figure \ref{connect1andsubdiv} illustrates  situations  (A) and (B) on one and the same graph.
\begin{figure}[H]
    \centering
\raisebox{0cm}{\includegraphics[scale=0.67]{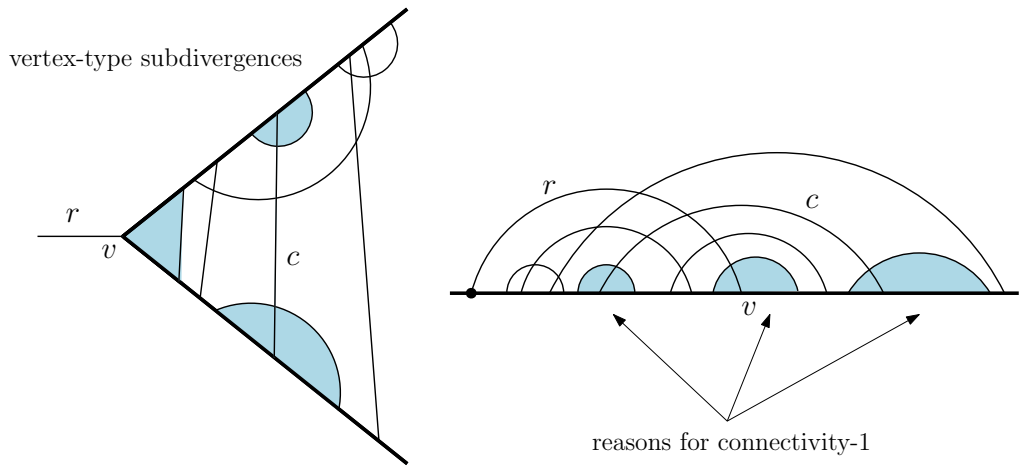}}
    \caption{Connectivity-1 and vertex-type subdivergences.}\label{connect1andsubdiv}
\end{figure}

Thus, every graph in $\mathcal{Q}$ is uniquely represented as a $2$-connected chord diagram with the number of chords equal to the number of internal photon edges and also equal to the loop number of the graph. The generating series $\;z_{|\psi_c|^2}(\hbar_{\text{ren}})\;$ counts the same diagrams, it only differs in not having the external photon leg $r$. The removal of the  external leg $r$ will not change the argument above: propagator-type subdivergences correspond to isolated components in the chord diagram representation and vertex-type subdivergences correspond to reasons for connectivity-1. This completes the proof.
\end{proof}
This gives 
$ \mathcal{A}^2_{\frac{1}{2}}z_{\phi_c|\psi_c|^2}(\hbar_{\text{ren}})=
\frac{1}{\hbar_{\text{ren}}}\mathcal{A}^2_{\frac{1}{2}}z_{|\psi_c|^2}(\hbar_{\text{ren}})=
\mathcal{A}^2_{\frac{1}{2}}C_{\geq2}(x),
$
and by equation (\ref{computC2asympt}):
\begin{align}
    [\hbar_{\text{ren}}^{n-1}]\;&z_{\phi_c|\psi_c|^2}(\hbar_{\text{ren}})=
[\hbar_{\text{ren}}^{n}]\;z_{|\psi_c|^2}(\hbar_{\text{ren}})=[x^n]\;C_{\geq2}(x)=\nonumber\\
    & = e^{-2}  \bigg((2n-1)!!-6(2n-3)!!-4(2n-5)!!-
    \displaystyle\frac{218}{3}(2n-7)!!- \nonumber\\
    &\; \qquad       -890(2n-9)!!-\displaystyle\frac{196838}{15}(2n-11)!!-\cdots\bigg)
    \nonumber\\
    & = e^{-2} (2n-1)!!\bigg(1-\displaystyle\frac{6}{2n-1}-\displaystyle\frac{4}{(2n-3)(2n-1)}-\displaystyle\frac{218}{3}\displaystyle\frac{1}{(2n-5)(2n-3)(2n-1)}- \nonumber\\
    &\;   \qquad     -\displaystyle\frac{890}{(2n-7)(2n-5)(2n-3)(2n-1)}-\displaystyle\frac{196838}{15}\displaystyle\frac{1}{(2n-9)\cdots(2n-1)}-\cdots\bigg),
\end{align}
which coincides with the result in \cite{michiq}.
\section{Yukawa Theory}\label{YYYYY}
In this section we will establish the connection between interacting Yukawa theory and chord diagrams. Unlike the case of quenched QED, the relation this time is well hidden. The problem was based on an observation, made by the author, upon seeing Table 18 (a) in \cite{michiq} while working on the quenched QED case. It can be seen that the rows in Table \ref{table4} below coincide with some of the expressions used in deriving the chord diagram identities earlier in this thesis, see Table \ref{table3} for example. For the sake of clarity, let us display Table 18 (a) of \cite{michiq} (with an extra column added to emphasize the relation to chord diagrams):

\begin{table}[H]
\center
\begin{tabular}{|c|c||c|c|c|c|c|c|c|}\hline
 &            &$\hbar^0$ & $\hbar^1$   & $\hbar^2$  & $\hbar^3$ & $\hbar^4$ & $\hbar^5$ & $\hbar^6$                           \\\hline\hline
1&$\partial_{\phi_c}^0(\partial_{\psi_c}\partial_{\bar{\psi}_c})^0
G^{\text{Yuk}}\big|_{\phi_c=\psi_c=0}$     
              &   0  & 0     & 1/2    & 1     & 9/2   & 31    & 283   
              \\\hline
2&$\partial_{\phi_c}^1(\partial_{\psi_c}\partial_{\bar{\psi}_c})^0
G^{\text{Yuk}}\big|_{\phi_c=\psi_c=0}$      
              &   0  & 1     & 1      & 4     & 27    & 248   & 2830       
              \\\hline
3&$\partial_{\phi_c}^2(\partial_{\psi_c}\partial_{\bar{\psi}_c})^0
G^{\text{Yuk}}\big|_{\phi_c=\psi_c=0}$      
              &  -1  & 1     & 3      & 20    & 189   & 2232  & 31130      
              \\\hline
4&$\partial_{\phi_c}^0(\partial_{\psi_c}\partial_{\bar{\psi}_c})^1
G^{\text{Yuk}}\big|_{\phi_c=\psi_c=0}$ 
              &  -1  & 1     & 3      & 20    & 189   & 2232  & 31130   
              \\\hline
5&$\partial_{\phi_c}^1(\partial_{\psi_c}\partial_{\bar{\psi}_c})^1
G^{\text{Yuk}}\big|_{\phi_c=\psi_c=0}$
              &   1  & 1     & 9      & 100   & 1323  & 20088 & 342430       
              \\\hline
 \end{tabular}\caption{The first coefficients of the proper Green functions $\left.\partial_{\phi_c}^i(\partial_{\psi_c}\partial_{\bar{\psi}_c})^jG^{\text{Yuk}}\right|_{\phi_c=\psi_c=0}$  of Yukawa theory such that $i+2j\in\{0,1,2,3\}$.  }
\label{table4}
\end{table}
As we have mentioned in Section \ref{partitionfunctionsection}, Yukawa theory is different from quenched QED in that  fermion loops are allowed in the diagrams. Thus, if $\mathcal{U}_{ij}$ denotes the class of 1PI Feynman graphs counted by the Green function $\left.\partial_{\phi_c}^i(\partial_{\psi_c}\partial_{\bar{\psi}_c})^jG^{\text{Yuk}}\right|_{\phi_c=\psi_c=0}$, then, to our combinatorial concern,   $\mathcal{U}_{ij}$ consists of all graphs $\gamma$ with the following specifications:
\begin{enumerate}
   \item two types of edges (as before), fermion and boson (meson) edges, represented as \raisebox{-0cm}{\includegraphics[scale=0.6]{Figures/fermionedge.eps}} \;and \raisebox{-0.2cm}{\includegraphics[scale=0.4]{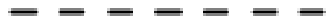}},  respectively;
    \item only three-valent vertices with the structure \raisebox{-0.64cm}{\includegraphics[scale=0.3]{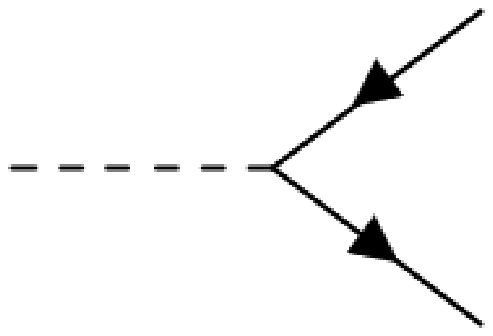}}, with one fermion in, one fermion out, and one boson;
    
    \item fermion loops are allowed;
    \item the residue $\text{res} (\gamma)$ has $i$ external boson legs,  $j$  external fermion-in legs, and $j$ external fermion-out legs; and 
    \item $\gamma$ is $1PI$, i.e. it has no bridges (2-edge-connected). This is implied by the definition of proper Green functions (see Section \ref{quenchedsec} or Section \ref{section}).
\end{enumerate}

Note that, unlike the quenched QED case, we are not restricting to primitive diagrams since we are not working with any expressions from renormalization in this part (yet).

\begin{lem}\label{property1Yukawa}
Let $\Gamma$ be a Yukawa theory 1PI graph with $2j$ external fermion legs. Then
\begin{equation}
    |V(\Gamma)|=f+j,
\end{equation}

where, as before, $|V(\Gamma)|$ is the number of vertices and $f$ is the number of internal fermion edges.
\end{lem}

\begin{proof}
Consider the graph $\Gamma'$ obtained from $\Gamma$ by removing all boson edges and half edges. The resulting graph $\Gamma'$ is generally a collection of fermion loops and fermion paths. 
The number of these paths should be $j$. Indeed, on one hand every such path will give two external fermion legs when the boson edges are present. To see this notice that, under the vertex condition, such paths can not end with a vertex: the boson edges can not then complete the degree of such a vertex. On the other hand, it is clear that the external fermion legs can only be at the ends of such paths. 

As a consequence of the above argument, the number of vertices can be calculated as follows: every fermion loop has  as many vertices as fermion edges. Whereas in every fermion path the number of vertices is more by one the number of internal fermion edges. This proves the lemma.
\end{proof}

Next we investigate the combinatorial meaning of the Green functions in 
Table \ref{table4} and its relation to chord diagrams. This is done by first proving that the number of \textit{1PI} tadpole graphs (line 2 in Table \ref{table4}) in Yukawa theory with loop number $n$  is equal to the number of connected chord diagrams with $n$ chords. In consequence, we get the other interpretations, in terms of chord diagrams, for the other Green functions listed in Table \ref{table4}.

\subsection{Yukawa Tadpole Graphs:\quad $\left.\partial_{\phi_c}^1(\partial_{\psi_c}\partial_{\bar{\psi}_c})^0G^{\text{Yuk}}(\hbar,\phi_c,\psi_c)\right|_{\phi_c=\psi_c=0}$}

By definition (see Sections \ref{quenchedsec} and \ref{section}), $\left.\partial_{\phi_c}^1(\partial_{\psi_c}\partial_{\bar{\psi}_c})^0G^{\text{Yuk}}(\hbar,\phi_c,\psi_c)\right|_{\phi_c=\psi_c=0}$ is the generating series of Yukawa theory graphs with exactly one external leg, which is of boson type, graded by loop number. In other words, 
\[[\hbar^n]\left.\partial_{\phi_c}^1(\partial_{\psi_c}\partial_{\bar{\psi}_c})^0G^{\text{Yuk}}\right|_{\phi_c=\psi_c=0}\]
is the number of \textit{1PI} tadpole graphs with one boson leg and loop number $n$.

From Table \ref{table4}, we can conjecture that $[\hbar^n]\left.\partial_{\phi_c}^1(\partial_{\psi_c}\partial_{\bar{\psi}_c})^0G^{\text{Yuk}}\right|_{\phi_c=\psi_c=0}=C_n$, the number of connected chord diagrams on $n$ chords. Interestingly, unlike the quenched QED graphs, tadpoles do a great job hiding their chord diagrammatic structure. This however is to be unveiled Theorem \ref{myresult1inYukawa} below. In Figure \ref{27tadpoles} below, we display the tadpoles counted in $[\hbar^4]\left.\partial_{\phi_c}^1(\partial_{\psi_c}\partial_{\bar{\psi}_c})^0G^{\text{Yuk}}\right|_{\phi_c=\psi_c=0}$, which, by our claim, should be $27$ as $C_4$. 

\begin{rem} For the sake of simplicity of drawings, we will drop the direction of the fermion loops and assume it is always counter-clockwise. Besides, we will draw no more dashed boson lines, and shall instead use a light line for bosons and a heavier line for fermions. Note that the relative direction of loops matters in some cases and have to be compensated sometimes by a twist in the boson edges. For example,  the following two tadpoles in Figure \ref{two diff tadpoles1} are different, and shall be represented as in Figure \ref{two diff tadpoles11}:
\begin{figure}[H]
    \centering
   \raisebox{0cm}{\includegraphics[scale=0.41]{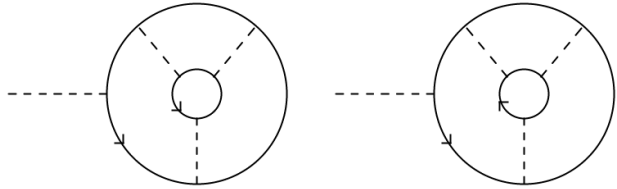}}
    \caption{Two tadpoles may differ due to the relative orientation of fermion loops.}
    \label{two diff tadpoles1}
\end{figure}
\begin{figure}[h]
    \centering
   \raisebox{0cm}{\includegraphics[scale=0.54]{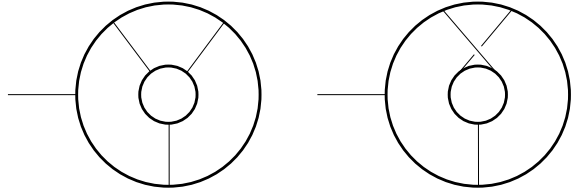}}
    \caption{All loops are now assumed oriented counter-clockwise and the attached boson edges have to be twisted accordingly; Fermion edges are drawn thicker than boson edges.}
    \label{two diff tadpoles11}
\end{figure}
\end{rem}

\newpage

\begin{figure}[h!]
    \centering
   \raisebox{0cm}{\includegraphics[scale=0.73]{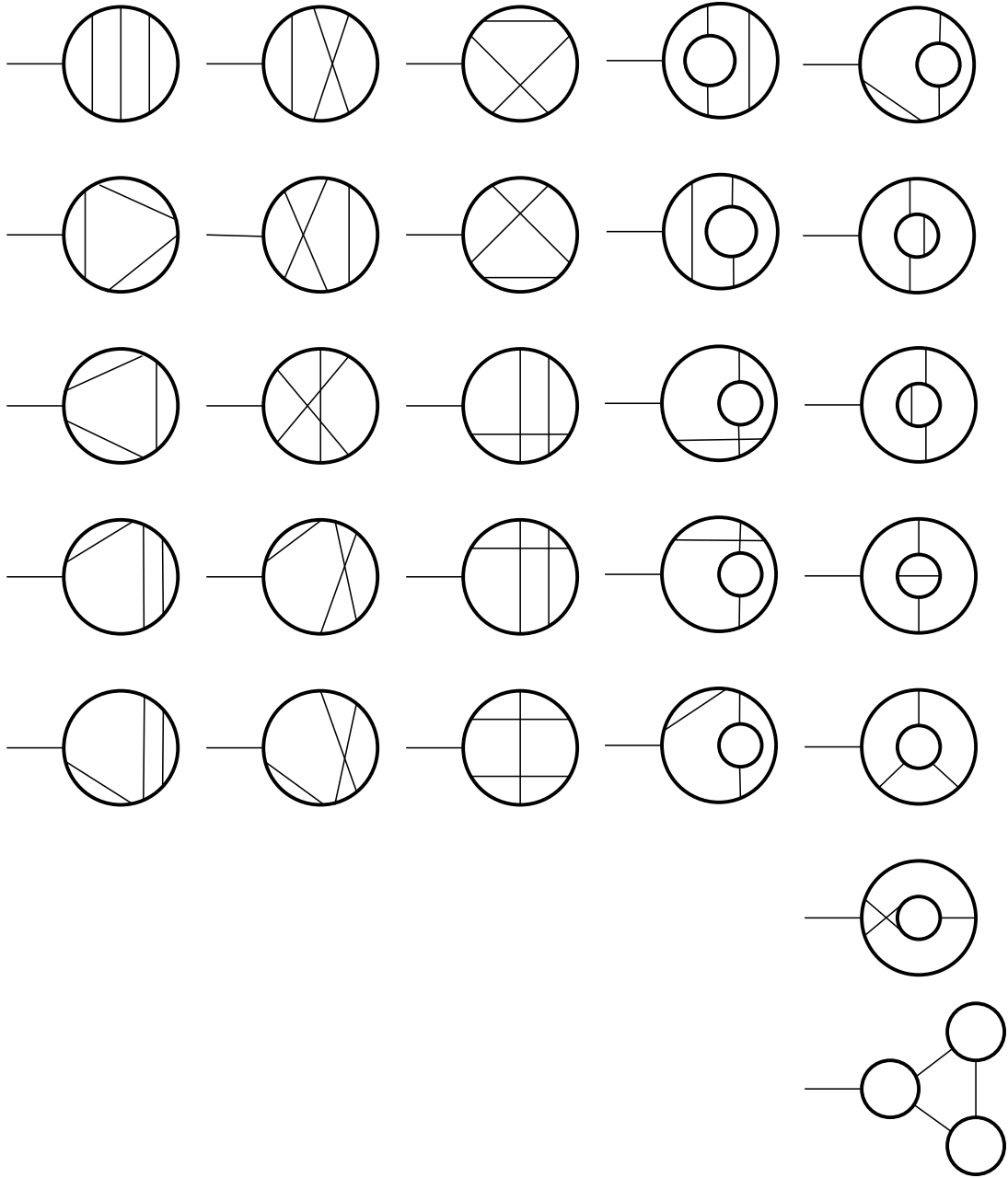}}
    \caption{The 27 \textit{1PI } tadpole graphs with loop number 4}
    \label{27tadpoles}
\end{figure}

The next lemma is a direct corollary to Lemma \ref{property1Yukawa} for the case of tadpoles.
\begin{lem}\label{property2yukawa}
For any Yukawa 1PI tadpole graph, the following are true for all $\Gamma\in\mathcal{U}_{10}$:

\begin{enumerate}
   
\item The number of vertices is equal to the number of fermion edges,
\[|V(\Gamma)|=f.\]

 \item The number of all boson edges (including the external boson leg) is equal to the loop number of the graph. That is,
\[p+1=\ell(\Gamma).\]

\item Fermion loops partition the set of vertices;
\end{enumerate}
 where, as before, $p$ is the number of internal boson edges (the $p$ is suitable in this case too as it stands for a \textit{pion}, pions  are often the bosons in a Yukawa interaction), and $f$ is the number of fermion edges (all fermion edges are internal in this case).

\end{lem}

\begin{proof}
By Lemma \ref{property1Yukawa}, we directly have $|V(\Gamma)|=f$.  Euler's formula now implies   \[1=|V(\Gamma)|-(p+f)+\ell(\Gamma)=f-(p+f)+\ell(\Gamma)=-p+\ell(\Gamma),\] for any $\Gamma\in\mathcal{U}_{10}$. Finally, since there are no external fermion edges, every fermion edge must be on a fermion loop. Thus all vertices are on fermion loops, and by the condition on the vertices in the theory, no vertex can lie on more than one fermion loop.
\end{proof}

Lemma \ref{property2yukawa} is useful in that we do not have to think about the loop number in proving the bijection to connected chord diagrams, and can instead focus on the more evident count of boson edges.

\begin{nota}\label{nota randv}
For a tadpole $T\in\mathcal{U}_{10}$, we will fix the notation that the external boson leg is denoted by $r_T$, and the vertex at the leg is denoted by $v_T$ (this is consistent with the notation used in the previous section). For a vertex $a \in V(T)$ we let Loop$(a)$ denote the unique fermion loop containing $a$, and we let Fermion$(a)$ be the fermion edge coming out of $a$ (i.e. the next on Loop$(a)$ counter-clockwise). Boson$(a)$ will denote the unique boson edge to which $a$ is incident.
\begin{center}

\raisebox{0cm}{\includegraphics[scale=0.67]{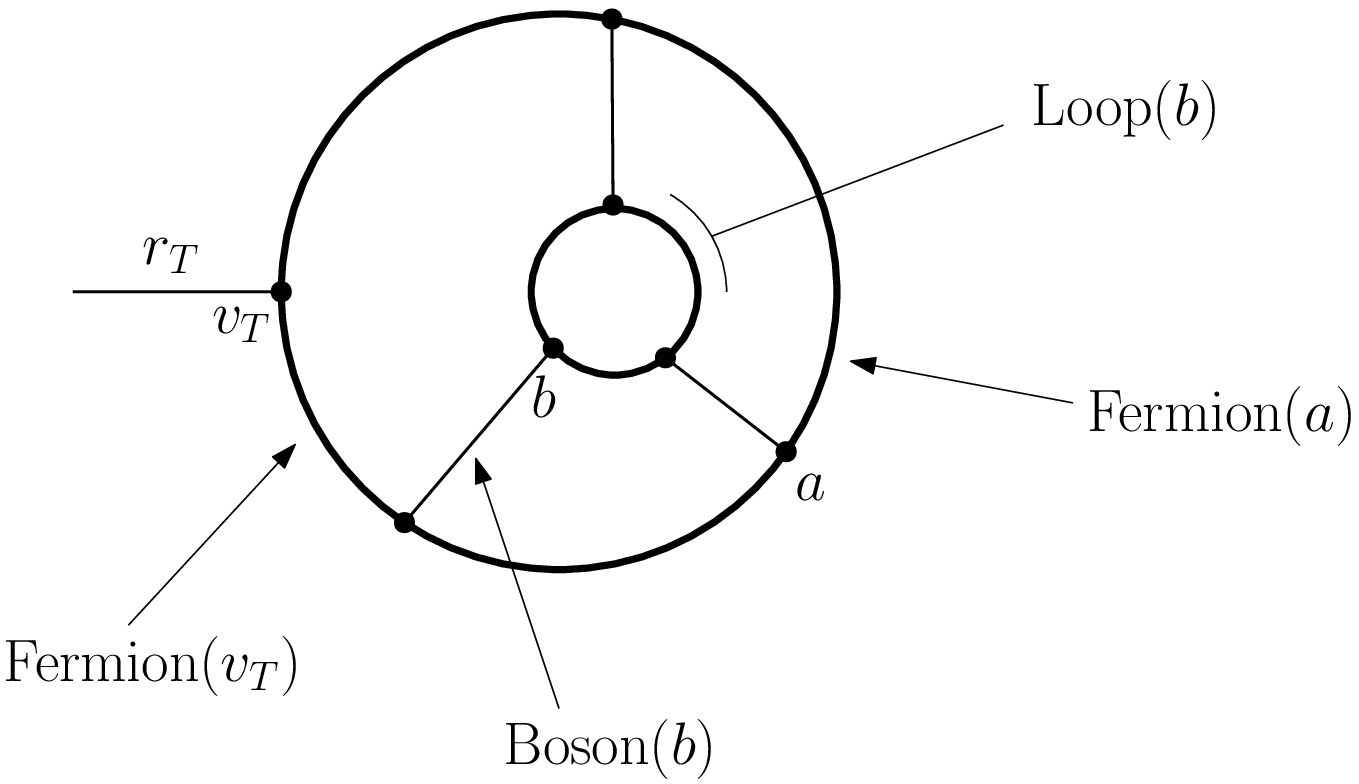}}
    
\end{center}

In the next proof we consider the free end of $r$ to be a vertex of degree 1. Then the number of vertices is twice the number of boson edges. 
\end{nota}

\begin{thm}\label{myresult1inYukawa}
The number of Yukawa 1PI tadpole graphs with loop number $n$ is equal to the number of connected chord diagrams on $n$ chords. In other words,
\[[\hbar^n]\left.\partial_{\phi_c}^1(\partial_{\psi_c}\partial_{\bar{\psi}_c})^0G^{\text{Yuk}}\right|_{\phi_c=\psi_c=0}=C_n.\]

\end{thm}

\begin{proof}
Let $T(x)$ be the generating series for tadpoles in $\mathcal{U}_{10}$, counted by the number of boson edges (including the external boson leg). We are taking advantage of Lemma \ref{property2yukawa} in order to use the number of bosons instead of the loop number. The theorem shall be proven through an algorithm that shows that  $T(x)$ obeys the same recurrence as $C(x)$ (see Lemma \ref{cd}), namely
\[2xT(x)T^{\prime}(x)=T(x)^2+T(x)-x.\]
First notice that the LHS stands for two tadpole diagrams, one of which has a distinguished end point of one of the boson edges (hence the 2 factor). For simplicity, we will treat the free end of an external boson leg as a vertex. Then let $\mathcal{U}_{10}^\bullet$ be the class of tadpoles with a distinguished vertex. Also let $\mathcal{U}_{10}-\{\mathcal{X}\}$ be the class of tadpoles excluding $\mathcal{X}$, the tadpole with one vertex ($\mathcal{X}$ has only the boson leg).

Let $T_1,T_2$ be tadpole graphs in $\mathcal{U}_{10}$, and assume that $T_2$ has a distinguished vertex $d$. By Notation \ref{nota randv}, we let $r_1$ and $r_2$ be the boson legs of  $T_1$ and $T_2$, and we let $v_1$ and $v_2$ be the 3-valent vertices incident to $r_1$ and $r_2$, respectively. Figure \ref{T1T2figure} illustrates the notation.

\begin{figure}[h!]
    \centering
    \raisebox{0cm}{\includegraphics[scale=0.67]{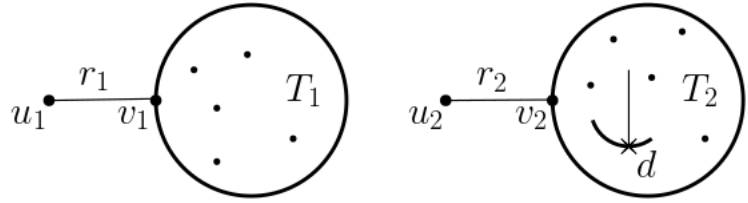}}
    \caption{Notation for $(T_1,(T_2,d))$}
    \label{T1T2figure}
\end{figure}

Now we can describe the reversible algorithm as follows.

\rule{\textwidth}{0.4pt}
\textbf{Algorithm $\Psi$:} \quad$(\mathcal{U}_{10}\times\mathcal{U}_{10}^\bullet)\longrightarrow (\mathcal{U}_{10}\times\mathcal{U}_{10})\;\bigcup\; (\mathcal{U}_{10}-\{\mathcal{X}\})$\\
\rule{\textwidth}{0.4pt}

\textbf{Input:} $(T_1,(T_2,d)) \in(\mathcal{U}_{10}\times\mathcal{U}_{10}^\bullet)$, with notation as described above.\\

 (a) If $d=u_2$ just \textbf{return} $(T_1,T_2)$.\\
 
 (b) If $d\neq u_2$, do the following: 
 
 Move (counter-clockwise) along Loop$(v_1)$ in $T_1$, determine Fermion$(v_1)$ and let $w$ be the first vertex met on the loop. Note that $w$ may be $v_1$ itself.
 
    \begin{enumerate}
  \item  If $w=v_1$, i.e. $T_1$ contains no internal boson edges, \textbf{return} the tadpole $T$ obtained as follows:
  
  \begin{enumerate}
      \item[(i)] Insert vertex $v_1$ together with the leg $r_1$ into Fermion$(d)$ in $T_2$ by making a subdivision of Fermion$(d)$.
      \item[(ii)] Insert $u_2$ into the new Fermion$(v_1)$ on Loop$(d)$.
  \end{enumerate}
  
  \item If $w\neq v_1$, \textbf{return} the tadpole $T$ obtained as follows:
  \begin{enumerate}
      \item[(i)] Insert $u_2$ into  Fermion$(v_1)$ in $T_1$. 
      \item[(ii)] Detach $w$ from Loop$(v_1)$ and insert it into Fermion$(d)$ in $T_2$.
  \end{enumerate}
 
        \end{enumerate}
     \rule{\textwidth}{0.4pt}       
     
Figures \ref{Psi1} and \ref{Psi2} illustrate the two cases for (b).

\begin{figure}[h!]
    \centering
    \includegraphics[scale=0.7]{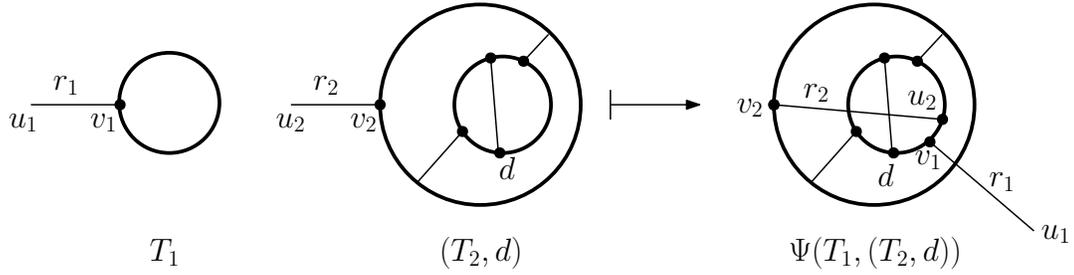}
    \caption{$\Psi(T_1,(T_2,d))$ when $T_1$ has exactly one vertex, the case $w=v_1$.}
    \label{Psi1}
\end{figure}

    \begin{figure}[h!]
    \centering
    \includegraphics[scale=0.7]{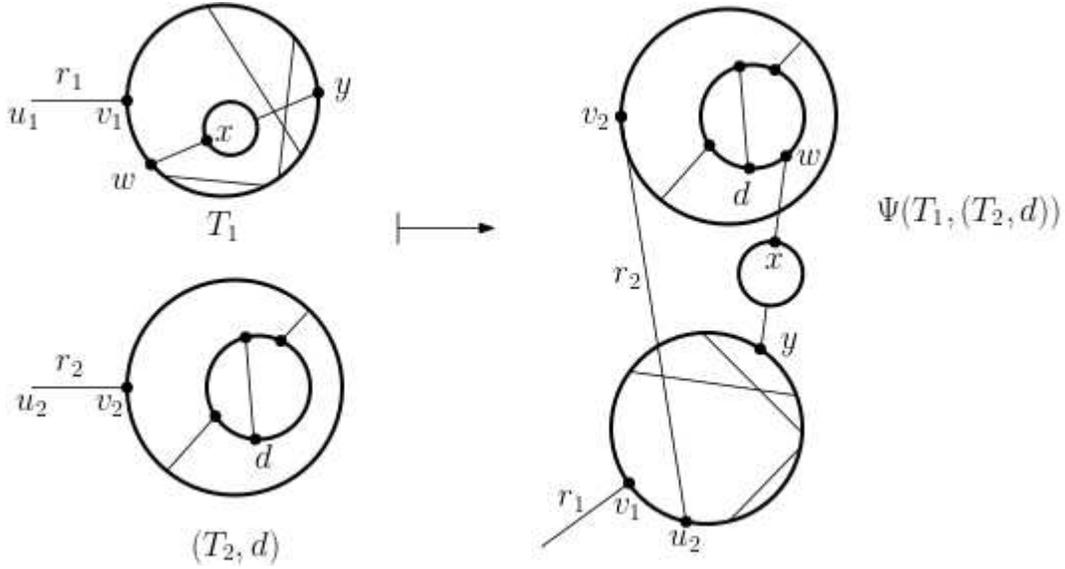}
    \caption{$\Psi(T_1,(T_2,d))$ for a general $T_1$, the case $w\neq v_1$.}
    \label{Psi2}
\end{figure}

For the reverse process we can devise the following algorithm.

\rule{\textwidth}{0.4pt}
\textbf{Algorithm $\Psi^{-1}$:} \quad$ (\mathcal{U}_{10}\times\mathcal{U}_{10})\;\bigcup\; (\mathcal{U}_{10}-\{\mathcal{X}\})\longrightarrow(\mathcal{U}_{10}\times\mathcal{U}_{10}^\bullet)$\\
\rule{\textwidth}{0.4pt}

\textbf{Input:} $(T_1,T_2)\in(\mathcal{U}_{10}\times\mathcal{U}_{10})$,\; or \;$ T \in \mathcal{U}_{10}-\{\mathcal{X}\}$\\

 (a) If the input is a pair $(T_1,T_2)$ then \textbf{return} $(T_1,(T_2,v_2))$.\\
 
 (b) If the input is a tadpole graph $T\in\mathcal{U}_{10}-\{\mathcal{X}\}$, do the following: 
 \begin{enumerate}
 \item Move (counter-clockwise) along Loop$(v_T)$, determine the first vertex $a$ met on the loop. Note that $a\neq v_T$  since we are excluding the tadpole with a single vertex.

  \item  Determine the other end vertex of Boson$(a)$ and denote it by $v_2$.
  
  \item Remove the vertex $a$ from $T$ and keep the resulting boson leg attached at $v_2$. Let the resulting graph be denoted by $\Gamma$.
  
  \item Check whether $\Gamma$ contains a bridge:
  
  \begin{enumerate}
      \item[Case 1:] If $\Gamma$ is 2-edge-connected (i.e. contains no bridges) do
      \begin{enumerate}
          \item[(1)] Determine the first vertex on Loop$(v_T)$ before $v_T$, denote it by $d$.
          \item[(2)] Remove $v_T$ and its boson leg $r_T$. Denote the remaining tadpole by $T_2$.
          
          \item[(3)] \textbf{Return} $(\mathcal{X}, (T_2,d))$.
      \end{enumerate}

\item[Case 2:] If $b_0$ is a bridge (must be a boson edge) in $\Gamma$,  undergo the following:
      \begin{enumerate}
        
    \item[(1)] Set $G=\Gamma$ and $b=b_0$. 
    \item[(2)] \{\textbf{while} $G$ has a bridge $g$ \textbf{do}
      \begin{enumerate}
      \item reset $b\longleftarrow g$;
      \item determine the component $\gamma$ that contains $v_2$ if $b$ is removed; 
      \item reset $G\longleftarrow \gamma$.\}
      \end{enumerate}
   \item[(3)] Let $w$ be the end vertex of $b$ that lies in $G$. Notice that, after the while-loop, $G$ contains no bridges. 
    \item[(4)] Determine the first vertex on Loop$(w)$ before $w$, denote it by $d$.
    \item[(5)] Detach $w$ from Loop$(d)$ in $G$ and insert into Fermion$(v_T)$ (i.e. next to $v_T$ on Loop$(v_T)$).
    \item[(6)] Let $T_2=G-{w}$ be the tadpole obtained from $G$ after $w$ is removed.
    \item[(7)] Let $T_1$ be that tadpole obtained in $\Gamma-G$ after $w$ is inserted on Loop$(v_T)$.
    \item[(8)] \textbf{Return} $(T_1,(T_2,d))$.
      \end{enumerate}
  \end{enumerate}

        \end{enumerate}
 \rule{\textwidth}{0.4pt}
 
Before discussing the algorithms, the reader may like to see Example \ref{examplePsiinverse} for applying $\Psi^{-1} $ to $\Psi(T_1,(T_2,d))$ from Figure \ref{Psi2}.

Now, for Algorithm $\Psi$, the two cases (a) and (b) are clearly distinguishable by the types of their outputs, so Let us focus on (b). 

\begin{enumerate}
    \item The special case in ($\Psi$ b:(1)) when $w=v_1$ returns a tadpole without bridges. Indeed, we only added an external leg $r_1$ at the position determined by $d$, and then we inserted the free end $u_2$ (the external leg of $T_2$). None of these steps changes the connectivity of $T_2$, and the return value is indeed in $\mathcal{U}_{10}-\{\mathcal{X}\}$. 
    
    \item In ($\Psi$ b:(2)), when $w\neq v_1$, the result has no bridges, since what we do is roughly joining $T_1$ and $T_2$ by means of two boson edges in a certain way. Thus the result is indeed a tadpole in $\mathcal{U}_{10}-\{\mathcal{X}\}$. Notice that the  first of these joints is attached next to the leg, and its removal leaves the graph with a bridge. This is of most importance in the reverse process.
 
\end{enumerate}
 
 Then, for Algorithm $\Psi^{-1}$, we have the following:
 
 \begin{enumerate}
     \item If the input is a pair, then this uniquely means that the distinguished vertex satisfies $d=u_2$.
     \item If the input is a tadpole that stays bridgeless after the the vertex $a$ next to $v_T$ is  removed then this uniquely means that $T_1=\mathcal{X}$. Indeed, we have seen above that in all the other cases we get a bridge if the first boson edge after the external leg is removed. 
     
     \item If the input reveals a bridge $b_0$ when Boson$(a)$ is removed, then we learn that $a$ and Boson$(a)$ formed the external leg of $T_2$ and we start disentangling $T_2$ from the graph.   Roughly speaking, we need to determine $d$ by using the fact that, in the absence of Boson$(a)$, Boson$(d)$ is a bridge coming from $T_1$. The while loop in the algorithm works on finding the last such bridge.
     
     \item By the engineering of the while-loop, the graph $G$ obtained at the end of the loop has no more bridges, besides, it carries the traces of the last bridge $b$, which determines the distinguished vertex $d$. 
     
     \item After modifying $G$ by removing $w$ we get $(T_2,d)$, and simultaneously we get $T_1$ from the remaining graph by attaching $w$ into Fermion$(v_T)$. By doing so, $T_1$ is also bridgeless.
 \end{enumerate}
This proves the theorm.
\end{proof}

\begin{exm}\label{examplePsiinverse}
Let us apply Algorithm $\Psi^{-1}$ to the tadpole $T$ given in Figure \ref{Psi2} by 
\begin{center}
    \raisebox{0cm}{\includegraphics[scale=0.45]{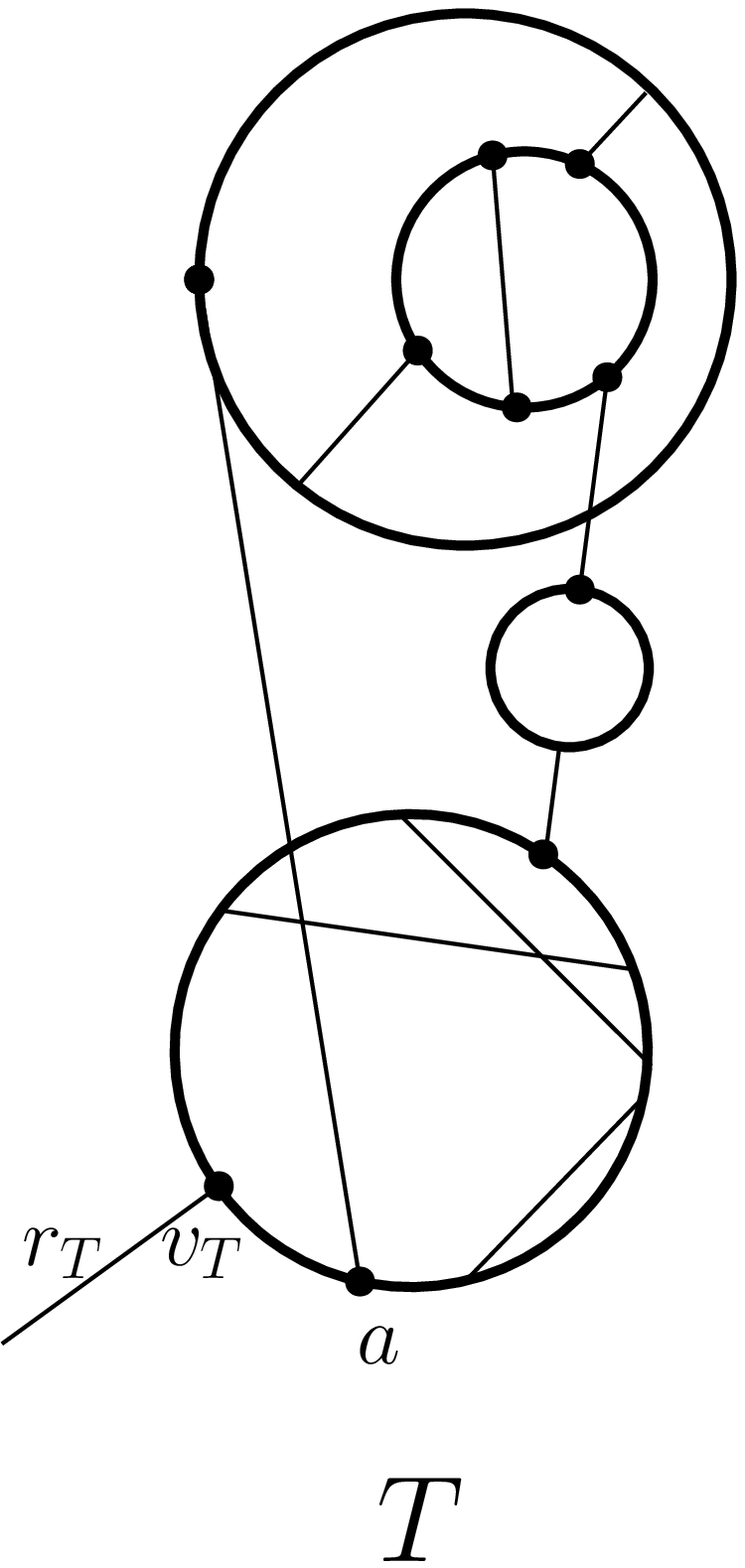}}
\end{center}

\begin{enumerate}
    \item The input is a tadpole in $\mathcal{U}_{10}-\{\mathcal{X}\}$ and so we apply (b).
    \item We determine $a$ as the vertex next to $v_T$ on Loop$(v_T)$, and with it we determine $v_2$, the other end of Boson$(a)$.
    \item We remove vertex $a$ from $T$ and keep Boson$(a)$ attached at $v_2$ as in the figure below.
    
    \begin{center}
    \raisebox{0cm}{\includegraphics[scale=0.75]{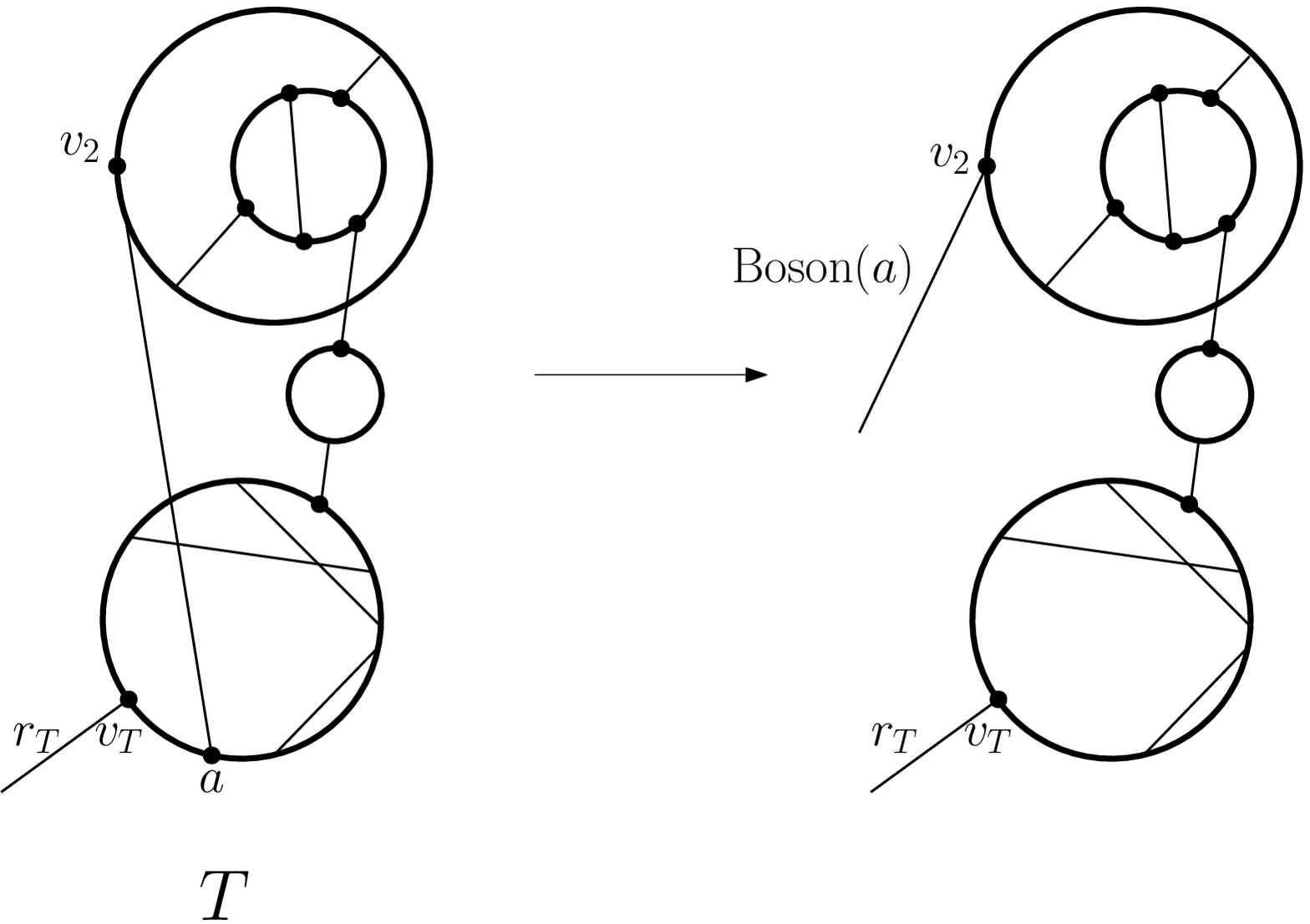}}
\end{center}

\item We check for bridges and we find one of them, we assume that our search provided the bridge $b_0$.
\begin{center}
    \raisebox{0cm}{\includegraphics[scale=0.5]{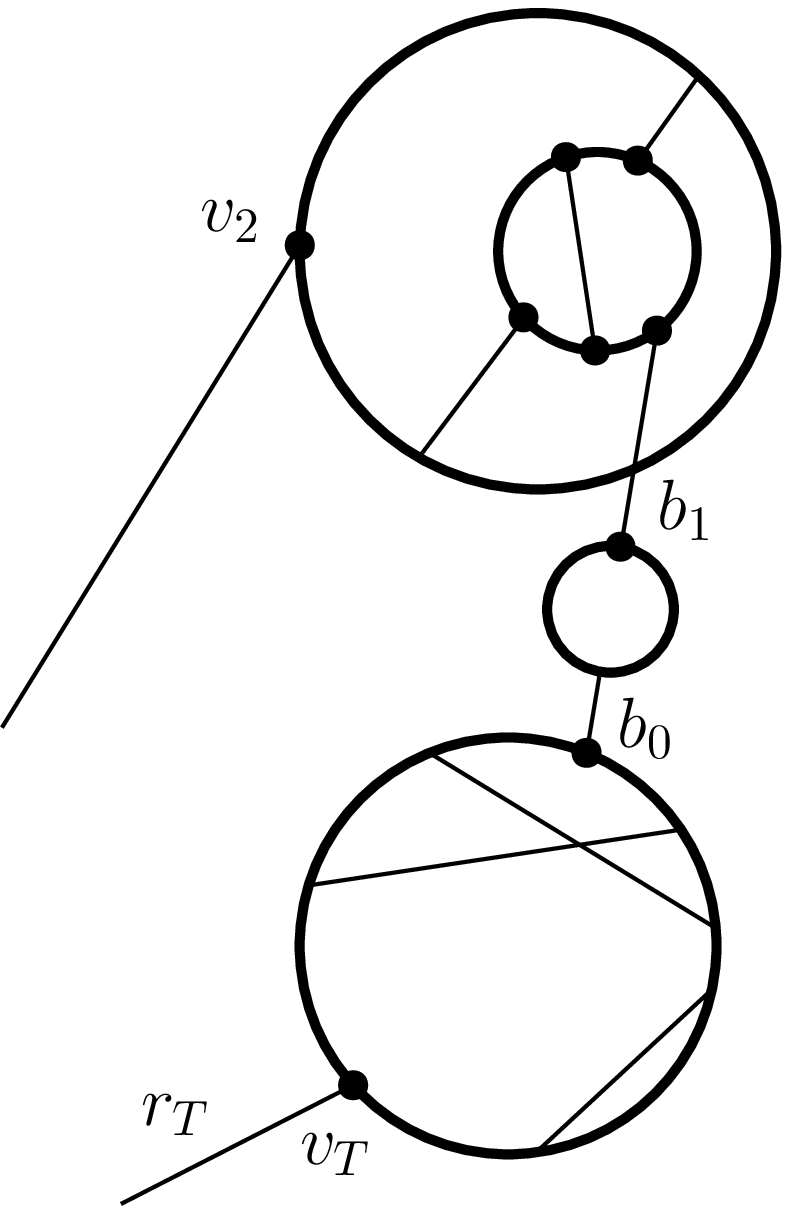}}
\end{center}

\item We enter the while loop with $G=\Gamma$ given above and $b=b_0$:\begin{enumerate}
    \item After the first iteration $G$ is modified to be 
    \begin{center}
    \raisebox{0cm}{\includegraphics[scale=0.5]{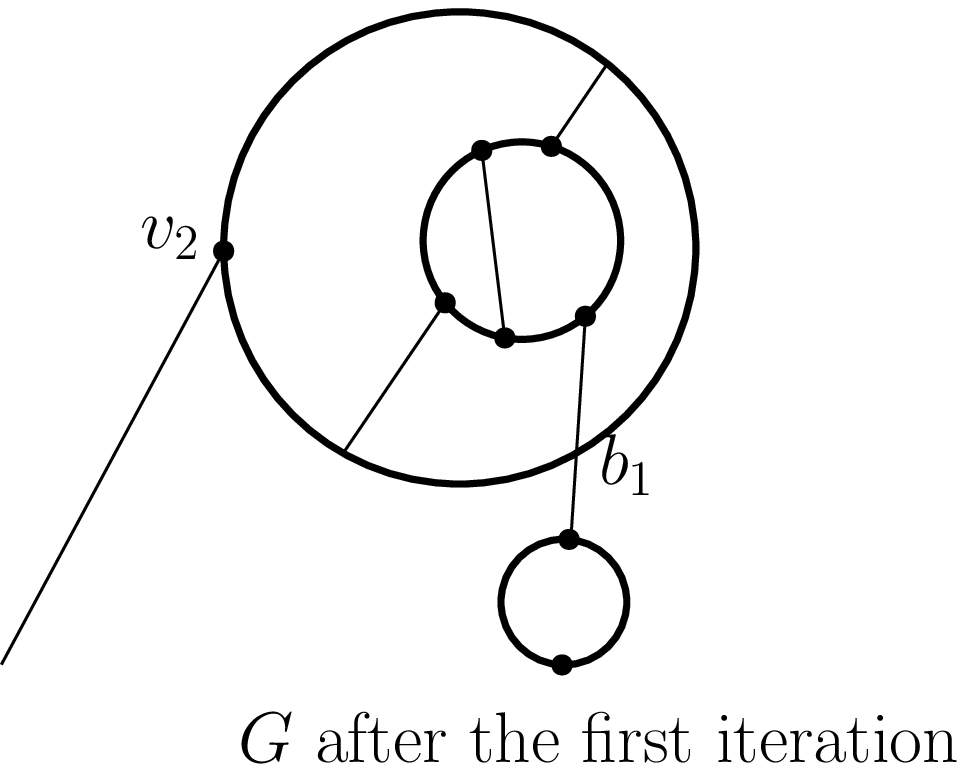}}\end{center}
    
    \item Now $G$ again has a bridge $b$, and, after the second iteration $G$ has no more bridges. After the while-loop $b=b_1$ and $G$ is given by

\begin{center}
    \raisebox{0cm}{\includegraphics[scale=0.5]{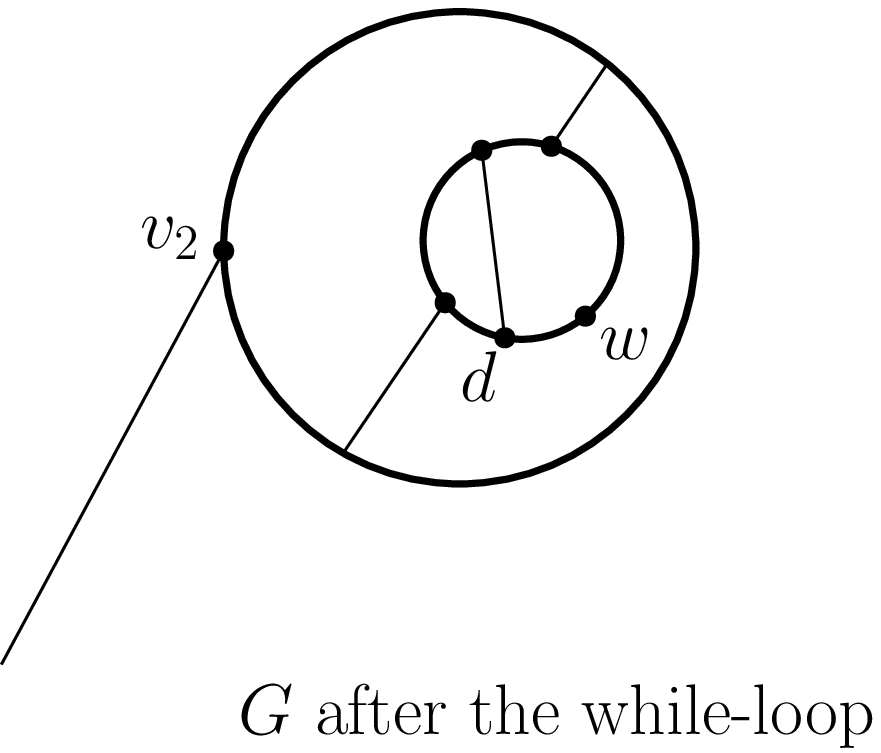}}\end{center}
    
    \item Finally, detach $w$ from $G$, set $T_2=G-w$, reset $G=G-w$, and insert $w$ next to $v_T$ in $\Gamma-G$ to get $T_1$. The result is shown in the figure below.
\end{enumerate} 

    \raisebox{0cm}{\includegraphics[scale=0.7]{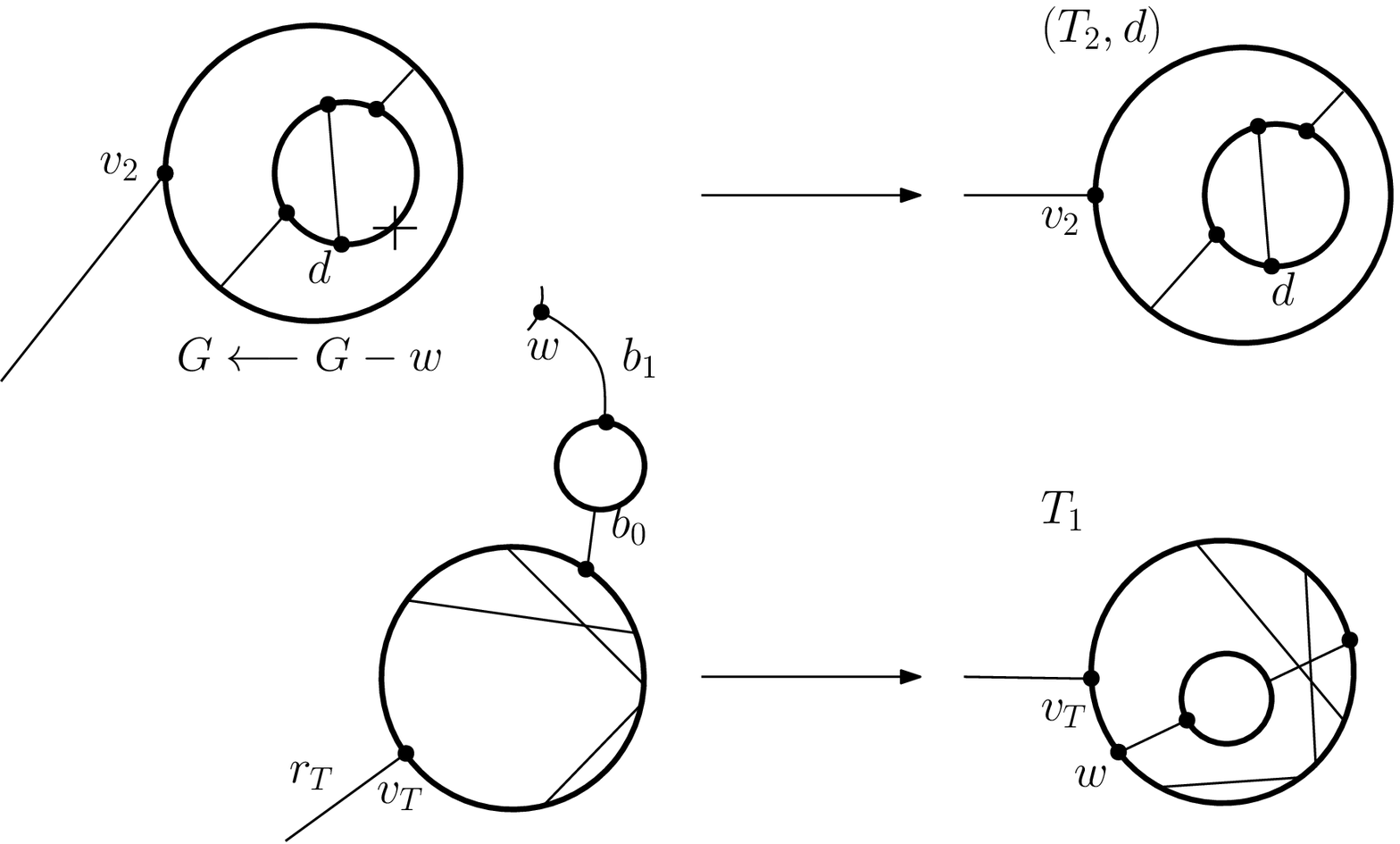}}
\end{enumerate}

\end{exm}

Theorem \ref{myresult1inYukawa} tells us that the two structures, Yukawa 1PI tadpoles and connected chord diagrams satisfy the same recurrence and hence there exists a bijection between the two classes obtained recursively. However, we still have to do one more bit of work to express this bijection. The bijection should respect the sizes, that is, a Yukawa 1PI tadpole with $n$ boson edges (with $n$ loops) should be mapped to a connected chord diagram on $n$ chords.

Moreover, as we can see, the vertices of a tadpole should correspond to the vertices in a connected chord diagram, and the fermion edges should accordingly correspond to the intervals. It has not been made clear so far  how we can order fermion edges in a way that resembles the natural linear order of the intervals in chord diagrams, an  order that is compatible with the decomposition in Theorem \ref{myresult1inYukawa}. Definition \ref{psiorder} below addresses this issue. To see why the order should be defined this way, we have to first recall the root share decomposition of connected chord diagrams. The root share decomposition has been mentioned and used in the proof of Lemma \ref{bij}, and now we need to define it properly:

\begin{dfn}[Root Share Decomposition]
The \textit{root share decomposition} is the map $\nabla: \mathcal{C}\longrightarrow \mathcal{C}\times(\mathcal{C},\mathbb{N})$ defined by 

\[\nabla C=(C_1,(C_2,k)),\]

where $1\leq k\leq |C_2|-1$, and $C_1,C_2$ are obtained as follows:
Among the components produced by removing the root of $C$, $C_2$ is taken to be the first in intersection order with the root. $k$ determines the interval in $C_2$ through which the root used to pass. $C_1$ is then obtained by removing the chords of $C_2$ from $C$. 

If $(C_1,(C_2,k))\in\mathcal{C}\times(\mathcal{C},\mathbb{N})$ is a valid triplet then $\nabla^{-1}(C_1,(C_2,k))$ is the connected chord diagram obtained by placing $C_1$ in the $k$th interval of $C_2$ and pulling the root of $C_1$ out to become the root of the whole diagram (i.e. place it to the left of the root of $C_2$). See Figure \ref{roootshare}.
\end{dfn}

\begin{figure}[h!]
    \centering
    \includegraphics[scale=1]{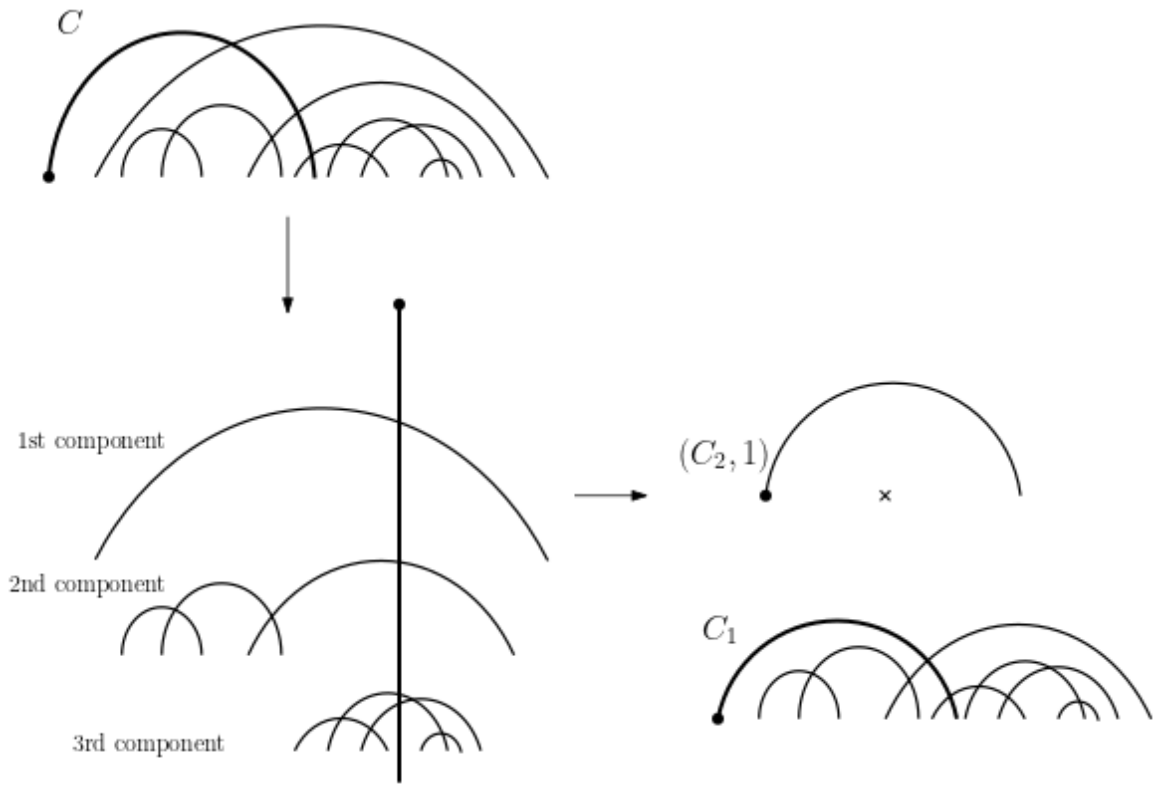}
    \caption{Root share decomposition of a connected chord diagram}
    \label{roootshare}
\end{figure}
\newpage
The effect of the root share decomposition on the linear order of the intervals in $C_1$ and $C_2$ leads us to the following order on fermion edges in a Yukawa 1PI tadpole.

\begin{dfn}[The $\Psi$-order]\label{psiorder}
We define the \textit{$\Psi$-order} on the fermion edges in a Yukawa 1PI tadpole inductively on the size $n$ of the tadpoles as follows:

 $\ast$ For a fermion edge $e$ in a tadpole $T$, its  $\Psi$-order takes values in $\mathbb{N}$ and is to be denoted by $\psi_T(e)$.
\begin{itemize}
   
    \item For $\mathcal{X}$ the unique fermion edge is ordered as 1.
    \item Assume all tadpoles of size less than $n$ are ordered and let $T$ be a tadpole of size $n$. To order $T$ do the following:
    \begin{enumerate}
        \item Apply $\Psi^{-1}$ to $T$ to determine a triplet $(T_1,(T_2,d))$. As before, let $v_T$, $v_1$ and $v_2$ be the leg vertices in $T$, $T_1$ and $T_2$ respectively. Let $w_1$ be the vertex next to $v_1$ in $T_1$. Also let $w_T$ be the vertex in $T$ next to $v_T$, and let $w_d$ be the vertex in $T$ next to the vertex $d$ (i.e. these are the vertices of subdivisions created by $\Psi$). Note that $w_d=v_T$ if $T_1=\mathcal{X}$.
        
        \item By the induction hypothesis it is assumed that we know the $\Psi$-ordering of $T_1$ and $T_2$. Let \[M=\max_{e\in T_1} \psi_{T_1}(e).\]
        
        \end{enumerate}
        \begin{enumerate}
        \item[Case 1:]  $T_1=\mathcal{X}$. Set $\psi_T(\text{Fermion}(v_T))=1$, and for any other fermion edge $e$ in $T$ define $\psi_T(e)$ as
        
        \begin{enumerate}
            \item $\psi_T(e)=\psi_{T_2}(e)+1$ if $ e\in T_2$ and $\psi_{T_2}(e)< \psi_{T_2}(\text{Fermion}(d))$.
            
            \item $\psi_T(\text{Fermion}(d))=\psi_{T_2}(\text{Fermion}(d))+1$ \quad (in $T$ this is the $dv_T$ edge).

            \item $\psi_T(\text{Fermion}(w_T))=\psi_{T_2}(\text{Fermion}(d))+2$.
            
            \item $\psi_T(e)=\psi_{T_2}(e)+2$ if $e\in T_2$.
            \end{enumerate}
        
        \item[Case 2:] $T_1\neq\mathcal{X}$. Set $\psi_T(\text{Fermion}(v_T))=1$, and for any other fermion edge $e$ in $T$ define $\psi_T(e)$ as
        
        \begin{enumerate}
            \item $\psi_T(e)=\psi_{T_2}(e)+1$ if $ e\in T_2$ and $\psi_{T_2}(e)< \psi_{T_2}(\text{Fermion}(d))$.
            
            \item $\psi_T(\text{Fermion}(d))=\psi_{T_2}(\text{Fermion}(d))+1$ \quad (in $T$ this is the $dw_2$ edge).

            \item $\psi_T(\text{Fermion}(w_T))=\psi_{T_1}(\text{Fermion}(w_1))+\psi_{T_2}(\text{Fermion}(d))$.

            \item $\psi_T(e)=\psi_{T_1}(e)+\psi_{T_2}(\text{Fermion}(d))$ \quad if $e\in T_1$. 
            
            \item $\psi_T(\text{Fermion}(w_d))=M+\psi_{T_2}(\text{Fermion}(d))+1$.
            
            \item $\psi_T(e)=\psi_{T_2}(e)+M+1$ if $e\in T_2$.
            
        \end{enumerate}
    \end{enumerate}
\end{itemize}

\end{dfn}

\begin{exm}\label{psiorderexm}
Two examples of the $\Psi$-order are given in Figures \ref{psiorderexm} and \ref{psiorderexm2}. In Figure \ref{psiorderexm3} we give the corresponding chord diagram for the tadpole in Figure \ref{psiorderexm2}. It is worth noticing how the two orders are constructed similarly.

\begin{figure}[h!]
    \centering
    \includegraphics[scale=0.65]{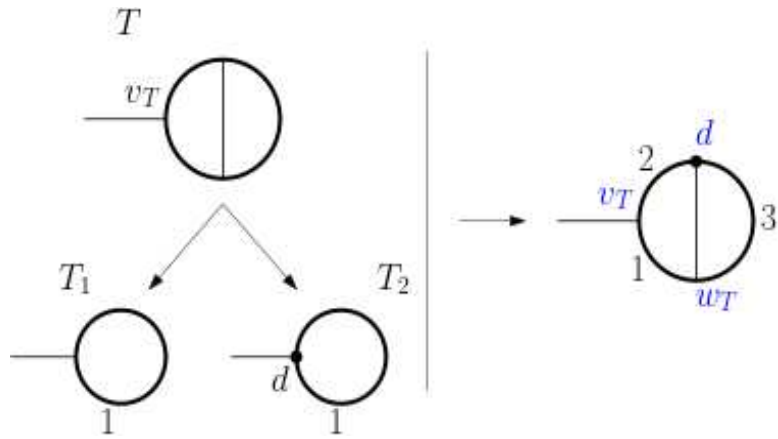}
    \caption{An example of the $\Psi$-order Case 1.}
    \label{psiorderexm44}
\end{figure}

\begin{figure}[h!]
    \centering
    \includegraphics[scale=0.85]{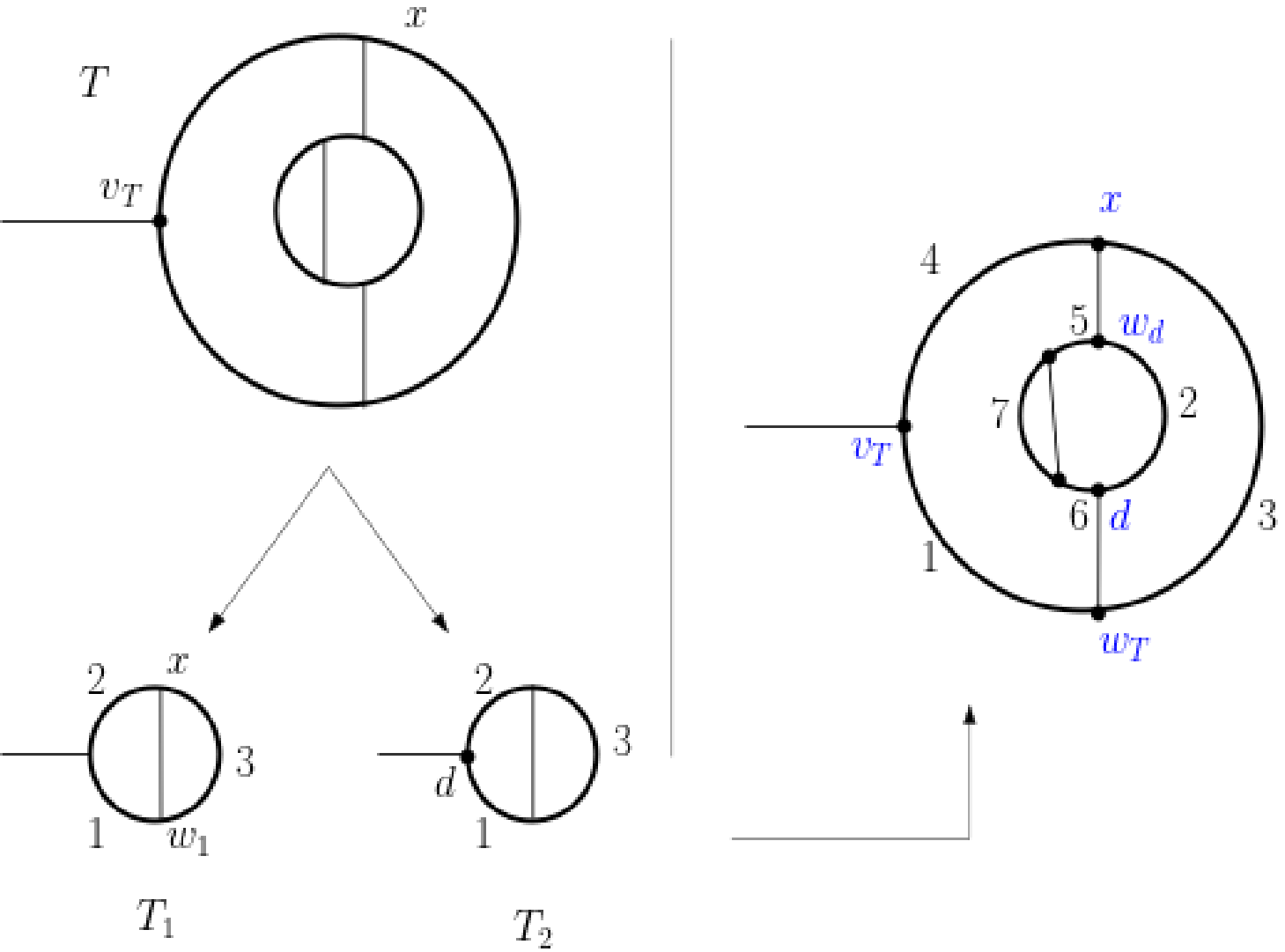}
    \caption{An example of the $\Psi$-order Case 2.}
    \label{psiorderexm2}
\end{figure}
\begin{figure}[H]
    \centering
    \includegraphics[scale=0.75]{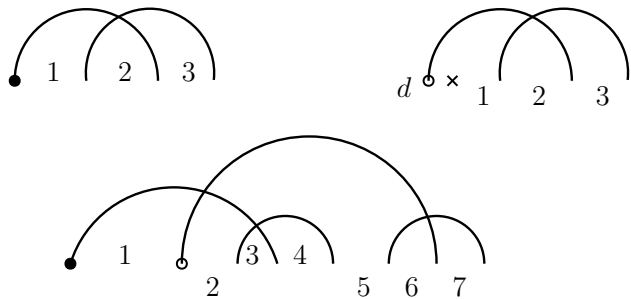}
    \caption{The corresponding chord diagrams for the graphs of Example \ref{psiorderexm}, in Figure \ref{psiorderexm2}.}
    \label{psiorderexm3}
\end{figure}

\end{exm}

Now we can finally express the bijection between Yukawa 1PI tadpoles and connected chord diagrams.

If we use a vertex to indicate an interval in a chord diagram, then we mean, as usual, the interval to the right of the vertex in the linear representation. Analogously we use the fermion edge that comes next to a vertex in counter-clockwise direction. Also note that the interval in the root share decomposition can not be the rightmost interval of the diagram. 

\subsubsection{An Explicit Bijection}

\begin{cor}\label{myresult2inYukawa}
Theorem \ref{myresult1inYukawa} can be used to give an explicit bijection $\Lambda$ between Yukawa theory 1PI tadpoles in (the class $\mathcal{U}_{10}$) and \;$\mathcal{C}$, the class of connected chord diagrams. Namely, $\Lambda: \mathcal{U}_{10}\longrightarrow \mathcal{C}$ is defined recursively as follows:

    \[\Lambda\big(\mathcal{X}\big)\;=\;\raisebox{0cm}{\includegraphics[scale=0.5]{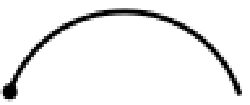}}\;;\qquad\text{and}\quad\]
    \[
    \Lambda(T)=\nabla^{-1}\big(\Lambda(T_1),(\Lambda(T_2),\psi(d)\big),
    \]
    
 where $\Psi^{-1}(T)=(T_1,(T_2,d))$, and, as defined earlier, $\mathcal{X}=\raisebox{-0.25cm}{\includegraphics[scale=0.5]{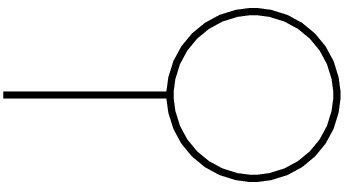}}$.
\end{cor}

\begin{proof}
The proof is straightforward from the definitions of the maps involved. The map is well defined by the uniqueness of the root share decomposition, and is a bijection since $\psi$, $\Psi$, and $\nabla$ are bijections. 
\end{proof}

\begin{exm}\label{completeexmpsi}
Figure \ref{complete} illustrates all the steps from a tadpole $T$ to its corresponding connected chord diagram. Namely, as the corollary states, given a Yukawa 1PI tadpole $T$, the process can be described as follows:
\begin{enumerate}
    \item Use $\Psi^{-1}$ to decompose $T$ all the way down to copies of $\mathcal{X}$, with extra information about positions $d_i$ at each step
    \item Go up again step by step and use the recursive definition of $\psi$ to order all the fermion edges in the graph.
    \item Again start from the bottom to create the corresponding chord diagrams using $\nabla^{-1}$ and the values $\psi(d)$. The last insertion up gives $\Lambda(T)$.
\end{enumerate}

\begin{figure}[p]
    \centering
    \includegraphics[scale=0.8]{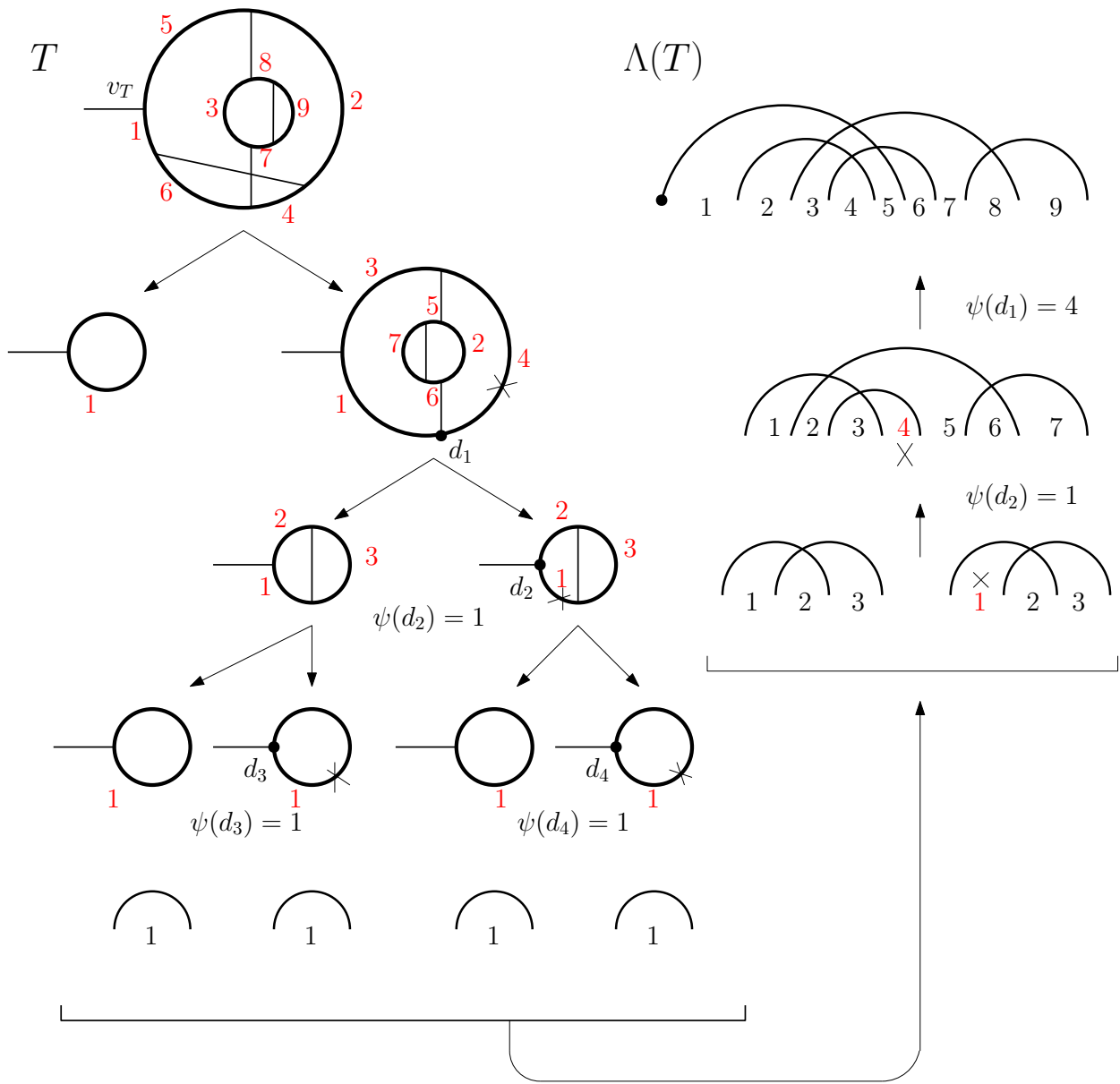}
    \caption{A complete example of the recursive calculation of $\Lambda$.}
    \label{complete}
\end{figure}
\end{exm}

\begin{rem}
It is surprising that, in light of the result in \cite{con}, the bijection between Yukawa 1PI tadpoles and connected chord diagrams gives a bijection between Yukawa 1PI tadpoles and rooted bridgeless combinatorial maps. This will be investigated in future work.
\end{rem}

\newpage\subsection{Yukawa Vacuum Graphs:\quad $\left.\partial_{\phi_c}^0(\partial_{\psi_c}\partial_{\bar{\psi}_c})^0G^{\text{Yuk}}(\hbar,\phi_c,\psi_c)\right|_{\phi_c=\psi_c=0}$}

Here we interpret line 1 in Table \ref{table4}. By definition, $\left.\partial_{\phi_c}^0(\partial_{\psi_c}\partial_{\bar{\psi}_c})^0G^{\text{Yuk}}(\hbar,\phi_c,\psi_c)\right|_{\phi_c=\psi_c=0}$ generates all 1PI Yukawa graphs with no external legs. In physics jargon these are called vacuum graphs. 

\begin{figure}[H]
    \centering
    \includegraphics[scale=0.75]{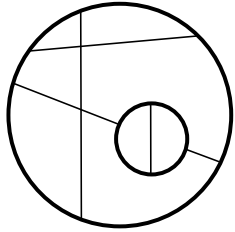}
    \caption{A 1PI vacuum graph in Yukawa theory.}
    \label{vaccc}
\end{figure}

By Lemma \ref{property1Yukawa} we know that for a vaccuum graph $\Gamma$ we will still have $|V(\Gamma)|=f$, the number of internal fermion edges. Consequently, we also still get $p=\ell(\Gamma)-1$ and $|V(\Gamma)|=2p$  (Euler's formula for the first), where $p$ is the number of internal boson edges. Let $V(x)$ be the generating series of 1PI Yukawa vacuum graphs counted by the number of boson edges.

Before giving the chord-diagrammatic interpretation, two things should be noticed. First, in Table \ref{table4} we have $$[\hbar^0]\left.\partial_{\phi_c}^0(\partial_{\psi_c}\partial_{\bar{\psi}_c})^0G^{\text{Yuk}}(\hbar,\phi_c,\psi_c)\right|_{\phi_c=\psi_c=0}=[\hbar^1]\left.\partial_{\phi_c}^0(\partial_{\psi_c}\partial_{\bar{\psi}_c})^0G^{\text{Yuk}}(\hbar,\phi_c,\psi_c)\right|_{\phi_c=\psi_c=0}=0,$$ since we do not consider the empty graph or the plain loop to be 1PI graphs. Second, the entries for the higher powers of $\hbar$ in Table \ref{table4} seem to coincide with the coefficients of \[\displaystyle\frac{C(x)^2}{2x}\]
(see Table \ref{table2}). The following proposition proves this conjecture and is a consequence of Theorem \ref{myresult1inYukawa}.

\begin{prop}\label{myresultykawavacuum}
 Let $V(x)$ be the generating series of 1PI Yukawa vacuum graphs counted by the number of boson edges. Then 
 \[V(x)=\displaystyle\frac{C(x)^2}{2x},\]
which implies that \[[\hbar^{n+1}]\left.\partial_{\phi_c}^0(\partial_{\psi_c}\partial_{\bar{\psi}_c})^0G^{\text{Yuk}}(\hbar,\phi_c,\psi_c)\right|_{\phi_c=\psi_c=0}=[x^n]\displaystyle\frac{C(x)^2}{2x}.\]
\end{prop}
\begin{proof}
Let $\mathcal{U}_{00}$ be the class of Yukawa 1PI vacuum graphs. It is clear that $\mathcal{U}_{10}-\{\mathcal{X}\}=X\ast\mathcal{U}_{00}^\bullet$, where $\mathcal{U}_{00}^\bullet$ is the class of vacuum graphs with a distinguished fermion edge (or equivalently, with a distinguished vertex), and $X$ as usual is used for the single constituent whose generating function is $x$ (in this case it refers to an external boson edge to be inserted). Also note that our vacuum graphs must have at least one boson edge (the plain fermion loop is not considered 1PI).

Since there are two vertices for each chosen boson edge, we have that the generating function for $\mathcal{U}_{00}^\bullet$ is 
\[2xV^\prime(x).\]
Thus, by Theorem \ref{myresult1inYukawa}, we should have 
\[C(x)-x=2x^2V^\prime(x).\] 
Then, by Lemma \ref{cd} we know that $2xCC'=C^2+C-x$, and we get
\[
  V^\prime(x)=\displaystyle\frac{1}{2x^2}(C(x)-x)
    =\displaystyle\frac{1}{2x^2}(2xC(x)C^\prime(x)-C(x)^2)
    =\displaystyle\frac{1}{2}\frac{(x(C(x)^2)^\prime-C(x)^2)}{x^2},
\]
which can then be integrated to give the result.
\end{proof}


\subsection{Yukawa Graphs from $\left.\partial_{\phi_c}^2(\partial_{\psi_c}\partial_{\bar{\psi}_c})^0G^{\text{Yuk}}(\hbar,\phi_c,\psi_c)\right|_{\phi_c=\psi_c=0}$}

The Yukawa 1PI graphs counted by loop number in line 3 of Table \ref{table4} are the graphs with exactly two external legs, each of which are boson-type. Again, for any such graph we have by Lemma \ref{property1Yukawa} and Euler's formula that $|V(\Gamma)|=f$, $p=\ell(\Gamma)-1$ and $|V(\Gamma)|=2p$;
 where $p$ is the number of internal boson edges and $f$ is the number of internal fermion edges. 
 
 By their definition, we see that these graphs are simply tadpoles with a distinguished fermion edge at which a \textit{second} external boson leg is inserted. 
 
 \begin{rem}\label{importremark}
 By the word `second' above we literally mean that the roles of the two boson edges are physically different. This will be reflected in that we will always assume that one boson leg is the `left' or `first' one. For example, the graphs 
 \begin{center}
     \raisebox{0cm}{\includegraphics[scale=0.87]{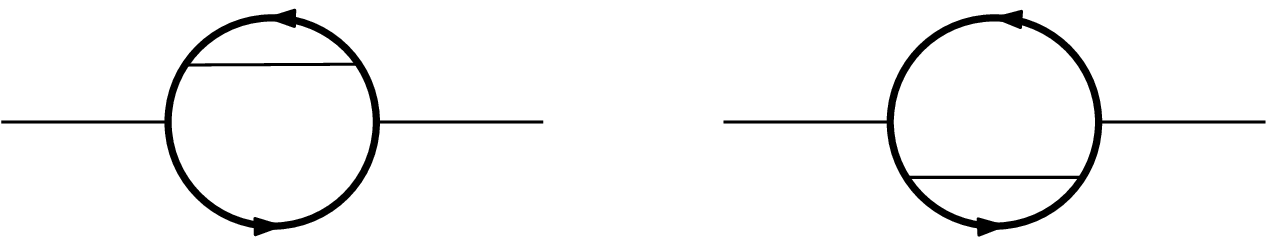}}
 \end{center}
 are considered different even though one of them can be rotated to get the other one.
 \end{rem}
 
 However, in  the process of distinguishing a fermion edge of a tadpole, we have to exclude the fermion edge immediately before the tadpole's leg vertex as it will yield the same graph if the next fermion edge is chosen instead. Thus the generating function of these graphs, counted by the number of all boson edges is given by
 \begin{equation}\label{myvacuumeq1}
     U_{20}(x)=x(2xT^\prime(x)-T(x))=x(2xC'(x)-C(x)),
 \end{equation} where $T$ is as in Theorem \ref{myresult1inYukawa}, the generating function for Yukawa 1PI tadpoles counted by the number of all boson edges. 
 \begin{figure}[H]
     \centering
     \includegraphics[scale=0.7]{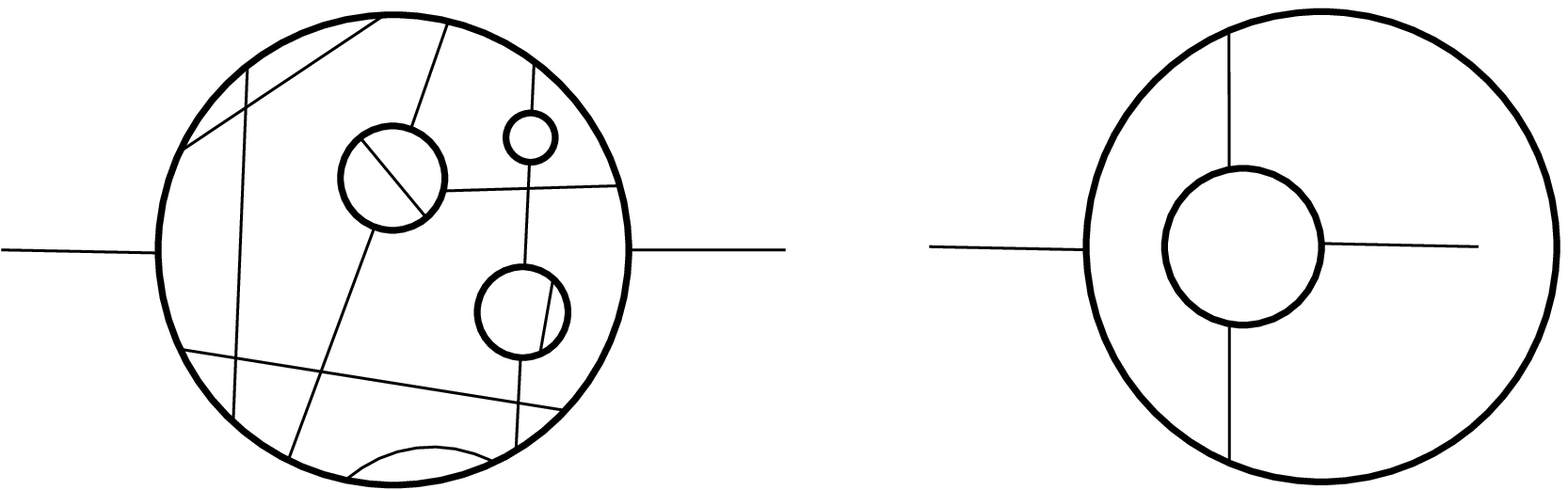}
     \caption{Two graphs generated by $\left.\partial_{\phi_c}^2(\partial_{\psi_c}\partial_{\bar{\psi}_c})^0G^{\text{Yuk}}(\hbar,\phi_c,\psi_c)\right|_{\phi_c=\psi_c=0}$.}
     \label{twoleg}
 \end{figure}
 
\begin{rem}
 Notice that the two boson legs are not necessarily on the same fermion loop, see the second graph in Figure \ref{twoleg} for example.
\end{rem}

Thus, equation \ref{myvacuumeq1} shows that 
\begin{align}\label{vacuumeq}
    [\hbar^{n}]\left.\partial_{\phi_c}^2(\partial_{\psi_c}\partial_{\bar{\psi}_c})^0G^{\text{Yuk}}(\hbar,\phi_c,\psi_c)\right|_{\phi_c=\psi_c=0}&=[x^{n+1}]U_{20}(x)=\;x(2xT^\prime(x)-T(x))\\&=[x^n](2xC'(x)-C(x)),
\end{align}
which is also verified for the first coefficients  by comparing Tables \ref{table1} and \ref{table4}. Further, by Proposition \ref{myproposition2connected}, it follows that  
\begin{align}\label{myequ1}
    [\hbar^{n}]\left.\partial_{\phi_c}^2(\partial_{\psi_c}\partial_{\bar{\psi}_c})^0G^{\text{Yuk}}(\hbar,\phi_c,\psi_c)\right|_{\phi_c=\psi_c=0}&=[x^n] \;\;\displaystyle\frac{C(x)^2}{x} \;\left[\left. \displaystyle\frac{C_{\geq2}(t)}{t^2}\right|_{t=C(x)^2/x}\right], \;\text{or equivalently}\nonumber\\
    U_{20}(x)&=C(x)^2 \;\left[\left. \displaystyle\frac{C_{\geq2}(t)}{t^2}\right|_{t=C(x)^2/x}\right].
\end{align}

This equation will be useful in providing a chord-diagrammatic interpretation for graphs generated by $\left.\partial_{\phi_c}^1(\partial_{\psi_c}\partial_{\bar{\psi}_c})^1G^{\text{Yuk}}(\hbar,\phi_c,\psi_c)\right|_{\phi_c=\psi_c=0}$, as we shall see next. For the next section we will need to give a name for the class of graphs considered here, and, as it became our good habit, we shall denote it by $\mathcal{U}_{20}$.

\subsection{Yukawa Graphs from $\left.\partial_{\phi_c}^1(\partial_{\psi_c}\partial_{\bar{\psi}_c})^1G^{\text{Yuk}}(\hbar,\phi_c,\psi_c)\right|_{\phi_c=\psi_c=0}$}
The graphs generated by $\left.\partial_{\phi_c}^1(\partial_{\psi_c}\partial_{\bar{\psi}_c})^1G^{\text{Yuk}}(\hbar,\phi_c,\psi_c)\right|_{\phi_c=\psi_c=0}$ are the Yukawa 1PI graphs with vertex-type residue. Line 5 in Table \ref{table4} gives the number of these graphs, sized with loop number, up to size 5. These have two external fermion legs in addition to one boson leg (see Figure \ref{exaaample1Y}).
\begin{figure}[H]
     \centering
     \includegraphics[scale=0.6]{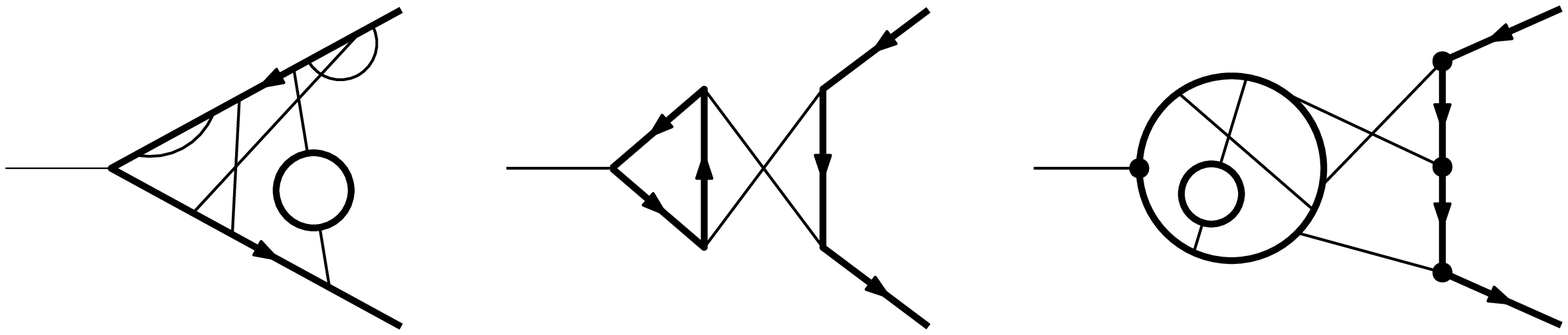}
     \caption{Examples of graphs generated by $\left.\partial_{\phi_c}^1(\partial_{\psi_c}\partial_{\bar{\psi}_c})^1G^{\text{Yuk}}(\hbar,\phi_c,\psi_c)\right|_{\phi_c=\psi_c=0}$.}
     \label{exaaample1Y}
 \end{figure}
Notice that if the ends of the two fermion external legs were identified we still wouldn't get a general tadpole. The reason for this is that, since the graph is  1PI, we can not have something like the one in Figure \ref{forboddengraph} below.
\begin{figure}[H]
    \centering
    \raisebox{0cm}{\includegraphics[scale=0.54]{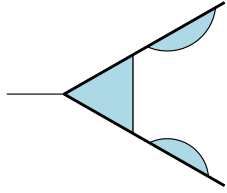}}
    
    \caption{A forbidden graph.}
    \label{forboddengraph}
\end{figure}
Knowing this is exactly the way we get the chord-diagrammatic interpretation. We will let $\mathcal{U}_{11}$ denote the class of al graphs generated by $\left.\partial_{\phi_c}^1(\partial_{\psi_c}\partial_{\bar{\psi}_c})^1G^{\text{Yuk}}(\hbar,\phi_c,\psi_c)\right|_{\phi_c=\psi_c=0}$. Before proceeding to the next theorem recall that by Lemma \ref{property1Yukawa} we have that for any $\Gamma\in\mathcal{U}_{11}$ 
$|V(\Gamma)|=f+1$, and hence by Euler's formula $p=\ell(\Gamma)$, where $f$ (and $p$) is the number of internal fermion (boson) edges. 

\begin{nota}\label{alinota1}
In representing the graphs in $\mathcal{U}_{11}$, we still stick to the counter-clockwise convention, even for the unique path of fermion edges. We shall always represent the graphs in $\mathcal{U}_{11}$ by fixing the boson external leg to the left, and then the fermion external half-edge directed towards the boson-leg vertex will be called the \textit{upper end} and will be denoted with $u_\Gamma$; on the other hand the fermion external half-edge directed away from the boson-leg vertex will be called the \textit{lower end} and will be denoted by $l_\Gamma$.
\end{nota}
\begin{center}
\raisebox{0cm}{\includegraphics[scale=0.6]{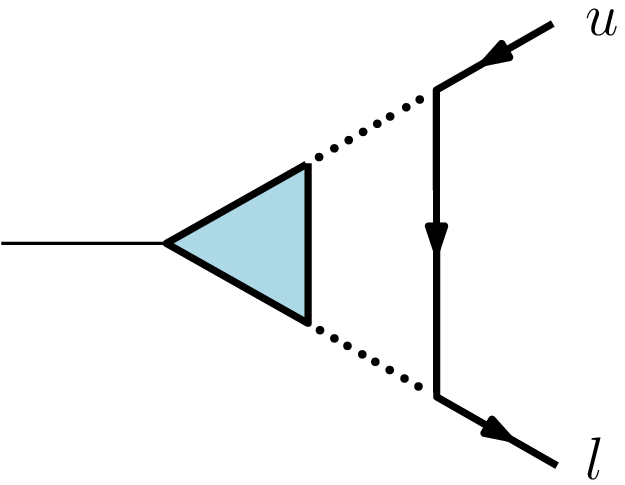}}
\end{center}

\begin{thm}
\label{mypropyukawa3}\label{ali}
 Let $\mathcal{U}_{11}$ be the class of Yukawa 1PI graphs (with vertex-type residue) generated by $\left.\partial_{\phi_c}^1(\partial_{\psi_c}\partial_{\bar{\psi}_c})^1G^{\text{Yuk}}(\hbar,\phi_c,\psi_c)\right|_{\phi_c=\psi_c=0}$, and let $U_{11}(x)$ be the generating series of these graphs, counted by the number of all boson edges. Then  
 
 \[U_{11}(x)=x\;\left. \displaystyle\frac{C_{\geq2}(t)}{t^2}\right|_{t=C(x)^2/x}.\]
\end{thm}

\begin{proof}
We start with another type of graphs, namely with the class $\mathcal{U}_{20}$ of the previous section. We will describe a bijection 
\[M:\mathcal{U}_{11}\ast (\mathcal{U}_{10}\ast\mathcal{U}_{10})\longrightarrow \mathcal{X}\ast\mathcal{U}_{20}.\]
The construction is simple:
Assume $(\Gamma,(T_1,T_2))$ is a triplet from $\mathcal{U}_{11}\ast (\mathcal{U}_{10}\ast\mathcal{U}_{10})$. Let $u$ and $l$ be the upper and lower ends of $\Gamma$ as described in Notation \ref{alinota1}. Now take the tadpoles $T_1$ and $T_2$ and do the following:
\begin{enumerate}
    \item For $T_1$, let $v_1$ be the vertex at the boson leg as before. Let $f_1$ be the fermion edge immediately before $v_1$ on Loop$(v_1)$. Detach $f_1$ from $v_1$ and denote the unique resulting graph with $u(T_1)$. This can be depicted as in Figure \ref{bell1fig} below.
    \begin{figure}[H]
        \centering
       \raisebox{0cm}{\includegraphics[scale=0.73]{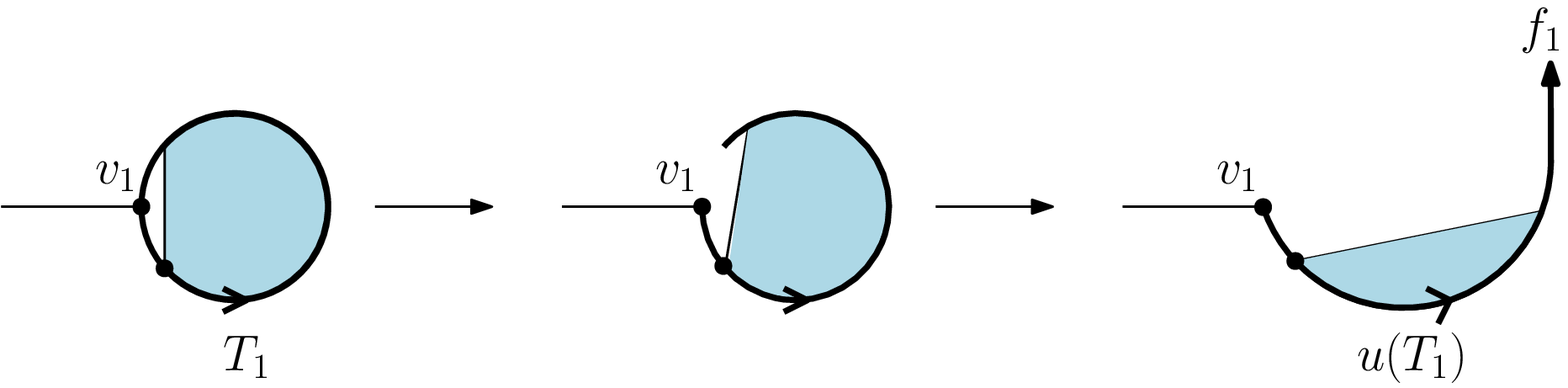}}
       \caption{$u(T_1)$}
        \label{bell1fig}
    \end{figure}
    \item For $T_2$, let $v_2$ be the vertex at the boson leg as before. Let $f_2$ be Fermion$(v_2)$, the fermion edge immediately next to $v_2$ on Loop$(v_2)$. Detach $f_2$ from $v_2$ and denote the unique resulting graph with $l(T_2)$. This can be depicted as in Figure \ref{bell2fig} below.
    \begin{figure}[H]
        \centering
       \raisebox{0cm}{\includegraphics[scale=0.73]{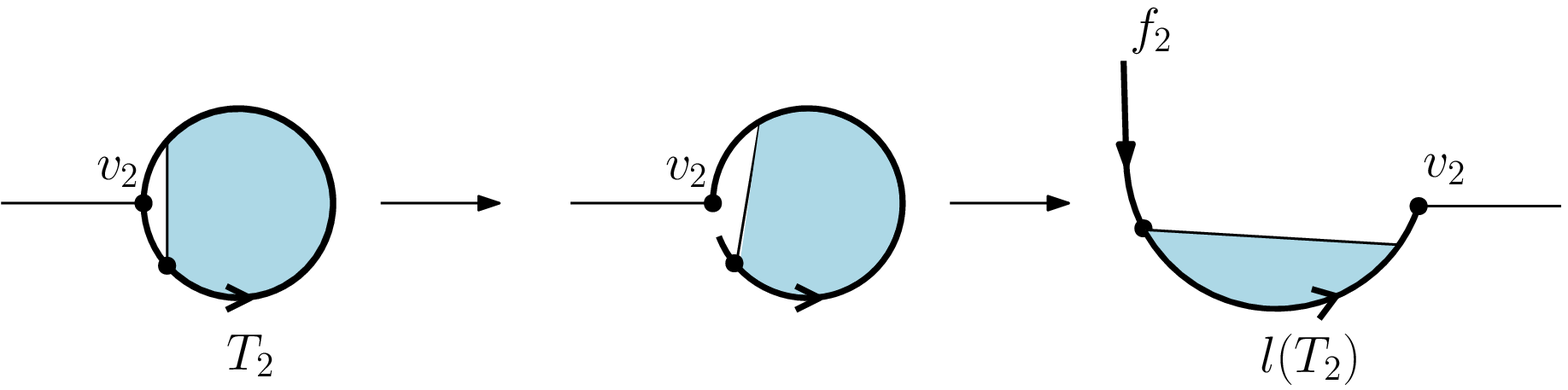}}
       \caption{$l(T_2)$}
        \label{bell2fig}
    \end{figure}
    
    \item Identify $f_1$ with $u$ of $\Gamma$ and identify $f_2$ with $l$. Denote the resulting graph with $\tilde{\Gamma}$
    \item Identify the vertices $v_1$ and $v_2$ in $\tilde{\Gamma}$.
    \item By removing one of the external boson legs we get $M(\Gamma, T_1,T_2)$.
    The process described here can be depicted as in Figure \ref{figm} below.
    \begin{figure}[H]
        \centering
       \raisebox{0cm}{\includegraphics[scale=0.7]{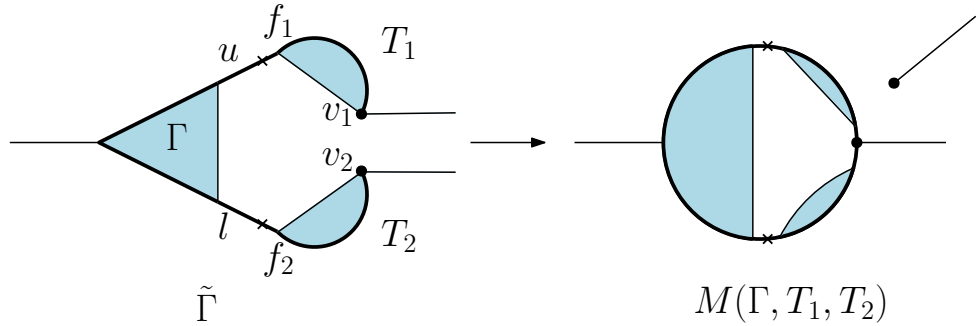}}
       \caption{$M(\Gamma, T_1,T_2)$}
        \label{figm}
    \end{figure}
\end{enumerate}

By our construction, and by the uniqueness of the representation of graphs described in Notation \ref{alinota1}, the map $M$ described this way is well-defined. Moreover, it is reversible: Assume we are given a graph $G$ in $\mathcal{U}_{20}$ in canonical representation (i.e. drawn according to our counter-clockwise convention). Remember from Remark \ref{importremark}, that the representation of $G$ is unique and one boson leg is the `left' external leg. Let $w_l$ be and $w_r$ be the corresponding vertices at the left and right  boson legs respectively. Let us refer to the half fermion loop from $w_r$ to $w_l$ by the \textit{upper half loop}, and similarly we will refer to the half fermion loop from $w_l$ to $w_r$ by the \textit{lower half loop} (see Figure \ref{notaM}).  Now we do the following:
\begin{enumerate}
    \item Starting from $w_r$, remove Fermion$(w_r)$ from the graph $G$, and move along the rest of the upper half loop (direction is as before) searching for the \underline{\textbf{last}} fermion edge whose removal disconnects $G-$Fermion$(w_r)$. Stop the search at $w_l$. Notice that it may happen that no such edge exists. If found, denote this edge by $f_1$. 
    If such an edge doesn't exist we set $f_1=$ Fermion$(w_r)$.  
    
    \item Similarly, starting from $w_l$, remove Fermion$(w_l)$ from the graph $G$, and move along the rest of the lower half loop searching for the \underline{\textbf{first}} fermion edge whose removal disconnects $G-$Fermion$(w_l)$. Stop the search at $w_r$. Notice that it may happen that no such edge exists. If found, denote this edge by $f_2$. 
    If such an edge doesn't exist we set $f_2 =$ (fermion edge immediately before $w_r$).
    
    \item Cut at $f_1$ and $f_2$, and identify their ends with $w_r$ such that the direction of $f_1$ with respect to $w_r$ is counter-clockwise, whereas the direction of $f_2$ with respect to $w_r$ is to be made clockwise.
    
    \item Obtain $T_1$ and $T_2$ by splitting $w_r$ and its boson leg into two copies. The remnant of $G$ is $\Gamma$. 
    
    \begin{figure}[H]
        \centering
       \raisebox{0cm}{\includegraphics[scale=0.77]{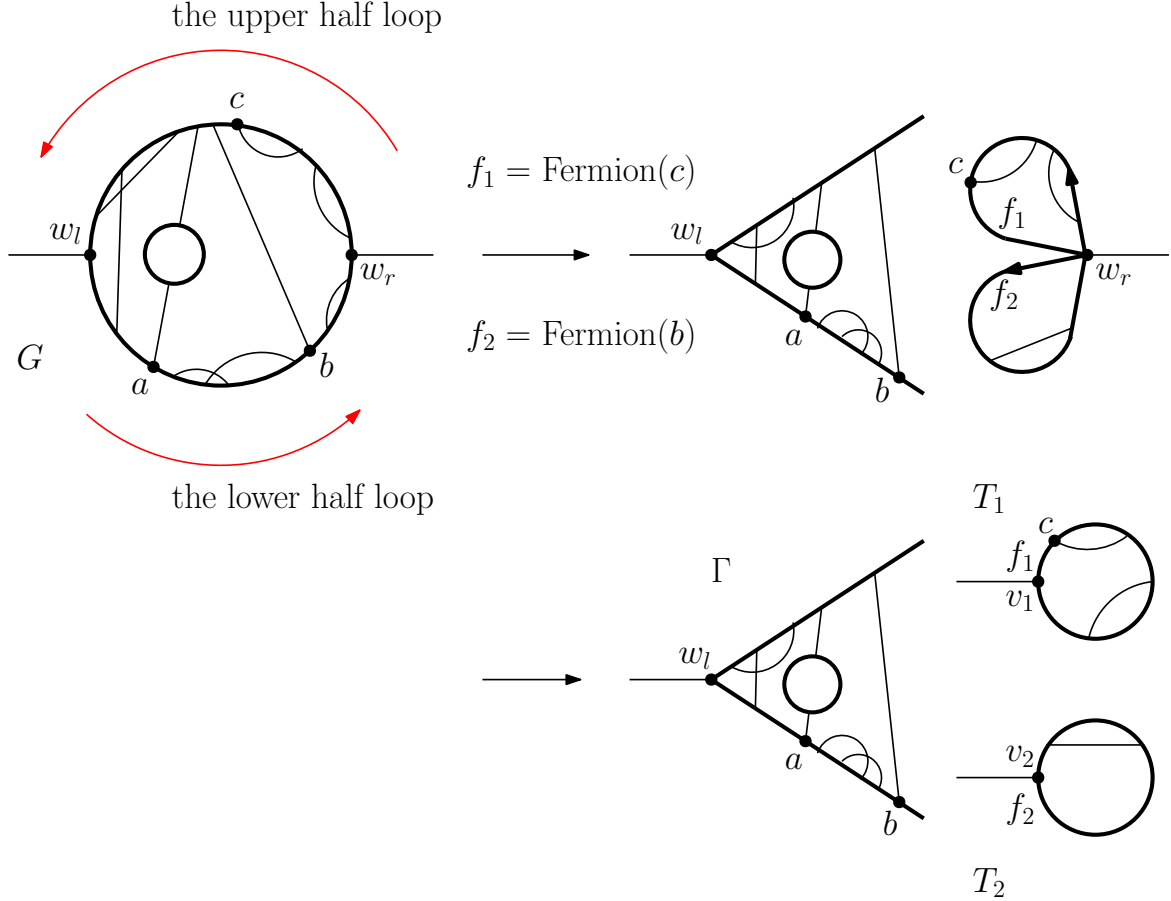}}
       \caption{Calculating $M^{-1}(G)$.}
        \label{notaM}
    \end{figure}
\end{enumerate}

It is worth noting that in Figure \ref{notaM} $f_2$ is Fermion$(b)$ and is not Fermion$(a)$ as can be checked using the definition.

Thus, the map $M:\mathcal{U}_{11}\ast (\mathcal{U}_{10}\ast\mathcal{U}_{10})\longrightarrow \mathcal{X}\ast\mathcal{U}_{20}$ is a bijection. Consequently, on the level of generating functions we will have 
\[
    U_{11}(x) T(x)^2= x \;U_{20}(x),
\]

and hence, by equation (\ref{myequ1}) and Theorem \ref{myresult1inYukawa}, we get  
\[U_{11}(x) C(x)^2=x\;C(x)^2 \;\left[\left. \displaystyle\frac{C_{\geq2}(t)}{t^2}\right|_{t=C(x)^2/x}\right],\]
that is,
\begin{equation}\label{myeq2}
    U_{11}(x)=\left. \displaystyle\frac{C_{\geq2}(t)}{t^2}\right|_{t=C(x)^2/x}, 
\end{equation}
as desired. 
\end{proof}

\begin{rem}\label{remarkafter3legged}

Recall that, from our work in Section \ref{functionalrecurrence2connected}, the RHS of equation (\ref{myeq2}) is the generating function for connected chord diagrams with the extra properties that
\begin{enumerate}
    \item neither the root chord nor the chord of the last end vertex (in the linear order) are cuts, and
    \item neither the root nor the last end vertex (in the linear order) are contained in any reason of connectivity-1.
\end{enumerate}

It is still not very clear what is the direct way to manifest this structure in the world of the graphs in $\mathcal{U}_{11}$, that is, the graphs generated by $\left.\partial_{\phi_c}^1(\partial_{\psi_c}\partial_{\bar{\psi}_c})^1G^{\text{Yuk}}(\hbar,\phi_c,\psi_c)\right|_{\phi_c=\psi_c=0}$.
\end{rem}

Compare Tables \ref{table2} and \ref{table4} for a numeric verification of Theorem \ref{ali} for the first few terms.

\subsection{Yukawa  Graphs from $\left.\partial_{\phi_c}^0(\partial_{\psi_c}\partial_{\bar{\psi}_c})^1G^{\text{Yuk}}(\hbar,\phi_c,\psi_c)\right|_{\phi_c=\psi_c=0}$}

Let $\mathcal{U}_{01}$ be the class of Yukawa 1PI graphs generated by $\left.\partial_{\phi_c}^0(\partial_{\psi_c}\partial_{\bar{\psi}_c})^1G^{\text{Yuk}}(\hbar,\phi_c,\psi_c)\right|_{\phi_c=\psi_c=0}$. Line 4 in Table \ref{table4} gives the number of these graphs, sized with loop number, up to size 5. These are graphs with two external fermion legs and with no  boson leg. 

 For any $\Gamma\in\mathcal{U}_{01}$, we have by Lemma \ref{property1Yukawa} that
$|V(\Gamma)|=f+1$, and hence by Euler's formula $p=\ell(\Gamma)$, where $f$ (and $p$) is the number of internal fermion (boson) edges. In particular, 
\begin{equation}\label{numinternalfermion}
    f=2p-1.
\end{equation}
We let $U_{01}(x)$ be the generating function of the graphs in $\mathcal{U}_{01}$ counted by the number of boson edges.

There are two ways to think about this type of graphs, both can be used to obtain $U_{01}(x)$:
\begin{enumerate}
    \item It is clear that the graphs in $\mathcal{U}_{11}$, which were considered in the last section, are the rooted versions of the graphs in $\mathcal{U}_{01}$, namely,  except for the one-vertex graph, every graph in $\mathcal{U}_{11}$  is obtained from a unique graph in $\mathcal{U}_{01}$ by distinguishing an internal fermion edge and inserting a boson leg. This means that, on the level of generating functions:

\begin{equation}\label{letsdoit}
    U_{11}(x)=x\;(2x\;U_{01}^\prime(x)-U_{01}(x)+1),
\end{equation}where the RHS uses the fact that $f=2p-1$ by equation (\ref{numinternalfermion}), and where the $1$ corresponds to the one-vertex graph graph in $\mathcal{U}_{11}$.

\item The other way is to think of a tadpole as being constructed from a list of graphs from $\mathcal{U}_{01}$. This way is more direct and we shall discuss it below.
\end{enumerate}

\begin{prop}
 A Yukawa 1PI tadpole graph can be decomposed as  boson leg together with a list of graphs from $\mathcal{U}_{01}$. In particular, on the level of generating functions we will have
 \begin{equation}\label{tobeusd}
     T(x)=\displaystyle\frac{x}{1-U_{01}(x)}.
 \end{equation}\label{tobeused2}
 This in turn implies that \begin{equation}
     U_{01}(x)=2xC^\prime(x)-C(x)=C(x)^2 \;\left[\left. \displaystyle\frac{C_{\geq2}(t)}{t^2}\right|_{t=C(x)^2/x}\right]=U_{20}(x).
 \end{equation}
\end{prop}

 \begin{proof}
 Given a tadpole $T\in\mathcal{U}_{10}$, we will give a unique decomposition into a boson leg together with  a list of graphs from $\mathcal{U}_{01}$. Let $v_T$ be the vertex at the boson external leg of $T$ as usual. Also let $f_0$ be the fermion edge on Loop$(v_T)$ immediately before $v_T$. 
 
 \begin{enumerate}
     \item If $f_0=$ Fermion$(v_T)$, return $(\raisebox{-0.cm}{\raisebox{-0.05cm}{\includegraphics[scale=0.5]{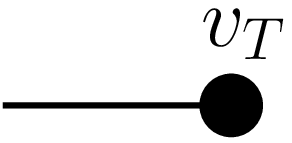}}};\varnothing)$.
     \item  Otherwise, detach $f_0$ and call the resulting graph $u(T)$. \item Starting from $v_T$, move along Loop$(v_T)$ and determine all fermion edges on Loop$(v_T)$ whose removal disconnects $u(T)$. Let this list of fermion edges be 
     \[f_1,f_2,\ldots,f_n.\]
     (Note that it will always be the case that $f_1=$ Fermion$(v_T)$).
     \item For $1\leq i\leq n-1$ define $G_i$ to be the graph obtained from the original $T$ by 
     
     \begin{itemize}
         \item cutting at $f_i$ and $f_{i+1}$, and
         \item deleting the component that contains $v_T$.
     \end{itemize}
     
     \item For $i=n$, define $G_n$ to be the graph obtained from the original $T$ by 
     
     \begin{itemize}
         \item cutting at $f_n$ and $f_{0}$, and
         \item deleting the component that contains $v_T$.
     \end{itemize}
     \item Return $(\raisebox{-0.05cm}{\includegraphics[scale=0.5]{Figures/vt.eps}};G_1,G_2,\ldots,G_n)$. 
 \end{enumerate}
 
 It is easily seen that the graphs $G_i$ are 2-edge connected, in fact the graphs $G_i$ are maximal 1PI's inserted along Loop$(v_T)$. Besides, by their definition, the graphs $G_i$ have a fermion-type residue and are therefore in $\mathcal{U}_{01}$. Moreover, this construction is clearly unique for every tadpole in  $\mathcal{U}_{10}$. 
 
 Conversely, every such list can be uniquely used to produce a tadpole. Figure \ref{Gs} below illustrates two cases of this decomposition.
 \begin{figure}[H]
        \centering
       \raisebox{0cm}{\includegraphics[scale=0.8]{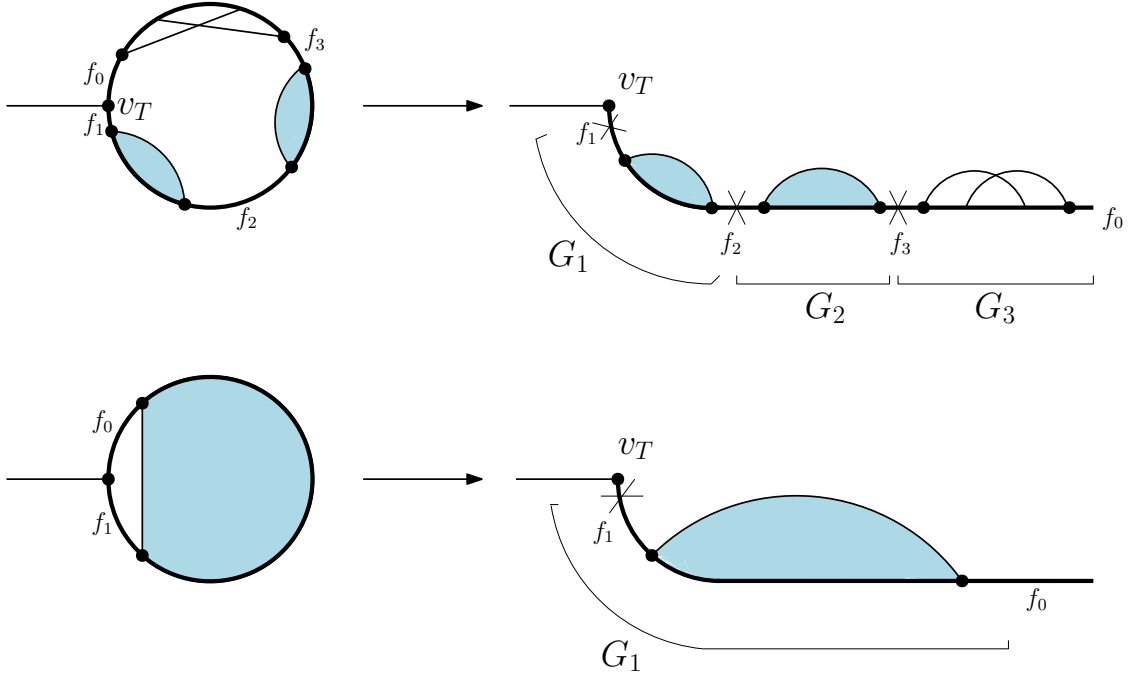}}
       \caption{Examples of the decomposition of tadpoles into a list of graphs from $\mathcal{U}_{01}$.}
        \label{Gs}
    \end{figure}

This gives that, on the level of generating functions,  
  \[T(x)=\displaystyle\frac{x}{1-U_{01}(x)}.\]
 
 By Lemma \ref{cd} recall that $C(x)=\frac{x}{1-(2xC'(x)-C(x))}$, and since we have $T(x)=C(x)$ we get 
     \[U_{01}(x)=2xC^\prime(x)-C(x).\]
 Then it follows from Proposition \ref{myproposition2connected} that  
     
     \[U_{01}(x)=C(x)^2 \;\left[\left. \displaystyle\frac{C_{\geq2}(t)}{t^2}\right|_{t=C(x)^2/x}\right]=U_{20}(x).\]
 \end{proof}

\textbf{Conclusion:} In this section we have obtained a chord-diagrammatic interpretation for a number of proper Green functions in Yukawa theorey and quenched QED. In \cite{michiq}, the asymptotics of these generating series are obtained by means of some sort of singularity analysis. Here, having obtained every generating series in terms of connected chord diagrams, the task of obtaining the asymptotics is straightforward. Indeed, we only need to use our knowledge of $\mathcal{A}_{\frac{1}{2}}^2C(x)$ (and $\mathcal{A}_{\frac{1}{2}}^2C_{\geq2}(x)$) to obtain the asymptotics for the different Yukawa and QQED green functions considered here.

\chapter{Diffeomorphisms of Scalar Quantum Fields via Generating Functions}\label{difchapter}
This chapter addresses a problem independent from the context of chord diagrams, and was actually the first problem to be considered into this PhD project. Remember that a free scalar quantum field is defined via a Lagrangian (density) 
\[
L(\phi) = \frac{1}{2}\partial_\mu \phi(x)\partial^\mu \phi(x) - \frac{m^2}{2}\phi^2(x)
\]

that contains no self-interaction terms, where $m$ is the mass of the $\phi$-particle. A \textit{field diffeomorphism} $F$ is going to be formally defined as a power series in the field
\[\phi\mapsto F(\phi)=a_0\phi(x)+a_1\phi(x)^2+\cdots \;= \sum_{j=0}^\infty a_j \phi^{j+1},\]

where $a_0=1$, i.e. $F$ is tangent to the identity. The problem is then to study the field theory expressed by the transformed Lagrangian, if one applies the diffeomorphism to the Lagrangian equation above. The result is seemingly an interacting theory. 

In classical field theory this is merely a canonical transformation that does not change the Poisson brackets \cite{paulrefers}, and it simply relates theories with different Lagrangians. However, for quantum fields, there are some ambiguities probably due to operator ordering in the path-integral formulation and the topic is therefore controversial \cite{diff1,diff2,diff3,diff4,diff5,diff6,diff7}.

\section{Motivation and Prior Work}

The approach followed in \cite{kreimer2} and \cite{karendiffeo} is a `least-action' approach:  they study field diffeomorphisms order by order in perturbation theory. In \cite{kreimer2} D. Kreimer and A. Velenich showed by direct calculations that, up to six external legs, interacting tree-level amplitudes do vanish. Yet, it was not still known how this can be generalized to higher orders. Note that the vanishing of tree-level amplitudes is crucial as it leads to the vanishing of loop amplitudes. In \cite{karendiffeo} Karen Yeats and Dirk Kreimer proved that if a point field diffeomorphism $\phi(x) \mapsto \sum_{j\geq 0} a_j \phi^{j+1}$ is applied to a free scalar field theory, the resulting field theory, while it appears to have many interaction terms, in fact remains a free theory by appropriate cancellations between diagrams. 

At tree level these cancellations hold whenever the external edges are on-shell, while at loop level they additionally require renormalization with a kinematical renormalization scheme.  This work followed up on the observations by Kreimer and Velenich \cite{kreimer2}.

The arguments of \cite{karendiffeo} proceeded first to reduce the tree level problem to a purely combinatorial problem of proving certain combinatorial identities. These were then proved using Bell polynomials.  Then the loop level results were bootstrapped off the tree level results.

However, the proofs of \cite{karendiffeo}, even at tree level, were unsatisfying as they were both opaque and intricate, consisting of delicate Bell polynomial manipulations which needed to reach fairly deeply into the repertoire of known Bell polynomial identities without obtaining insight.  The authors in \cite{karendiffeo} conjectured that a proof on the level of generating functions could be possible, and could give better insight, especially given the fact that Bell polynomials come from series composition. 

\textbf{Our Contribution:}
This is what we do in this chapter, reproving the tree level cancellations of \cite{karendiffeo} at the level of generating functions, and then leveraging the extra insight gained to see exactly how the solution appears as a compositional inverse, and making an explicit connection with the combinatorial Legendre transform of Jackson, Kempf, and Morales.  The latter is particularly interesting because of the role of the on-shell condition in the outcome of the combinatorial Legendre transform in our situation.  We show that the series whose coefficients are the tree-level amplitudes for the transformed theory is exactly the compositional inverse of the diffeomorphism $F$ applied. Additionally, along the way we give new combinatorial proofs of some Bell polynomial identities due to Cvijovi\'c (see \cite{cvijovic8}).

Note that ultimately every problem considered in this chapter is purely combinatorial.  

In terms of more physical considerations, note that no appeal to the path integral or its measure is used in \cite{karendiffeo} nor here.  All results are proven by rigorous arguments at the diagram level.  Consequently these results are ground truth, and the correct transformations for the path integral and path integral measure can be reverse engineered from them.  That a field diffeomorphism ought to pass nicely though the path integral is often viewed as a near triviality, though others have argued that in fact it does not (see \cite{diffeonot}). Different lines of thought can also be seen in \cite{diff1,diff2,diff3,diff4,diff5,diff6,diff7}.  Settling this rigourously while side stepping the path integral entirely was a major motivation for \cite{karendiffeo} as well as for us here.

Paul Balduf independently arrived at the fact that the series whose coefficients are the tree-level amplitudes is the compositional inverse of the diffeomorphism through analyzing the $S$-matrix \cite{Paul, Paulthesis}. Our method was used in the proof of Theorem 3.2 in \cite{Paul}.

\section{Field Theory Set-up}
Let $F$ be a field diffeomorphism $F:\phi\mapsto F(\phi)=\overset{\infty}{\underset{j=0}\sum} a_j\phi^{j+1}$ (with $a_0=1$). When $F$ is applied to a free field $\phi(x)$ with Lagrangian density \[L(\phi)=\frac{1}{2}\partial_\mu\phi(x)\partial^\mu\phi(x)-\frac{m^2}{2}\phi^2(x),\]
it gives the new Lagrangian \[L_F(\phi)=\frac{1}{2}\partial_\mu F(\phi)\partial^\mu F(\phi)-\frac{m^2}{2}F(\phi)F(\phi),\] where the field $\phi$ is a scalar field from the 4-dimensional Minkowski space-time ($\phi:\mathbb{R}^4\rightarrow\mathbb{R}$).
Expanding out the transformed Lagrangian we obtain
\begin{align*}
L_F(\phi)&=\frac{1}{2}\partial_\mu F(\phi)\partial^\mu F(\phi)-\frac{m^2}{2}F(\phi)F(\phi) \\
    & = \frac{1}{2}\partial_\mu (\phi+a_1\phi^2+a_2\phi^3+\cdots)\partial^\mu (\phi+a_1\phi^2+a_2\phi^3+\cdots)-\frac{m^2}{2}(\phi+a_1\phi^2+\cdots)^2 \\
    & = \frac{1}{2}\partial_\mu \phi\partial^\mu \phi+ a_1\frac{1}{2}\partial_\mu \phi^2\partial^\mu \phi+a_1\frac{1}{2}\partial_\mu \phi\partial^\mu\phi^2+a_1^2\frac{1}{2}\partial_\mu \phi^2\partial^\mu \phi^2 \\
    & \quad + a_2\frac{1}{2}\partial_\mu \phi^3\partial^\mu \phi+a_2\frac{1}{2}\partial_\mu \phi\partial^\mu\phi^3+\cdots\\
    &-\frac{m^2}{2}\phi^2 - (2a_1)\frac{m^2}{2}\phi^3 - (a_1^2+2a_2)\frac{m^2}{2}\phi^4 - \cdots \\
    & = \frac{1}{2}\partial_\mu \phi\partial^\mu \phi+ (4a_1)\frac{1}{2}\phi \partial_\mu \phi\partial^\mu \phi+(4a_1^2+6a_2)\frac{1}{2}\phi^2\partial_\mu \phi\partial^\mu\phi + \cdots  \\
    & \quad -\frac{m^2}{2}\phi^2 - (2a_1)\frac{m^2}{2}\phi^3 - (a_1^2+2a_2)\frac{m^2}{2}\phi^4 - \cdots \\
    & = \frac{1}{2}\partial_\mu \phi\partial^\mu \phi+ \frac{1}{2}\partial_\mu \phi\partial^\mu \phi\sum_{n=1}^\infty\frac{d_n}{n!}\phi^n - \frac{m^2}{2}\phi^2 - \frac{m^2}{2}\sum_{n=1}^{\infty}\frac{c_n}{(n+2)!}\phi^{n+2},
\end{align*}

where $d_n = n!\sum_{j=0}^n (j+1)(n-j+1)a_ja_{n-j}$ and $c_n = (n+2)!\sum_{j=0}^n a_ja_{n-j}$, (see equation 15 of \cite{kreimer2} with slightly different conventions). 

We can see from each term of the original free Lagrangian we obtain a vertex of each order $\geq 3$, (thus we have two types of vertices of each order) which we will call the kinematic and massive vertices respectively.  We read off the Feynman rules to be 

\begin{itemize}
    \item 

\[
    i\frac{d_{n-2}}{2}(p_1^2+p_2^2+\cdots + p_n^2)
\]
for the $n$-point kinematic vertex where $p_1, \ldots, p_n$ are the momenta of the incident edges; and 

\item 
\[
    -i\frac{m^2}{2}c_{n-2}
\]
for the $n$-point massive vertex.  

\item The free part of the Lagrangian is unchanged so the propagator remains
\[
\frac{i}{p^2-m^2}
\]
for momentum $p$.  We are interested in the on-shell $n$-point tree level amplitude.
\end{itemize}

For the combinatorial reader let us spell out in a bit more detail how the above leads to a purely combinatorial problem on trees.  We are working with graphs with external edges.  For a graph theorist such graphs can be constructed as bipartite graphs where if the bipartition is $(A,B)$ then we require that all vertices in $B$ are either of degree $1$ or $2$ and all vertices in $A$ are of degree $\geq 3$.  Then $B$ in fact contains no additional information: the 2-valent vertices of $B$ just mark the \emph{internal edges} of the original graph, while the 1-valent vertices of $B$ mark some bare half-edges of the original graph, known as \emph{external edges} or \emph{legs}.

To calculate the $n$-point tree level amplitude, we must sum over all trees (connected acyclic graphs of the type above) with $n$ external edges and with vertices either kinematic or massive.  For each tree we compute as follows: To each internal and external edge of the tree assign a momentum $p$ in Minkowski space, that is in $\mathbb{R}^4$ but using the pseudo-metric $|(a_0, a_1, a_2, a_3)|^2 = -a_1^2 + a_2^2+a_3^2+a_4^2$ (in fact the choice of signature will not matter).  Following the usual convention we will write $p^2$ for $|p|^2$.  Impose momentum conservation at each vertex, that is the sum of the momenta for the edges incident to any given vertex must be $0$.  Now multiply the factors given by the Feynman rules for each vertex and the propagator for each internal edge to get the contribution of this tree.

The on-shell condition applies only to the external edges and is that $p^2=m^2$ for each external momentum $p$.  Because we are working with a pseudo-metric, note that $p^2=0$ does not imply $p=0$.

In summary, combinatorially we have the following:
\begin{enumerate}
    \item The $n$-point tree level amplitude is the sum over all trees with $n$ external edges. 
    \item Each external edge is labelled, nothing else is.
    \item A momentum variable $p$ is assigned to every internal and external edge.
    \item The on-shell condition is that $p^2=m^2$ holds for the momentum of every external edge, where ($p^2=p\cdot p$).
    \item Conservation of momenta holds at every vertex.
    \item
    The vertices come in two kinds, massive and kinematic, each with its own contribution to the sum given by the Feynman rules. 
    \item The Feynman rule for a kinematic vertex of degree $n$ with momenta $p_1,\ldots,p_n$ for the incident legs is
    \[i\frac{d_{n-2}}{2}(p_1^2+p_2^2+\cdots+p_n^2),\]where
    \[d_r=r!\sum_{j=0}^r(j+1)(r-j+1)a_ja_{r-j}. \]
    \item The Feynman rule for a massive vertex of degree $n$ is
    \[-i\frac{m^2}{2}n!\sum_{j=0}^{n-2}a_ja_{n-2-j}.\]
    
    \item  The Feynman rule for an internal edge is \[\frac{i}{p^2-m^2},\] where $p$ is the momentum assigned to this propagator.
    
    \item  The contribution of each tree is the product of the Feynman rules for its vertices and internal edges (no contribution from external edges).
\end{enumerate}

Now notice that since we are summing over all such trees, we are to get the same value if we consider only a single type of combined vertices each of which is the sum of kinematic and massive vertices of degree $n$, for each $n$.

\subsubsection{Tree-level Amplitudes}
The best way to explore the problem combinatorially is through a small example. Fix an internal edge $e$, and consider all the possible subtrees that may occur below $e$ for a fixed number of legs, $n$. Let $b_n$ be the sum over all such subtrees with the Feynman rules applied to the vertices and edges of the subtree along with the edge $e$ itself.

\begin{exm}
When $n=3$ we have the contributors from the tree graphs in Figure \ref{diffeotr} below.

\begin{figure}[H]
    \centering
    \includegraphics[scale=0.75]{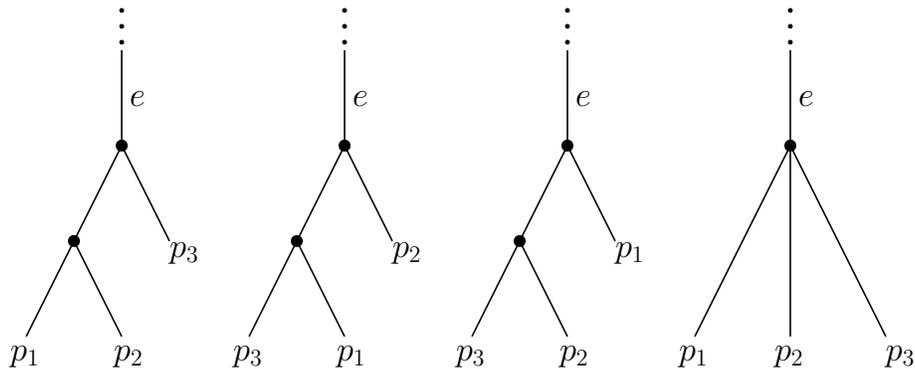}
    \caption{Subtrees below $e$ corresponding to $n=3$ external legs.}
    \label{diffeotr}
\end{figure}

 where the vertical dots above $e$ indicate the rest of the tree.
 Let us recursively compute $b_3$, knowing that $b_2=-2a_1$ (and $b_1=1$) which is easily verified. We shall denote the sum $p_1+p_2+p_3$ by $\mathbf{p}$.
 
 $$b_3=\frac{i^2(\frac{d_2}{2}(p_1^2+p_2^2+p_3^2+\mathbf{p}^2)+c_2)}{\mathbf{p}^2-m^2}+\frac{i^2(\frac{d_1}{2}(p_1^2+(p_2+p_3)^2+\mathbf{p}^2)+c_1)b_2}{\mathbf{p}^2-m^2}+$$
 
 $$+\frac{i^2(\frac{d_1}{2}(p_2^2+(p_1+p_3)^2+\mathbf{p}^2)+c_1)b_2}{\mathbf{p}^2-m^2}+\frac{i^2(\frac{d_1}{2}(p_3^2+(p_1+p_2)^2+\mathbf{p}^2)+c_1)b_2}{\mathbf{p}^2-m^2}.$$
 
 This simplifies to $b_3=-6a_2+12a_1^2$ as the reader may check. It is these cancellations that required an interpretation and triggered the research in \cite{karendiffeo,kreimer2}, especially because of the independence from momenta and masses in the resulting values.

\end{exm}

Returning to the general case, we are interested in the quantity $b_n$, which is the result of fixing an internal edge $e$ and summing over all possible subtrees with $n$ external edges labelled $p_1, p_2, \ldots,p_n$.  Equivalently $b_n$ is the sum over all trees with $n+1$ external edges with one of these edges (called $e$) is not necessarily on-shell, and where additionally we include the propagator for $e$ itself in each term.  

Fortunately, $b_n$ can be computed recursively. Consider the edges below $e$ which are incident to $e$.  Each of them is either external or has another subtree rooted at their other end.  Summing over all possibilities below $e$ means summing over all possibilities for each of these subtrees.  By induction, this gives the following recurrence, see \cite{karendiffeo} for a full proof:
\begin{equation}\label{eq orig b rec}
\begin{gathered}
b_n=-\underset{\underset{P_i\neq\emptyset\;\text{and disjoint}}{P_1\cup\cdots\cup P_k=\{1,\ldots,n\}}}{\sum}b_{|P_1|}\cdots
b_{|P_k|}\times \\\frac{\frac{(k-1)!}{2}\underset{j=0}{\overset{k-1}{\sum}}a_ja_{k-1-j}\Big(-m^2(k+1)k+(j+1)(k-j)\big(\sum^k_{i=1}(\sum_{e\in P_i}p_e)^2+(\sum^n_{s=1}p_s)^2\Big)}{(\sum^n_{s=1}p_s)^2-m^2}.\vspace{0.6cm}\end{gathered}\end{equation}
 
 The idea then was to break this into two recurrences and it turned out that some intricate Bell Polynomial identities give one way to solve for $b_n$. The approach in \cite{karendiffeo} makes an extensive use of Bell polynomials identities on different levels. In our case we will show that the use of Bell polynomials identities can be shortened in a minimal way. Besides, we will give new proofs for a number of these identities. For example, in the Section \ref{bellpolyidsection} we give a simple combinatorial proof for the recent identity obtained by Cvijovi\'c (see \cite{cvijovic8}) in 2013.

 \section{The Role of Bell Polynomials}
 
 \begin{dfn}[Partial Bell Polynomial]\label{bellpoly} The \textit{partial Bell polynomial}, for parameters $n,k$, in an infinite set of indeterminates, is defined by
\begin{align*}B_{n,k}(x_1,x_2,\ldots) &= \underset{P_i\text{'s disjoint, nonempty   }}{\underset{P_1\cup\cdots\cup P_k=\{1,\ldots,n\}}{\underset{\{P_1,\ldots,P_k\}}{\sum}}}x_{|P_1|}x_{|P_2|}\cdots x_{|P_k|}\\&=\underset{\lambda(n,k)}{\sum}\frac{n!}{j_1!j_2!\cdots} \left(\frac{x_1}{1!}\right)^{j_1}\left(\frac{x_2}{2!}\right)^{j_2}\cdots, \end{align*}

where the second sum ranges over all partitions $\lambda=1^{j_1}2^{j_2}\ldots$ of $n$ with $k$ parts, that is, such that \[j_1+j_2+j_3+\cdots=k\;\;\;\text{and}\,\,\,\;j_1+2j_2+3j_3+\cdots=n\;\;\;\text{and}\;\;\; j_i\geq1.\]
\end{dfn}

  Note that, by this definition, the largest index appearing should be $x_{n-k+1}$, thus, any given Bell polynomial uses only a finite number of variables (and is indeed a polynomial). 
  
  On the level of generating functions, one gets
\[ \exp\bigg(u \underset{m\geq1}{\sum}x_m\frac{t^m}{m!}\bigg)=\underset{n,k\geq0}{\sum}B_{n,k}(x_1,x_2,\ldots)\;\frac{t^n}{n!}u^k .\]  This can be used as an alternative definition for Bell polynomials.
 
 To split equation \eqref{eq orig b rec} up usefully, we will use the fact that the problem is symmetric in the external momenta along with the on-shell condition.  Consider expanding all the $p_e^2$ in the numerator and the denominator into sums of squares of external momenta and dot products of distinct external momenta.  All the squares of external momenta are $m^2$ by the on-shell condition, so in both the numerator and denominator collect all of these along with the explicit $m^2$.  The remaining terms all have a factor which is a dot product of distinct external momenta.  By the symmetry in the external momenta, we know that each dot product appears the same number of times, so it suffices to keep track of how many dot product terms there are, without keeping track of which momenta are involved.  So to satisfy equation \eqref{eq orig b rec} it suffices to separately satisfy the $m^2$ part and the dot product part of it.  These two parts are respectively the equations of the following Lemma (see \cite{karendiffeo} for details). 
 
 \begin{lem}{\cite{karendiffeo}} Let $b_n$ be, as before, the sum over all amplitudes of rooted trees with $n+1$ external legs, one of which is off-shell and has a propagator contribution (see the Feynman rules above). Then the sequence $\mathbf{b}=(b_n)$ satisfies equation \eqref{eq orig b rec} if and only if it satisfies the following two recurrences :
 
 \begin{align}
     0&=\overset{n}{\underset{k=1}{\sum}}B_{n,k}(\mathbf{b})\;\frac{(k-1)!}{2}\;\;\overset{k-1}{\underset{j=0}{\sum}}a_ja_{k-1-j}\;\Big[\;2n(j+1)(k-j)-k(k+1)\;\Big]\label{recurrence}\;,\\
 0&=\overset{n}{\underset{k=1}{\sum}}\overset{k-1}{\underset{j=0}{\sum}}a_ja_{k-1-j}(j+1)(k-j)\frac{(k-1)!}{2k}\overset{n}{\underset{s=1}{\sum}}\frac{b_s}{s!(n-s)!}B_{n-s,k-1}(\mathbf{b})(ks(s-1)+n(n-1));\label{di}
 \end{align}
  where $\mathbf{b}=(b_1,b_2,\cdots)$.\end{lem}

  \setlength{\parindent}{0.5cm}
  From these two equations, Karen Yeats and Dirk Kreimer proved that  \[\tag{$\dagger$}b_{n+1}=\overset{n}{\underset{k=1}{\sum}}\frac{(n+k)!}{n!}B_{n,k}(-1!a_1,-2!a_2,\ldots)\;\;,\label{form} \] which might be suggested by the calculation of the first examples of $b_n$'s. 
  
  
  \section{Bell Polynomial Identities}\label{bellpolyidsection}

 A number of Bell polynomial identities are needed in the sequel of this chapter. The identities we are most concerned with were introduced by D. Cvijovi\'c in \cite{cvijovic8} (2013). These identities are key ingredients in the arguments of \cite{karendiffeo}, they are also combinatorially significant \cite{bighunt}. In 2015, S. Eger re-proved some of these identities by translating into integer-valued distributions \cite{eger}. It is surprising, however, that elementary combinatorial proofs are actually quite applicable, and this section is devoted  for displaying them. The reason that these proofs, despite being simple, were not discovered before is probably because the proofs are only seen clear if the appropriate identity is chosen to start with.
 
\setlength{\parindent}{0cm}

\begin{lem}\label{lemma1}

Suppose $n,k>0$, then$$k\;B_{n,k}(x_1,x_2,\ldots)=\;\overset{n}{\underset{s=1}{\sum}}\;{n\choose s} \;x_s \;B_{n-s,k-1}(x_1,x_2,\ldots)\;,$$
and
\[n\;B_{n,k}(x_1,x_2,\ldots)=\;\overset{n}{\underset{s=1}{\sum}}\;{n\choose s}\;s\;x_s\;B_{n-s,k-1}(x_1,x_2,\ldots).\]
                                    \end{lem}
\begin{proof}For the first identity, the left hand side is the generating function for partitions with $k$ parts which are rooted at one part (localization). Seen another way, we may first choose $s$ arbitrary  elements from $\{1,2,\ldots,n\}$ to form our root part, and then generate all possible partitions with $k-1$ parts over the remaining $n-s$ elements, thus getting the right hand side.

For the second identity, the left hand side stands for partitions with $k$ parts, which are rooted in a finer way than in the previous setting, namely, they rooted at one on the $n$ elements. Again, we can do this rather differently (localizing in two levels): First choose $s$ special elements that will form the part which hosts the root, then choose the root from amongst them (in $s$ ways); finally generate all partitions with $k-1$ elements over the remaining elements, hence getting the right hand side.   \end{proof}

The next theorem is the main theorem in \cite{cvijovic8}. The proof presented here is new and does not make any reference to the analytic methods used in proving the identities in \cite{cvijovic8}. The proof only depends on the combinatorial meaning of Definition \ref{bellpoly}.

\begin{thm}\label{report}
The following Bell identities hold, where $B_{n,k}$ stands for $B_{n,k}(x_1,x_2,\ldots)$, the partial Bell polynomial with $k$ parts.
\begin{equation}
    B_{n,k}=\frac{1}{x_1}\cdot\frac{1}{n-k}\;\overset{n-k}{\underset{\alpha=1}{\sum}}\;{n\choose \alpha} \Big[(k+1)-\frac{n+1}{\alpha+1}\Big]\;x_{\alpha+1} \;B_{n-\alpha,k}\;\;,\label{id1}
\end{equation}

\begin{equation}
    B_{n,k_1+k_2}=\frac{k_1!\;k_2!}{(k_1+k_2)!}\;\overset{n}{\underset{\alpha=0}{\sum}}\;{n\choose \alpha} \;B_{\alpha,k_1}\;B_{n-\alpha,k_2}\;\;,\label{id2}
\end{equation}

\begin{equation}
    B_{n,k+1}=\frac{1}{(k+1)!}
\overset{n-1}{\underset{\alpha_1=k}{\sum}}
\overset{\alpha_1-1}{\underset{\alpha_2=k-1}{\sum}}\cdots
\overset{\alpha_{k-1}-1}{\underset{\alpha_k=1}{\sum}}
{n\choose \alpha_1}{\alpha_1\choose \alpha_2}\cdots{\alpha_{k-1}\choose \alpha_k} x_{n-\alpha_1}\cdots x_{\alpha_{k-1}-\alpha_k}x_{\alpha_k}.\label{id3}
\end{equation}
\end{thm}

\begin{proof}
Identity (\ref{id3}) is immediate from (\ref{id2}), so we start by proving (\ref{id1}). The following `starter' identity is clear from the definition of partial Bell polynomials:
\[B_{n+1,k+1}=\overset{n-k}{\underset{\alpha=0}{\sum}}\;{n\choose \alpha}\;x_{\alpha+1} \;B_{n-\alpha,k}\;.\]
Indeed, the identity exactly describes the natural passage from partitions of the set \;$\mathbb{N}_n=\{1,2,\ldots,n\}$ to partitions of $\mathbb{N}_{n+1}=\{1,2,\ldots,n,n+1\}$. Namely, to form all partitions of $\mathbb{N}_{n+1}$ with $k+1$ parts in which $n+1$ appears in a part of size $\alpha+1$, we can first choose $\alpha$ elements (inall possible ways) from $\mathbb{N}_n$ to be in the same part with $n+1$, and then generate all partitions with $k$ parts on the remaining $n-\alpha$ elements of $\mathbb{N}_n$. 

Now, by Lemma \ref{lemma1}, multiply both sides by $(k+1)$ to further get 
\[\overset{(n+1)-k}{\underset{s=1}{\sum}}\;{n+1\choose s}\;x_s\;B_{n+1-s,k}= (k+1)\;B_{n+1,k+1}=(k+1)\;\overset{n-k}{\underset{\alpha=0}{\sum}}\;{n\choose \alpha}\;x_{\alpha+1} \;B_{n-\alpha,k}\;.\]
Reindexing the left sum by $\alpha=s-1$, and explicitly writing the first term ($\alpha=0$) of both sides, we arrive at
\[(n+1) x_1B_{n,k}\;+\;\overset{n-k}{\underset{\alpha=1}{\sum}}\;{n+1\choose \alpha+1}x_{\alpha+1}B_{n-\alpha,k}= (k+1)x_1 B_{n,k}\;+\;(k+1)\overset{n-k}{\underset{\alpha=1}{\sum}}{n\choose \alpha}x_{\alpha+1} B_{n-\alpha,k}.\]

Hence,
\begin{align*}
(n-k) \;x_1\; B_{n,k}\;&=\;\overset{n-k}{\underset{\alpha=1}{\sum}}\;\Big[(k+1){n\choose \alpha}-{n+1\choose \alpha+1}\Big]\;x_{\alpha+1}\;B_{n-\alpha,k}\\&=\overset{n-k}{\underset{\alpha=1}{\sum}}\;{n\choose \alpha} \Big[(k+1)-\frac{n+1}{\alpha+1}\Big]\;x_{\alpha+1} \;B_{n-\alpha,k}\;,\end{align*}which gives identity (\ref{id1}).

Finally, identity \ref{id2} is actually easier. Given $k=k_1+k_2$, the generating function for partitions into $k$ parts with $k_1$ `distinguished' parts is given by ${k_1+k_2\choose k_1} B_{n,k_1+k_2}$. Another way is to first select $\alpha$ elements and use them to build a partition on $k_1$ parts (these are now naturally `highlighted' by this choice), and then generate a partition of the remaining $n-\alpha$ elements on $k_2$ parts.

\end{proof}

    \section{Generating Functions Method}
  
The formula for $b_n$ in (\ref{form}) turns out to be exactly the compositional inverse of the diffeomorphism $F$. This can be seen through an old result that is mentioned in \cite[p.150-151]{comtet}, seemingly obtained independently by B\"{o}dewadt (1942) and Riordan (1968) among others. However, it seems that this expansion for compositional inverses is not widely known, and was not known to the authors in \cite{karendiffeo}. 
 We derive a functional differential equation whose solution is the inverse of $F$. The idea is to rewrite equations (\ref{recurrence}) and (\ref{di}) so that we can apply the next Fa\`a di Bruno's composition of power series relation.

 \begin{lem}[Fa\`a di Bruno]\label{faa}
 Given two power series $f(t)=\overset{\infty}{\underset{n=0}{\sum}}\;f_n\displaystyle\frac{t^n}{n!}$ and $g(t)=\overset{\infty}{\underset{n=0}{\sum}}\;g_n\displaystyle\frac{t^n}{n!}$, the composition $\;h(t):=f(g(t))$ can be written as $h(t)=\overset{\infty}{\underset{n=0}{\sum}}h_n\displaystyle\frac{t^n}{n!},$ where
\[h_n=\overset{n}{\underset{k=0}{\sum}}\;f_k\; B_{n,k}(g_1,g_2,\ldots).\]
 \end{lem}
 
 \begin{prop}\label{1pm}Let $F(t)=\overset{\infty}{\underset{j=0}{\sum}}a_j\;t^{j+1}$ be a diffeomorphism of fields as before, and set $\;G(t):=\overset{\infty}{\underset{n=1}{\sum}}b_n\displaystyle\frac{t^n}{n!}$ where the $b_n$'s satisfy the recurrences \ref{recurrence} and \ref{di}. Define 
 \begin{equation}\label{PandQ}
     Q(t):=\frac{1}{2} \;\frac{d}{dt}\Big(\big(F(t)\big)^2\Big)\;\;\;\;\;\text{and}\;\;\;\;\;P(t):=\int\Big(\frac{d}{dt}F(t)\Big)^2 dt\;.
 \end{equation}
 Then, on the level of generating functions, the recurrence (\ref{recurrence}) is equivalent to the differential equation 
 \begin{equation}
     0=t\;\frac{d}{dt}P\big(G(t)\big)\;-\; Q\big(G(t)\big)\;,  \label{myequation}
 \end{equation} and the recurrence (\ref{di}) is equivalent to the differential equation 
 \begin{equation}
     0=\;\frac{d^2}{dt^2}P\big(G(t)\big)\;+\; \frac{d^2G}{dt^2}\cdot\frac{d}{dG}P\big(G(t)\big)\;.\; \label{myequation2}
 \end{equation}\end{prop}

  \begin{proof}
 (1) Proving (\ref{myequation}): By their definition, $P$ and $Q$ can be expanded as 
 \begin{equation}
     Q(t)=\frac{1}{2} \;\frac{d}{dt}\Big((F(t))^2\Big)=\overset{\infty}{\underset{k=1}{\sum}}\;\overbrace{\Big(k!\;\frac{(k+1)}{2}\;\;\overset{k-1}{\underset{j=0}{\sum}}a_ja_{k-1-j}\Big)}^{q_k}\;\frac{t^k}{k!}\;\;,
 \end{equation} and
  \begin{equation}
      P(t)=\int\Big(\frac{d}{dt}F(t)\Big)^2 dt=\overset{\infty}{\underset{k=1}{\sum}}\;\overbrace{\Big(k!\;\frac{1}{k}\;\overset{k-1}{\underset{j=0}{\sum}}a_ja_{k-1-j}\;(j+1)(k-j)\Big)}^{p_k}\;\frac{t^k}{k!}\;.
  \end{equation}
 Consequently, by the Fa\`a di Bruno's formula (Lemma \ref{faa}), we have 
 \begin{align}
     Q(G(t))&=\overset{\infty}{\underset{n=1}{\sum}}\;\Big(\overset{n}{\underset{k=1}{\sum}}\;q_k\;B_{n,k}(b_1,b_2,\ldots)\Big)\;\frac{t^n}{n!}\;,\\
     P(G(t))&=\overset{\infty}{\underset{n=1}{\sum}}\;\Big(\overset{n}{\underset{k=1}{\sum}}\;p_k\;B_{n,k}(b_1,b_2,\ldots)\Big)\;\frac{t^n}{n!}\;.
 \end{align}
In particular, 
  \[t\;\frac{d}{dt}P(G(t))=\overset{\infty}{\underset{n=1}{\sum}}\;n\;\Big(\overset{n}{\underset{k=1}{\sum}}\;p_k\;B_{n,k}(b_1,b_2,\ldots)\Big)\;\frac{t^n}{n!}\;.\]
 
 Then (\ref{myequation}) is given by
  \begin{align*}0&=t\;\frac{d}{dt}P\big(G(t)\big)\;-\; Q\big(G(t)\big)\\&=\overset{\infty}{\underset{n=1}{\sum}}\;\Big(\overset{n}{\underset{k=1}{\sum}}B_{n,k}(b_1,b_2,\ldots) \;(n\;p_k-q_k)\Big)\;\frac{t^n}{n!}.
 \end{align*}
 
 That is, equation (\ref{myequation}) is equivalent to that for all $n\geq1$ 
  \begin{align*}0&=\overset{n}{\underset{k=1}{\sum}}B_{n,k}(b_1,b_2,\ldots) \;(n\;p_k-q_k)\\&=\overset{n}{\underset{k=1}{\sum}}B_{n,k}(b_1,b_2,\ldots)\left[n\Big(k!\frac{1}{k}\overset{k-1}{\underset{j=0}{\sum}}a_ja_{k-1-j}(j+1)(k-j)\Big)-\Big(k!\frac{(k+1)}{2}\overset{k-1}{\underset{j=0}{\sum}}a_ja_{k-1-j}\Big)\right]\\&=\overset{n}{\underset{k=1}{\sum}}B_{n,k}(b_1,b_2,\ldots)\frac{(k-1)!}{2}\;\;\overset{k-1}{\underset{j=0}{\sum}}a_ja_{k-1-j}\;\Big[\;2n(j+1)(k-j)-k(k+1)\;\Big],\end{align*}
 which is exactly the recurrence (\ref{recurrence}).\\

 (2) Proving (\ref{myequation2}):
 
  We can cancel out the factor of $1/2$  and rearrange recurrence (\ref{di}) as 
 \begin{align*}0=\;n(n-1)&\overset{n}{\underset{k=1}{\sum}}\;p_k\cdot\frac{1}{k}\overset{n-k+1}{\underset{s=1}{\sum}}\frac{b_s}{s!(n-s)!}B_{n-s,k-1}(b_1,b_2,\ldots)\;\\+\;&\overset{n}{\underset{k=1}{\sum}}\;p_k\;\overset{n-k+1}{\underset{s=1}{\sum}}\frac{b_s\;s(s-1)}{s!(n-s)!}B_{n-s,k-1}(b_1,b_2,\ldots)\;.\end{align*}
 
 Now, recall that
 \begin{align*}k\;B_{n,k}(b_1,b_2,\ldots)&=n!\;\overset{n}{\underset{s=0}{\sum}}\;\frac{b_s}{s!(n-s)!}B_{n-s,k-1}(b_1,b_2,\ldots)\\&=n!\overset{n-k+1}{\underset{s=1}{\sum}}\frac{b_s}{s!(n-s)!}B_{n-s,k-1}(b_1,b_2,\ldots)\;,\end{align*}
 since $b_0=0$ by convention, and since Bell polynomials vanish whenever the number of parts is greater than the size. The latter reason is exactly what allows us to also change the bounds of the summations in the above equation to finally get  
  \begin{align*}0&=\;\frac{n(n-1)}{n!}\overset{n}{\underset{k=1}{\sum}}\;p_k\;B_{n,k}(b_1,b_2,\ldots)\;\\&+\;\overset{n}{\underset{s=0}{\sum}}\;\frac{s(s-1)\;b_s  }{s!(n-s)!}\overset{n-s+1}{\underset{k=1}{\sum}}p_k \;B_{n-s,k-1}(b_1,b_2,\ldots)\\&=[t^n]\; t^2\;\frac{d^2}{dt^2}P\big(G(t)\big) \;+\; \overset{n}{\underset{s=0}{\sum}} [t^s] \big(t^2\frac{d^2G}{dt^2}\big) \cdot    [t^{n-s}] \Bigg(\frac{dP}{dt}\Bigg)(G),\end{align*}
 which establishes the claim (notice that $ \Big(\frac{dP}{dt}\Big)(G(t))= \frac{d}{dG}P(G)$ ).\\
 \end{proof}

 \begin{cor}\label{corcorcor}The compositional inverse $F^{-1}$ of the diffeomorphism $F$ is a solution for (\ref{myequation}) and (\ref{myequation2}).\end{cor}
 
 \begin{proof}

 (1) Equation (\ref{myequation}) can be simplified further: 
 \begin{align*}
0&=t\;\frac{d}{dt}P\big(G(t)\big)\;-\; Q\big(G(t)\big)\\&=t\;\frac{d}{dG}P\big(G(t)\big)\;\frac{dG}{dt}-\; \frac{1}{2}\frac{d}{dG}\Big(F\big(G(t)\big)\Big)^2\\&=t\;\Big(\frac{d}{dG}F\big(G(t)\big)\Big)^2\;\frac{dG}{dt}\;-\;F\big(G(t)\big)\;\frac{d}{dG}F\big(G(t)\big).
\end{align*} 
 
 Now we might assume that $\frac{d}{dG}F\big(G(t)\big)\neq 0$, that is just $F(G(t))$ is not a constant, and then we have 

\[0=t\;\frac{d}{dG}F\big(G(t)\big)\;\frac{dG}{dt}\;-\;F\big(G(t)\big).\]

 That is to say,
 \[F\big(G(t)\big)=t\;\frac{d}{dt}F\big(G(t)\big),\] and a simple separation of variables then gives that $G=F^{-1}$ is a solution.
 
 \bigskip
 \bigskip
 (2) Also Equation (\ref{myequation2}) boils down maybe more insightfully:
 
 \begin{align*}
 0&=\;\frac{d^2}{dt^2}P\big(G(t)\big)\;+\; \frac{d^2G}{dt^2}\cdot\frac{d}{dG}P\big(G(t)\big)\\
  &=\;\frac{d}{dt}\Big(\frac{d}{dG}P\big(G(t)\big)\cdot\frac{dG}{dt}\Big)\;+\; \frac{d^2G}{dt^2}\cdot\frac{d}{dG}P\big(G(t)\big)\\
  &=2\;\frac{d^2G}{dt^2}\cdot\frac{d}{dG}P\big(G(t)\big)\;+\;\frac{d}{dt}\Big(\frac{d}{dG}P\big(G(t)\big)\Big)\cdot\frac{dG}{dt}\\
  &=2\;\frac{d^2G}{dt^2}\cdot\Big(\frac{d}{dG}F\big(G(t)\big)\Big)^2\;+\;\frac{d}{dt}\Bigg(\Big(\frac{d}{dG}F\big(G(t)\big)\Big)^2\Bigg)\cdot\frac{dG}{dt}\\
  &=2\;\frac{d^2G}{dt^2}\cdot\Big(\frac{d}{dG}F\big(G(t)\big)\Big)^2\;+\;2\;\frac{d}{dG}F\big(G(t)\big)\cdot\frac{d}{dG}\Big(\frac{d}{dG}F\big(G(t)\big)\Big)\cdot\Big(\frac{dG}{dt}\Big)^2.
 \end{align*}
 
Again we assume that $G(t)$ is not a constant, and so neither $F(G(t))$ is. Hence,

\begin{align*}
 0&=\frac{d^2G}{dt^2}\cdot\frac{d}{dG}F\big(G(t)\big)\;+\;\frac{d^2}{dG^2}F\big(G(t)\big)\cdot\Big(\frac{dG}{dt}\Big)^2\\
  &=\frac{d}{dG}\Big(\frac{dG}{dt}\Big)\cdot\frac{dG}{dt}\cdot\frac{d}{dG}F\big(G(t)\big)\;+\;\frac{d^2}{dG^2}F\big(G(t)\big)\cdot\Big(\frac{dG}{dt}\Big)^2,
 \end{align*}
  
  Then by our assumption above
 \begin{align*}
 0&=\frac{d}{dG}\Big(\frac{dG}{dt}\Big)\cdot\frac{d}{dG}F\big(G(t)\big)\;+\;\frac{d^2}{dG^2}F\big(G(t)\big)\cdot\frac{dG}{dt}\\&=\frac{d}{dG}\Big[\frac{dG}{dt}\cdot\frac{d}{dG}F\big(G(t)\big)\Big]\\&=\frac{d}{dG}\Bigg(\frac{d}{dt}F\big(G(t)\big)\Bigg).
 \end{align*}
 
 Again this leads to that $G=F^{-1}$ is a solution, which proves the corollary.
 \end{proof}

 Thus, we have shown that the series whose coefficients are the solution $(b_n)$ of  recurrences (\ref{recurrence}) and (\ref{di})  is exactly the compositional inverse of the diffeomorphism $F$. We further conjecture that there may be another way for arriving at this fact, and is probably related to the Legendre transform, since, as we shall see in the next section, the definition of the combinatorial Legendre transform goes through an almost identical process of summing over tree graphs and involves compositional inverses \cite{kjm1,kjm2}.

\section{Relation with the Legendre Transform}

Given a function $f:x\mapsto f(x)$, it might be desirable, in many contexts, to express every thing in terms of $y=f'(x)$ instead of $x$, without losing information about the function. The (analytic) Legendre transform does indeed achieve this goal but for a restricted class of functions, namely convex smooth functions. However, physicists usually use the Legendre transform even when the functions involved fail to satisfy these requirements. The surprise is when such calculations match with experimental results. In \cite{kjm1}, a combinatorial Legendre transform is defined which generalizes the analytic version and unveils the hidden robust algebraic structure of the Legendre transform. 

The defining relation for the Legendre transform $\mathrm{L}$ is given as \[\mathrm{L}f(z)=-z g(z)+ (f\circ g)(z),\]

where $g$ is the compositional inverse of the derivative, namely, $x=g(y)$.

The idea in \cite{kjm1} was to realize this relation as coming from the following combinatorial bijection between classes:

\[v_1 \partial_{v_1} \mathcal{T}^l\uplus\biguplus_{k\geq 2} v_k\circ (\partial_{v_1}\mathcal{T}^l)\sim \mathcal{T}^l\uplus (e\ast(\mathcal{E}_2\circ e^{-1}\partial_{v_1} \mathcal{T}^l),
\]
where $\mathcal{T}^l$ is the class of labelled trees, and $v_1\partial_{v_1}$ is the operation of rooting or distinguishing at a 1-vertex (vertex of degree 1), $e,e^{-1}$ for edges and anti-edges (the inverse with respect to the gluing operation) respectively. 

\begin{dfn}[Combinatorial Legendre Transform \cite{kjm2}]If $A\in R[[x]]$ and $A'^{-1}$ exists, then the \textit{combinatorial Legendre transform} is defined to be \[(\mathrm{L}A)(x)=(A\circ A'^{-1})(x)-x A'^{-1}(x).\]\end{dfn}
\begin{thm}[\cite{kjm2}]

Let $T_F(y)= [[\mathcal{T}^l,\omega_{v_1}\otimes\omega_e\otimes\omega]](y,u,\lambda_2,\lambda_3,\ldots)$, be the generating series of labelled tree graphs with indeterminates $y,u,\lambda_i$ standing for leaves, edges and vertices of higher orders, respectively; and where  $F(x)=-u^{-1}\frac{x^2}{2!}+\sum_{k\geq2} \lambda_k \frac{x^k}{k!}$. Then $(\mathrm{L}F)(y)=T_F(-y)$. That is, the Legendre transform is exactly the generating series of tree graphs with the prescribed weights and variables.
\end{thm}

This alternative Legendre transform is applicable to all formal power series with vanishing constant and linear terms, and the most important is that the new Legendre transform drops the convexity constraint.

We are looking for the right tie between this version of the Legendre transform and our result for the tree-level amplitudes. In the two situations we form the all-tree sum of a given power series and eventually compositional inverses are involved. A physical reason should be also accessible.

\begin{center}
    $\sim$
\end{center}

\chapter{Conclusions and Further Speculations}\label{chfurther}

In this final chapter we list the results explicitly and mention a number of questions that could not be settled in this thesis. The questions offer a wide workspace for the future, and may be of interest to researchers and grad students working in the same area.

\section{Results}
\begin{enumerate}
    \item In Chapter \ref{chchords1} our ultimate goal is Corollary \ref{endcoro}, in which we prove that \href{https://oeis.org/A088221}{A088221} counts pairs of connected chord diagrams. In the way to prove this corollary we had to prove a number of propositions: Lemma \ref{bij}, Theorem \ref{mainth}, Lemma \ref{inde}, and Corollary \ref{coro}.
    
    \item In Chapter \ref{chapterchords2} we study $2$-connected chord diagrams and their asymptotic behaviour. We prove Lemma \ref{B} and then use it to derive a decomposition for connectivity-$1$ chord diagrams represented by equation \ref{C11eq}. Then, in Proposition \ref{myproposition2connected}, we prove a functional relation between connected and $2$-connected chord diagrams. We use this result to compute an asymptotic behaviour for $2$-connected chord diagrams as in equation \ref{computC2asympt}. This expansion extends the result by Kleitman \cite{kleit} in the case of $2$-connected chord diagrams. Kleitman's argument approximates the proportion of $2$-connected chord diagrams (irreducible linked diagrams) to $e^{-2}$. This value corresponds only to the first term in the expansion obtained here.
    
    The coefficients of $2$-connected diagrams were noticed to coincide with some other series in more physical contexts in the work of M. Borinsky and D. J. Broadhurst. Namely, they agreed with the counterterms series $\;z_{\phi_c|\psi_c|^2}(\hbar_{\text{ren}})\;$ and $\;z_{|\psi_c|^2}(\hbar_{\text{ren}})\;$ in quenched QED. This relation is what we prove in Theorem \ref{myresultinquenched}, then we use our information about the asymptotics of $2$-connected chord diagrams to obtain the asymptotic behaviour of these series without referring to singularity analysis.
    
    We also interpret a number of Yukawa theory green functions in terms of connected chord diagrams. These results are represented in Section \ref{YYYYY}. All lemmas, propositions, theorems and corollaries in Section \ref{YYYYY} are original work.
    
    \item In Chapter \ref{difchapter} we study diffeomorphisms of quantum fields. We prove a cancellation property in the tree-level amplitudes of theories obtained from free theories through field diffeomorphisms. We show this through Proposition \ref{1pm} and Corollary \ref{corcorcor}. The proofs presented here establish the cancellation property is significantly simpler than the proofs used in \cite{karendiffeo}. One important advantage of the proofs on the level of generating series is that we could show that the series $\;G(t):=\overset{\infty}{\underset{n=1}{\sum}}b_n\displaystyle\frac{t^n}{n!}$ is exactly the compositional inverse of the applied diffeomorphism. In Theorem \ref{report} we also give simple proofs of some Bell Polynomial identities  that were recently established by D. Cvijovi\'c in \cite{cvijovic8} (2013). Lemma \ref{lemma1} is proved and is then used to prove Theorem \ref{report}.
\end{enumerate}

\section{Questions and Discussions}
\begin{enumerate}
    \item In Chapter \ref{difchapter}, we saw that the tree-level amplitudes of the transformed theory generate the compositional inverse of the diffeomorphism which was applied to the original free theory. We conjecture that there should be a better, more evident, way of deriving this fact. Probably this should be also related to the Legendre transform.
    
    \item The asymptotic expansion $\mathcal{A}^2_\frac{1}{2}C$ of connected chord diagrams consists of a polynomial and an exponential function in $C(x)$, this phenomenon is related to the mathematical area of resurgence and transseries \cite{resurgence} where non-perturbative information is grafted into perturbative ones. This can be studied by means of the Borel summation of the coefficients of the series considered. In \cite{Geraldmichi} M. Borinsky and G. Dunne study transseries solutions for the differential equation of connected chord diagrams. They seek a combinatorial interpretation of the non-perturbative terms that arise when the Borel summation is used for Yukawa theory generating functions, and, of course, this type of studies can lead to many interesting combinatorial problems.
    
    \item The results by Kleitman \cite{kleit}, although they stop at the first term of what we would like to obtain as an asymptotic expansion, they give a positive sign that a functional relation for $k$-connected chord diagrams can be derived in general for arbitrary $k$.  
    
    \item It seems that connected chord diagrams can be used in expressing many Green functions in QFT as we saw in the case of quenched QED and the many Yukawa theory Green functions. So what other Green functions counting problems can benefit the enumeration of connected chord diagrams?
    
    \item A particular case for the previous question is that probably the renormalization quantities in Table 19 in \cite{michiq} are also related to connected chord diagrams, how can we express this as we did in the case of quenched QED?
\end{enumerate}



\bibliographystyle{plain}
\cleardoublepage 
\phantomsection  
\renewcommand*{\bibname}{References}

\addcontentsline{toc}{chapter}{\textbf{References}}

\bibliography{References}

\nocite{*}


\appendix
\chapter*{APPENDICES}
\addcontentsline{toc}{chapter}{APPENDICES}
\chapter{Enumerative Tools: Combinatorial Structures and Lagrange Inversion}\label{Appendix1}

This appendix is a primer of the language of enumerative combinatorics, following the lecture notes I took after Prof. David Wagner in CO 630 during Winter 2017 (University of Waterloo).

\subsection{Combinatorial Structures}

\begin{dfn}[Combinatorial Structures]
A \textit{combinatorial structure} on a finite set means a finite set together with some additional information defined on the set. This can be a graph structure on a set of vertices, or a permutaion, etc.
\end{dfn}

We will talk about structures in terms of \textit{classes of structures}. For example, the class $\mathcal{G}$ of all graphs consists of all finite sets $V$ and all graph structures on this set of verices. The class $\mathcal{G}$ is not a set in the sense of rigorous set theory and has to be defined to be a `class'. Notice that $\mathcal{G}$  associates to every finite set $V$ a finite set $\mathcal{G}_V$ of all possible  $\mathcal{G}$-structures that can be defined over $V$.

\begin{dfn}[Classes of Structures]
A \textit{class} $\mathcal{A}$ of structures associates to every finite set $X$ another finite set $\mathcal{A}_X $ such that 
\begin{enumerate}
    \item $\mathcal{A}_X\cap\mathcal{A}_Y=\varnothing$ whenever $X$ and $Y$ are distinct; and
    
    \item $|\mathcal{A}_X|=|\mathcal{A}_Y|$ if and only if $|X|=|Y|$.
\end{enumerate}
\end{dfn}

By the above definition we can restrict to only one representative set of each size, namely we set $\mathcal{A}_n$ be the set of non-isomorphic structures associated with the set $\{1,\ldots,n\}$. The  generating function for this class then should be 
\[A(x)=\sum_{n=0}^\infty |\mathcal{A}_n| x^n.\]

\begin{dfn}[Product Operations on Classes]
The following are some of the operations on combinatorial classes that are most often needed in enumerative problems:
\begin{enumerate}
    \item The \textit{product} $\mathcal{A}\ast\mathcal{B}$ of two classes $\mathcal{A}$ and $\mathcal{B}$ will stand for the class of structures that are defined over a finite set $X$ by first partitioning the set $X$ into two parts and applying an $\mathcal{A}$-structure to one part and applying a $\mathcal{B}$-structure to the other part.
    \[(\mathcal{A}\ast\mathcal{B})_X:=\bigcup_{S\subseteq X}(\mathcal{A}_S\times \mathcal{B}_{(X-S)}).\]
    The generating function of the product is easily seen to be $A(x)B(x)$.
    \item The class of \textit{sequences} of a given class $\mathcal{A}$, denoted by $\mathcal{A}^\ast$ is similarly defined as the product, the set is subdivided into an ordered sequence of parts, each of which receives an $\mathcal{A}$-structure. The generating function is 
    
    \[A^\ast(x)=\displaystyle\frac{1}{1-A(x)}.\]
\end{enumerate}

\end{dfn}

\subsection{Lagrange's Inversion}
The material provided in this section can be found in many  algebraic combinatorics references, e.g. the reader can refer to \cite{goulden}.
\begin{thm}[Lagrange's Implicit Function Theorem (LIFT)] Let $G(x)\in\mathbb{C}[[x]]$ be invertible, then there is a unique solution in $\mathbb{C}[[x]]$ to the functional equation $$R(x)=x\;G(R(x)).$$ Moreover, if $F\in\mathbb{C}[[x]]$, then, for $n\geq1$, \[[x^n]F(R(x))=\displaystyle\frac{1}{n}[t^{n-1}]F'(t)G(t)^n.\]
\end{thm}

\begin{cor}\label{corolift}In the same context of LIFT, if $a_n=[x^n]H(x)G^n(x)$ where $H\in\mathbb{C}[[x]]$, then 
\[\sum_{n\geq0}a_n x^n=\displaystyle\frac{H(R(x))}{1-x\;G'(R(x))}.\]

\end{cor}

\begin{proof}Differentiating the defining equation of $R$ we get\[R'=G(R)+x G'(R) R'.\] Now, given that $G$ is invertible we can write $H=F'G$ for some $F\in\mathbb{C}[[x]]$. Thus we have \[\displaystyle\frac{F'(R)G(R)}{1-x\;G'(R)}=\displaystyle\frac{dF(R)}
{dx}=\sum_{n\geq1}x^{n-1}[t^{n-1}]F'(t)G(t)^n=\sum_{n\geq0}x^{n}[t^{n}]F'(t)G(t)^{n+1},\] and the result follows. Note that the second equality above is through LIFT. \end{proof}

\begin{center}
    $\sim$
\end{center}

\end{document}